\documentclass[12pt]{article}

\usepackage{a4wide,latexsym,amsfonts,amssymb,exscale,enumerate,amscd}
\usepackage{amsmath}
\usepackage{amsthm}

\usepackage{pstricks}
\usepackage[tiling]{pst-fill}

\psset{linewidth=0.3pt,dimen=middle} \psset{xunit=.70cm,yunit=0.70cm}
\psset{arrowsize=1pt 5,arrowlength=0.6,arrowinset=0.7}

\usepackage[all]{xy}
\SelectTips{cm}{}

\usepackage{graphicx}

\newcommand{\seq}{{\rm Seq}}
\newcommand{\seqd}{{\rm Seqd}}
\newcommand{\sseq}{{\rm SSeq}}
\newcommand{\sseqd}{{\rm SSeqd}}

\newcommand{\Pol}{\cal{P}o\ell}
\newcommand{\ii}{ \textbf{\textit{i}}}
\newcommand{\jj}{ \textbf{\textit{j}}}

\newcommand{\onel}{{\mathbf 1}_{\lambda}}
\newcommand{\onelp}{{\mathbf 1}_{\lambda'}}

\newcommand{\Eol}{\cal{E}_i{\mathbf 1}_{\lambda}}

\newcommand{\GrG}{\cat{EqFlag}_{N}^{*}}
\newcommand{\HG}{H^G}
\newcommand{\uk}{\underline{k}}
\newcommand{\ukep}{{}_{+i}\uk}
\newcommand{\ukem}{{}_{-i}\uk}
\newcommand{\ukp}{\uk^{+i}}
\newcommand{\ukm}{\uk^{-i}}
\newcommand{\Gr}{\cat{Flag}_{N}^{*}}

\newcommand{\refequal}[1]{\xy {\ar@{=}^{#1}
(-1,0)*{};(1,0)*{}};
\endxy}

\newcommand{\Uq}{{\bf U}_q(\mathfrak{sl}_n)}
\newcommand{\U}{\dot{{\bf U}}}
\newcommand{\UA}{{_{\cal{A}}\dot{{\bf U}}}}
\newcommand{\Upr}{{}'{\bf U}}
\newcommand{\Ucat}{\cal{U}}
\newcommand{\Ucats}{\Ucat_{\scriptscriptstyle \rightarrow}}
\newcommand{\Ucatq}{\cal{U}^{\ast}}
\newcommand{\UcatD}{\dot{\cal{U}}}

\newcommand{\qbin}[2]{
\left[
 \begin{array}{c}
 #1 \\
 #2 \\
 \end{array}
 \right]
}

\newcommand{\xsum}[2]{
  \vcenter{\xy
  (0,.4)*{\sum};
  (0,3.8)*{\scs #2};
  (0,-3.2)*{\scs #1};
  \endxy}
}

\newcommand{\Uup}{\xy {\ar (0,-3)*{};(0,3)*{} };(0,0)*{\bullet};(2,0)*{};(-2,0)*{};\endxy}
\newcommand{\Udown}{\xy {\ar (0,3)*{};(0,-3)*{} };(0,0)*{\bullet};(2,0)*{};(-2,0)*{};\endxy}
\newcommand{\Ucupr}{\xy (-2,1)*{}; (2,1)*{} **\crv{(-2,-3) & (2,-3)} ?(1)*\dir{>}; \endxy}
\newcommand{\Ucupl}{\xy (2,1)*{}; (-2,1)*{} **\crv{(2,-3) & (-2,-3)}?(1)*\dir{>};\endxy}
\newcommand{\Ucapr}{\xy (-2,-1)*{}; (2,-1)*{} **\crv{(-2,3) & (2,3)}?(1)*\dir{>};\endxy\;\;}
\newcommand{\Ucapl}{\xy (2,-1)*{}; (-2,-1)*{} **\crv{(2,3) &(-2,3) }?(1)*\dir{>};\endxy\;}
\newcommand{\Ucross}{\xy {\ar (2.5,-2.5)*{};(-2.5,2.5)*{}}; {\ar (-2.5,-2.5)*{};(2.5,2.5)*{} };
(4,0)*{};(-4,0)*{};\endxy}
\newcommand{\Ucrossd}{\xy {\ar (2.5,2.5)*{};(-2.5,-2.5)*{}}; {\ar (-2.5,2.5)*{};(2.5,-2.5)*{} };
(4,0)*{};(-4,0)*{};\endxy}

\usepackage{fancyheadings}
\newcommand{\cat}[1]{\ensuremath{\mbox{\bfseries {\upshape {#1}}}}}

\newcommand{\BOX}{\hbox {$\sqcap$ \kern -1em $\sqcup$}}

\newcommand{\To}{\Rightarrow}

\newcommand{\Hom}{{\rm Hom}}
\newcommand{\HOM}{{\rm HOM}}
\newcommand{\HOMU}{{\rm HOM_{\Ucat}}}
\renewcommand{\to}{\rightarrow}
\newcommand{\maps}{\colon}
\newcommand{\op}{{\rm op}}
\newcommand{\co}{{\rm co}}

\newcommand{\End}{{\rm End}}
\newcommand{\ENDU}{{\rm END_{\Ucat}}}

\newcommand{\la}{\langle}
\newcommand{\ra}{\rangle}
\newcommand{\sla}{\langle}
\newcommand{\sra}{\rangle}

\newcommand{\scs}{\scriptstyle}

\theoremstyle{definition}
\newtheorem{thm}{Theorem}[section]
\newtheorem{cor}[thm]{Corollary}

\newtheorem{lem}[thm]{Lemma}
\newtheorem{rem}[thm]{Remark}
\newtheorem{prop}[thm]{Proposition}
\newtheorem{defn}[thm]{Definition}

        \newcommand{\be}{\begin{equation}}
        \newcommand{\ee}{\end{equation}}
        \newcommand{\ba}{\begin{eqnarray}}
        \newcommand{\ea}{\end{eqnarray}}
        \newcommand{\ban}{\begin{eqnarray*}}
        \newcommand{\ean}{\end{eqnarray*}}
        \newcommand{\barr}{\begin{array}}
        \newcommand{\earr}{\end{array}}

\numberwithin{equation}{section}

\def\emph#1{{\sl #1\/}}
\def\ie{{\sl i.e.\/}}

\let\hat=\widehat
\let\tilde=\widetilde

\let\phi=\varphi
\let\epsilon=\varepsilon

\usepackage{bbm}
\def\C{{\mathbbm C}}
\def\N{{\mathbbm N}}
\def\R{{\mathbbm R}}
\def\Z{{\mathbbm Z}}
\def\Q{{\mathbbm Q}}

\def\cal#1{\mathcal{#1}}%
\def\1{\mathbbm{1}}%
\def\nn{\notag}

\def\lra{{\longrightarrow}}
\def\gdim{{\mathrm{gdim}}}

\def\Id{\mathrm{Id}}
\def\mc{\mathcal}
\def\mf{\mathfrak}
\def\Af{{_{\mc{A}}\mathbf{f}}}    
\def\shuffle{\,\raise 1pt\hbox{$\scriptscriptstyle\cup{\mskip
               -4mu}\cup$}\,}

\renewcommand{\sup}[1]{\xybox{
   (-3,-7)*{};
  (3,6)*{};
 (0,0)*{\includegraphics[scale=0.5]{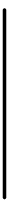}};
 (-.1,-7)*{\scs #1};
 }}

 \newcommand{\supdot}[1]{\xybox{
   (-3,-7)*{};
  (3,6)*{};
 (0,0)*{\includegraphics[scale=0.5]{short_up.eps}};
 (-.1,-7)*{\scs #1}; (0,0)*{\bullet};
 }}

\newcommand{\dcross}[2]{\xybox{
 (-6,-7)*{};
 (6,6)*{};
 (0,0)*{\includegraphics[scale=0.5]{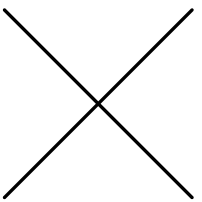}};
 (-5.1,-7)*{\scs #1};
 (4.7,-7)*{\scs #2};
}}

\newcommand{\lineu}[1]{\xybox{%
  (-2,0)*{};
  (2,0)*{};
  (0,0)*{}; (0,-18)*{} **\dir{-}; ?(.5)*\dir{<}+(1.7,-7)*{\scs #1};
}}
\newcommand{\lined}[1]{\xybox{%
  (-2,0)*{};
  (2,0)*{};
  (0,0)*{}; (0,-18)*{} **\dir{-}; ?(.5)*\dir{>}+(1.7,-7)*{\scs #1};
}}

\newcommand{\twoIu}[1]{\xybox{%
  (-6,0)*{};
  (6,0)*{};
  (0,6)*{}="f'";
  (0,12)*{}="f";
  (-3,18)*{}="t1";
  (3,18)*{}="t2";
    (-3,0)*{}="t1'";
  (3,0)*{}="t2'";
  (0,0)*{}="b";
  "t1";"f" **\crv{(-3,14)};?(.1)*\dir{<};
  "t2";"f" **\crv{(3,14)};?(.1)*\dir{<};
  "f"+(.5,0);"f'"+(.5,0) **\dir{-};
  "f"+(-.5,0);"f'"+(-.5,0) **\dir{-};
  "t1'";"f'" **\crv{(-3,4)};
  "t2'";"f'" **\crv{(3,4)} ?(.15)*\dir{ }+(2,0)*{\scs #1};
}}

\newcommand{\lowrru}[1]{\xybox{%
  (-8,0)*{};
  (8,0)*{};
  (-6,-18)*{};(6,-9)*{} **\crv{(-6,-13) & (6,-15)} ?(1)*\dir{>};
  (6,-9)*{};(6,0)*{}  **\dir{-} ?(.3)*\dir{ }+(2,0)*{\scs {\bf j}};
}}

\newcommand{\lowllu}[1]{\xybox{%
  (-8,0)*{};
  (8,0)*{};
  (6,-18)*{};(-6,-9)*{} **\crv{(6,-13) & (-6,-15)} ?(1)*\dir{>};
  (-6,-9)*{};(-6,0)*{}  **\dir{-} ?(.3)*\dir{ }+(-2,0)*{\scs {\bf j}};
}}

\newcommand{\bbe}[1]{\xybox{%
  (-2,0)*{};
  (2,0)*{};
  (0,0);(0,-18) **\dir{-}; ?(.5)*\dir{<}+(2.3,0)*{\scriptstyle{#1}};
}}

\newcommand{\bbsid}{\xybox{%
  (-2,0)*{};
  (2,0)*{};
  (0,10);(0,4) **\dir{-};
}}
\newcommand{\bbpef}[1]{\xybox{%
  (-6,0)*{};
  (6,0)*{};
  (-4,0)*{}="t1";
  (4,0)*{}="t2";
  "t1";"t2" **\crv{(-4,-6) & (4,-6)}; ?(.15)*\dir{>} ?(.9)*\dir{>}
   ?(.5)*\dir{}+(0,-2)*{\scriptstyle{#1}};
}}
\newcommand{\bbpfe}[1]{\xybox{%
  (-6,0)*{};
  (6,0)*{};
  (-4,0)*{}="t1";
  (4,0)*{}="t2";
  "t2";"t1" **\crv{(4,-6) & (-4,-6)}; ?(.15)*\dir{>} ?(.9)*\dir{>}
  ?(.5)*\dir{}+(0,-2)*{\scriptstyle{#1}};
}}

\newcommand{\bbcfe}[1]{\xybox{%
  (-6,0)*{};
  (6,0)*{};
  (-4,0)*{}="t1";
  (4,0)*{}="t2";
  "t1";"t2" **\crv{(-4,6) & (4,6)}; ?(.15)*\dir{>} ?(.9)*\dir{>}
  ?(.5)*\dir{}+(0,2)*{\scriptstyle{#1}};
}}
\newcommand{\bbcef}[1]{\xybox{%
  (-6,0)*{};
  (6,0)*{};
  (-4,0)*{}="t1";
  (4,0)*{}="t2";
  "t2";"t1" **\crv{(4,6) & (-4,6)}; ?(.15)*\dir{>}
  ?(.9)*\dir{>} ?(.5)*\dir{}+(0,2)*{\scriptstyle{#1}};
}}

\newcommand{\ccbub}[2]{
\xybox{%
 (-6,0)*{};
  (6,0)*{};
  (-4,0)*{}="t1";
  (4,0)*{}="t2";
  "t2";"t1" **\crv{(4,6) & (-4,6)}; ?(.7)*\dir{}+(-2,0)*{\scs #2}
  ?(.05)*\dir{>} ?(1)*\dir{>};
  "t2";"t1" **\crv{(4,-6) & (-4,-6)};
   ?(.3)*\dir{}+(0,0)*{\bullet}+(0,-3)*{\scs {#1}};
}}
\newcommand{\cbub}[2]{
\xybox{%
 (-6,0)*{};
  (6,0)*{};
  (-4,0)*{}="t1";
  (4,0)*{}="t2";
  "t2";"t1" **\crv{(4,6) & (-4,6)};?(.7)*\dir{}+(-2,0)*{\scs #2};
   ?(0)*\dir{<} ?(.95)*\dir{<};
  "t2";"t1" **\crv{(4,-6) & (-4,-6)};
   ?(.3)*\dir{}+(0,0)*{\bullet}+(0,-3)*{\scs {#1}};
}}

\newcommand{\bbdl}[1]{\xybox{%
  (2,0);(0,-8) **\crv{(2,-2)&(0,-6)}; ?(.5)*\dir{>}
}}
\newcommand{\bbdlu}[1]{\xybox{%
  (2,0);(0,-8) **\crv{(2,-2)&(0,-6)}; ?(.5)*\dir{<}
}}
\newcommand{\bbdr}[1]{\xybox{%
  (-2,0);(0,-8) **\crv{(-2,-2)&(0,-6)}; ?(.5)*\dir{>}
}}
\newcommand{\bbdru}[1]{\xybox{%
  (-2,0);(0,-8) **\crv{(-2,-2)&(0,-6)}; ?(.5)*\dir{<}
}}

\newgray{whitegray}{.9}

\newcommand{\chern}[1]{
    \rput(0,0){\psframebox[framearc=.4,fillstyle=solid, linewidth=.8pt]{\small $\scriptstyle #1$}}
}

\newcommand{\dchern}[1]{
\rput(0,0){\psframebox[framearc=.4,fillstyle=solid,
fillcolor=whitegray,linewidth=2pt]{\small $\scriptstyle #1$}}
}

\newcommand{\Eline}[1]{
  \psline[linewidth=1pt](0,.3)(0,1.8)
  \psline[linewidth=1pt]{->}(0,.3)(0,1.05)
  \rput(0,0){$\scs #1$}
}
\newcommand{\Elinedot}[2]{
  \psline[linewidth=1pt](0,.3)(0,1.8)
  \psline[linewidth=1pt]{->}(0,.3)(0,1.05)
  \rput(0,0){$\scs #1$}
  \psdot[linewidth=1.5pt](0,1.3)
   \rput(0.4,1.4){$\scriptstyle #2$}
}
\newcommand{\Fline}[1]{
\psline[linewidth=1pt](0,.3)(0,1.8)
  \psline[linewidth=1pt]{->}(0,1.8)(0,1.05)
  \rput(0,0){$\scs #1$}
}

\newcommand{\Flinedot}[2]{
  \psline[linewidth=1pt](0,.3)(0,1.8)
  \psline[linewidth=1pt]{->}(0,1.8)(0,.95)
  \rput(0,0){$\scs #1$}
  \psdot[linewidth=1.5pt](0,1.3)
  \rput(0.4,1.4){$\scriptstyle #2$}
}

\title{A diagrammatic approach to categorification of quantum groups III}
      \author{ Mikhail Khovanov and Aaron D. Lauda}

\begin{document}
%
\date{July 22, 2008}
\maketitle

\begin{abstract}
\begin{center}
We categorify the idempotented form of quantum sl(n).
\end{center}
\end{abstract}

\setcounter{tocdepth}{2} \tableofcontents

\section{Introduction}
%

In this paper we categorify the Beilinson-Lusztig-Macpherson idempotented
modification $\U(\mf{sl}_n)$ of ${\bf U}_q(\mf{sl}_n)$ for any $n$,  generalizing
\cite{Lau1,Lau2} where such categorification was described for $n=2$ and using
constructions and results of~\cite{KL,KL2} which contain a categorification of
${\bf U}^-$ for any Cartan datum.

In~\cite{Lus4} Lusztig associates a quantum group $\mathbf{U}$ to any root datum;
the latter consists of a perfect pairing $\langle,\rangle$ between two free
abelian groups $X$ and $Y$, embeddings of the set $I$ of simple roots into $X,
Y$, and a bilinear form on $\Z[I]$ subject to certain compatibility and
integrality conditions. Lusztig's definition is slightly different from the
original ones due to Drinfeld~\cite{Drinfeld} and Jimbo~\cite{Jimbo}. Lusztig
then modifies $\mathbf{U}$ to the nonunital ring $\U$ which contains a system of
idempotents $\{1_{\lambda}\}$ over all weights $\lambda\in X$ as a substitute for
the unit element,
\[\U = \bigoplus_{\lambda,\mu\in X} 1_{\mu}\U 1_{\lambda}.
\]
In the $\mf{sl}_n$ case $\U$ was originally defined by Beilinson, Lusztig, and
Macpherson~\cite{BLM} and then appeared in \cite{Lus7,Kas3} in greater
generality. It is clear from Lusztig's work that the $\Q(q)$-algebra $\U$ is
natural for at least the following reasons:
\begin{enumerate}
\item A $\U$-module is the same as  a $\mathbf{U}$-module which has an  integral
weight decomposition. These modules are of prime importance in the representation
theory of $\mathbf{U}$.
\item $\U$ has analogues of the comultiplication, the antipode, and other standard
 symmetries of $\mathbf{U}$.
\item $\U$ is a $\mathbf{U}$-bimodule.
\item The Peter-Weyl theorem and the theory of cells can be intrinsically stated
in terms of the algebra $\U$.
\item $\U$ has an integral form $\UA$, a $\Z[q,q^{-1}]$-lattice closed
under multiplication and comultiplication. The integral form comes with a
canonical basis $\dot{\bf B}$. Conjecturally, multiplication and comultiplication
in this basis have coefficients in $\N[q,q^{-1}]$ when the Cartan datum is
symmetric.
\item The braid group associated to the Cartan datum acts on $\U$.
\end{enumerate}

Moreover, $\U$ appears throughout the categorification program for quantum
groups. Representations of quantum groups that are known to have
categorifications all have integral weight decompositions, and thus automatically
extend to representations of $\U$. In most or all of these examples,
see~\cite{BFK,CR,FKS,Sussan,Zheng2}, the weight decomposition of representations
lifts to a direct sum decomposition of categories, so one obtains a
categorification of the idempotent $1_{\lambda}$ as the functor of projection
onto the corresponding direct summand (the only possible exception is the
categorification of tensor products via the affine Grassmannian \cite{CK1,CK2}).
 In the categorification of tensor
powers of the fundamental ${\bf U}_q(\mf{sl}_2)$-representation~\cite{BFK,FKS},
each canonical basis element of $\U$ acts as an indecomposable projective functor
or as the zero functor.  The idea that $\U$ rather than $\mathbf{U}$ should be
categorified goes back to Crane and Frenkel~\cite{CF}.

$\U$ is generated by elements $E_i1_{\lambda}$, $F_i1_{\lambda}$, and
$1_{\lambda}$, where $\lambda \in X$ is an element of the weight lattice and $i$
is a simple root. We will often write $E_{+i}$ instead of $E_i$ and $E_{-i}$
instead of $F_i$. We have
\[
 E_{\pm i}1_{\lambda} = 1_{\mu}E_{\pm i}1_{\lambda},
\]
where, in our notations, explained in Section~\ref{subsec_quantumgroups}, $\mu =
\lambda \pm i_X$, and $i_X$ is the element of $X$ associated to the simple root
$i$.  Algebra $\U$ is spanned by products
\[ E_{\ii} 1_{\lambda} := E_{\pm i_1} E_{\pm i_2} \dots  E_{\pm i_m}
 1_{\lambda}=1_{\mu}E_{\pm i_1} E_{\pm i_2} \dots E_{\pm i_m} 1_{\lambda},
 \]
where $\ii=(\pm i_1, \dots, \pm i_m)$ is a signed sequence of simple roots, and
$\mu=\lambda+\ii_X$.

The integral form $\UA \subset \U$ is the $\Z[q,q^{-1}]$-algebra generated by
divided powers \[ E_{i^{(a)}1_{\lambda}} = \frac{1}{[a]_i!}E_i^{a}1_{\lambda}.
\]

\vspace{0.1in}

Note that $\U$ can, alternatively, be viewed as a pre-additive category with
objects $\lambda\in X$ and morphisms from $\lambda$ to $\mu$ being
$1_{\mu}\U1_{\lambda}$. Of course, any ring with a collection of
mutually-orthogonal idempotents as a substitute for the unit element can be
viewed as a pre-additive category and vice versa. From this perspective, though,
we can expect the categorification of $\U$ to be a 2-category.

\vspace{0.1in}

In Section~\ref{subsec_Uprime} we associate a 2-category $\Ucat$ to a root datum.
The objects of this 2-category are integral weights $\lambda\in X$, the morphisms
from $\lambda$ to $\mu$ are finite formal sums of symbols
$\mc{E}_{\ii}\onel\{t\}$, where $\ii=(\pm i_1, \dots, \pm i_m)$ is a signed
sequence of simple roots such that the left weight of the symbol is $\mu$
($E_{\ii}1_{\lambda}=1_{\mu}E_{\ii}1_{\lambda}$), and $t \in \Z$ is a grading
shift. When $\ii$ consists of a single term, we get 1-morphisms
$\mc{E}_{+i}\onel$ and $\mc{E}_{-i}\onel$, which should be thought of as
categorifying elements $E_i1_{\lambda}$ and $F_i1_{\lambda}$ of $\U$,
respectively. Grading shift $\{t\}$ categorifies multiplication by $q^t$. The
one-morphism $ \mc{E}_{\ii} {\mathbf 1}_{\lambda}: \lambda\lra \mu $ should be
thought of as a categorification of the element $E_{\ii}1_{\lambda}$. When the
sequence $\ii$ is empty, we get the identity morphism $\mathbf{1}_{\lambda}:
\lambda \lra \lambda$, a categorification of the element $1_{\lambda}$.

Two-morphisms between $\mc{E}_{\ii}{\mathbf 1}_{\lambda}\{t\}$ and
$\mc{E}_{\jj}{\mathbf 1}_{\lambda}\{t'\}$ are given by linear combinations of
degree $t-t'$ diagrams drawn on the strip $\R\times [0,1]$ of the plane. The
diagrams
 consist of immersed oriented one-manifolds,
with every component labelled by a simple root, and dots placed on the
components. Labels and orientations at the lower and upper endpoints of the
one-manifold must match the sequences $\ii$ and $\jj$, respectively. Integral
weights label regions of the plane cut out by the one-manifold, with the
rightmost region labelled $\lambda$. Each diagram has an integer degree assigned
to it. We work over a ground field $\Bbbk$, and define a two-morphism between
$\mc{E}_{\ii}{\mathbf 1}_{\lambda}\{t\}$ and $\mc{E}_{\jj} {\mathbf
1}_{\lambda}\{t'\}$ as a linear combination of such diagrams of degree $t-t'$,
with coefficients in $\Bbbk$, modulo isotopies and a collection of very carefully
chosen local relations. The set of 2-morphisms $\Ucat( \mc{E}_{\ii}\onel\{t\},
\mc{E}_{\jj} \onel \{t'\})$ is a $\Bbbk$-vector space.  We also form graded
vector space
\begin{equation}
  \HOM_{\Ucat}(\mc{E}_{\ii}\onel, \mc{E}_{\jj} \onel) := \bigoplus_{t\in \Z}
  \Ucat(\mc{E}_{\ii}\onel\{t\},\mc{E}_{\jj} \onel ).
\end{equation}
Vertical composition of 2-morphisms is given by concatenation of diagrams,
horizontal composition consists of placing diagrams next to each other.

In each graded $\Bbbk$-vector space $\HOM_{\Ucat}( \mc{E}_{\ii}{\mathbf
1}_{\lambda}, \mc{E}_{\jj} \onel)$ we construct a homogeneous spanning set
$B_{\ii,\jj,\lambda}$ which depends on extra choices. The Laurent power series in
$q$, with the coefficient at $q^r$ equal to the number of spanning set elements
of degree $r$, is proportional to suitably normalized inner product $ \langle
E_{\ii}1_{\lambda}, E_{\jj}1_{\lambda}\rangle$, where the semilinear form
$\langle,\rangle $ is a mild modification of the Lusztig bilinear form on $\U$.
The proportionality coefficient $\pi$ depends only on the root datum.

We say that our graphical calculus is \emph{nondegenerate} for a given root datum
and field $\Bbbk$ if for each $\ii, \jj$ and $\lambda$ the homogeneous spanning
set $B_{\ii,\jj,\lambda}$ is a basis of the $\Bbbk$-vector space $\HOMU(
\mc{E}_{\ii}{\mathbf 1}_{\lambda}, \mc{E}_{\jj} {\mathbf 1}_{\lambda})$.
Nondegeneracy will be crucial for our categorification constructions.

\vspace{0.1in}

The 2-category $\Ucat$ is $\Bbbk$-additive, and we form its Karoubian envelope
$\UcatD$, the smallest 2-category which contains $\Ucat$ and has splitting
idempotents. Namely, for each $\lambda, \mu \in X$, the category
$\UcatD(\lambda,\mu)$ of morphisms $\lambda\lra \mu$ is defined as the Karoubian
envelope of the additive $\Bbbk$-linear category $\Ucat(\lambda,\mu)$. The split
Grothendieck category $K_0(\UcatD)$ is a pre-additive category with objects
$\lambda$, and the abelian group of morphisms from $\lambda$ to $\mu$ is the
split Grothendieck group $K_0(\UcatD(\lambda,\mu))$ of the additive category
$\UcatD(\lambda,\mu))$. The grading shift functor on $\UcatD(\lambda,\mu)$ turns
$K_0(\UcatD(\lambda,\mu))$ into a $\Z[q,q^{-1}]$-module. This module is free with
the basis given by isomorphism classes of indecomposable objects of $
\UcatD(\lambda,\mu)$, up to grading shifts. The split Grothendieck category
$K_0(\UcatD)$ can also be viewed as a nonunital $\Z[q,q^{-1}]$-algebra with a
collection of idempotents $[{\mathbf 1}_{\lambda}]$ as a substitute for the unit
element.

In Section~\ref{subsec_KzeroU} we set up a $\Z[q,q^{-1}]$-algebra homomorphism
$$ \gamma \ : \  \  \UA \lra K_0(\UcatD)$$
which takes $1_{\lambda}$ to $[{\mathbf 1}_{\lambda}]$ and $E_{\ii}{\mathbf
1}_{\lambda}$ to $[\mc{E}_{\ii}{\mathbf 1}_{\lambda}]$, for any ``divided power''
signed sequence $\ii$.

The main results of this paper are the following theorems.

\begin{thm} \label{thm-surjective} The map $\gamma$ is surjective for any
root datum and field $\Bbbk$.
\end{thm}

\begin{thm} \label{thm-injective} The map $\gamma$ is injective  if the graphical
calculus for the root datum and field $\Bbbk$ is nondegenerate.
\end{thm}

\begin{thm} \label{thm-nondegenerate}
The graphical calculus is nondegenerate for the root datum of $\mf{sl}_n$ and any
field $\Bbbk$.
\end{thm}

The three theorems together immediately imply

\begin{prop} \label{prop-iso} The map $\gamma$ is an isomorphism for
the root datum of $\mf{sl}_n$ and any field $\Bbbk$.
\end{prop}

The last result establishes a canonical isomorphism
\[
\UA(\mf{sl}_n) \cong K_0\big(\UcatD(\mf{sl}_n)\big)
\]
and allows us to view $\UcatD(\mf{sl}_n)$ as a categorification of
$\U(\mf{sl}_n)$.

\vspace{0.2in}

Theorem~\ref{thm-surjective}, proved in Section~\ref{subsec_surjectivity},
follows from the results of~\cite{KL,KL2,Lau1} and basic properties of
Grothendieck groups and idempotents. Theorem~\ref{thm-injective}, proved in
Section~\ref{subsec_injectivity}, follows from the nondegeneracy of the
semilinear form on $\U$ and its pictorial interpretation explained in
Section~\ref{subsec_geometric}. To prove theorem~\ref{thm-nondegenerate} we
construct a family of 2-representations of $\UcatD$ and check that the elements
of each spanning set $B_{\ii,\jj,\lambda}$ act linearly independently on vector
spaces in these 2-representations, implying nondegeneracy of the graphical
calculus. Sections~\ref{sec_Usln}--\ref{sec_rep} are devoted to these
constructions.

\vspace{0.2in}

Indecomposable one-morphisms, up to isomorphism and grading shifts, constitute a
basis of $K_0(\UcatD(\mf{sl}_n))\cong \UA(\mf{sl}_n)$, which might potentially
depend on the ground field $\Bbbk$. The multiplication in this basis has
coefficients in $\N[q,q^{-1}]$. It is an open problem whether this basis
coincides with the Lusztig canonical basis of $\UA(\mf{sl}_n)$. The answer is
positive when $n=2$, see~\cite{Lau1}.

Another major problem is to determine for which root data the graphical calculus
is nondegenerate. Nondegeneracy immediately implies, via
theorems~\ref{thm-surjective} and~\ref{thm-injective}, that $\UcatD$ categorifies
$\U$ for a given root datum.

We believe that $\UcatD$ will prove ubiquitous in representation theory. This
2-category or its mild modifications is expected to act on parabolic-singular
blocks of highest weight categories for $\mf{sl}_N$ in the context of
categorification of $\mf{sl}_n$ representations~\cite{CR,FKS,Sussan}, on derived
categories of coherent sheaves on Kronheimer-Nakajima~\cite{KN} and
Nakajima~\cite{Na} quiver varieties and on their Fukaya-Floer counterparts, on
categories of modules over cyclotomic Hecke and degenerate Hecke
algebras~\cite{Ari1,KleBook}, on categories of perverse sheaves in Zheng's
categorifications of tensor products~\cite{Zheng2}, on categories of modules over
cyclotomic quotients of rings $R(\nu)$ in~\cite{KL,KL2}, on categories of matrix
factorizations that appear in~\cite[Section 11]{KR}, et cetera. A possible
approach to proving that the calculus is nondegenerate for other root systems is
to show that $\UcatD$ acts on a sufficiently large 2-category and verify  that
the spanning set elements act linearly independently.  It would also be
interesting to relate our constructions with those of Rouquier~\cite{Rou1,Rou2}.

Categories of projective modules over rings $R(\nu)$, defined in~\cite{KL,KL2},
categorify ${\bf U}^-$ weight spaces. A subset of our defining local relations on
two-morphisms gives the relations for rings $R(\nu)$.  This subset consists
exactly of the relations whose diagrams have no critical points (U-turns) on
strands and have all strand orientations going in the same direction. In other
words, the relations on braid-like diagrams allow us to categorify ${\bf U}^-$,
while the relations without these restrictions lead to a categorification of the
entire $\U$, at least in the $\mf{sl}_n$ case. Informally, the passage from a
categorification of ${\bf U}^-$ to a categorification of $\U$ is analogous to
generalizing from braids to tangles.

\vspace{0.2in}

{\bf Acknowledgments.} M.K. was partially supported by the NSF grant DMS-0706924
and, during the early stages of this work, by the Institute for Advanced Study.

%
\section{Graphical Interpretation of the bilinear form}
%

%
\subsection{Quantum groups} \label{subsec_quantumgroups}
%

%
\subsubsection{Algebras ${\bf f}$ and ${\bf U}$} \label{subsubsec_algebras_f}
%

We recall several definitions, following~\cite{Lus4}. A \emph{Cartan datum}
\index{I@$(I,\cdot)$ Cartan datum}$(I,\cdot)$ consists of a finite set $I$ and a
symmetric $\Z$-valued bilinear form on $\Z[I]$, subject to conditions
\begin{itemize}
\item  $i\cdot i \in \{2, 4, 6, \dots \}$ for  $i\in I$,
\item $d_{ij}:=-2\frac{i\cdot j}{i\cdot i}\in \{0, 1, 2, \dots \}$ for any $i\not= j$ in $I$.
\end{itemize}
Let $q_i=q^{\frac{i\cdot i}{2}}$, $[a]_i=q_i^{a-1}+q_i^{a-3}+\dots + q_i^{1-a}$,
$[a]_i!=[a]_i[a-1]_i\dots [1]_i$. Denote by $'\mathbf{f}$\index{f@$'\mathbf{f}$}
the free associative algebra over $\Q(q)$ with generators
\index{t@$\theta_i$}$\theta_i$, $i\in I$, and introduce q-divided powers
$\theta_i^{(a)}= \theta_i^a/[a]_i!$. The algebra $'\mathbf{f}$ is $\N[I]$-graded,
with $\theta_i$ in degree $i$. The tensor square $'\mathbf{f}\otimes
{}'\mathbf{f}$ is an associative algebra with twisted multiplication
$$ (x_1\otimes x_2) (x_1'\otimes x_2') =q^{- |x_2|\cdot |x_1'|} x_1 x_1' \otimes x_2 x_2'$$
for homogeneous $x_1, x_2, x_1', x_2'$. The assignment $r(\theta_i) = \theta_i
\otimes 1 + 1\otimes \theta_i$ extends to a unique algebra homomorphism $r:
{}'\mathbf{f}\lra {}'\mathbf{f}\otimes {}'\mathbf{f}$.

Algebra $'\mathbf{f}$ carries a $\Q(q)$-bilinear form determined by the
conditions\footnote{Our bilinear form $(,)$ corresponds to Lusztig's bilinear
form $\{,\}$, see Lusztig~\cite[1.2.10]{Lus4}.}
\begin{itemize}
\item $(1,1)=1$,
\item $(\theta_i, \theta_j) = \delta_{i,j} (1-q_i^2)^{-1} $ for $i,j\in I$,
\item $(x, y y') = (r(x), y \otimes y')$ for $x,y, y' \in {}'\mathbf{f}$,
\item $(x x', y) = (x \otimes x', r(y))$ for $x, x', y\in {}'\mathbf{f}$.
\end{itemize}
The bilinear form $(,)$ is symmetric. Its radical $\mf{I}$ is a two-sided ideal
of $'\mathbf{f}$. The form $(,)$ descends to a nondegenerate form on the
associative $\Q(q)$-algebra $\mathbf{f} = {}'\mathbf{f}/\mf{I}$.

\begin{thm} The ideal $\mf{I}$ is generated by the elements
$$ \sum_{a+b=d_{ij}+1} (-1)^a \theta_i^{(a)} \theta_j \theta_i^{(b)} $$
over all $i,j\in I, $ $i\not= j$.
\end{thm}
It seems that the only known proof of this theorem, for a general Cartan datum,
requires Lusztig's geometric realization of $\mathbf{f}$ via perverse sheaves.
This proof is given in his book~\cite[Theorem~33.1.3]{Lus4}. Less sophisticated
proofs exist when the Cartan datum is finite.

We see that $\mathbf{f}$ is the quotient of $'\mathbf{f}$ by the quantum Serre
relations
\begin{equation}\label{rels-serre}
 \sum_{a+b=d_{ij}+1} (-1)^a \theta_i^{(a)} \theta_j \theta_i^{(b)} =0.
\end{equation}

Let $\Af$ be the $\Z[q,q^{-1}]$-subalgebra of $\mathbf{f}$ generated by the
divided powers $\theta_i^{(a)}$, over all $i\in I$ and $a\in \N$.

A root datum of type $(I,\cdot)$ consists of
\begin{itemize}
  \item free finitely generated abelian groups $X$,$Y$ and a perfect
  pairing $\langle,\rangle \maps Y \times X \to \Z$;
  \item inclusions $I \subset X $ $\;(i \mapsto i_X)$ and $I \subset
  Y$ $\; (i \mapsto i)$ such that
  \item $\langle i, j_X \rangle = 2\frac{i \cdot j}{i \cdot i}= -d_{ij}$ for all $i,j \in
  I$.
\end{itemize}
This implies $\la i,i_X \ra =2$ for all $i$.  We write $i_X$, rather than
Lusztig's $i'$, to denote the image of $i$ in $X$.

The quantum group ${\bf U}$ associated to a root datum as above is the unital
associative $\Q(q)$-algebra given by generators $E_i$, $F_i$, $K_{\mu}$ for $i
\in I$ and $\mu \in Y$, subject to the relations:
\begin{center}
\begin{enumerate}[i)]
 \item $K_0=1$, $K_{\mu}K_{\mu'}=K_{\mu+\mu'}$ for all $\mu,\mu' \in Y$,
 \item $K_{\mu}E_i = q^{\la \mu,i_X\ra}E_iK_{\mu}$ for all $i \in I$, $\mu \in
 Y$,
 \item $K_{\mu}F_i = q^{-\la \mu, i_X\ra}F_iK_{\mu}$ for all $i \in I$, $\mu \in
 Y$,
 \item $E_iF_j - F_jE_i = \delta_{ij}
 \frac{\tilde{K}_i-\tilde{K}_{-i}}{q_i-q_i^{-1}}$, where
 $\tilde{K}_{\pm i}=K_{\pm (i\cdot i/2) i}$,
 \item For all $i\neq j$ $$\sum_{a+b=-\la i, j_X \ra +1}(-1)^{a} E_i^{(a)}E_jE_i^{(b)} = 0
 \qquad {\rm and} \qquad
 \sum_{a+b=-\la i, j_X \ra +1}(-1)^{a} F_i^{(a)}F_jF_i^{(b)} = 0 .$$
\end{enumerate} \end{center}

If  $f(\theta_i) \in \mf{I}$ for a polynomial $f$ in noncommutative variables
$\theta_i$, $i \in I$, then $f(E_i)=0$ and $f(F_i)=0$ in ${\bf U}$. This gives a
pair of injective algebra homomorphisms $\mathbf{f} \to {\bf U}$.

%
\subsubsection{Some automorphisms of ${\bf U}$}
%

Let $\;\bar{}\;$ be the $\Q$-linear involution of $\Q(q)$ which maps $q$ to
$q^{-1}$.  ${\bf U}$ has the following standard algebra (anti)automorphisms.
\begin{itemize} \index{$\psi$}
  \item The $\Q(q)$-antilinear algebra involution  $\psi \maps {\bf U} \to {\bf U}$
given by
\[
 \psi(E_i)=E_i, \;\; \psi(F_i)=F_i,\;\;  \psi(K_{\mu}) = K_{-\mu}, \;\;
 \psi(fx)=\bar{f}\psi(x) \quad \text{for $f \in \Q(q)$ and $x \in {\bf U}$}.\]

 \item The $\Q(q)$-linear algebra involution $\omega\maps {\bf U} \to {\bf U}$
 given by \index{$\omega$}
\begin{eqnarray*}
\omega(E_i)=F_i, && \omega(F_i)=E_i, \qquad \omega(K_{\mu}) = K_{-\mu}.
\end{eqnarray*}

 \item  The $\Q(q)$-linear algebra antiinvolution $\sigma \maps {\bf U} \to {\bf U}^{\op}$
 given by \index{$\sigma$}
\begin{eqnarray*}
\sigma(E_i)=E_i, && \sigma(F_i)=F_i, \qquad \sigma(K_{\mu}) = K_{\mu}.
\end{eqnarray*}

 \item \index{$\rho$}The $\Q(q)$-linear algebra antiinvolution $\rho \maps {\bf U} \to {\bf U}^{\op}$ given by
\begin{eqnarray*}
\rho(E_i)=q_i\tilde{K}_iF_i, && \rho(F_i)=q_i\tilde{K}_{-i}E_i, \qquad \rho(K_{\mu}) = K_{\mu}.
\end{eqnarray*}
We denote by $\overline{\rho}$ the $\Q(q)$-linear antiinvolution $\psi \rho \psi
\maps {\bf U} \to {\bf U}^{\op}$. \index{$\overline{\rho}$}

 \item The $\Q(q)$-antilinear antiautomorphism $\tau \maps {\bf U} \to
{\bf U}^{\op}$ defined as the composite $\tau = \psi\rho$:\index{$\tau$}
\begin{eqnarray*}
\label{eq_tau_def}
\tau(E_i)=q_i^{-1}\tilde{K}_{-i}F_i, && \tau(F_i)=q_i^{-1}\tilde{K}_{i}E_i, \quad \tau(K_{\mu})=K_{-\mu}.
\end{eqnarray*}
\end{itemize}

%
\subsubsection{$\U$ and the bilinear form}\index{U@$\U$}
%

The $\Q(q)$-algebra $\U$ is obtained from ${\bf U}$ by adjoining a collection of
orthogonal idempotents $1_{\lambda}$ for each $\lambda \in X$,
\begin{equation}
  1_{\lambda}1_{\lambda'} = \delta_{\lambda,\lambda'}1_{\lambda'},
\end{equation}
such that
\begin{equation}
  K_{\mu}1_{\lambda} = 1_{\lambda}K_{\mu} = q^{\la \mu, \lambda\ra}1_{\lambda}, \quad
 E_i1_{\lambda} = 1_{\lambda+i_X}E_i, \qquad F_i1_{\lambda} = 1_{\lambda-i_X}F_i.
\end{equation}
The algebra $\U$ decomposes as direct sum of weight spaces
\[
 \U = \bigoplus_{\lambda,\lambda' \in X}1_{\lambda'}\U1_{\lambda} .
\]
We say that $\lambda$, respectively $\lambda'$, is the right, respectively left,
weight of $x \in 1_{\lambda'}\U1_{\lambda}$. The algebra $\UA$ is the
$\Z[q,q^{-1}]$-subalgebra of $\U$ generated by products of divided powers
$E_i^{(a)}1_{\lambda}$ and $F_i^{(a)}1_{\lambda}$, and has a similar weight
decomposition
\[
 \UA = \bigoplus_{\lambda,\lambda' \in X}1_{\lambda'}(\UA)1_{\lambda} .
\]

The following identities hold in $\U$ and $\UA$:
\begin{equation}
    (E_iF_j-F_jE_i)1_{\lambda} = \delta_{i,j}[\la i, \lambda \ra]_i 1_{\lambda},
\end{equation}
\begin{equation}
  E_i^{(a)}1_{\lambda} =  1_{\lambda+ai_X}E_i^{(a)}, \qquad F_i^{(a)}1_{\lambda} =
   1_{\lambda-ai_X}F_i^{(a)},
\end{equation}
\begin{equation}
  E_i^{(a)}F_i^{(b)}1_{\lambda} = \sum_{t = 0}^{\min(a,b)}\qbin{a-b+\la
  i,\lambda\ra}{t}_i F_i^{(b-t)}E_i^{(a-t)}1_{\lambda},
\end{equation}
\begin{equation}
  F_i^{(b)}E_i^{(a)}1_{\lambda} = \sum_{t =0}^{\min(a,b)}\qbin{-a+b-\la
  i,\lambda\ra}{t}_i E_i^{(a-t)}F_i^{(b-t)}1_{\lambda}.
\end{equation}

The (anti) automorphisms $\psi$, $\omega$, $\sigma$, $\rho$ and $\tau$ all
naturally extend to $\U$ and $\UA$ if we set
\begin{equation}
    \psi(1_{\lambda}) = 1_{\lambda}, \quad \omega(1_{\lambda}) = 1_{-\lambda},
    \quad \sigma(1_{\lambda})=1_{-\lambda}, \quad \rho(1_{\lambda}) = 1_{\lambda},
    \quad \tau(1_{\lambda}) = 1_{\lambda}.\nn
\end{equation}
Taking direct sums of the induced maps on each summand
$1_{\lambda'}\U1_{\lambda}$ allows these maps to be extended to $\U$.  For
example, the antiautomorphism $\tau$ induces for each $\lambda$ and $\lambda'$ in
$X$ an isomorphism $1_{\lambda'}\U1_{\lambda} \to 1_{\lambda}\U1_{\lambda'}$.
Restricting to the $\Z[q,q^{-1}]$-subalgebra $\UA$ and taking direct sums, we
obtain an algebra antiautomorphism $\tau\maps \UA\to\UA$ such that
$\tau(1_{\lambda})=1_{\lambda}$, $\tau(1_{\lambda+i_X}E_i
1_{\lambda})=q_i^{-1-\la i,\lambda\ra}1_{\lambda}F_i1_{\lambda+i_X}$, and
$\tau(1_{\lambda}F_i1_{\lambda+i_X})=q_i^{1+\la i,\lambda\ra
}1_{\lambda+i_X}E_i1_{\lambda}$ for all $\lambda \in X$.

The following result is taken from Lusztig~\cite[26.1.1]{Lus4}, but our bilinear
form is normalized slightly differently from his.

\begin{prop}  There exists a unique pairing $( ,) \maps \U \times\U \to \Q(q)$ with the
properties:
\begin{enumerate}[(i)]
  \item $( ,)$ is bilinear, \ie,
$ ( fx,y)=f( x,y)$, $( x,fy ) =f ( x,y)$, for $f \in \Q(q)$ and $x,y \in \U$,
 \item
$ ( 1_{\lambda_1}x1_{\lambda_2}, 1_{\lambda_1'}y1_{\lambda_2'} ) =0 \quad
\text{for all $x,y \in \U$ unless $\lambda_1=\lambda_1'$ and
$\lambda_2=\lambda_2'$}$,
 \item $( ux,y)= ( x,\overline{\rho}(u)y )$ for $u\in {\bf U}$ and $x,y \in
 \U$,
 \item $(x 1_{\lambda}, x'1_{\lambda}) = (x,x')$ for all $x, x' \in \mathbf{f}\cong{\bf U}^+$ and all $\lambda$
(here $(x,x')$ is as in Section~\ref{subsubsec_algebras_f}),
  \item $( x,y ) = ( y,x)$ for all $x,y \in \U$.
\end{enumerate}
\end{prop}

\begin{defn} \label{defn_semilinear}
Define a semilinear form $\sla,\sra \maps \U \times \U \to \Q(q)$ by
\[
 \sla x,y\sra :=( \psi(x),y ) \qquad \text{for all $x,y \in \U$}.
\]
\end{defn}

\begin{prop} \label{prop_H}
The map $\sla ,\sra \maps \U \times \U \to \Q(q)$ has the following
properties
\begin{enumerate}[(i)]
 \item $\sla ,\sra$ is semilinear, \ie,
$ \sla fx,y\sra =\bar{f}\sla x,y\sra$, $\sla x,fy \sra =f\sla x,y\sra$, for $f
\in \Q(q)$ and $x,y \in \U$,
 \item
$ \big\sla 1_{\lambda_1}x1_{\lambda_2}, 1_{\lambda_1'}y1_{\lambda_2'} \big\sra=0
\quad \text{for all $x,y \in \U$ unless $\lambda_1=\lambda_1'$ and
$\lambda_2=\lambda_2'$}$,
 \item $\sla ux,y\sra=\sla x,\tau(u)y\sra$ for $u\in {\bf U}$ and $x,y \in
 \U$,
 \item $\sla x1_{\lambda}, x'1_{\lambda}\sra =  ( \psi(x),x')$
for $x, x' \in \mathbf{f}$ and $\lambda\in X$,
 \item $\sla x,y\sra  = \sla \psi(y),\psi(x)\sra$ for all $x,y \in \U$.
\end{enumerate}
\end{prop}

\begin{proof}
Immediate.
\end{proof}

\begin{prop} \label{prop_nondeg}
The bilinear form $(,)$ and the semilinear form $\sla,\sra$ are both
nondegenerate on $\U$.
\end{prop}

\begin{proof}
Nondegeneracy of $(,)$ is implicit throughout \cite[Chapter 26]{Lus4} and follows
for instance from Theorem 26.3.1 of \cite{Lus4}.
\end{proof}

The bilinear and semilinear forms restrict to pairings
\begin{eqnarray}
 (,) \maps \UA \times \UA \to \Z[q,q^{-1}], \qquad \quad
 \sla, \sra \maps \UA \times \UA \to \Z[q,q^{-1}].
\end{eqnarray}

%
\subsubsection{Signed sequences}
%

Throughout the paper we write $E_{+i}$ for $E_i$ and $E_{-i}$ for $F_i$ and need
notation for products of these elements. By a signed sequence $\ii =
(\epsilon_1i_1,\epsilon_2i_2, \dots, \epsilon_mi_m)$, where $\epsilon_1, \dots,
\epsilon_m \in \{ +,-\}$ and $i_1, \dots, i_m \in I$, we mean a finite sequence
of elements of $I$ with signs.   We may also write $\ii$ as
$\epsilon_1i_1,\epsilon_2i_2, \dots, \epsilon_mi_m$ or even
$\epsilon_1i_1\epsilon_2i_2\dots\epsilon_mi_m$.  Let $|i|= \sum_{k=1}^m
\epsilon_ki_k$, viewed as an element of $\Z[I]$. The length of $\ii$, denoted
$\parallel \ii
\parallel$, is the number of elements in the sequence $\ii$ ($m$ in the above
notations).  Define $\sseq$ to be the set of signed sequences.  Let
\begin{equation}
  E_{\ii} := E_{\epsilon_1i_1}E_{\epsilon_2i_2}\dots E_{\epsilon_mi_m} \in {\bf
  U},
\end{equation}
\begin{equation}
E_{\ii}1_{\lambda} := E_{\epsilon_1i_1}E_{\epsilon_2i_2}\dots E_{\epsilon_mi_m}
1_{\lambda} \in \U.
\end{equation}

Recall that $i_X \in X$ denotes the image of $i$ under the embedding $I \to X$.
Let
\begin{equation}
  \ii_X := (\epsilon_1i_1)_X+\dots+(\epsilon_mi_m)_X \in X.
\end{equation}
Then $E_{\ii}1_{\lambda} = 1_{\lambda+\ii_X} E_{\ii}1_{\lambda}$, so that
$\lambda$, respectively $\lambda+\ii_X$, is the right, respectively left, weight
of $E_{\ii}1_{\lambda}$.

A sequence $\ii$ is {\em positive} if all signs $\epsilon_1, \dots,
\epsilon_m=+$.  We may write a positive sequence as $(i_1, \dots, i_m)$ or even
$i_1\dots i_m$.  Denote by $\sseq^+$ the set of all positive sequences.  As in
\cite{KL}, for $\nu \in \N[I]$, $\nu = \sum \nu_i \cdot i$, denote by $\seq(\nu)$
the set of all sequences $\ii = i_1 \cdots i_m$ such that $i$ appears $\nu_i$
times in the sequence. There is an obvious bijection
\begin{equation}
  \sseq^+ \cong \coprod_{\nu \in \N[I]}\seq(\nu).
\end{equation}
A sequence is {\em negative} if all signs $\epsilon_1, \dots, \epsilon_m = -$.
For a positive sequence $\ii$ denote by $-\ii$ the corresponding negative
sequence.  For any signed sequence $\ii$ denote by $-\ii$ the sequence given by
reversing all signs in $\ii$.  We write concatenation of sequences $\ii$ and
$\jj$ as $\ii\jj$, attaching $\pm i$ to the left of $\ii$ as $\pm i \ii$, etc.

By a divided powers signed sequence (dpss for short) we mean a sequence
\begin{equation}
  \ii = (\epsilon_1 i_1^{(a_1)},\epsilon_2 i_2^{(a_2)},\dots,\epsilon_m i_m^{(a_m)})
\end{equation}
where $\epsilon$'s and $i$'s are as before and $a_1, \dots, a_m \in \{
1,2,\dots\}$. Let $|\ii|= \sum_{k=1}^m \epsilon_ka_ki_k \in \Z[I]$ and define
$\parallel \ii \parallel$, the length of $\ii$, as $\sum_{k=1}^{m}a_k \in \N$.
For a dpss $\ii$ let $\hat{\ii}$ be a signed sequence given by expanding $\ii$
\begin{equation}
  \hat{\ii} =(\epsilon_1 i_1, \dots,\epsilon_1 i_1, \epsilon_2 i_2, \dots, \epsilon_2 i_2,
  \dots , \epsilon_m i_m, \dots, \epsilon_m i_m) = ((\epsilon_1 i_1)^{a_1} (\epsilon_2
  i_2)^{a_2} \dots (\epsilon_m i_m)^{a_m}), \nn
\end{equation}
with term $\epsilon_1i_1$ repeating $a_1$ times, term $\epsilon_2i_2$ repeating
$a_2$ times, etc.

Define
\begin{equation}
  E_{\epsilon i^{(a)}} :=E_{\epsilon i}^{(a)} = \frac{E_{\epsilon i}^a}{[a]_i!} \in
  {\bf U},
\end{equation}
the quantum divided powers of generators $E_{\epsilon i}$. Then
\begin{equation}E_{\epsilon i^{(a)}}1_{\lambda} = E_{\epsilon
i}^{(a)}1_{\lambda} \in \U
\end{equation}
and has left weight $\lambda+\epsilon a i_X$.  More generally, for a dpss $\ii$
\begin{equation}
  E_{\ii} := E_{\epsilon_1i_1}^{(a_1)}E_{\epsilon_2i_2}^{(a_2)} \dots
  E_{\epsilon_mi_m}^{(a_m)} \in {\bf U}
\end{equation}
and
\begin{equation}
  E_{\ii}1_{\lambda} := E_{\epsilon_1i_1}^{(a_1)}E_{\epsilon_2i_2}^{(a_2)} \dots
  E_{\epsilon_mi_m}^{(a_m)}1_{\lambda} \in \U,
\end{equation}
with left weight $\lambda + \sum_{r=1}^m\epsilon_r a_r(i_r)_X$.

Let $\sseqd$ be the set of all dpss.  Elements $E_{\ii}1_{\lambda}$, over all
$\ii \in \sseqd$ and $\lambda \in X$, span $\UA$ as a $\Z[q,q^{-1}]$-module.  For
$\nu \in \N[I]$, ${\rm Seqd}(\nu)$ was defined in \cite{KL} as the set of all
expressions $i_1^{(a_1)} \dots i_m^{(a_m)}$ such that $\sum_{r=1}^m a_r i_r =
\nu$.  Each element of ${\rm Seqd}(\nu)$ gives a positive divided power sequence,
and
\begin{equation}
  \sseqd^+ \cong \coprod_{\nu \in \N[I]} {\rm Seqd}(\nu).
\end{equation}
The length $\parallel \ii \parallel$ of a dpss is defined as the sum $a_1+a_2 +
\dots + a_m$ in the above notation.

For the reader's convenience our notations for sequences are collected below:

\[
 \xy
 (90,-15)*++{\txt{ $\sseqd$\\divided powers \\ signed sequences }}="1";
  (90,-50)*++{\txt{$\sseq$\\signed sequences  }}="2";
  (30,-50)*++{\txt{$\seq^+$ \\positive sequences }}="4";
 (30,-15)*++{\txt{$\sseqd^+$ \\positive divided power \\signed sequences }}="6";
  (-30,-15)*++{\txt{$\seqd(\nu)$ \\positive divided \\ powers sequences \\ of weight $\nu$ }}="3";
 (-30,-50)*++{\txt{$\seq(\nu)$ \\positive sequences \\ of weight $\nu$ }}="5";
 {\ar@{^{(}->} "2";"1"};
  {\ar@{^{(}->} "4";"2"};
 {\ar@{^{(}->} "3";"6"}; {\ar@{^{(}->} "6";"1"};
 {\ar@{^{(}->} "5";"4"};
 {\ar@{^{(}->} "5";"3"};
  {\ar@{^{(}->} "4";"6"};
 \endxy
\]

%
\subsection{Geometric interpretation of the bilinear form}
\label{subsec_geometric}
%

Mark $m$ points $1 \times \{0\},2 \times \{0\}, \ldots, m \times \{0\}$ on the
lower boundary $\R \times \{0\}$ of the strip $\R \times [0,1]$ and $k$ points $1
\times \{1\},2 \times \{1\}, \ldots, k \times \{1\}$ on the upper boundary $\R
\times \{1\}$. Assuming $m+k$ is even, choose an immersion of $\frac{m+k}{2}$
strands into $\R\times [0,1]$ with these $m+k$ points as the endpoints.  Orient
each strand and label it by an element of $I$.  Then endpoints inherit
orientations and labels from the strands:
\[
 \xy
  (-20,12)*{}; (12,12) **\dir{-};(-20,-12)*{}; (12,-12) **\dir{-};
  (-16,12)*{};(0,12)*{} **\crv{(-12,2) & (-1,2)}?(.22)*\dir{<};
  (-16,-12)*{};(8,-12)*{} **\crv{(-7,-4) & (7,-4)}?(.8)*\dir{>};
  (-8,-12); (-8,12) **\crv{ (-5,-4) & (-9,4)}?(.5)*\dir{<};
  (0,-12); (8,12) **\crv{ (5,-4) & (6,6)}?(.5)*\dir{>};
  (-16,8)*{ \scs i};(-9,-2)*{\scs j};(8,2)*{\scs i};(9,-9)*{\scs \ell};
 \endxy
 \qquad
 \longrightarrow \qquad
  \xy
  (-20,12); (12,12) **\dir{-};(-20,-12)*{}; (12,-12) **\dir{-};
  (-16,12)*{\bullet};(-8,12)*{\bullet};(0,12)*{\bullet}; (8,12)*{\bullet};
  (-16,-12)*{\bullet};(-8,-12)*{\bullet};(0,-12)*{\bullet};(8,-12)*{\bullet};
  (-16,15)*{+};(-8,15)*{-};(0,15)*{-}; (8,15)*{+};
   (-16,9)*{\scs i};(-8,9)*{\scs j};(0,9)*{\scs i}; (8,9)*{\scs i};
    (-16,-15)*{+};(-8,-15)*{-};(0,-15)*{+}; (8,-15)*{-};
   (-16,-9)*{\scs \ell};(-8,-9)*{\scs j};(0,-9)*{\scs i}; (8,-9)*{\scs \ell};
   (23,12)*{k=4};(23,-12)*{m=4};
 \endxy
\]
Orientation and labels at lower and upper endpoints define signed sequences
$\ii,\jj \in \sseq$. In the above example, $\ii=(+\ell,-j,+i,-\ell)$ and
$\jj=(+i,-j,-i,+i)$. We consider immersions modulo boundary--preserving
homotopies and call them pairings between sequences $\ii$ and $\jj$, or simply
$(\ii,\jj)$-pairings.

Clearly, $(\ii,\jj)$-pairings are in a bijection with complete matchings of $m+k$
points such that the two points in each matching pair share the same label and
their orientations are compatible.  Denote by $p'(\ii,\jj)$ the set of all
$(\ii,\jj)$-pairings.

A minimal diagram $D$ of a $(\ii,\jj)$-pairing is a generic immersion that
realizes the pairing such that strands have no self--intersections and any two
strands intersect at most once.  We consider minimal diagrams up to
boundary--preserving isotopies.  A $(\ii,\jj)$-pairing has at least one minimal
diagram, in the example below it has two:
\[
 \xy
  (-20,12)*{}; (12,12) **\dir{-};(-20,-12)*{}; (12,-12) **\dir{-};
  (-16,12)*{};(8,12)*{} **\crv{(-12,2) & (7,2)}?(.22)*\dir{<};
  (0,-12)*{};(8,-12)*{} **\crv{(0,-6) & (7,-6)}?(.8)*\dir{>};
  (-16,-12); (0,12) **\crv{ (-15,-6) & (0,4)}?(.2)*\dir{<};
  (-8,-12); (-8,12) **\crv{ (-4,-4) & (-8,4)}?(.2)*\dir{<};
  (-16,8)*{ \scs i};(-14,-4)*{\scs j};(-4,-4)*{\scs i};(9,-9)*{\scs \ell};
 \endxy
 \qquad \qquad
  \xy
  (-20,12)*{}; (12,12) **\dir{-};(-20,-12)*{}; (12,-12) **\dir{-};
  (-16,12)*{};(8,12)*{} **\crv{(-16,-8) & (7,-8)}?(.22)*\dir{<};
  (0,-12)*{};(8,-12)*{} **\crv{(0,-6) & (7,-6)}?(.8)*\dir{>};
  (-16,-12); (0,12) **\crv{ (-15,-6) & (0,4)}?(.2)*\dir{<};
  (-8,-12); (-8,12) **\crv{ (-4,-4) & (-8,4)}?(.2)*\dir{<};
  (-18,8)*{ \scs i};(-10,7)*{\scs i};(1,7)*{\scs j};(9,-9)*{\scs \ell};
 \endxy
 \]
Here is an immersion which is not a minimal diagram
\[
  \xy
  (-20,12)*{}; (12,12) **\dir{-};(-20,-12)*{}; (12,-12) **\dir{-};
  (-16,12)*{};(8,12)*{} **\crv{(-12,2) & (7,2)}?(.22)*\dir{<};
  (-16,-12)*{};(-8,-12)*{} **\crv{(-12,8) & (4,8)}?(.2)*\dir{>};
  (8,-12); (-8,12) **\crv{(8,-10) & (-8,-4) & (-9,4)}?(.2)*\dir{<};
  (-16,8)*{ \scs i};(-17,-8)*{\scs j};(9,2)*{\scs i};(9,-9)*{\scs \ell};
  (0,12)*{}="T";
(0,-12)*{}="B"; (0,5)*{}="T'"; (0,-5)*{}="B'"; "T";"T'" **\dir{-}?(.6)*\dir{<};
"B";"B'" **\dir{-}; (3,0)*{}="MB"; (7,0)*{}="LB"; "T'";"LB" **\crv{(1,-4) &
(7,-4)}; \POS?(.25)*{}="2z"; "LB"; "2z" **\crv{(8,6) & (2,6)}; "2z"; "B'"
**\crv{(0,-3)};
 \endxy
\]

Given $\lambda \in X$, color the regions of a minimal diagram $D$ by elements of
$X$ such that the rightmost region is colored $\lambda$, and the two regions
between the two sides on an $i$-labelled strand differ by $i_X$
\[
 \xy
 (0,-8)*{};(0,8)*{} **\dir{-}?(.8)*\dir{>};
 (0,-11)*{\scs i}; (6,0)*{\lambda}; (-8,0)*{\lambda+i_X}
 \endxy
\]
Isotope $D$ so that the strands at each crossing are either both oriented up or
down, and define the degree $\deg(D,\lambda)$ of $D$ relative to $\lambda$ as the
integer which is the sum of contributions from all crossings, local maxima and
minima of $D$:
\[\begin{tabular}{|l|c|c|c|c|c|c|}
  \hline
 &
\xy (0,0)*{\xybox{
  (-7,-10)*{}; (5,3)*{};
  (3,0)*{}; (-3,0)*{} **\crv{(3,-6) & (-3,-6)}; ?(.05)*\dir{<}  ?(.9)*\dir{<}
  ?(.5)*\dir{}+(0,-2)*{\scriptstyle i};
    (8,-5)*{ \lambda};
    }}; \endxy
  & \xy (0,0)*{\xybox{
  (-7,-10)*{}; (5,3)*{};
  (3,0)*{}; (-3,0)*{} **\crv{(3,-6) & (-3,-6)}; ?(.1)*\dir{>}  ?(.95)*\dir{>}
  ?(.5)*\dir{}+(0,-2)*{\scriptstyle i};
    (8,-5)*{ \lambda};
    }}; \endxy
  & \xy (0,0)*{\xybox{
  (-7,10)*{}; (5,-3)*{};
  (3,0)*{}; (-3,0)*{} **\crv{(3,6) & (-3,6)}; ?(.15)*\dir{>}  ?(.9)*\dir{>}
  ?(.5)*\dir{}+(0,2)*{\scriptstyle i};
    (8,5)*{ \lambda};
    }}; \endxy
  & \xy (0,0)*{\xybox{
  (-7,10)*{}; (5,-3)*{};
  (3,0)*{}; (-3,0)*{} **\crv{(3,6) & (-3,6)}; ?(.05)*\dir{<}  ?(.9)*\dir{<}
  ?(.5)*\dir{}+(0,2)*{\scriptstyle i};
    (8,5)*{ \lambda};
    }}; \endxy
 &
   \xy
  (0,0)*{\xybox{
    (-3,-4)*{};(3,4)*{} **\crv{(-3,-1) & (3,1)}?(1)*\dir{>} ;
    (3,-4)*{};(-3,4)*{} **\crv{(3,-1) & (-3,1)}?(1)*\dir{>};
    (-4.5,-3)*{\scs i};
     (4.5,-3)*{\scs j};
     (8,1)*{ \lambda};
     (-10,0)*{};(12,0)*{};
     }};
  \endxy
 &
   \xy
  (0,0)*{\xybox{
    (-3,4)*{};(3,-4)*{} **\crv{(-3,1) & (3,-1)}?(1)*\dir{>} ;
    (3,4)*{};(-3,-4)*{} **\crv{(3,1) & (-3,-1)}?(1)*\dir{>};
    (-4.5,3)*{\scs i};
     (5.1,3)*{\scs j};
     (8,1)*{ \lambda};
     (-12,0)*{};(12,0)*{};
     }};
  \endxy \\ \hline
  {\rm deg} &\xy (0,-3)*{};(0,6)*{}; \endxy
  $c_{+i,\lambda}$ &
 $c_{-i,\lambda}$
 &$c_{+i,\lambda}$& $c_{-i,\lambda}$
 & $-i\cdot j$ & $-i\cdot j$\\
  \hline
\end{tabular}
\]
where
\begin{equation} \label{eq_cpm_lambda}
  c_{\pm i,\lambda} := \frac{i\cdot i}{2}(1\pm \la i,\lambda\ra).
\end{equation}
Notice that
\begin{equation}
 q^{c_{\pm i,\lambda}} = q_i^{1 \pm \la i, \lambda \ra}.
\end{equation}
In the simply-laced case $q_i=q$ and $c_{\pm i,\lambda} = 1 \pm \la i, \lambda
\ra$.

The degrees of the other crossings are determined from the rules above
\begin{eqnarray}
 \deg\left( \; \xy
  (0,0)*{\xybox{
    (-4,-4)*{};(4,4)*{} **\crv{(-4,-1) & (4,1)}?(1)*\dir{>} ;
    (4,-4)*{};(-4,4)*{} **\crv{(4,-1) & (-4,1)}?(0)*\dir{<};
    (-5,-3)*{\scs i};
     (-5,3)*{\scs j};
     (8,2)*{ \lambda}; }};
  \endxy \right)
&:=&
 \deg\left( \; \xy 0;/r.19pc/:
  (0,0)*{\xybox{
    (4,-4)*{};(-4,4)*{} **\crv{(4,-1) & (-4,1)}?(1)*\dir{>};
    (-4,-4)*{};(4,4)*{} **\crv{(-4,-1) & (4,1)};
     (-4,4);(-4,12) **\dir{-};
     (-12,-4);(-12,12) **\dir{-};
     (4,-4);(4,-12) **\dir{-};(12,4);(12,-12) **\dir{-};
     (16,1)*{\lambda};
     (-10,0)*{};(10,0)*{};
     (-4,-4)*{};(-12,-4)*{} **\crv{(-4,-10) & (-12,-10)}?(1)*\dir{<}?(0)*\dir{<};
      (4,4)*{};(12,4)*{} **\crv{(4,10) & (12,10)}?(1)*\dir{>}?(0)*\dir{>};
      (-14,11)*{\scs j};(-2,11)*{\scs i};
      (14,-11)*{\scs j};(2,-11)*{\scs i};
     }};
  \endxy \right)
  \quad \;= \quad
 \deg\left( \;  \xy 0;/r.19pc/:
  (0,0)*{\xybox{
    (-4,-4)*{};(4,4)*{} **\crv{(-4,-1) & (4,1)}?(1)*\dir{<};
    (4,-4)*{};(-4,4)*{} **\crv{(4,-1) & (-4,1)};
     (4,4);(4,12) **\dir{-};
     (12,-4);(12,12) **\dir{-};
     (-4,-4);(-4,-12) **\dir{-};(-12,4);(-12,-12) **\dir{-};
     (16,1)*{\lambda};
     (10,0)*{};(-10,0)*{};
     (4,-4)*{};(12,-4)*{} **\crv{(4,-10) & (12,-10)}?(1)*\dir{>}?(0)*\dir{>};
      (-4,4)*{};(-12,4)*{} **\crv{(-4,10) & (-12,10)}?(1)*\dir{<}?(0)*\dir{<};
      (14,11)*{\scs i};(2,11)*{\scs j};
      (-14,-11)*{\scs i};(-2,-11)*{\scs j};
     }};
  \endxy \; \right) \quad = 0 \nn \\
 \deg\left( \;  \xy
  (0,0)*{\xybox{
    (-4,-4)*{};(4,4)*{} **\crv{(-4,-1) & (4,1)}?(0)*\dir{<} ;
    (4,-4)*{};(-4,4)*{} **\crv{(4,-1) & (-4,1)}?(1)*\dir{>};
    (5.1,-3)*{\scs i};
     (5.1,3)*{\scs j};
     (-8,2)*{ \lambda};
     }};
  \endxy  \; \right)
&:=&
  \deg\left( \;\xy 0;/r.19pc/:
  (0,0)*{\xybox{
    (-4,-4)*{};(4,4)*{} **\crv{(-4,-1) & (4,1)}?(1)*\dir{>};
    (4,-4)*{};(-4,4)*{} **\crv{(4,-1) & (-4,1)};
     (4,4);(4,12) **\dir{-};
     (12,-4);(12,12) **\dir{-};
     (-4,-4);(-4,-12) **\dir{-};(-12,4);(-12,-12) **\dir{-};
     (-16,1)*{\lambda};
     (10,0)*{};(-10,0)*{};
     (4,-4)*{};(12,-4)*{} **\crv{(4,-10) & (12,-10)}?(1)*\dir{<}?(0)*\dir{<};
      (-4,4)*{};(-12,4)*{} **\crv{(-4,10) & (-12,10)}?(1)*\dir{>}?(0)*\dir{>};
      (14,11)*{\scs j};(2,11)*{\scs i};
      (-14,-11)*{\scs j};(-2,-11)*{\scs i};
     }};
  \endxy  \; \right)
  \quad = \quad
  \deg\left( \; \xy 0;/r.19pc/:
  (0,0)*{\xybox{
    (4,-4)*{};(-4,4)*{} **\crv{(4,-1) & (-4,1)}?(1)*\dir{<};
    (-4,-4)*{};(4,4)*{} **\crv{(-4,-1) & (4,1)};
     (-4,4);(-4,12) **\dir{-};
     (-12,-4);(-12,12) **\dir{-};
     (4,-4);(4,-12) **\dir{-};(12,4);(12,-12) **\dir{-};
     (-16,1)*{\lambda};
     (-10,0)*{};(10,0)*{};
     (-4,-4)*{};(-12,-4)*{} **\crv{(-4,-10) & (-12,-10)}?(1)*\dir{>}?(0)*\dir{>};
      (4,4)*{};(12,4)*{} **\crv{(4,10) & (12,10)}?(1)*\dir{<}?(0)*\dir{<};
      (-14,11)*{\scs i};(-2,11)*{\scs j};
      (14,-11)*{\scs i};(2,-11)*{\scs j};
     }};
  \endxy  \; \right) \quad = 0 \nn
\end{eqnarray}
Since both of these crossings have degree zero, we will refer to them as {\em
balanced} crossings.

\begin{prop}
 $\deg(D,\lambda)$ depends only on $\lambda$ and on the $(\ii,\jj)$-pairing
 realized by $D$.  Thus, $\deg(D,\lambda)$ is independent of the choice of a
 minimal diagram for a pairing.
\end{prop}

\begin{proof}
Invariance under cancellation of U-turns can be shown as follows:
\[
  \deg \left(\;  \xy   0;/r.18pc/:
    (-8,0)*{}="1";
    (0,0)*{}="2";
    (8,0)*{}="3";
    (-8,-10);"1" **\dir{-};
    "1";"2" **\crv{(-8,8) & (0,8)} ?(0)*\dir{>} ?(1)*\dir{>};
    "2";"3" **\crv{(0,-8) & (8,-8)}?(1)*\dir{>};
    "3"; (8,10) **\dir{-};
    (14,9)*{\lambda};
    (-6,9)*{\lambda+i_X};
    \endxy \;\right)
     =
    \frac{i\cdot i}{2}(1+\la i,\lambda\ra)+\frac{i\cdot i}{2}(1-\la i,\lambda+i_X\ra)
    = 0 = \deg\left(\;
    \xy   0;/r.18pc/:
    (-8,0)*{}="1";
    (0,0)*{}="2";
    (8,0)*{}="3";
    (0,-10);(0,10)**\dir{-} ?(.5)*\dir{>};
    (5,9)*{\lambda};
    (-8,9)*{\lambda+i_X};
    \endxy \right).
\]
Invariance of $\deg(D,\lambda)$ under other isotopies of $D$ is straightforward
and is left to the reader.  Any two minimal diagrams of the same
$(\ii,\jj)$-pairing are related by a sequence of isotopies and
``triple--crossing" moves for various orientations, see examples below.
\begin{equation}
 \vcenter{
 \xy 0;/r.18pc/:
    (-4,-4)*{};(4,4)*{} **\crv{(-4,-1) & (4,1)}?(1)*\dir{>};
    (4,-4)*{};(-4,4)*{} **\crv{(4,-1) & (-4,1)}?(1)*\dir{>};
    (4,4)*{};(12,12)*{} **\crv{(4,7) & (12,9)}?(1)*\dir{>};
    (12,4)*{};(4,12)*{} **\crv{(12,7) & (4,9)}?(1)*\dir{>};
    (-4,12)*{};(4,20)*{} **\crv{(-4,15) & (4,17)}?(1)*\dir{>};
    (4,12)*{};(-4,20)*{} **\crv{(4,15) & (-4,17)}?(1)*\dir{>};
    (-4,4)*{}; (-4,12) **\dir{-};
    (12,-4)*{}; (12,4) **\dir{-};
    (12,12)*{}; (12,20) **\dir{-};
  (18,8)*{\lambda};
  (-6,-3)*{\scs i};
  (6,-3)*{\scs j};
  (15,-3)*{\scs k};
\endxy}
 \;\; \leftrightarrow\;\;
 \vcenter{
 \xy 0;/r.18pc/:
    (4,-4)*{};(-4,4)*{} **\crv{(4,-1) & (-4,1)}?(1)*\dir{>};
    (-4,-4)*{};(4,4)*{} **\crv{(-4,-1) & (4,1)}?(1)*\dir{>};
    (-4,4)*{};(-12,12)*{} **\crv{(-4,7) & (-12,9)}?(1)*\dir{>};
    (-12,4)*{};(-4,12)*{} **\crv{(-12,7) & (-4,9)}?(1)*\dir{>};
    (4,12)*{};(-4,20)*{} **\crv{(4,15) & (-4,17)}?(1)*\dir{>};
    (-4,12)*{};(4,20)*{} **\crv{(-4,15) & (4,17)}?(1)*\dir{>};
    (4,4)*{}; (4,12) **\dir{-};
    (-12,-4)*{}; (-12,4) **\dir{-};
    (-12,12)*{}; (-12,20) **\dir{-};
  (10,8)*{\lambda};
  (7,-3)*{\scs k};
  (-6,-3)*{\scs j};
  (-14,-3)*{\scs i};
\endxy}
\qquad \qquad \vcenter{
 \xy 0;/r.18pc/:
    (-4,-4)*{};(4,4)*{} **\crv{(-4,-1) & (4,1)}?(1)*\dir{>};
    (4,-4)*{};(-4,4)*{} **\crv{(4,-1) & (-4,1)}?(1)*\dir{<};
    ?(0)*\dir{<};
    (4,4)*{};(12,12)*{} **\crv{(4,7) & (12,9)}?(1)*\dir{>};
    (12,4)*{};(4,12)*{} **\crv{(12,7) & (4,9)}?(1)*\dir{>};
    (-4,12)*{};(4,20)*{} **\crv{(-4,15) & (4,17)};
    (4,12)*{};(-4,20)*{} **\crv{(4,15) & (-4,17)}?(1)*\dir{>};
    (-4,4)*{}; (-4,12) **\dir{-};
    (12,-4)*{}; (12,4) **\dir{-};
    (12,12)*{}; (12,20) **\dir{-};
  (18,8)*{\lambda};
  (-6,-3)*{\scs i};
  (7,-3)*{\scs j};
  (15,-3)*{\scs k};
\endxy}
 \;\; \leftrightarrow\;\;
 \vcenter{
 \xy 0;/r.18pc/:
    (4,-4)*{};(-4,4)*{} **\crv{(4,-1) & (-4,1)}?(1)*\dir{>};
    (-4,-4)*{};(4,4)*{} **\crv{(-4,-1) & (4,1)}?(0)*\dir{<};
    (-4,4)*{};(-12,12)*{} **\crv{(-4,7) & (-12,9)}?(1)*\dir{>};
    (-12,4)*{};(-4,12)*{} **\crv{(-12,7) & (-4,9)}?(1)*\dir{>};
    (4,12)*{};(-4,20)*{} **\crv{(4,15) & (-4,17)};
    (-4,12)*{};(4,20)*{} **\crv{(-4,15) & (4,17)}?(1)*\dir{>};
    (4,4)*{}; (4,12) **\dir{-} ?(.5)*\dir{<};
    (-12,-4)*{}; (-12,4) **\dir{-};
    (-12,12)*{}; (-12,20) **\dir{-};
  (10,8)*{\lambda};
  (7,-3)*{\scs k};
  (-7,-3)*{\scs j};
  (-14,-3)*{\scs i};
\endxy}
\end{equation}
The invariance under these moves is manifestly obvious.
\end{proof}

From now on, for each $(\ii,\jj)$-pairing we choose one minimal diagram $D$
representing this pairing and denote by $p(\ii,\jj)$ the set of these diagrams.
Recall the pairings $(,)$ and $\sla,\sra$ on $\U$ from
Section~\ref{subsec_quantumgroups}.

\begin{thm} \label{thm_form_formula} For any $\ii,\jj\in \sseq$ and $\lambda, \mu \in X$
\begin{equation} \label{eq_thm_pairing}
 ( E_{\ii}1_{\lambda}, E_{\jj}1_{\mu} ) \quad = \quad \sla E_{\ii}1_{\lambda}, E_{\jj}1_{\mu} \sra
  \quad = \quad \delta_{\lambda\mu}
  \sum_{D \in p(\ii,\jj)} q^{\deg(D,\lambda)} \prod_{i\in I}
  \frac{1}{(1-q_i^2)^{\kappa_i}},
\end{equation}
where $\kappa_i$ is the number of $i$-colored strands in $D$.
\end{thm}

For example,
\[
 p\left((-j,+i,+j,-i),(+i,-i)\right) \quad = \quad
 \left\{\;\;\;\xy
  (-20,12)*{}; (12,12) **\dir{-};(-20,-12)*{}; (12,-12) **\dir{-};
  (-16,12)*{};(-8,12)*{} **\crv{(-12,2) & (-7,2)}?(.22)*\dir{<};
  (-16,-12)*{};(0,-12)*{} **\crv{(-16,-2) & (0,-2)}?(.14)*\dir{<};
  (-8,-12)*{};(8,-12)*{} **\crv{(-8,-2) & (8,-2)}?(.16)*\dir{>};
  (-16,8)*{ \scs i};(-14,-4)*{\scs j};(8,-4)*{\scs i};
  (8,5)*{\lambda};
 \endxy \quad,
 \qquad
  \xy
  (-20,12)*{}; (12,12) **\dir{-};(-20,-12)*{}; (12,-12) **\dir{-};
  (-16,12)*{};(-8,-12)*{} **\crv{(-12,2) & (-8,2)}?(.22)*\dir{<};
  (-16,-12)*{};(0,-12)*{} **\crv{(-16,-2) & (0,-2)}?(.14)*\dir{<};
  (-8,12)*{};(8,-12)*{} **\crv{(-8,2) & (8,-2)}?(.25)*\dir{>};
  (-16,8)*{ \scs i};(-14,-4)*{\scs j};(8,-4)*{\scs i};
  (8,5)*{\lambda};
 \endxy \;\;\; \right\}
 \]
so that
\begin{eqnarray}
(E_{(-j,+i,+j,-i)}1_{\lambda},E_{(+i,-i)}1_{\lambda}  ) &=&
\left(q^{2c_{-i,\lambda}+c_{+j,\lambda}-i\cdot j}+q^{c_{+j,\lambda-i_{ X}}}
\right) \frac{1}{(1-q_i)^2}\frac{1}{(1-q_j)}
\\ &=&
\left(q_i^{2(1-\la i,\lambda \ra)}q_j^{ (1+\la j,\lambda \ra)}q^{ - i \cdot j}
+q_j^{(1+\la j,\lambda \ra)}q^{- i \cdot j} \right)
\frac{1}{(1-q_i)^2}\frac{1}{(1-q_j)}. \nn
\end{eqnarray}

The product term in \eqref{eq_thm_pairing} is independent of $D$ since
$2\kappa_i$ is the number of times label $i$ appears in the sequence $\ii\jj$.

\begin{proof}
The first equality in $\eqref{eq_thm_pairing}$ is obvious since
$E_{\ii}1_{\lambda}$ is invariant under the involution $\psi$. For $\ii$, $\jj$,
$\lambda, \mu$ as in Theorem~\ref{thm_form_formula} define
\begin{equation} \label{eq_pariing prime}
(E_{\ii}1_{\lambda},E_{\jj}1_{\mu})' = \delta_{\lambda\mu}\sum_{D \in p(\ii,\jj)}
q^{\deg(D,\lambda)} \prod_{i\in I}
  \frac{1}{(1-q_i^2)^{\kappa_i}}.
\end{equation}
During the proof we view $E_{\ii}1_{\lambda}$ and $E_{\jj}1_{\mu}$ in
$(E_{\ii}1_{\lambda},E_{\jj}1_{\mu})'$ as formal symbols rather than elements of
$\U$, since we have not yet proved that $(,)'$ descends to $\U$. We want to show
\begin{equation}
  (E_{\ii}1_{\lambda},E_{\jj}1_{\mu})' =
  (E_{\ii}1_{\lambda},E_{\jj}1_{\mu}) \nn
\end{equation}
for all $\ii$, $\jj$, $\lambda$ and $\mu$.

The nontrivial case is $\lambda = \mu$. We can also interpret the formula
$(E_{\ii}1_{\lambda}, E_{\jj}1_{\mu})'=0$ for $\mu \neq \lambda$ via the sum of
diagrams, since then the rightmost region must be colored by both $\lambda$ and
$\mu$, which is not allowed.
\[
 \xy
  (-20,12)*{}; (12,12) **\dir{-};(-20,-12)*{}; (12,-12) **\dir{-};
  (-16,12)*{};(8,12)*{} **\crv{(-12,2) & (7,2)}?(.22)*\dir{<};
  (0,-12)*{};(8,-12)*{} **\crv{(0,-6) & (7,-6)}?(.8)*\dir{>};
  (-16,-12); (0,12) **\crv{ (-15,-6) & (0,4)}?(.2)*\dir{<};
  (-8,-12); (-8,12) **\crv{ (-4,-4) & (-8,4)}?(.2)*\dir{<};
  (-16,8)*{ \scs i};(-14,-4)*{\scs j};(-4,-4)*{\scs i};(9,-9)*{\scs \ell};
  (12,5)*{\lambda};(12,-5)*{\mu};
 \endxy
 \]

Notice that in the absence of $(\ii,\jj)$-pairings the right hand side of
\eqref{eq_thm_pairing} is zero, since there are no diagrams to sum over.  If the
left weights of $E_{\ii}1_{\lambda}$ and $E_{\jj}1_{\lambda}$ are different the
inner product $(E_{\ii}1_{\lambda},E_{\jj}1_{\lambda})=0$ and
$(E_{\ii}1_{\lambda},E_{\jj}1_{\lambda})'=0$ as well since the set of
$(\ii,\jj)$-pairings is empty in this case.

Assume that both $\ii$ and $\jj$ are positive.  Then any minimal
$(\ii,\jj)$-diagram can be isotoped to be braid--like, with the strands going
upward:
\[
  \vcenter{
 \xy 0;/r.18pc/:
    (-12,-12)*{};(4,4)*{} **\crv{(-4,-1) & (4,1)}?(1)*\dir{>};
    (4,-12)*{};(-4,4)*{} **\crv{(4,-1) & (-4,1)}?(1)*\dir{>};
    (12,-4)*{};(20,4)*{} **\crv{(12,-1) & (20,1)}?(1)*\dir{>};
    (20,-4)*{};(12,4)*{} **\crv{(20,-1) & (12,1)}?(0)*\dir{<};
    (-4,12)*{};(4,20)*{} **\crv{(-4,15) & (4,17)}?(1)*\dir{>};
    (20,4)*{};(-4,20)*{} **\crv{(20,8) & (-4,15)}?(1)*\dir{>};
    (4,4)*{};(12,4)*{} **\crv{(4,8) & (12,8)}?(1)*\dir{};
    (20,-4)*{};(28,-4)*{} **\crv{(20,-8) & (28,-8)}?(1)*\dir{};
    (28,-4)*{};(12,20)*{} **\crv{(28,8) & (12,14)}?(1)*\dir{};
    (-4,4)*{}; (-4,12) **\dir{-};
    (12,-12)*{}; (12,-4) **\dir{-};
  (24,14)*{\lambda};
    (-6,0)*{}="1";
    (-10,0)*{}="2";
    (-14,0)*{}="3";
    (-4,-12);"1" **\crv{(-4,-8) & (-4,-4)};
    "1";"2" **\crv{(-8,8) & (-12,8)} ?(0)*\dir{>} ?(1)*\dir{>};
    "2";"3" **\crv{(-8,-8) & (-12,-8)}?(1)*\dir{>};
    "3"; (-12,20) **\crv{(-18,8) & (-9,16)};
        (-12,23)*{\scs j_1};(0,23)*{\scs \cdots};(12,23)*{\scs j_m};
   (-12,-15)*{\scs i_1};(0,-15)*{\scs \cdots};(12,-15)*{\scs i_m};
\endxy}
\quad \rightsquigarrow \quad
  \vcenter{
 \xy 0;/r.18pc/:
    (-12,-12)*{};(4,4)*{} **\crv{(-4,-1) & (4,1)}?(1)*\dir{>};
    (4,-12)*{};(-4,4)*{} **\crv{(4,-1) & (-4,1)}?(1)*\dir{>};
    (4,4)*{};(12,12)*{} **\crv{(4,7) & (12,9)}?(1)*\dir{>};
    (12,4)*{};(4,12)*{} **\crv{(12,7) & (4,9)}?(1)*\dir{>};
    (-4,12)*{};(4,20)*{} **\crv{(-4,15) & (4,17)}?(1)*\dir{>};
    (4,12)*{};(-4,20)*{} **\crv{(4,15) & (-4,17)}?(1)*\dir{>};
    (-4,4)*{}; (-4,12) **\dir{-};
    (12,-12)*{}; (12,4) **\dir{-};
    (12,12)*{}; (12,20) **\dir{-};
  (18,8)*{\lambda};
    (-4,-12);(-12,20) **\crv{(-4,-8) & (-12,-4)}?(1)*\dir{>};
    (-12,23)*{\scs j_1};(0,23)*{\scs \cdots};(12,23)*{\scs j_m};
    (-12,-15)*{\scs i_1};(0,-15)*{\scs \cdots};(12,-15)*{\scs i_m};
\endxy}
\]
Moreover, in this case $k=m$, otherwise $p(\ii,\jj)$ is empty.  The sum on the
right hand side of \eqref{eq_thm_pairing} describes the canonical bilinear form
on ${\bf U}^+ \cong \mathbf{f}$ evaluated on $E_{\ii}=E_{i_1}\cdots E_{i_m}$ and
$E_{\jj}=E_{j_1}\ldots E_{j_m}$, see \cite{Lus6,Reineke,KL,KL2}.

In view of property (iv) of Lusztig's bilinear form, we obtain the following:

\begin{lem} \label{lem_ij_pos}
  If $\ii$, $\jj$ are positive,
\begin{equation}
(E_{\ii}1_{\lambda},E_{\jj}1_{\lambda})' =
(E_{\ii}1_{\lambda},E_{\jj}1_{\lambda})
\end{equation}
for any $\lambda \in X$.
\end{lem}

Thus, for positive $\ii$, $\jj$, the geometrically defined bilinear form $(,)'$
coincides with $(,)$.

\begin{lem}
Equation \eqref{eq_thm_pairing} holds for the pair $(\pm i \ii,\jj)$ if and only
if it holds for the pair $(\ii,\mp i \jj)$.
\end{lem}

\begin{proof}
Attaching a $U$-turn gives a bijection between $p( \pm i\ii,\jj)$ and $p(\ii,\mp
i\jj)$.
\[
D \quad
 \xy
  (-20,12)*{}; (12,12) **\dir{-};(-20,-12)*{}; (12,-12) **\dir{-};
  (-16,-12)*{};(8,-12)*{} **\crv{(-12,-2) & (7,-2)}?(.22)*\dir{};
  (0,12)*{};(8,12)*{} **\crv{(0,6) & (7,6)}?(.8)*\dir{>};
  (-16,12); (0,-12) **\crv{ (-15,6) & (2,-4)}?(.2)*\dir{>};
  (-8,12); (-8,-12) **\crv{ (-6,4) & (-8,-4)}?(.2)*\dir{<};
  (12,2)*{\lambda};(-20,2)*{\mu};
  (0,-14)*{\underbrace{\hspace{.7in}}};(-4,17)*{\scs \jj};(-16,-15)*{\pm i};
 (-4,14)*{\overbrace{\hspace{1.1in}}};(0,-17)*{\scs \ii};
 \endxy \qquad \qquad \qquad
 \xy (0,0)*{D}; (0,-2); (-4,2) **\crv{(0,-5) & (-5,-5)} \endxy
 \quad
  \xy
  (-28,12)*{}; (12,12) **\dir{-};(-28,-12)*{}; (12,-12) **\dir{-};
  (-12,-6)*{};(8,-12)*{} **\crv{(-10,0) & (7,-2)}?(.15)*\dir{};
  (-12,-6)*{}; (-24,12) **\crv{(-14,-12) & (-20,-12)};
  (0,12)*{};(8,12)*{} **\crv{(0,6) & (7,6)}?(.8)*\dir{>};
  (-16,12); (0,-12) **\crv{ (-15,6) & (2,-4)}?(.2)*\dir{>};
  (-8,12); (-8,-12) **\crv{ (-6,4) & (-8,-4)}?(.2)*\dir{<};
 (12,2)*{\lambda};(-16,2)*{\mu};
  (0,-14)*{\underbrace{\hspace{.7in}}};(-4,17)*{\scs \jj};
   (-4,14)*{\overbrace{\hspace{1.1in}}};(0,-17)*{\scs \ii};
   (-24,15)*{\mp i};
 \endxy
 \]
We have
\[
 \deg(\xy (0,0)*{D}; (0,-2); (-4,2) **\crv{(0,-5) & (-5,-5)} \endxy, \lambda) =
  \deg(D,\lambda)+\frac{i \cdot i}{2}(1\pm\la i, \mu\mp i_X \ra) =
 \deg(D,\lambda)-\frac{i \cdot i}{2}(1\mp \la i, \mu \ra),
\]
where the additional term matches the power of $q$ in the formula
$\overline{\rho}(1_{\mu}E_{\pm i}) = q_i^{1\mp \la i ,\mu \ra}E_{\mp i}1_{\mu}$.
\end{proof}

The previous lemma implies that it is enough to check the equality
\begin{equation}
  (E_{\ii}1_{\lambda},E_{\jj}1_{\lambda})' =
  (E_{\ii}1_{\lambda},E_{\jj}1_{\lambda})
\end{equation}
when $\jj$ is the empty sequence.  The next three lemmas show that
\begin{equation}
  (E_{\ii}1_{\lambda},1_{\lambda})' = (E_{\ii}1_{\lambda},1_{\lambda})
\end{equation}
for any $\ii$ and $\lambda$.

\begin{lem} \label{lem_minum_then_plus}
Equation \eqref{eq_thm_pairing} holds for all pairs $((-\jj)\ii,\emptyset)$ with
$\ii$, $\jj$ positive.
\end{lem}

\begin{proof}
Bending up $-\jj$ by adding U-turns transforms the pair $((-\jj)\ii,\emptyset)$
to $(\ii,\jj)$ for which \eqref{eq_thm_pairing} holds (see
Lemma~\ref{lem_ij_pos}). We have
\begin{eqnarray}
  ( E_{(-\jj)\ii}1_{\lambda},1_{\lambda} ) = q^{\alpha}(
  E_{\ii}1_{\lambda},E_{\jj}1_{\lambda} ),  \qquad \quad
 (E_{(-\jj)\ii}1_{\lambda},1_{\lambda} )' = q^{\alpha}(
  E_{\ii}1_{\lambda},E_{\jj}1_{\lambda})' ,
\end{eqnarray}
where $\alpha$ is the power of $q$ in the formula
$\overline{\rho}(E_{-\jj}1_{\lambda+\ii_X}) = \tau(E_{-\jj}1_{\lambda+\ii_X}) =
q^{\alpha}1_{\lambda + \ii_X}E_{\jj}$.
\end{proof}

\begin{lem}
For $i,j \in I$, $i \neq j$
\begin{equation}
  (E_{\ii'\pm i\mp j\ii''}1_{\lambda},1_{\lambda})' =
  (E_{\ii'\pm i\mp j\ii''}1_{\lambda},1_{\lambda})
\end{equation}
if and only if
\begin{equation}
  (E_{\ii'\mp j\pm i\ii''}1_{\lambda},1_{\lambda})' =
  (E_{\ii'\mp j\pm i\ii''}1_{\lambda},1_{\lambda}).
\end{equation}
\end{lem}

\begin{proof}
Attach a crossing at the $\pm i \mp j$ location to a diagram $D$ in $p(\ii'\pm i
\mp j \ii'', \emptyset)$.
\[
 \xy
  (-24,-12)*{}; (24,-12) **\dir{-};
  (-4,-12)*{};(12,-12)*{} **\crv{(-4,-2) & (12,-2)}?(.15)*\dir{>};
   (4,-12)*{};(-20,-12)*{} **\crv{(5,-2) & (-20,-2)}?(.11)*\dir{<};
     (-12,-12)*{};(20,-12)*{} **\crv{(-12,6) & (20,6)}?(.3)*\dir{};
     (-4,-15)*{\pm i};(4,-15)*{\mp j};(-18,0)*{D};(20,0)*{\lambda};
   (-16,-14)*{\underbrace{\hspace{.45in}}};(-16,-17)*{\scs \ii'};
    (16,-14)*{\underbrace{\hspace{.45in}}};(16,-17)*{\scs \ii''};
 \endxy
 \qquad \longrightarrow \qquad
  \xy
  (-24,-12)*{}; (24,-12) **\dir{-};
  (-4,-12)*{};(12,-12)*{} **\crv{(-4,-2) & (12,-2)}?(.15)*\dir{};
   (4,-12)*{};(-20,-12)*{} **\crv{(5,-2) & (-20,-2)}?(.1)*\dir{};
     (-12,-12)*{};(20,-12)*{} **\crv{(-12,6) & (20,6)}?(.3)*\dir{};
     (4,-23)*{\pm i};(-4,-23)*{\mp j};
     (-4,-20)*{};(4,-12)*{} **\crv{(-4,-17) & (4,-15)}?(1)*\dir{<} ;
    (4,-20)*{};(-4,-12)*{} **\crv{(4,-17) & (-4,-15)}?(1)*\dir{>};
    (-18,0)*{D'};(20,0)*{\lambda};
    (-16,-14)*{\underbrace{\hspace{.45in}}};(-16,-17)*{\scs \ii'};
    (16,-14)*{\underbrace{\hspace{.45in}}};(16,-17)*{\scs \ii''};
 \endxy
 \]
The resulting diagram $D'$ is minimal if the $\pm i$ and $\mp j$ strands of $D$
do not intersect.  Otherwise, it is not minimal, but the homotopy
\[
    \vcenter{\xy 0;/r.18pc/:
    (-4,-4)*{};(4,4)*{} **\crv{(-4,-1) & (4,1)}?(1)*\dir{};?(0)*\dir{<};
    (4,-4)*{};(-4,4)*{} **\crv{(4,-1) & (-4,1)}?(0)*\dir{};
    (-4,4)*{};(4,12)*{} **\crv{(-4,7) & (4,9)}?(1)*\dir{>};
    (4,4)*{};(-4,12)*{} **\crv{(4,7) & (-4,9)}?(1)*\dir{};
 \endxy}
 \quad \longrightarrow \quad
\xy 0;/r.18pc/:
  (3,9);(3,-9) **\dir{-}?(.5)*\dir{<}+(2.3,0)*{};
  (-3,9);(-3,-9) **\dir{-}?(.55)*\dir{>}+(2.3,0)*{};
 \endxy
 \]
will make it minimal (the orientations in the picture are for the $+i-j$ case).

We get a bijection
\begin{equation}
 \xymatrix{ p(\ii' \pm i \mp j \ii'' , \emptyset)\ar[r]^{\cong} &
 p(\ii'\mp j \pm i \ii'', \emptyset) }
\end{equation}
which preserves the degree of a diagram, for any $\lambda$, since $\deg\left(\;
   \xy
  (0,0)*{\xybox{
    (-3,-4)*{};(3,4)*{} **\crv{(-3,-1) & (3,1)}?(0)*\dir{<} ;
    (3,-4)*{};(-3,4)*{} **\crv{(3,-1) & (-3,1)}?(1)*\dir{>};
     (8,1)*{ \mu};
     }};
  \endxy \; \right)=0$ for any $\mu$.  Hence,
\begin{equation}
(E_{\ii' \pm i \mp j \ii'' }1_{\lambda}, 1_{\lambda})' = (E_{\ii'\mp j \pm i
\ii''} 1_{\lambda}, 1_{\lambda})'.
\end{equation}
Since $E_{\ii' \pm i \mp j \ii''} 1_{\lambda} = E_{\ii'\mp j\pm i  \ii''}
1_{\lambda}$ in $\U$, we see that
\begin{equation}
(E_{\ii' \pm i \mp j \ii'' }1_{\lambda}, 1_{\lambda}) = (E_{\ii'\mp j \pm i \ii''
}1_{\lambda}, 1_{\lambda})
\end{equation}
and the lemma follows.
\end{proof}

\begin{lem}
Assume
$(E_{\ii'\ii''}1_{\lambda},1_{\lambda})'=(E_{\ii'\ii''}1_{\lambda},1_{\lambda})$.
Then
\begin{equation}
(E_{\ii' +i-i\ii''}1_{\lambda},1_{\lambda})' = (E_{\ii'
+i-i\ii''}1_{\lambda},1_{\lambda})
\end{equation}
if and only if
\begin{equation}
(E_{\ii' -i+i\ii''}1_{\lambda},1_{\lambda})' = (E_{\ii'
-i+i\ii''}1_{\lambda},1_{\lambda}).
\end{equation}
\end{lem}

\begin{proof}
Decompose the diagrams in $p(\ii' +i-i\ii'',\emptyset)$ into three classes
\begin{enumerate}[(1)]
 \item $+i$ and $-i$ strands do not intersect $ \vcenter{ \xy
  (-12,-12)*{}; (12,-12) **\dir{-};
  (-4,-12)*{};(-8,-4)*{} **\crv{ (-6,-6)}?(.3)*\dir{>};
  (4,-12)*{};(8,-4)*{} **\crv{ (6,-6)}?(.15)*\dir{<};
     (-4,-15)*{+ i};(4,-15)*{-i};
 \endxy}$
 \item $+i$ and $-i$ strands intersect $\vcenter{  \xy
  (-12,-12)*{}; (12,-12) **\dir{-};
  (-4,-12)*{};(8,-4)*{} **\crv{(-3,-4) & (0,-6)}?(.18)*\dir{>};
  (4,-12)*{};(-8,-4)*{} **\crv{(3,-4) & (0,-6)}?(.11)*\dir{<};
     (-4,-15)*{+ i};(4,-15)*{-i};
 \endxy}$
 \item a strand connects $+i$ and $-i$ $ \vcenter{\xy
  (-12,-12)*{}; (12,-12) **\dir{-};
  (-4,-12)*{};(4,-12)*{} **\crv{(-4,-4) & (4,-4)}?(.16)*\dir{>};
     (-4,-15)*{+i};(4,-15)*{-i};
 \endxy}$
\end{enumerate}
Likewise, decompose the diagrams in $p(\ii' -i+i \ii'', \emptyset)$ into three
classes:
\begin{enumerate}[(1)]
 \item $-i$ and $+i$ strands do not intersect $ \vcenter{ \xy
  (-12,-12)*{}; (12,-12) **\dir{-};
  (-4,-12)*{};(-8,-4)*{} **\crv{ (-6,-6)}?(.15)*\dir{<};
  (4,-12)*{};(8,-4)*{} **\crv{ (6,-6)}?(.3)*\dir{>};
     (-4,-15)*{- i};(4,-15)*{+i};
 \endxy}$
 \item $-i$ and $+i$ strands intersect $\vcenter{  \xy
  (-12,-12)*{}; (12,-12) **\dir{-};
  (-4,-12)*{};(8,-4)*{} **\crv{(-3,-4) & (0,-6)}?(.17)*\dir{<};
  (4,-12)*{};(-8,-4)*{} **\crv{(3,-4) & (0,-6)}?(.18)*\dir{>};
     (-4,-15)*{-i};(4,-15)*{+i};
 \endxy}$
 \item a strand connects $-i$ and $+i$ $ \vcenter{\xy
  (-12,-12)*{}; (12,-12) **\dir{-};
  (-4,-12)*{};(4,-12)*{} **\crv{(-4,-4) & (4,-4)}?(.11)*\dir{<};
     (-4,-15)*{- i};(4,-15)*{+i};
 \endxy}$
\end{enumerate}
Set up a bijection
\begin{equation}
\xymatrix{p(\ii' +i-i \ii'',\emptyset) \ar[r]^{\cong} & p(\ii'-i+i\ii'',\emptyset
)}
\end{equation}
that takes diagrams from class (1) in the first set to diagrams of class (2) in
the second set by adding a crossing, diagrams of class (2) to diagrams of class
(1) by removing the crossing $ \left(\;\xy
    (-3,-4)*{};(3,4)*{} **\crv{(-3,-1) & (3,1)}?(1)*\dir{>} ;
    (3,-4)*{};(-3,4)*{} **\crv{(3,-1) & (-3,1)}?(0)*\dir{<};
    (-6,-3)*{\scs +i};
     (6,-3)*{\scs -i};
  \endxy\; \right)$, and diagrams of class (3) to diagrams of
class (3) by reversing the orientation of the $+i$ $-i$ strand.

Let $E_{\ii'+i-i\ii''}1_{\lambda} = E_{\ii' +i-i}1_{\mu}E_{\ii''}1_{\lambda}$,
that is, $\mu = \lambda + \ii''_X$ is the weight of the region to the right of
the strand near $-i$
\[
\vcenter{ \xy
  (-12,-12)*{}; (12,-12) **\dir{-};
  (-4,-12)*{};(-8,-4)*{} **\crv{ (-6,-6)}?(.25)*\dir{>};
  (4,-12)*{};(8,-4)*{} **\crv{ (6,-6)}?(.2)*\dir{<};
     (-4,-15)*{+ i};(4,-15)*{-i};(12,-7)*{\mu};
 \endxy}
 \]
We have in $\U$
\begin{equation}
E_{\ii'+i-i\ii''}1_{\lambda} = E_{\ii'-i+i\ii''}1_{\lambda}+ [\la
i,\mu\ra]_iE_{\ii'\ii''}1_{\lambda}.
\end{equation}
Therefore,
\begin{equation}
(E_{\ii'+i-i\ii''}1_{\lambda},1_{\lambda}) =
(E_{\ii'-i+i\ii''}1_{\lambda},1_{\lambda})+[\la
i,\mu\ra]_i(E_{\ii'\ii''}1_{\lambda},1_{\lambda}).
\end{equation}

Since $\deg\left(\; \xy (0,0)*{\xybox{ (-3,-4)*{};(3,4)*{} **\crv{(-3,-1) &
(3,1)}?(0)*\dir{<} ; (3,-4)*{};(-3,4)*{} **\crv{(3,-1) &
(-3,1)}?(1)*\dir{>};(8,1)*{ \mu};  }};\endxy \; \right)=\deg\left(\; \xy
(0,0)*{\xybox{ (-3,-4)*{};(3,4)*{} **\crv{(-3,-1) & (3,1)}?(1)*\dir{>} ;
(3,-4)*{};(-3,4)*{} **\crv{(3,-1) & (-3,1)}?(0)*\dir{<};(8,1)*{ \mu};  }};\endxy
\; \right)=0$, a diagram of class (1), respectively class (2), contributes to
$(E_{\ii'+i-i\ii''}1_{\lambda},1_{\lambda})'$ as much as its image in class (2),
respectively class (1), contributes to
$(E_{\ii'-i+i\ii''}1_{\lambda},1_{\lambda})'$.

A diagrams of class (3) in $p(\ii'+i-i\ii'',\emptyset)$ comes from some diagram
in $p(\ii'\ii'',\emptyset)$ by adding a $(+i-i)$ cap to the correct position.
Similarly, a diagram of class (3) in $p(\ii' -i+i\ii'',\emptyset)$ comes from a
diagram in $p(\ii'\ii'',\emptyset)$ by adding a $(-i+i)$ cap. Thus, the only
contribution to the difference $(E_{\ii'+i-i\ii''}1_{\lambda},1_{\lambda})'
-(E_{\ii'-i+i\ii''}1_{\lambda},1_{\lambda})'$ comes from class $(3)$ and is given
by
\begin{equation}
(E_{\ii'+i-i\ii''}1_{\lambda},1_{\lambda})'
-(E_{\ii'-i+i\ii''}1_{\lambda},1_{\lambda})' = \frac{q^{\deg
 \left(\;\xy
 (0,0)*{\xybox{ (-3,10)*{}; (3,-3)*{}; (3,0)*{}; (-3,0)*{} **\crv{(3,6) &
 (-3,6)}; ?(.05)*\dir{<}  ?(.9)*\dir{<} ?(.5)*\dir{}+(0,2)*{\scriptstyle i};
 (5,7)*{ \mu}; }}; \endxy
 \;\right)}}{1-q_i^2}
(E_{\ii'\ii''}1_{\lambda},1_{\lambda})' -\frac{q^{\deg\left(\; \xy (0,0)*{\xybox{
  (-3,10)*{}; (3,-3)*{};
  (3,0)*{}; (-3,0)*{} **\crv{(3,6) & (-3,6)}; ?(.1)*\dir{>}  ?(.95)*\dir{>}
  ?(.5)*\dir{}+(0,2)*{\scriptstyle i};
    (5,7)*{ \mu};
    }}; \endxy  \;\right)}}{1-q_i^2}
 (E_{\ii'\ii''}1_{\lambda},1_{\lambda})' \nn
\end{equation}
or
\begin{eqnarray}
 (E_{\ii'+i-i\ii''}1_{\lambda},1_{\lambda})' & =&
 (E_{\ii'-i+i\ii''}1_{\lambda},1_{\lambda})' +
  \frac{q_i^{1-\la i, \mu \ra}
       -q_i^{1+ \la i, \mu \ra} }{1-q_i^2}
 (E_{\ii'\ii''}1_{\lambda},1_{\lambda})' \nn \\
 &=&
 (E_{\ii'-i+i\ii''}1_{\lambda},1_{\lambda})'+
  \frac{q_i^{ \la i, \mu \ra}
       -q_i^{-\la i, \mu \ra} }{q_i-q_i^{-1}}
 (E_{\ii'\ii''}1_{\lambda},1_{\lambda})' \nn
\end{eqnarray}
where the coefficient of $(E_{\ii'\ii''}1_{\lambda},1_{\lambda})'$ in the last
term is easily recognized as $[\la i, \mu \ra]_i$, completing the proof of the
lemma.
\end{proof}

We can finish the proof that
\begin{equation}
  (E_{\ii}1_{\lambda},1_{\lambda})' = (E_{\ii}1_{\lambda},1_{\lambda})
\end{equation}
by induction on $\parallel i \parallel$, the length of $\ii$.  During the
induction step we move all negative entries of $\ii$ to the left of all positive
entries, sometimes adding terms $(E_{\jj}1_{\lambda},1_{\lambda})'$, respectively
$(E_{\jj}1_{\lambda},1_{\lambda})$ with $\parallel\jj\parallel
=\parallel\ii\parallel-2$ to the equation. We can then reduce to the case when
all negative entries of $\ii$ precede all positive entries, which is
Lemma~\ref{lem_minum_then_plus}. Theorem \ref{thm_form_formula} follows.
\end{proof}

We can turn this proof around to define $\U$ in a more geometric way than in
Lusztig~\cite{Lus4}.  Start with the $\Q(q)$-algebra $\Upr$ with
mutually--orthogonal idempotents $1_{\lambda}$ and basis
$\{E_{\ii}1_{\lambda}\}_{i,\lambda}$, over all finite signed sequences $\ii \in
\sseq$ and $\lambda \in X$. The multiplication is
\begin{equation}
  E_{\ii'}1_{\mu}E_{\ii}1_{\lambda} = \left\{
   \begin{array}{ccl}
   E_{\ii'\ii}1_{\lambda} & \quad &\text{if $\mu = \lambda+\ii_X$} ,\\
   0 & \quad&\text{otherwise.}
   \end{array}
   \right.
\end{equation}
When $\ii$ is the empty sequence, $E_{\emptyset}1_{\lambda} = 1_{\lambda}$.

Define a $\Q(q)$-bilinear form $(,)'$ on $\Upr$ via the sum over diagrams,
formula \eqref{eq_pariing prime}. Let $\cal{I} \subset \Upr$ be the kernel of
this bilinear form.  Then
\begin{equation}
\cal{I}=\bigoplus_{\mu,\lambda \in X} {_\mu \cal{I}_{\lambda}}, \qquad \qquad
{_{\mu}\cal{I}_{\lambda}}
 := \cal{I} \cap {_{\mu}\Upr_{\lambda}},
\end{equation}
where ${_{\mu}\Upr_{\lambda}}$ is spanned by $E_{\ii}1_{\lambda}$ for all $\ii$
such that $\mu = \lambda+\ii_X$. It follows from the definition of the bilinear
form that $\cal{I}$ is an ideal of $\Upr$.

It is not hard to check that
\begin{equation}
  \U \cong \Upr / \cal{I}  \qquad \text{and} \qquad (,)'=(,),
\end{equation}
following the above proof of Theorem~\ref{thm_form_formula}.  This definition is
not too far off from Lusztig's original definition, which utilizes ${\bf f} \cong
{\bf U}^{\pm}$, defined as the quotient of the free associative algebra $'{\bf
f}$ by the kernel of a bilinear form on $'{\bf f}$.  The latter bilinear form is
the restriction of $(,)'$ on $\Upr$ to $'{\bf f} \subset \Upr$.  Here we map
$'{\bf f} \to \Upr$ by sending $\theta_{\ii}=\theta_{i_1}\dots\theta_{i_m}$ to
$E_{\ii}1_{\lambda}=E_{+i_1}\dots E_{+i_m}1_{\lambda}$ for each positive sequence
$\ii$ (this map is not a homomorphism), alternatively we can send $\theta_{\ii}$
to $E_{-\ii}1_{\lambda}$.
\[
 \xy
 (-25,-10)*++{ '{\bf f} }="tl";(0,-10)*++{ {\bf f}}="tm";(25,-10)*++{{}_{\cal{A}}{\bf f} }="tr";
  (-25,10)*++{ \Upr}="bl";(0,10)*++{\U}="bm";(25,10)*++{\UA}="br";
  {\ar@{_{(}->} "tl";"bl"};{\ar@{_{(}->} "tm";"bm"};{\ar@{_{(}->} "tr";"br"};
  {\ar@{->>} "tl";"tm"};{\ar@{->>} "bl";"bm"};{\ar@{_{(}->} "tr";"tm"};{\ar@{_{(}->} "br";"bm"};
  (60,0)*{\left(\txt{vertical arrows \\do not respect \\ algebra structure}\right)};
  (-12.5,-16)*{\underbrace{\hspace{1.2in}}};(-12.5,-20)*{\text{$\Q(q)$-algebras}};
   (25,-16)*{\underbrace{\hspace{.6in}}};(29,-20)*{\text{$\Z[q,q^{-1}]$-algebras}};
 \endxy
\]
For each $\lambda \in X$ there are two inclusions (corresponding to ${\bf U}^+$
and ${\bf U}^-$) of the lower half of the diagram to the upper half. Restricting
to weight spaces, we get the following diagram
\[
 \xy
 (-35,-10)*++{ '{\bf f}_{\nu} }="tl";(0,-10)*++{ {\bf f}_{\nu}}="tm";
 (35,-10)*++{{}_{\cal{A}}{\bf f} _{\nu}}="tr";
  (-35,10)*++{ {}_{\lambda\pm\nu_X}\Upr_{\lambda}}="bl";(0,10)*++{{}_{\lambda\pm\nu_X}\U_{\lambda}}="bm";
  (35,10)*++{{}_{\lambda\pm\nu_X}(\UA)_{\lambda}}="br";
  {\ar@{_{(}->} "tl";"bl"};{\ar@{_{(}->} "tm";"bm"};{\ar@{_{(}->} "tr";"br"};
  {\ar@{->>} "tl";"tm"};{\ar@{->>} "bl";"bm"};{\ar@{_{(}->} "tr";"tm"};{\ar@{_{(}->} "br";"bm"};
  (-17.5,-16)*{\underbrace{\hspace{1.9in}}};(-17.5,-20)*{\text{$\Q(q)$-vector spaces}};
   (35,-16)*{\underbrace{\hspace{.6in}}};(39,-20)*{\text{$\Z[q,q^{-1}]$-modules}};
 \endxy
\]
where ${}_{\lambda\pm\nu_X}\U_{\lambda}=1_{\lambda\pm\nu_X}\U1_{\lambda}$, etc.

%
\section{Graphical calculus for $\U$ categorification} \label{sec_graphical}
%

%
\subsection{The 2-category $\Ucat$} \label{subsec_Uprime}
%

%
\subsubsection{Definition} \label{subsubsec_definition}
%

We define a 2-category $\Ucat$ for any root datum $(Y,X,\la,\ra, \dots)$ of type
$(I,\cdot)$.  This 2-category has the structure of an additive $\Bbbk$-linear
2-category, see \cite{GK} and \cite[Section 5]{Lau1}.  Thus the hom sets between
any two objects form a $\Bbbk$-linear category, and composition and identities
are given by additive $\Bbbk$-linear functors. The 2-morphisms in $\Ucat$ are
represented graphically using string diagrams, see \cite[Section 4]{Lau1} and the
references therein.

\begin{defn} \label{def_Ucat}
Let $(Y,X,\la,\ra, \dots)$ be a root datum of type $(I,\cdot)$. $\Ucat$ is an
additive $\Bbbk$-linear 2-category. The 2-category $\Ucat$ consists of
\begin{itemize}
  \item objects: $\lambda$ for $\lambda \in X$.
\end{itemize}
The homs $\Ucat(\lambda,\lambda')$ between two objects $\lambda$, $\lambda'$ are
additive $\Bbbk$-linear categories consisting of:
\begin{itemize}
  \item objects\footnote{We refer to objects of the category
$\Ucat(\lambda,\lambda')$ as 1-morphisms of $\Ucat$.  Likewise, the morphisms of
$\Ucat(\lambda,\lambda')$ are called 2-morphisms in $\Ucat$. } of
$\Ucat(\lambda,\lambda')$: a 1-morphism in $\Ucat$ from $\lambda$ to $\lambda'$
is a formal finite direct sum of 1-morphisms
  \[
 \cal{E}_{\ii} \onel\{t\} =\onelp \cal{E}_{\ii} \onel\{t\}
  \]
for any $t\in \Z$ and signed sequence $\ii \in \sseq$ such that
$\lambda'=\lambda+\ii_X$.
  \item morphisms of $\Ucat(\lambda,\lambda')$: for 1-morphisms $\cal{E}_{\ii} \onel\{t\}
  ,\cal{E}_{\jj} \onel\{t'\} \in \Ucat$, hom
sets $\Ucat(\cal{E}_{\ii} \onel\{t\},\cal{E}_{\jj} \onel\{t'\})$ of
$\Ucat(\lambda,\lambda')$ are graded $\Bbbk$-vector spaces given by linear
combinations of degree $t-t'$ diagrams, modulo certain relations, built from
composites of:
\begin{enumerate}[i)]
  \item  Degree zero identity 2-morphisms $1_x$ for each 1-morphism $x$ in
$\Ucat$; the identity 2-morphisms $1_{\cal{E}_{+i} \onel}\{t\}$ and
$1_{\cal{E}_{-i} \onel}\{t\}$, for $i \in I$, are represented graphically by
\[
\begin{array}{ccc}
  1_{\cal{E}_{+i} \onel\{t\}} &\quad  & 1_{\cal{E}_{-i} \onel\{t\}} \\ \\
    \xy
 (0,8);(0,-8); **\dir{-} ?(.5)*\dir{>}+(2.3,0)*{\scriptstyle{}};
 (0,-11)*{ i};(0,11)*{ i};
 (6,2)*{ \lambda};
 (-8,2)*{ \lambda +i_X};
 (-10,0)*{};(10,0)*{};
 \endxy
 & &
 \;\;   \xy
 (0,8);(0,-8); **\dir{-} ?(.5)*\dir{<}+(2.3,0)*{\scriptstyle{}};
 (0,-11)*{ i};(0,11)*{i};
 (6,2)*{ \lambda};
 (-8,2)*{ \lambda -i_X};
 (-12,0)*{};(12,0)*{};
 \endxy
\\ \\
   \;\;\text{ {\rm deg} 0}\;\;
 & &\;\;\text{ {\rm deg} 0}\;\;
\end{array}
\]
and more generally, for a signed sequence $\ii=\epsilon_1i_1 \epsilon_2i_2 \dots
\epsilon_mi_m$, the identity $1_{\cal{E}_{\ii} \onel\{t\}}$ 2-morphism is
represented as
\begin{equation*}
\begin{array}{ccc}
  \xy
 (-12,8);(-12,-8); **\dir{-};
 (-4,8);(-4,-8); **\dir{-};
 (4,0)*{\cdots};
 (12,8);(12,-8); **\dir{-};
 (-12,11)*{i_1}; (-4,11)*{ i_2};(12,11)*{ i_m };
  (-12,-11)*{ i_1}; (-4,-11)*{ i_2};(12,-11)*{ i_m};
 (18,2)*{ \lambda}; (-20,2)*{ \lambda+\ii_X};
 \endxy
\end{array}
\end{equation*}
where the strand labelled $i_{\alpha}$ is oriented up if $\epsilon_{\alpha}=+$
and oriented down if $\epsilon_{\alpha}=-$. We will often place labels with no
sign on the side of a strand and omit the labels at the top and bottom.  The
signs can be recovered from the orientations on the strands as explained in
Section~\ref{subsec_geometric}.

  \item For each $\lambda \in X$ the 2-morphisms
\[
\begin{tabular}{|l|c|c|c|c|}
\hline
 {\bf Notation:} \xy (0,-5)*{};(0,7)*{}; \endxy&
 $\Uup_{i,\lambda}$  &  $\Udown_{i,\lambda}$  &$\Ucross_{i,j,\lambda}$
 &$\Ucrossd_{i,j,\lambda}$  \\
 \hline
 {\bf 2-morphism:} &   \xy
 (0,7);(0,-7); **\dir{-} ?(.75)*\dir{>}+(2.3,0)*{\scriptstyle{}}
 ?(.1)*\dir{ }+(2,0)*{\scs i};
 (0,-2)*{\txt\large{$\bullet$}};
 (6,4)*{ \lambda};
 (-8,4)*{ \lambda +i_X};
 (-10,0)*{};(10,0)*{};
 \endxy
 &
     \xy
 (0,7);(0,-7); **\dir{-} ?(.75)*\dir{<}+(2.3,0)*{\scriptstyle{}}
 ?(.1)*\dir{ }+(2,0)*{\scs i};
 (0,-2)*{\txt\large{$\bullet$}};
 (-6,4)*{ \lambda};
 (8,4)*{ \lambda +i_X};
 (-10,0)*{};(10,9)*{};
 \endxy
 &
   \xy
  (0,0)*{\xybox{
    (-4,-4)*{};(4,4)*{} **\crv{(-4,-1) & (4,1)}?(1)*\dir{>} ;
    (4,-4)*{};(-4,4)*{} **\crv{(4,-1) & (-4,1)}?(1)*\dir{>};
    (-5,-3)*{\scs i};
     (5.1,-3)*{\scs j};
     (8,1)*{ \lambda};
     (-12,0)*{};(12,0)*{};
     }};
  \endxy
 &
   \xy
  (0,0)*{\xybox{
    (-4,4)*{};(4,-4)*{} **\crv{(-4,1) & (4,-1)}?(1)*\dir{>} ;
    (4,4)*{};(-4,-4)*{} **\crv{(4,1) & (-4,-1)}?(1)*\dir{>};
    (-6,-3)*{\scs i};
     (6,-3)*{\scs j};
     (8,1)*{ \lambda};
     (-12,0)*{};(12,0)*{};
     }};
  \endxy
\\ & & & &\\
\hline
 {\bf Degree:} & \;\;\text{  $i \cdot i$ }\;\;
 &\;\;\text{  $i\cdot i$}\;\;& \;\;\text{  $-i \cdot j$}\;\;
 & \;\;\text{  $-i \cdot j$}\;\; \\
 \hline
\end{tabular}
\]

\[
\begin{tabular}{|l|c|c|c|c|}
\hline
  {\bf Notation:} \xy (0,-5)*{};(0,7)*{}; \endxy&\text{$\Ucupr_{i,\lambda}$} & \text{$\Ucupl_{i,\lambda}$} & \text{$\Ucapl_{i,\lambda}$} &
 \text{$\Ucapr_{i,\lambda}$} \\
 \hline
  {\bf 2-morphism:} &  \xy
    (0,-3)*{\bbpef{i}};
    (8,-5)*{ \lambda};
    (-12,0)*{};(12,0)*{};
    \endxy
  & \xy
    (0,-3)*{\bbpfe{i}};
    (8,-5)*{ \lambda};
    (-12,0)*{};(12,0)*{};
    \endxy
  & \xy
    (0,0)*{\bbcef{i}};
    (8,5)*{ \lambda};
    (-12,0)*{};(12,0)*{};
    \endxy
  & \xy
    (0,0)*{\bbcfe{i}};
    (8,5)*{ \lambda};
    (-12,0)*{};(12,0)*{};
    \endxy\\& & &  &\\ \hline
 {\bf Degree:} & \;\;\text{  $c_{+i,\lambda}$}\;\;
 & \;\;\text{ $c_{-i,\lambda}$}\;\;
 & \;\;\text{ $c_{+i,\lambda}$}\;\;
 & \;\;\text{  $c_{-i,\lambda}$}\;\;
 \\
 \hline
\end{tabular}
\]
\end{enumerate}
with $c_{\pm i,\lambda}$ defined in \eqref{eq_cpm_lambda}, such that the
following identities hold.

\item The $\mathfrak{sl}_2$ relations\footnote{
The vertical $ii$-crossing was represented by the diagram $ \xy 0;/r.18pc/:
    (0,0)*{\twoIu{}{}};
    (6,0)*{};
    (-8,0)*{};
    \endxy
$ in \cite{Lau1,Lau2}.  Here we use a standard crossing for simplicity.} ( all of
the strands are labelled by $i$):
\begin{enumerate}[i)]
\item  $\mathbf{1}_{\lambda+i_X}\cal{E}_{+i}\onel$ and
$\onel\cal{E}_{-i}\mathbf{1}_{\lambda+i_X}$ are biadjoint, up to grading shifts:
\begin{equation} \label{eq_biadjoint1}
  \xy   0;/r.18pc/:
    (-8,0)*{}="1";
    (0,0)*{}="2";
    (8,0)*{}="3";
    (-8,-10);"1" **\dir{-};
    "1";"2" **\crv{(-8,8) & (0,8)} ?(0)*\dir{>} ?(1)*\dir{>};
    "2";"3" **\crv{(0,-8) & (8,-8)}?(1)*\dir{>};
    "3"; (8,10) **\dir{-};
    (12,-9)*{\lambda};
    (-6,9)*{\lambda+i_X};
    \endxy
    \; =
    \;
\xy   0;/r.18pc/:
    (-8,0)*{}="1";
    (0,0)*{}="2";
    (8,0)*{}="3";
    (0,-10);(0,10)**\dir{-} ?(.5)*\dir{>};
    (5,8)*{\lambda};
    (-9,8)*{\lambda+i_X};
    \endxy
\qquad \quad  \xy   0;/r.18pc/:
    (-8,0)*{}="1";
    (0,0)*{}="2";
    (8,0)*{}="3";
    (-8,-10);"1" **\dir{-};
    "1";"2" **\crv{(-8,8) & (0,8)} ?(0)*\dir{<} ?(1)*\dir{<};
    "2";"3" **\crv{(0,-8) & (8,-8)}?(1)*\dir{<};
    "3"; (8,10) **\dir{-};
    (12,-9)*{\lambda+i_X};
    (-6,9)*{ \lambda};
    \endxy
    \; =
    \;
\xy   0;/r.18pc/:
    (-8,0)*{}="1";
    (0,0)*{}="2";
    (8,0)*{}="3";
    (0,-10);(0,10)**\dir{-} ?(.5)*\dir{<};
   (9,8)*{\lambda+i_X};
    (-6,8)*{ \lambda};
    \endxy
\end{equation}

\begin{equation}
 \xy   0;/r.18pc/:
    (8,0)*{}="1";
    (0,0)*{}="2";
    (-8,0)*{}="3";
    (8,-10);"1" **\dir{-};
    "1";"2" **\crv{(8,8) & (0,8)} ?(0)*\dir{>} ?(1)*\dir{>};
    "2";"3" **\crv{(0,-8) & (-8,-8)}?(1)*\dir{>};
    "3"; (-8,10) **\dir{-};
    (12,9)*{\lambda};
    (-5,-9)*{\lambda+i_X};
    \endxy
    \; =
    \;
      \xy 0;/r.18pc/:
    (8,0)*{}="1";
    (0,0)*{}="2";
    (-8,0)*{}="3";
    (0,-10);(0,10)**\dir{-} ?(.5)*\dir{>};
    (5,-8)*{\lambda};
    (-9,-8)*{\lambda+i_X};
    \endxy
\qquad \quad \xy  0;/r.18pc/:
    (8,0)*{}="1";
    (0,0)*{}="2";
    (-8,0)*{}="3";
    (8,-10);"1" **\dir{-};
    "1";"2" **\crv{(8,8) & (0,8)} ?(0)*\dir{<} ?(1)*\dir{<};
    "2";"3" **\crv{(0,-8) & (-8,-8)}?(1)*\dir{<};
    "3"; (-8,10) **\dir{-};
    (12,9)*{\lambda+i_X};
    (-6,-9)*{ \lambda};
    \endxy
    \; =
    \;
\xy  0;/r.18pc/:
    (8,0)*{}="1";
    (0,0)*{}="2";
    (-8,0)*{}="3";
    (0,-10);(0,10)**\dir{-} ?(.5)*\dir{<};
    (9,-8)*{\lambda+i_X};
    (-6,-8)*{ \lambda};
    \endxy
\end{equation}
\item
\begin{equation} \label{eq_cyclic_dot}
    \xy
    (-8,5)*{}="1";
    (0,5)*{}="2";
    (0,-5)*{}="2'";
    (8,-5)*{}="3";
    (-8,-10);"1" **\dir{-};
    "2";"2'" **\dir{-} ?(.5)*\dir{<};
    "1";"2" **\crv{(-8,12) & (0,12)} ?(0)*\dir{<};
    "2'";"3" **\crv{(0,-12) & (8,-12)}?(1)*\dir{<};
    "3"; (8,10) **\dir{-};
    (15,-9)*{ \lambda+i_X};
    (-12,9)*{\lambda};
    (0,4)*{\txt\large{$\bullet$}};
    (10,8)*{\scs };
    (-10,-8)*{\scs i};
    \endxy
    \quad = \quad
      \xy
 (0,10);(0,-10); **\dir{-} ?(.75)*\dir{<}+(2.3,0)*{\scriptstyle{}}
 ?(.1)*\dir{ }+(2,0)*{\scs };
 (0,0)*{\txt\large{$\bullet$}};
 (-6,5)*{ \lambda};
 (8,5)*{ \lambda +i_X};
 (-10,0)*{};(10,0)*{};(-2,-8)*{\scs i};
 \endxy
    \quad = \quad
    \xy
    (8,5)*{}="1";
    (0,5)*{}="2";
    (0,-5)*{}="2'";
    (-8,-5)*{}="3";
    (8,-10);"1" **\dir{-};
    "2";"2'" **\dir{-} ?(.5)*\dir{<};
    "1";"2" **\crv{(8,12) & (0,12)} ?(0)*\dir{<};
    "2'";"3" **\crv{(0,-12) & (-8,-12)}?(1)*\dir{<};
    "3"; (-8,10) **\dir{-};
    (15,9)*{\lambda+i_X};
    (-12,-9)*{\lambda};
    (0,4)*{\txt\large{$\bullet$}};
    (-10,8)*{\scs };
    (10,-8)*{\scs i};
    \endxy
\end{equation}
\item  All dotted bubbles of negative degree are zero. That is,
\begin{equation} \label{eq_positivity_bubbles}
 \xy
 (-12,0)*{\cbub{\alpha}{i}};
 (-8,8)*{\lambda};
 \endxy
  = 0
 \qquad
  \text{if $\alpha<\la i,\lambda\ra-1$} \qquad
 \xy
 (-12,0)*{\ccbub{\alpha}{i}};
 (-8,8)*{\lambda};
 \endxy = 0\quad
  \text{if $\alpha< -\la i,\lambda\ra-1$}
\end{equation}
for all $\alpha \in \Z_+$, where a dot carrying a label $\alpha$ denotes the
$\alpha$-fold iterated vertical composite of $\Uup_{i,\lambda}$ or
$\Udown_{i,\lambda}$ depending on the orientation.  A dotted bubble of degree
zero equals 1:
\[
\xy 0;/r.18pc/:
 (0,0)*{\cbub{\la i,\lambda\ra-1}{i}};
  (4,8)*{\lambda};
 \endxy
  = 1 \quad \text{for $\la i,\lambda\ra \geq 1$,}
  \qquad \quad
  \xy 0;/r.18pc/:
 (0,0)*{\ccbub{-\la i,\lambda\ra-1}{i}};
  (4,8)*{\lambda};
 \endxy
  = 1 \quad \text{for $\la i,\lambda\ra \leq -1$.}
\]
\item For the following relations we employ the convention that all summations
are increasing, so that $\sum_{f=0}^{\alpha}$ is zero if $\alpha < 0$.
\begin{eqnarray}
  \text{$\xy 0;/r.18pc/:
  (14,8)*{\lambda};
  (-3,-8)*{};(3,8)*{} **\crv{(-3,-1) & (3,1)}?(1)*\dir{>};?(0)*\dir{>};
    (3,-8)*{};(-3,8)*{} **\crv{(3,-1) & (-3,1)}?(1)*\dir{>};
  (-3,-12)*{\bbsid};
  (-3,8)*{\bbsid};
  (3,8)*{}="t1";
  (9,8)*{}="t2";
  (3,-8)*{}="t1'";
  (9,-8)*{}="t2'";
   "t1";"t2" **\crv{(3,14) & (9, 14)};
   "t1'";"t2'" **\crv{(3,-14) & (9, -14)};
   "t2'";"t2" **\dir{-} ?(.5)*\dir{<};
   (9,0)*{}; (-6,-8)*{\scs i};
 \endxy$} \;\; = \;\; -\sum_{f=0}^{-\la i,\lambda\ra}
   \xy
  (19,4)*{\lambda};
  (0,0)*{\bbe{}};(-2,-8)*{\scs i};
  (12,-2)*{\cbub{\la i,\lambda\ra-1+f}{i}};
  (0,6)*{\bullet}+(8,1)*{\scs -\la i,\lambda\ra-f};
 \endxy
\qquad \qquad
  \text{$ \xy 0;/r.18pc/:
  (-12,8)*{\lambda};
   (-3,-8)*{};(3,8)*{} **\crv{(-3,-1) & (3,1)}?(1)*\dir{>};?(0)*\dir{>};
    (3,-8)*{};(-3,8)*{} **\crv{(3,-1) & (-3,1)}?(1)*\dir{>};
  (3,-12)*{\bbsid};
  (3,8)*{\bbsid}; (6,-8)*{\scs i};
  (-9,8)*{}="t1";
  (-3,8)*{}="t2";
  (-9,-8)*{}="t1'";
  (-3,-8)*{}="t2'";
   "t1";"t2" **\crv{(-9,14) & (-3, 14)};
   "t1'";"t2'" **\crv{(-9,-14) & (-3, -14)};
  "t1'";"t1" **\dir{-} ?(.5)*\dir{<};
 \endxy$} \;\; = \;\;
 \sum_{g=0}^{\la i,\lambda\ra}
   \xy
  (-12,8)*{\lambda};
  (0,0)*{\bbe{}};(2,-8)*{\scs i};
  (-12,-2)*{\ccbub{-\la i,\lambda\ra-1+g}{i}};
  (0,6)*{\bullet}+(8,-1)*{\scs \la i,\lambda\ra-g};
 \endxy
\end{eqnarray}
\begin{eqnarray}
 \vcenter{\xy 0;/r.18pc/:
  (-8,0)*{};
  (8,0)*{};
  (-4,10)*{}="t1";
  (4,10)*{}="t2";
  (-4,-10)*{}="b1";
  (4,-10)*{}="b2";(-6,-8)*{\scs i};(6,-8)*{\scs i};
  "t1";"b1" **\dir{-} ?(.5)*\dir{<};
  "t2";"b2" **\dir{-} ?(.5)*\dir{>};
  (10,2)*{\lambda};
  (-10,2)*{\lambda};
  \endxy}
&\quad = \quad&
 -\;\;
 \vcenter{   \xy 0;/r.18pc/:
    (-4,-4)*{};(4,4)*{} **\crv{(-4,-1) & (4,1)}?(1)*\dir{>};
    (4,-4)*{};(-4,4)*{} **\crv{(4,-1) & (-4,1)}?(1)*\dir{<};?(0)*\dir{<};
    (-4,4)*{};(4,12)*{} **\crv{(-4,7) & (4,9)};
    (4,4)*{};(-4,12)*{} **\crv{(4,7) & (-4,9)}?(1)*\dir{>};
  (8,8)*{\lambda};(-6,-3)*{\scs i};
     (6.5,-3)*{\scs i};
 \endxy}
  \quad + \quad
   \sum_{f=0}^{\la i,\lambda\ra-1} \sum_{g=0}^{f}
    \vcenter{\xy 0;/r.18pc/:
    (-10,10)*{\lambda};
    (-8,0)*{};
  (8,0)*{};
  (-4,-15)*{}="b1";
  (4,-15)*{}="b2";
  "b2";"b1" **\crv{(5,-8) & (-5,-8)}; ?(.05)*\dir{<} ?(.93)*\dir{<}
  ?(.8)*\dir{}+(0,-.1)*{\bullet}+(-5,2)*{\scs f-g};
  (-4,15)*{}="t1";
  (4,15)*{}="t2";
  "t2";"t1" **\crv{(5,8) & (-5,8)}; ?(.15)*\dir{>} ?(.95)*\dir{>}
  ?(.4)*\dir{}+(0,-.2)*{\bullet}+(3,-2)*{\scs \;\;\; \la i,\lambda\ra-1-f};
  (0,0)*{\ccbub{\scs \quad\;\;\;-\la i,\lambda\ra-1+g}{i}};
  \endxy} \nn
 \\  \; \nn \\
 \vcenter{\xy 0;/r.18pc/:
  (-8,0)*{};(-6,-8)*{\scs i};(6,-8)*{\scs i};
  (8,0)*{};
  (-4,10)*{}="t1";
  (4,10)*{}="t2";
  (-4,-10)*{}="b1";
  (4,-10)*{}="b2";
  "t1";"b1" **\dir{-} ?(.5)*\dir{>};
  "t2";"b2" **\dir{-} ?(.5)*\dir{<};
  (10,2)*{\lambda};
  (-10,2)*{\lambda};
  \endxy}
&\quad = \quad&
 -\;\;
   \vcenter{\xy 0;/r.18pc/:
    (-4,-4)*{};(4,4)*{} **\crv{(-4,-1) & (4,1)}?(1)*\dir{<};?(0)*\dir{<};
    (4,-4)*{};(-4,4)*{} **\crv{(4,-1) & (-4,1)}?(1)*\dir{>};
    (-4,4)*{};(4,12)*{} **\crv{(-4,7) & (4,9)}?(1)*\dir{>};
    (4,4)*{};(-4,12)*{} **\crv{(4,7) & (-4,9)};
  (8,8)*{\lambda};(-6,-3)*{\scs i};
     (6,-3)*{\scs i};
 \endxy}
  \quad + \quad
\sum_{f=0}^{-\la i,\lambda\ra-1} \sum_{g=0}^{f}
    \vcenter{\xy 0;/r.18pc/:
    (-8,0)*{};
  (8,0)*{};
  (-4,-15)*{}="b1";
  (4,-15)*{}="b2";
  "b2";"b1" **\crv{(5,-8) & (-5,-8)}; ?(.1)*\dir{>} ?(.95)*\dir{>}
  ?(.8)*\dir{}+(0,-.1)*{\bullet}+(-5,2)*{\scs f-g};
  (-4,15)*{}="t1";
  (4,15)*{}="t2";
  "t2";"t1" **\crv{(5,8) & (-5,8)}; ?(.15)*\dir{<} ?(.97)*\dir{<}
  ?(.4)*\dir{}+(0,-.2)*{\bullet}+(3,-2)*{\scs \;\;-\la i,\lambda\ra-1-f};
  (0,0)*{\cbub{\scs \quad\; \la i,\lambda\ra-1+g}{i}};
  (-10,10)*{\lambda};
  \endxy} \label{eq_ident_decomp}
\end{eqnarray}
for all $\lambda\in X$.  Notice that for some values of $\lambda$ the dotted
bubbles appearing above have negative labels. A composite of $\Uup_{i,\lambda}$
or $\Udown_{i,\lambda}$ with itself a negative number of times does not make
sense. These dotted bubbles with negative labels, called {\em fake bubbles}, are
formal symbols inductively defined by the equation
\begin{center}
\begin{eqnarray}
 \makebox[0pt]{ $
\left( \xy 0;/r.15pc/:
 (0,0)*{\ccbub{-\la i,\lambda\ra-1}{i}};
  (4,8)*{\lambda};
 \endxy
 +
 \xy 0;/r.15pc/:
 (0,0)*{\ccbub{-\la i,\lambda\ra-1+1}{i}};
  (4,8)*{\lambda};
 \endxy t
 + \cdots +
 \xy 0;/r.15pc/:
 (0,0)*{\ccbub{-\la i,\lambda\ra-1+\alpha}{i}};
  (4,8)*{\lambda};
 \endxy t^{\alpha}
 + \cdots
\right)
%
\left( \xy 0;/r.15pc/:
 (0,0)*{\cbub{\la i,\lambda\ra-1}{i}};
  (4,8)*{\lambda};
 \endxy
 + \cdots +
 \xy 0;/r.15pc/:
 (0,0)*{\cbub{\la i,\lambda\ra-1+\alpha}{i}};
 (4,8)*{\lambda};
 \endxy t^{\alpha}
 + \cdots
\right) =1 .$ } \nn \\ \label{eq_infinite_Grass}
\end{eqnarray}
\end{center}
and the additional condition
\[
\xy 0;/r.18pc/:
 (0,0)*{\cbub{-1}{i}};
  (4,8)*{\lambda};
 \endxy
 \quad = \quad
  \xy 0;/r.18pc/:
 (0,0)*{\ccbub{-1}{i}};
  (4,8)*{\lambda};
 \endxy
  \quad = \quad 1 \quad \text{if $\la i,\lambda\ra =0$.}
\]
Although the labels are negative for fake bubbles, one can check that the overall
degree of each fake bubble is still positive, so that these fake bubbles do not
violate the positivity of dotted bubble axiom. The above equation, called the
infinite Grassmannian relation, remains valid even in high degree when most of
the bubbles involved are not fake bubbles.  See \cite{Lau1} for more details.
\item NilHecke relations:
 \begin{equation}
  \vcenter{\xy 0;/r.18pc/:
    (-4,-4)*{};(4,4)*{} **\crv{(-4,-1) & (4,1)}?(1)*\dir{>};
    (4,-4)*{};(-4,4)*{} **\crv{(4,-1) & (-4,1)}?(1)*\dir{>};
    (-4,4)*{};(4,12)*{} **\crv{(-4,7) & (4,9)}?(1)*\dir{>};
    (4,4)*{};(-4,12)*{} **\crv{(4,7) & (-4,9)}?(1)*\dir{>};
  (8,8)*{\lambda};(-5,-3)*{\scs i};
     (5.1,-3)*{\scs i};
 \endxy}
 =0, \qquad \quad
 \vcenter{
 \xy 0;/r.18pc/:
    (-4,-4)*{};(4,4)*{} **\crv{(-4,-1) & (4,1)}?(1)*\dir{>};
    (4,-4)*{};(-4,4)*{} **\crv{(4,-1) & (-4,1)}?(1)*\dir{>};
    (4,4)*{};(12,12)*{} **\crv{(4,7) & (12,9)}?(1)*\dir{>};
    (12,4)*{};(4,12)*{} **\crv{(12,7) & (4,9)}?(1)*\dir{>};
    (-4,12)*{};(4,20)*{} **\crv{(-4,15) & (4,17)}?(1)*\dir{>};
    (4,12)*{};(-4,20)*{} **\crv{(4,15) & (-4,17)}?(1)*\dir{>};
    (-4,4)*{}; (-4,12) **\dir{-};
    (12,-4)*{}; (12,4) **\dir{-};
    (12,12)*{}; (12,20) **\dir{-}; (-5.5,-3)*{\scs i};
     (5.5,-3)*{\scs i};(14,-3)*{\scs i};
  (18,8)*{\lambda};
\endxy}
 \;\; =\;\;
 \vcenter{
 \xy 0;/r.18pc/:
    (4,-4)*{};(-4,4)*{} **\crv{(4,-1) & (-4,1)}?(1)*\dir{>};
    (-4,-4)*{};(4,4)*{} **\crv{(-4,-1) & (4,1)}?(1)*\dir{>};
    (-4,4)*{};(-12,12)*{} **\crv{(-4,7) & (-12,9)}?(1)*\dir{>};
    (-12,4)*{};(-4,12)*{} **\crv{(-12,7) & (-4,9)}?(1)*\dir{>};
    (4,12)*{};(-4,20)*{} **\crv{(4,15) & (-4,17)}?(1)*\dir{>};
    (-4,12)*{};(4,20)*{} **\crv{(-4,15) & (4,17)}?(1)*\dir{>};
    (4,4)*{}; (4,12) **\dir{-};
    (-12,-4)*{}; (-12,4) **\dir{-};
    (-12,12)*{}; (-12,20) **\dir{-};(-5.5,-3)*{\scs i};
     (5.5,-3)*{\scs i};(-14,-3)*{\scs i};
  (10,8)*{\lambda};
\endxy} \label{eq_nil_rels}
  \end{equation}
\begin{eqnarray}
  \xy
  (4,4);(4,-4) **\dir{-}?(0)*\dir{<}+(2.3,0)*{};
  (-4,4);(-4,-4) **\dir{-}?(0)*\dir{<}+(2.3,0)*{};
  (9,2)*{\lambda};     (-5,-3)*{\scs i};
     (5.1,-3)*{\scs i};
 \endxy
 \quad =
\xy
  (0,0)*{\xybox{
    (-4,-4)*{};(4,4)*{} **\crv{(-4,-1) & (4,1)}?(1)*\dir{>}?(.25)*{\bullet};
    (4,-4)*{};(-4,4)*{} **\crv{(4,-1) & (-4,1)}?(1)*\dir{>};
    (-5,-3)*{\scs i};
     (5.1,-3)*{\scs i};
     (8,1)*{ \lambda};
     (-10,0)*{};(10,0)*{};
     }};
  \endxy
 \;\; -
 \xy
  (0,0)*{\xybox{
    (-4,-4)*{};(4,4)*{} **\crv{(-4,-1) & (4,1)}?(1)*\dir{>}?(.75)*{\bullet};
    (4,-4)*{};(-4,4)*{} **\crv{(4,-1) & (-4,1)}?(1)*\dir{>};
    (-5,-3)*{\scs i};
     (5.1,-3)*{\scs i};
     (8,1)*{ \lambda};
     (-10,0)*{};(10,0)*{};
     }};
  \endxy
 \;\; =
\xy
  (0,0)*{\xybox{
    (-4,-4)*{};(4,4)*{} **\crv{(-4,-1) & (4,1)}?(1)*\dir{>};
    (4,-4)*{};(-4,4)*{} **\crv{(4,-1) & (-4,1)}?(1)*\dir{>}?(.75)*{\bullet};
    (-5,-3)*{\scs i};
     (5.1,-3)*{\scs i};
     (8,1)*{ \lambda};
     (-10,0)*{};(10,0)*{};
     }};
  \endxy
 \;\; -
  \xy
  (0,0)*{\xybox{
    (-4,-4)*{};(4,4)*{} **\crv{(-4,-1) & (4,1)}?(1)*\dir{>} ;
    (4,-4)*{};(-4,4)*{} **\crv{(4,-1) & (-4,1)}?(1)*\dir{>}?(.25)*{\bullet};
    (-5,-3)*{\scs i};
     (5.1,-3)*{\scs i};
     (8,1)*{ \lambda};
     (-10,0)*{};(10,0)*{};
     }};
  \endxy \nn \\ \label{eq_nil_dotslide}
\end{eqnarray}
We will also include \eqref{eq_cyclic_cross-gen} for $i =j$ as an
$\mf{sl}_2$-relation.
\end{enumerate}

  \item All 2-morphisms are cyclic\footnote{See \cite{Lau1} and the references therein for
  the definition of a cyclic 2-morphism with respect to a biadjoint structure.} with respect to the above biadjoint
   structure.  This is ensured by the relations \eqref{eq_cyclic_dot}, and the
   relations
\begin{equation} \label{eq_cyclic_cross-gen}
\xy 0;/r.19pc/:
  (0,0)*{\xybox{
    (-4,-4)*{};(4,4)*{} **\crv{(-4,-1) & (4,1)}?(1)*\dir{>};
    (4,-4)*{};(-4,4)*{} **\crv{(4,-1) & (-4,1)};
     (4,4)*{};(-18,4)*{} **\crv{(4,16) & (-18,16)} ?(1)*\dir{>};
     (-4,-4)*{};(18,-4)*{} **\crv{(-4,-16) & (18,-16)} ?(1)*\dir{<}?(0)*\dir{<};
     (18,-4);(18,12) **\dir{-};(12,-4);(12,12) **\dir{-};
     (-18,4);(-18,-12) **\dir{-};(-12,4);(-12,-12) **\dir{-};
     (8,1)*{ \lambda};
     (-10,0)*{};(10,0)*{};
      (4,-4)*{};(12,-4)*{} **\crv{(4,-10) & (12,-10)}?(1)*\dir{<}?(0)*\dir{<};
      (-4,4)*{};(-12,4)*{} **\crv{(-4,10) & (-12,10)}?(1)*\dir{>}?(0)*\dir{>};
      (20,11)*{\scs j};(10,11)*{\scs i};
      (-20,-11)*{\scs j};(-10,-11)*{\scs i};
     }};
  \endxy
\quad =  \quad \xy
  (0,0)*{\xybox{
    (-4,-4)*{};(4,4)*{} **\crv{(-4,-1) & (4,1)}?(0)*\dir{<} ;
    (4,-4)*{};(-4,4)*{} **\crv{(4,-1) & (-4,1)}?(0)*\dir{<};
    (-5,3)*{\scs i};
     (5.1,3)*{\scs j};
     (-8,0)*{ \lambda};
     (-12,0)*{};(12,0)*{};
     }};
  \endxy \quad =  \quad
 \xy 0;/r.19pc/:
  (0,0)*{\xybox{
    (4,-4)*{};(-4,4)*{} **\crv{(4,-1) & (-4,1)}?(1)*\dir{>};
    (-4,-4)*{};(4,4)*{} **\crv{(-4,-1) & (4,1)};
     (-4,4)*{};(18,4)*{} **\crv{(-4,16) & (18,16)} ?(1)*\dir{>};
     (4,-4)*{};(-18,-4)*{} **\crv{(4,-16) & (-18,-16)} ?(1)*\dir{<}?(0)*\dir{<};
     (-18,-4);(-18,12) **\dir{-};(-12,-4);(-12,12) **\dir{-};
     (18,4);(18,-12) **\dir{-};(12,4);(12,-12) **\dir{-};
     (8,1)*{ \lambda};
     (-10,0)*{};(10,0)*{};
     (-4,-4)*{};(-12,-4)*{} **\crv{(-4,-10) & (-12,-10)}?(1)*\dir{<}?(0)*\dir{<};
      (4,4)*{};(12,4)*{} **\crv{(4,10) & (12,10)}?(1)*\dir{>}?(0)*\dir{>};
      (-20,11)*{\scs i};(-10,11)*{\scs j};
      (20,-11)*{\scs i};(10,-11)*{\scs j};
     }};
  \endxy
\end{equation}
The cyclic condition on 2-morphisms expressed by \eqref{eq_cyclic_dot} and
\eqref{eq_cyclic_cross-gen} ensures that diagrams related by isotopy represent
the same 2-morphism in $\Ucat$.

It will be convenient to introduce degree zero 2-morphisms:
\begin{equation} \label{eq_crossl-gen}
  \xy
  (0,0)*{\xybox{
    (-4,-4)*{};(4,4)*{} **\crv{(-4,-1) & (4,1)}?(1)*\dir{>} ;
    (4,-4)*{};(-4,4)*{} **\crv{(4,-1) & (-4,1)}?(0)*\dir{<};
    (-5,-3)*{\scs i};
     (-5,3)*{\scs j};
     (8,2)*{ \lambda};
     (-12,0)*{};(12,0)*{};
     }};
  \endxy
:=
 \xy 0;/r.19pc/:
  (0,0)*{\xybox{
    (4,-4)*{};(-4,4)*{} **\crv{(4,-1) & (-4,1)}?(1)*\dir{>};
    (-4,-4)*{};(4,4)*{} **\crv{(-4,-1) & (4,1)};
     (-4,4);(-4,12) **\dir{-};
     (-12,-4);(-12,12) **\dir{-};
     (4,-4);(4,-12) **\dir{-};(12,4);(12,-12) **\dir{-};
     (16,1)*{\lambda};
     (-10,0)*{};(10,0)*{};
     (-4,-4)*{};(-12,-4)*{} **\crv{(-4,-10) & (-12,-10)}?(1)*\dir{<}?(0)*\dir{<};
      (4,4)*{};(12,4)*{} **\crv{(4,10) & (12,10)}?(1)*\dir{>}?(0)*\dir{>};
      (-14,11)*{\scs j};(-2,11)*{\scs i};
      (14,-11)*{\scs j};(2,-11)*{\scs i};
     }};
  \endxy
  \quad = \quad
  \xy 0;/r.19pc/:
  (0,0)*{\xybox{
    (-4,-4)*{};(4,4)*{} **\crv{(-4,-1) & (4,1)}?(1)*\dir{<};
    (4,-4)*{};(-4,4)*{} **\crv{(4,-1) & (-4,1)};
     (4,4);(4,12) **\dir{-};
     (12,-4);(12,12) **\dir{-};
     (-4,-4);(-4,-12) **\dir{-};(-12,4);(-12,-12) **\dir{-};
     (16,1)*{\lambda};
     (10,0)*{};(-10,0)*{};
     (4,-4)*{};(12,-4)*{} **\crv{(4,-10) & (12,-10)}?(1)*\dir{>}?(0)*\dir{>};
      (-4,4)*{};(-12,4)*{} **\crv{(-4,10) & (-12,10)}?(1)*\dir{<}?(0)*\dir{<};
      (14,11)*{\scs i};(2,11)*{\scs j};
      (-14,-11)*{\scs i};(-2,-11)*{\scs j};
     }};
  \endxy
\end{equation}
\begin{equation} \label{eq_crossr-gen}
  \xy
  (0,0)*{\xybox{
    (-4,-4)*{};(4,4)*{} **\crv{(-4,-1) & (4,1)}?(0)*\dir{<} ;
    (4,-4)*{};(-4,4)*{} **\crv{(4,-1) & (-4,1)}?(1)*\dir{>};
    (5.1,-3)*{\scs i};
     (5.1,3)*{\scs j};
     (-8,2)*{ \lambda};
     (-12,0)*{};(12,0)*{};
     }};
  \endxy
:=
 \xy 0;/r.19pc/:
  (0,0)*{\xybox{
    (-4,-4)*{};(4,4)*{} **\crv{(-4,-1) & (4,1)}?(1)*\dir{>};
    (4,-4)*{};(-4,4)*{} **\crv{(4,-1) & (-4,1)};
     (4,4);(4,12) **\dir{-};
     (12,-4);(12,12) **\dir{-};
     (-4,-4);(-4,-12) **\dir{-};(-12,4);(-12,-12) **\dir{-};
     (-16,1)*{\lambda};
     (10,0)*{};(-10,0)*{};
     (4,-4)*{};(12,-4)*{} **\crv{(4,-10) & (12,-10)}?(1)*\dir{<}?(0)*\dir{<};
      (-4,4)*{};(-12,4)*{} **\crv{(-4,10) & (-12,10)}?(1)*\dir{>}?(0)*\dir{>};
      (14,11)*{\scs j};(2,11)*{\scs i};
      (-14,-11)*{\scs j};(-2,-11)*{\scs i};
     }};
  \endxy
  \quad = \quad
  \xy 0;/r.19pc/:
  (0,0)*{\xybox{
    (4,-4)*{};(-4,4)*{} **\crv{(4,-1) & (-4,1)}?(1)*\dir{<};
    (-4,-4)*{};(4,4)*{} **\crv{(-4,-1) & (4,1)};
     (-4,4);(-4,12) **\dir{-};
     (-12,-4);(-12,12) **\dir{-};
     (4,-4);(4,-12) **\dir{-};(12,4);(12,-12) **\dir{-};
     (-16,1)*{\lambda};
     (-10,0)*{};(10,0)*{};
     (-4,-4)*{};(-12,-4)*{} **\crv{(-4,-10) & (-12,-10)}?(1)*\dir{>}?(0)*\dir{>};
      (4,4)*{};(12,4)*{} **\crv{(4,10) & (12,10)}?(1)*\dir{<}?(0)*\dir{<};
      (-14,11)*{\scs i};(-2,11)*{\scs j};
      (14,-11)*{\scs i};(2,-11)*{\scs j};
     }};
  \endxy
\end{equation}
where the second equality in \eqref{eq_crossl-gen} and \eqref{eq_crossr-gen}
follow from \eqref{eq_cyclic_cross-gen}.

\item For $i \neq j$
\begin{equation} \label{eq_downup_ij-gen}
 \vcenter{   \xy 0;/r.18pc/:
    (-4,-4)*{};(4,4)*{} **\crv{(-4,-1) & (4,1)}?(1)*\dir{>};
    (4,-4)*{};(-4,4)*{} **\crv{(4,-1) & (-4,1)}?(1)*\dir{<};?(0)*\dir{<};
    (-4,4)*{};(4,12)*{} **\crv{(-4,7) & (4,9)};
    (4,4)*{};(-4,12)*{} **\crv{(4,7) & (-4,9)}?(1)*\dir{>};
  (8,8)*{\lambda};(-6,-3)*{\scs i};
     (6,-3)*{\scs j};
 \endxy}
 \;\;= \;\;
\xy 0;/r.18pc/:
  (3,9);(3,-9) **\dir{-}?(.55)*\dir{>}+(2.3,0)*{};
  (-3,9);(-3,-9) **\dir{-}?(.5)*\dir{<}+(2.3,0)*{};
  (8,2)*{\lambda};(-5,-6)*{\scs i};     (5.1,-6)*{\scs j};
 \endxy
 \qquad
    \vcenter{\xy 0;/r.18pc/:
    (-4,-4)*{};(4,4)*{} **\crv{(-4,-1) & (4,1)}?(1)*\dir{<};?(0)*\dir{<};
    (4,-4)*{};(-4,4)*{} **\crv{(4,-1) & (-4,1)}?(1)*\dir{>};
    (-4,4)*{};(4,12)*{} **\crv{(-4,7) & (4,9)}?(1)*\dir{>};
    (4,4)*{};(-4,12)*{} **\crv{(4,7) & (-4,9)};
  (8,8)*{\lambda};(-6,-3)*{\scs i};
     (6,-3)*{\scs j};
 \endxy}
 \;\;=\;\;
\xy 0;/r.18pc/:
  (3,9);(3,-9) **\dir{-}?(.5)*\dir{<}+(2.3,0)*{};
  (-3,9);(-3,-9) **\dir{-}?(.55)*\dir{>}+(2.3,0)*{};
  (8,2)*{\lambda};(-5,-6)*{\scs i};     (5.1,-6)*{\scs j};
 \endxy
\end{equation}

\item The $R(\nu)$-relations:
\begin{enumerate}[i)]
\item For $i \neq j$
\begin{eqnarray}
  \vcenter{\xy 0;/r.18pc/:
    (-4,-4)*{};(4,4)*{} **\crv{(-4,-1) & (4,1)}?(1)*\dir{>};
    (4,-4)*{};(-4,4)*{} **\crv{(4,-1) & (-4,1)}?(1)*\dir{>};
    (-4,4)*{};(4,12)*{} **\crv{(-4,7) & (4,9)}?(1)*\dir{>};
    (4,4)*{};(-4,12)*{} **\crv{(4,7) & (-4,9)}?(1)*\dir{>};
  (8,8)*{\lambda};(-5,-3)*{\scs i};
     (5.1,-3)*{\scs j};
 \endxy}
 \qquad = \qquad
 \left\{
 \begin{array}{ccc}
     \xy 0;/r.18pc/:
  (3,9);(3,-9) **\dir{-}?(.5)*\dir{<}+(2.3,0)*{};
  (-3,9);(-3,-9) **\dir{-}?(.5)*\dir{<}+(2.3,0)*{};
  (8,2)*{\lambda};(-5,-6)*{\scs i};     (5.1,-6)*{\scs j};
 \endxy &  &  \text{if $i \cdot j=0$,}\\ \\
  \vcenter{\xy 0;/r.18pc/:
  (3,9);(3,-9) **\dir{-}?(.5)*\dir{<}+(2.3,0)*{};
  (-3,9);(-3,-9) **\dir{-}?(.5)*\dir{<}+(2.3,0)*{};
  (8,2)*{\lambda}; (-3,4)*{\bullet};(-6.5,5)*{\scs d_{ij}};
  (-5,-6)*{\scs i};     (5.1,-6)*{\scs j};
 \endxy} \quad
 + \quad
 \vcenter{\xy 0;/r.18pc/:
  (3,9);(3,-9) **\dir{-}?(.5)*\dir{<}+(2.3,0)*{};
  (-3,9);(-3,-9) **\dir{-}?(.5)*\dir{<}+(2.3,0)*{};
  (12,2)*{\lambda}; (3,4)*{\bullet};(7,5)*{\scs d_{ji}};
  (-5,-6)*{\scs i};     (5.1,-6)*{\scs j};
 \endxy}
   &  & \text{if $i \cdot j \neq 0$.}
 \end{array}
 \right. \nn \\\label{eq_r2_ij-gen}
\end{eqnarray}

\begin{eqnarray} \label{eq_dot_slide_ij-gen}
\xy
  (0,0)*{\xybox{
    (-4,-4)*{};(4,4)*{} **\crv{(-4,-1) & (4,1)}?(1)*\dir{>}?(.75)*{\bullet};
    (4,-4)*{};(-4,4)*{} **\crv{(4,-1) & (-4,1)}?(1)*\dir{>};
    (-5,-3)*{\scs i};
     (5.1,-3)*{\scs j};
     (8,1)*{ \lambda};
     (-10,0)*{};(10,0)*{};
     }};
  \endxy
 \;\; =
\xy
  (0,0)*{\xybox{
    (-4,-4)*{};(4,4)*{} **\crv{(-4,-1) & (4,1)}?(1)*\dir{>}?(.25)*{\bullet};
    (4,-4)*{};(-4,4)*{} **\crv{(4,-1) & (-4,1)}?(1)*\dir{>};
    (-5,-3)*{\scs i};
     (5.1,-3)*{\scs j};
     (8,1)*{ \lambda};
     (-10,0)*{};(10,0)*{};
     }};
  \endxy
\qquad  \xy
  (0,0)*{\xybox{
    (-4,-4)*{};(4,4)*{} **\crv{(-4,-1) & (4,1)}?(1)*\dir{>};
    (4,-4)*{};(-4,4)*{} **\crv{(4,-1) & (-4,1)}?(1)*\dir{>}?(.75)*{\bullet};
    (-5,-3)*{\scs i};
     (5.1,-3)*{\scs j};
     (8,1)*{ \lambda};
     (-10,0)*{};(10,0)*{};
     }};
  \endxy
\;\;  =
  \xy
  (0,0)*{\xybox{
    (-4,-4)*{};(4,4)*{} **\crv{(-4,-1) & (4,1)}?(1)*\dir{>} ;
    (4,-4)*{};(-4,4)*{} **\crv{(4,-1) & (-4,1)}?(1)*\dir{>}?(.25)*{\bullet};
    (-5,-3)*{\scs i};
     (5.1,-3)*{\scs j};
     (8,1)*{ \lambda};
     (-10,0)*{};(12,0)*{};
     }};
  \endxy
\end{eqnarray}

\item Unless $i = k$ and $i \cdot j \neq 0$
\begin{equation}
 \vcenter{
 \xy 0;/r.18pc/:
    (-4,-4)*{};(4,4)*{} **\crv{(-4,-1) & (4,1)}?(1)*\dir{>};
    (4,-4)*{};(-4,4)*{} **\crv{(4,-1) & (-4,1)}?(1)*\dir{>};
    (4,4)*{};(12,12)*{} **\crv{(4,7) & (12,9)}?(1)*\dir{>};
    (12,4)*{};(4,12)*{} **\crv{(12,7) & (4,9)}?(1)*\dir{>};
    (-4,12)*{};(4,20)*{} **\crv{(-4,15) & (4,17)}?(1)*\dir{>};
    (4,12)*{};(-4,20)*{} **\crv{(4,15) & (-4,17)}?(1)*\dir{>};
    (-4,4)*{}; (-4,12) **\dir{-};
    (12,-4)*{}; (12,4) **\dir{-};
    (12,12)*{}; (12,20) **\dir{-};
  (18,8)*{\lambda};
  (-6,-3)*{\scs i};
  (6,-3)*{\scs j};
  (15,-3)*{\scs k};
\endxy}
 \;\; =\;\;
 \vcenter{
 \xy 0;/r.18pc/:
    (4,-4)*{};(-4,4)*{} **\crv{(4,-1) & (-4,1)}?(1)*\dir{>};
    (-4,-4)*{};(4,4)*{} **\crv{(-4,-1) & (4,1)}?(1)*\dir{>};
    (-4,4)*{};(-12,12)*{} **\crv{(-4,7) & (-12,9)}?(1)*\dir{>};
    (-12,4)*{};(-4,12)*{} **\crv{(-12,7) & (-4,9)}?(1)*\dir{>};
    (4,12)*{};(-4,20)*{} **\crv{(4,15) & (-4,17)}?(1)*\dir{>};
    (-4,12)*{};(4,20)*{} **\crv{(-4,15) & (4,17)}?(1)*\dir{>};
    (4,4)*{}; (4,12) **\dir{-};
    (-12,-4)*{}; (-12,4) **\dir{-};
    (-12,12)*{}; (-12,20) **\dir{-};
  (10,8)*{\lambda};
  (7,-3)*{\scs k};
  (-6,-3)*{\scs j};
  (-14,-3)*{\scs i};
\endxy} \label{eq_r3_easy-gen}
\end{equation}

For $i \cdot j \neq 0$
\begin{equation}
 \vcenter{
 \xy 0;/r.18pc/:
    (-4,-4)*{};(4,4)*{} **\crv{(-4,-1) & (4,1)}?(1)*\dir{>};
    (4,-4)*{};(-4,4)*{} **\crv{(4,-1) & (-4,1)}?(1)*\dir{>};
    (4,4)*{};(12,12)*{} **\crv{(4,7) & (12,9)}?(1)*\dir{>};
    (12,4)*{};(4,12)*{} **\crv{(12,7) & (4,9)}?(1)*\dir{>};
    (-4,12)*{};(4,20)*{} **\crv{(-4,15) & (4,17)}?(1)*\dir{>};
    (4,12)*{};(-4,20)*{} **\crv{(4,15) & (-4,17)}?(1)*\dir{>};
    (-4,4)*{}; (-4,12) **\dir{-};
    (12,-4)*{}; (12,4) **\dir{-};
    (12,12)*{}; (12,20) **\dir{-};
  (18,8)*{\lambda};
  (-6,-3)*{\scs i};
  (6,-3)*{\scs j};
  (14,-3)*{\scs i};
\endxy}
\quad - \quad
 \vcenter{
 \xy 0;/r.18pc/:
    (4,-4)*{};(-4,4)*{} **\crv{(4,-1) & (-4,1)}?(1)*\dir{>};
    (-4,-4)*{};(4,4)*{} **\crv{(-4,-1) & (4,1)}?(1)*\dir{>};
    (-4,4)*{};(-12,12)*{} **\crv{(-4,7) & (-12,9)}?(1)*\dir{>};
    (-12,4)*{};(-4,12)*{} **\crv{(-12,7) & (-4,9)}?(1)*\dir{>};
    (4,12)*{};(-4,20)*{} **\crv{(4,15) & (-4,17)}?(1)*\dir{>};
    (-4,12)*{};(4,20)*{} **\crv{(-4,15) & (4,17)}?(1)*\dir{>};
    (4,4)*{}; (4,12) **\dir{-};
    (-12,-4)*{}; (-12,4) **\dir{-};
    (-12,12)*{}; (-12,20) **\dir{-};
  (10,8)*{\lambda};
  (6,-3)*{\scs i};
  (-6,-3)*{\scs j};
  (-14,-3)*{\scs i};
\endxy}
 \;\; =\;\;
 \sum_{a=0}^{d_{ij}-1} \;\;
\xy 0;/r.18pc/:
  (4,12);(4,-12) **\dir{-}?(.5)*\dir{<};
  (-4,12);(-4,-12) **\dir{-}?(.5)*\dir{<}?(.25)*\dir{}+(0,0)*{\bullet}+(-3,0)*{\scs a};
  (12,12);(12,-12) **\dir{-}?(.5)*\dir{<}?(.25)*\dir{}+(0,0)*{\bullet}+(10,0)*{\scs d_{ij}-1-a};
  (22,-2)*{\lambda}; (-6,-9)*{\scs i};     (6.1,-9)*{\scs j};
  (14,-9)*{\scs i};
 \endxy
 \label{eq_r3_hard-gen}
\end{equation}
\end{enumerate}
For example, for any shift $t$ there are 2-morphisms
\begin{eqnarray}
  \xy
 (0,7);(0,-7); **\dir{-} ?(.75)*\dir{>}+(2.3,0)*{\scriptstyle{}};
 (0.1,-2)*{\txt\large{$\bullet$}};
 (6,4)*{ \lambda};
 (-10,0)*{};(10,0)*{};(0,-10)*{i };(0,10)*{i};
 \endxy
 \maps  \cal{E}_{+i}\onel\{t\} \To \cal{E}_{+i}\onel\{t-2\}\quad
\xy
  (0,0)*{\xybox{
    (-4,-6)*{};(4,6)*{} **\crv{(-4,-1) & (4,1)}?(1)*\dir{>} ;
    (4,-6)*{};(-4,6)*{} **\crv{(4,-1) & (-4,1)}?(1)*\dir{>};
    (-4,-9)*{ i};
     (4,-9)*{ j};
     (8,1)*{ \lambda};
     (-12,0)*{};(12,0)*{};
     }};
  \endxy
  \maps \cal{E}_{+i+j}\onel\{t\} \To \cal{E}_{+j+i}\onel\{t-i\cdot j\} \nn \\
  \xy
    (0,-3)*{\bbpef{i}};
    (8,-5)*{ \lambda};
    (-12,0)*{};(12,0)*{};
    \endxy \maps
    \onel\{t\} \To \cal{E}_{-i+i}\onel\{t-c_{+i,\lambda}\} \quad \xy
    (0,0)*{\bbcfe{i}};
    (8,5)*{ \lambda};
    (-12,0)*{};(12,0)*{};
    \endxy \maps
    \cal{E}_{-i+i}\onel\{t\} \To \onel\{t-c_{-i,\lambda}\} \nn
\end{eqnarray}
in $\Ucat$,  and the diagrammatic relation
\[
\vcenter{
 \xy 0;/r.18pc/:
    (-4,-4)*{};(4,4)*{} **\crv{(-4,-1) & (4,1)}?(1)*\dir{>};
    (4,-4)*{};(-4,4)*{} **\crv{(4,-1) & (-4,1)}?(1)*\dir{>};
    (4,4)*{};(12,12)*{} **\crv{(4,7) & (12,9)}?(1)*\dir{>};
    (12,4)*{};(4,12)*{} **\crv{(12,7) & (4,9)}?(1)*\dir{>};
    (-4,12)*{};(4,20)*{} **\crv{(-4,15) & (4,17)}?(1)*\dir{>};
    (4,12)*{};(-4,20)*{} **\crv{(4,15) & (-4,17)}?(1)*\dir{>};
    (-4,4)*{}; (-4,12) **\dir{-};
    (12,-4)*{}; (12,4) **\dir{-};
    (12,12)*{}; (12,20) **\dir{-}; (-5.5,-3)*{\scs i};
     (5.5,-3)*{\scs i};(14,-3)*{\scs i};
  (18,8)*{\lambda};
\endxy}
 \;\; =\;\;
 \vcenter{
 \xy 0;/r.18pc/:
    (4,-4)*{};(-4,4)*{} **\crv{(4,-1) & (-4,1)}?(1)*\dir{>};
    (-4,-4)*{};(4,4)*{} **\crv{(-4,-1) & (4,1)}?(1)*\dir{>};
    (-4,4)*{};(-12,12)*{} **\crv{(-4,7) & (-12,9)}?(1)*\dir{>};
    (-12,4)*{};(-4,12)*{} **\crv{(-12,7) & (-4,9)}?(1)*\dir{>};
    (4,12)*{};(-4,20)*{} **\crv{(4,15) & (-4,17)}?(1)*\dir{>};
    (-4,12)*{};(4,20)*{} **\crv{(-4,15) & (4,17)}?(1)*\dir{>};
    (4,4)*{}; (4,12) **\dir{-};
    (-12,-4)*{}; (-12,4) **\dir{-};
    (-12,12)*{}; (-12,20) **\dir{-};(-5.5,-3)*{\scs i};
     (5.5,-3)*{\scs i};(-14,-3)*{\scs i};
  (10,8)*{\lambda};
\endxy}
\]
gives rise to relations in
$\Ucat\big(\cal{E}_{iii}\onel\{t\},\cal{E}_{iii}\onel\{t+3i\cdot i\}\big)$ for
all $t\in \Z$.

\item the additive $\Bbbk$-linear composition functor $\Ucat(\lambda,\lambda')
 \times \Ucat(\lambda',\lambda'') \to \Ucat(\lambda,\lambda'')$ is given on
 1-morphisms of $\Ucat$ by
\begin{equation}
  \cal{E}_{\jj}\mathbf{1}_{\lambda'}\{t'\} \times \cal{E}_{\ii}\onel\{t\} \mapsto
  \cal{E}_{\jj\ii}\onel\{t+t'\}
\end{equation}
for $\ii_X=\lambda-\lambda'$, and on 2-morphisms of $\Ucat$ by juxtaposition of
diagrams
\[
\left(\;\;\vcenter{\xy 0;/r.16pc/:
 (-4,-15)*{}; (-20,25) **\crv{(-3,-6) & (-20,4)}?(0)*\dir{<}?(.6)*\dir{}+(0,0)*{\bullet};
 (-12,-15)*{}; (-4,25) **\crv{(-12,-6) & (-4,0)}?(0)*\dir{<}?(.6)*\dir{}+(.2,0)*{\bullet};
 ?(0)*\dir{<}?(.75)*\dir{}+(.2,0)*{\bullet};?(0)*\dir{<}?(.9)*\dir{}+(0,0)*{\bullet};
 (-28,25)*{}; (-12,25) **\crv{(-28,10) & (-12,10)}?(0)*\dir{<};
  ?(.2)*\dir{}+(0,0)*{\bullet}?(.35)*\dir{}+(0,0)*{\bullet};
 (-36,-15)*{}; (-36,25) **\crv{(-34,-6) & (-35,4)}?(1)*\dir{>};
 (-28,-15)*{}; (-42,25) **\crv{(-28,-6) & (-42,4)}?(1)*\dir{>};
 (-42,-15)*{}; (-20,-15) **\crv{(-42,-5) & (-20,-5)}?(1)*\dir{>};
 (6,10)*{\cbub{}{}};
 (-23,0)*{\cbub{}{}};
 (8,-4)*{\lambda'};(-44,-4)*{\lambda''};
 \endxy}\;\;\right) \;\; \times \;\;
\left(\;\;\vcenter{ \xy 0;/r.18pc/: (-14,8)*{\xybox{
 (0,-10)*{}; (-16,10)*{} **\crv{(0,-6) & (-16,6)}?(.5)*\dir{};
 (-16,-10)*{}; (-8,10)*{} **\crv{(-16,-6) & (-8,6)}?(1)*\dir{}+(.1,0)*{\bullet};
  (-8,-10)*{}; (0,10)*{} **\crv{(-8,-6) & (-0,6)}?(.6)*\dir{}+(.2,0)*{\bullet}?
  (1)*\dir{}+(.1,0)*{\bullet};
  (0,10)*{}; (-16,30)*{} **\crv{(0,14) & (-16,26)}?(1)*\dir{>};
 (-16,10)*{}; (-8,30)*{} **\crv{(-16,14) & (-8,26)}?(1)*\dir{>};
  (-8,10)*{}; (0,30)*{} **\crv{(-8,14) & (-0,26)}?(1)*\dir{>}?(.6)*\dir{}+(.25,0)*{\bullet};
   }};
 (-2,-4)*{\lambda}; (-26,-4)*{\lambda'};
 \endxy} \;\;\right)
 \;\;\mapsto \;\;
\vcenter{\xy 0;/r.16pc/:
 (-4,-15)*{}; (-20,25) **\crv{(-3,-6) & (-20,4)}?(0)*\dir{<}?(.6)*\dir{}+(0,0)*{\bullet};
 (-12,-15)*{}; (-4,25) **\crv{(-12,-6) & (-4,0)}?(0)*\dir{<}?(.6)*\dir{}+(.2,0)*{\bullet};
 ?(0)*\dir{<}?(.75)*\dir{}+(.2,0)*{\bullet};?(0)*\dir{<}?(.9)*\dir{}+(0,0)*{\bullet};
 (-28,25)*{}; (-12,25) **\crv{(-28,10) & (-12,10)}?(0)*\dir{<};
  ?(.2)*\dir{}+(0,0)*{\bullet}?(.35)*\dir{}+(0,0)*{\bullet};
 (-36,-15)*{}; (-36,25) **\crv{(-34,-6) & (-35,4)}?(1)*\dir{>};
 (-28,-15)*{}; (-42,25) **\crv{(-28,-6) & (-42,4)}?(1)*\dir{>};
 (-42,-15)*{}; (-20,-15) **\crv{(-42,-5) & (-20,-5)}?(1)*\dir{>};
 (6,10)*{\cbub{}{}};
 (-23,0)*{\cbub{}{}};
 \endxy}
 \vcenter{ \xy 0;/r.16pc/: (-14,8)*{\xybox{
 (0,-10)*{}; (-16,10)*{} **\crv{(0,-6) & (-16,6)}?(.5)*\dir{};
 (-16,-10)*{}; (-8,10)*{} **\crv{(-16,-6) & (-8,6)}?(1)*\dir{}+(.1,0)*{\bullet};
  (-8,-10)*{}; (0,10)*{} **\crv{(-8,-6) & (-0,6)}?(.6)*\dir{}+(.2,0)*{\bullet}?
  (1)*\dir{}+(.1,0)*{\bullet};
  (0,10)*{}; (-16,30)*{} **\crv{(0,14) & (-16,26)}?(1)*\dir{>};
 (-16,10)*{}; (-8,30)*{} **\crv{(-16,14) & (-8,26)}?(1)*\dir{>};
  (-8,10)*{}; (0,30)*{} **\crv{(-8,14) & (-0,26)}?(1)*\dir{>}?(.6)*\dir{}+(.25,0)*{\bullet};
   }};
 (0,-5)*{\lambda};
 \endxy}
\]
\end{itemize}
\end{defn}

\begin{rem}
By choosing an orientation of the graph associated to Cartan datum $(I,\cdot)$
the $R(\nu)$-relations above can be modified by replacing them with signed
$R(\nu)$ relations determined by invertible elements $\tau_{ij}$, $\tau_{ji}$
chosen for each edge of the graph (see \cite{KL2} for more details).
\end{rem}

$\Ucat$ has graded 2-homs defined by
\begin{equation}
  \HOMU(x,y) := \bigoplus_{t \in \Z} \Hom_{\Ucat}(x\{t\},y).
\end{equation}
Also define graded endomorphisms
\begin{equation}
  \ENDU(x):= \HOMU(x,x).
\end{equation}
The 2-category with the same objects and 1-morphisms as $\Ucat$ and 2-homs given
by $\HOM_{\Ucat}(x,y)$ is denoted $\Ucatq$, so that
\begin{equation}
 \Ucatq(x,y) = \HOMU(x,y).
\end{equation}
$\Ucatq$ is a graded additive $\Bbbk$-linear 2-category with
translation~\cite[Section 5.1]{Lau1}.

%
\subsubsection{Relations in $\Ucat$} \label{subsubsec_relations}
%

We now collect some other relations that follow from those above.  These
relations simplify computations in the graphical calculus and can be used to
reduce complex diagrams into simpler ones (see the proofs of
Proposition~\ref{prop_bubbles_same_orient} and Lemma~\ref{lem_surjective}).

\begin{prop}[Bubble Slides] \label{prop_bubble_slide1}
The following identities hold in $\Ucat$
\begin{eqnarray}
    \xy
  (14,8)*{\lambda};
  (0,0)*{\bbe{}};
  (0,-12)*{\scs j};
  (12,-2)*{\ccbub{-\la i,\lambda \ra-1+\alpha}{i}};
  (0,6)*{ }+(7,-1)*{\scs  };
 \endxy
 &\quad = \quad&
 \left\{
 \begin{array}{ccl}
  \xsum{f=0}{\alpha}(\alpha+1-f)
   \xy
  (0,8)*{\lambda+j_X};
  (12,0)*{\bbe{}};
  (12,-12)*{\scs j};
  (0,-2)*{\ccbub{-\la i,\lambda +i_X\ra-1+f}{i}};
  (12,6)*{\bullet}+(5,-1)*{\scs \alpha-f};
 \endxy
    &  & \text{if $i=j$} \\ \\
        \xy
  (0,8)*{\lambda+j_X};
  (12,0)*{\bbe{}};
  (11,-12)*{\scs j};
  (0,-2)*{\ccbub{-\la i,\lambda+j_X \ra -1+\alpha}{i}};
 \endxy
 \quad + \quad
  \xy
  (0,8)*{\lambda+j_X};
  (12,0)*{\bbe{}};
  (12,-12)*{\scs j};
  (0,-2)*{\ccbub{-\la i,\lambda +j_X\ra-2+\alpha}{i}};
  (12,6)*{\bullet}+(5,-1)*{\scs };
 \endxy
   &  & \text{if $i \cdot j =-1$} \\
 \qquad \qquad  \xy
  (0,8)*{\lambda+j_X};
  (12,0)*{\bbe{}};
  (12,-12)*{\scs j};
  (0,-2)*{\ccbub{-\la i,\lambda \ra-1+\alpha}{i}};
 \endxy  &  & \text{if $i \cdot j=0$}
 \end{array}
 \right. \nn
\end{eqnarray}
\begin{eqnarray}
    \xy
  (15,8)*{\lambda};
  (11,0)*{\bbe{}};
  (11,-12)*{\scs j};
  (0,-2)*{\cbub{\la i,\lambda+j_X\ra-1+\alpha\quad }{i}};
 \endxy
   &\quad = \quad&
  \left\{\begin{array}{ccl}
     \xsum{f=0}{\alpha}(\alpha+1-f)
     \xy
  (18,8)*{\lambda};
  (0,0)*{\bbe{}};
  (0,-12)*{\scs j};
  (14,-4)*{\cbub{(\la i,\lambda\ra-1)+f}{i}};
  (0,6)*{\bullet }+(5,-1)*{\scs \alpha-f};
 \endxy
         &  & \text{if $i=j$}  \\
    \xy
  (18,8)*{\lambda};
  (0,0)*{\bbe{}};
  (0,-12)*{\scs j};
  (12,-2)*{\cbub{(\la i,\lambda \ra-1)+ \alpha-1}{i}};
    (0,6)*{\bullet }+(5,-1)*{\scs};
 \endxy
 \quad + \quad
\xy
  (18,8)*{\lambda};
  (0,0)*{\bbe{}};
  (0,-12)*{\scs j};
  (12,-2)*{\cbub{(\la i,\lambda \ra-1)+ \alpha}{i}};
 \endxy
         & &  \text{if $i\cdot j =-1$}\\
    \xy
  (15,8)*{\lambda};
  (0,0)*{\bbe{}};
  (0,-12)*{\scs j};
  (12,-2)*{\cbub{(\la i,\lambda \ra-1)+\alpha}{i}};
 \endxy          &  & \text{if $i \cdot j = 0$}
         \end{array}
 \right. \nn
\end{eqnarray}
\end{prop}

\begin{proof}
When $i = j$ the proof appears in \cite{Lau1}.  For $i \neq j$ the equation
follows from decomposing
\[
 \vcenter{\xy 0;/r.2pc/:
    (-4,-4)*{};(4,4)*{} **\crv{(-4,-1) & (4,1)}?(1)*\dir{>};
    (4,4)*{};(-4,12)*{} **\crv{(4,7) & (-4,9)}?(1)*\dir{>};
  (18,8)*{\lambda};(-5,-3)*{\scs {\bf }};
  (6,1.5)*{\ccbub{\qquad\quad -\la i,\lambda+j_X \ra-2+\alpha}{}};
     (-7,-3)*{\scs j};
     (8,9)*{\scs i};
 \endxy}
\quad = \quad
  \vcenter{\xy 0;/r.2pc/:
    (4,-4)*{};(-4,4)*{} **\crv{(4,-1) & (-4,1)}?(1)*\dir{>};
    (-4,4)*{};(4,12)*{} **\crv{(-4,7) & (4,9)}?(1)*\dir{>};
  (13,8)*{\lambda};(-5,-3)*{\scs {\bf }};
  (-2,1.5)*{\ccbub{-\la i,\lambda+j_X\ra-2+\alpha\qquad \quad }{}};
   (6,-3)*{\scs j};
     (-8,9)*{\scs i};
 \endxy}
 \qquad  \vcenter{\xy 0;/r.2pc/:
    (-4,-4)*{};(4,4)*{} **\crv{(-4,-1) & (4,1)}?(1)*\dir{>};
    (4,4)*{};(-4,12)*{} **\crv{(4,7) & (-4,9)}?(1)*\dir{>};
  (18,8)*{\lambda};(-5,-3)*{\scs {\bf }};
  (6,1.5)*{\cbub{\qquad\quad \la i,\lambda+j_X\ra-1+\alpha}{}};
     (-7,-3)*{\scs j};
     (8,9)*{\scs i};
 \endxy}
\quad = \quad
  \vcenter{\xy 0;/r.2pc/:
    (4,-4)*{};(-4,4)*{} **\crv{(4,-1) & (-4,1)}?(1)*\dir{>};
    (-4,4)*{};(4,12)*{} **\crv{(-4,7) & (4,9)}?(1)*\dir{>};
  (13,8)*{\lambda};(-5,-3)*{\scs {\bf }};
  (-2,1.5)*{\cbub{\la i,\lambda+j_X\ra-1+\alpha\qquad \quad }{}};
   (6,-3)*{\scs j};
     (-8,9)*{\scs i};
 \endxy}
 \]
 using the relations \eqref{eq_downup_ij-gen} and \eqref{eq_r2_ij-gen}.
\end{proof}

\begin{prop}[More bubble slides] \label{prop_bubble_slide2}
The following identities hold in $\Ucat$
  \[
      \xy
  (15,8)*{\lambda};
  (0,0)*{\bbe{}};
  (0,-12)*{\scs j};
  (12,-2)*{\cbub{(\la i,\lambda \ra-1)+\alpha}{i}};
 \endxy
   =
  \left\{
  \begin{array}{cc}
     \xy
  (0,8)*{\lambda+i_X};
  (12,0)*{\bbe{}};
  (12,-12)*{\scs j};
  (0,-2)*{\cbub{(\la i,\lambda \ra+1)+(\alpha-2)}{i}};
  (12,6)*{\bullet}+(3,-1)*{\scs 2};
 \endxy
   -2 \;
         \xy
  (0,8)*{\lambda+i_X};
  (12,0)*{\bbe{}};
  (12,-12)*{\scs j};
  (0,-2)*{\cbub{(\la i,\lambda \ra+1)+(\alpha-1)}{i}};
  (12,6)*{\bullet}+(8,-1)*{\scs };
 \endxy
 + \;\;
     \xy
  (0,8)*{\lambda+i_X};
  (12,0)*{\bbe{}};
  (12,-12)*{\scs j};
  (0,-2)*{\cbub{(\la i,\lambda \ra+1)+\alpha}{i}};
  (12,6)*{}+(8,-1)*{\scs };
 \endxy
  &   \text{if $i = j$} \\
  \xsum{f=0}{\alpha} (-1)^{f}
  \xy
  (0,8)*{\lambda+j_X};
  (14,0)*{\bbe{}};
  (12,-12)*{\scs j};
  (0,-2)*{\cbub{\la i,\lambda+j_X\ra-1+(\alpha-f)\quad}{i}};
  (14,6)*{\bullet}+(3,-1)*{\scs f};
 \endxy &   \text{if $i\cdot j =-1$}
  \end{array}
 \right.
\]
\[
    \xy
  (0,8)*{\lambda+j_X};
  (12,0)*{\bbe{}};
  (12,-12)*{\scs j};
  (0,-2)*{\ccbub{-\la i,\lambda+j_X\ra-1+\alpha}{i}};
  (12,6)*{}+(8,-1)*{\scs };
 \endxy
  =
\left\{
\begin{array}{cc}
    \xy
  (15,8)*{\lambda};
  (0,0)*{\bbe{}};
  (0,-12)*{\scs j};
  (12,-2)*{\ccbub{\quad(-\la i,\lambda \ra-1)+(\alpha-2)}{i}};
  (0,6)*{\bullet }+(3,1)*{\scs 2};
 \endxy
  -2 \;
      \xy
  (15,8)*{\lambda};
  (0,0)*{\bbe{}};
  (0,-12)*{\scs j};
  (12,-2)*{\ccbub{\quad(-\la i,\lambda \ra-1)+(\alpha-1)}{i}};
  (0,6)*{\bullet }+(5,-1)*{\scs };
 \endxy
 + \;\;
      \xy
  (15,8)*{\lambda};
  (0,0)*{\bbe{}};
  (0,-12)*{\scs j};
  (12,-2)*{\ccbub{(-\la i,\lambda \ra-1)+\alpha}{i}};
 \endxy
  &   \text{if $i=j$} \\
   \xsum{f=0}{\alpha}(-1)^f
    \xy
  (15,8)*{\lambda};
  (0,0)*{\bbe{}};
  (0,-12)*{\scs j};
  (14,-2)*{\ccbub{(-\la i,\lambda \ra-1)+(\alpha-f)}{i}};
  (0,6)*{\bullet }+(3,1)*{\scs f};
 \endxy
    &   \text{if $i \cdot j =-1$}
\end{array}
\right.
\]
\end{prop}

\begin{proof}
These equations follow from the previous Proposition.
\end{proof}

\begin{prop} \label{prop_other_triangle}
Unless $i=k=j$ we have
\begin{equation} \label{eq_other_r3_1}
 \vcenter{
 \xy 0;/r.15pc/:
    (-4,-4)*{};(4,4)*{} **\crv{(-4,-1) & (4,1)}?(1)*\dir{>};
    (4,-4)*{};(-4,4)*{} **\crv{(4,-1) & (-4,1)}?(1)*\dir{<};
    ?(0)*\dir{<};
    (4,4)*{};(12,12)*{} **\crv{(4,7) & (12,9)}?(1)*\dir{>};
    (12,4)*{};(4,12)*{} **\crv{(12,7) & (4,9)}?(1)*\dir{>};
    (-4,12)*{};(4,20)*{} **\crv{(-4,15) & (4,17)};
    (4,12)*{};(-4,20)*{} **\crv{(4,15) & (-4,17)}?(1)*\dir{>};
    (-4,4)*{}; (-4,12) **\dir{-};
    (12,-4)*{}; (12,4) **\dir{-};
    (12,12)*{}; (12,20) **\dir{-};
  (18,8)*{\lambda};
  (-6,-3)*{\scs i};
  (7,-3)*{\scs j};
  (15,-3)*{\scs k};
\endxy}
 \;\; =\;\;
 \vcenter{
 \xy 0;/r.15pc/:
    (4,-4)*{};(-4,4)*{} **\crv{(4,-1) & (-4,1)}?(1)*\dir{>};
    (-4,-4)*{};(4,4)*{} **\crv{(-4,-1) & (4,1)}?(0)*\dir{<};
    (-4,4)*{};(-12,12)*{} **\crv{(-4,7) & (-12,9)}?(1)*\dir{>};
    (-12,4)*{};(-4,12)*{} **\crv{(-12,7) & (-4,9)}?(1)*\dir{>};
    (4,12)*{};(-4,20)*{} **\crv{(4,15) & (-4,17)};
    (-4,12)*{};(4,20)*{} **\crv{(-4,15) & (4,17)}?(1)*\dir{>};
    (4,4)*{}; (4,12) **\dir{-} ?(.5)*\dir{<};
    (-12,-4)*{}; (-12,4) **\dir{-};
    (-12,12)*{}; (-12,20) **\dir{-};
  (10,8)*{\lambda};
  (7,-3)*{\scs k};
  (-7,-3)*{\scs j};
  (-14,-3)*{\scs i};
\endxy}
\end{equation}
and when $i=j=k$ we have
\begin{equation}
 \vcenter{
 \xy 0;/r.17pc/:
    (-4,-4)*{};(4,4)*{} **\crv{(-4,-1) & (4,1)}?(1)*\dir{>};
    (4,-4)*{};(-4,4)*{} **\crv{(4,-1) & (-4,1)}?(1)*\dir{<};
    ?(0)*\dir{<};
    (4,4)*{};(12,12)*{} **\crv{(4,7) & (12,9)}?(1)*\dir{>};
    (12,4)*{};(4,12)*{} **\crv{(12,7) & (4,9)}?(1)*\dir{>};
    (-4,12)*{};(4,20)*{} **\crv{(-4,15) & (4,17)};
    (4,12)*{};(-4,20)*{} **\crv{(4,15) & (-4,17)}?(1)*\dir{>};
    (-4,4)*{}; (-4,12) **\dir{-};
    (12,-4)*{}; (12,4) **\dir{-};
    (12,12)*{}; (12,20) **\dir{-};
  (18,8)*{\lambda};
  (-6,-3)*{\scs i};
  (7,-3)*{\scs i};
  (15,-3)*{\scs i};
\endxy}
-\;
   \vcenter{
 \xy 0;/r.17pc/:
    (4,-4)*{};(-4,4)*{} **\crv{(4,-1) & (-4,1)}?(1)*\dir{>};
    (-4,-4)*{};(4,4)*{} **\crv{(-4,-1) & (4,1)}?(0)*\dir{<};
    (-4,4)*{};(-12,12)*{} **\crv{(-4,7) & (-12,9)}?(1)*\dir{>};
    (-12,4)*{};(-4,12)*{} **\crv{(-12,7) & (-4,9)}?(1)*\dir{>};
    (4,12)*{};(-4,20)*{} **\crv{(4,15) & (-4,17)};
    (-4,12)*{};(4,20)*{} **\crv{(-4,15) & (4,17)}?(1)*\dir{>};
    (4,4)*{}; (4,12) **\dir{-} ?(.5)*\dir{<};
    (-12,-4)*{}; (-12,4) **\dir{-};
    (-12,12)*{}; (-12,20) **\dir{-};
  (10,8)*{\lambda};
  (7,-3)*{\scs i};
  (-7,-3)*{\scs i};
  (-14,-3)*{\scs i};
\endxy}
  \; = \;
 \sum_{} \; \xy 0;/r.17pc/:
    (-4,12)*{}="t1";
    (4,12)*{}="t2";
  "t2";"t1" **\crv{(5,5) & (-5,5)}; ?(.15)*\dir{} ?(.9)*\dir{>}
  ?(.2)*\dir{}+(0,-.2)*{\bullet}+(3,-2)*{\scs f_1};
    (-4,-12)*{}="t1";
    (4,-12)*{}="t2";
  "t2";"t1" **\crv{(5,-5) & (-5,-5)}; ?(.05)*\dir{} ?(.9)*\dir{<}
  ?(.15)*\dir{}+(0,-.2)*{\bullet}+(3,2)*{\scs f_3};
    (-8,1)*{\ccbub{\scs -\la i,\lambda \ra-3+f_4}{i}};
    (13,12)*{};(13,-12)*{} **\dir{-} ?(.5)*\dir{<};
    (13,8)*{\bullet}+(3,2)*{\scs f_2};
  (19,-6)*{\lambda};
  \endxy
+\;
  \sum_{}
\; \xy 0;/r.17pc/: (-10,12)*{};(-10,-12)*{} **\dir{-} ?(.5)*\dir{<};
  (-10,8)*{\bullet}+(-3,2)*{\scs g_2};
  (-4,12)*{}="t1";
  (4,12)*{}="t2";
  "t1";"t2" **\crv{(-4,5) & (4,5)}; ?(.15)*\dir{>} ?(.9)*\dir{>}
  ?(.4)*\dir{}+(0,-.2)*{\bullet}+(3,-2)*{\scs \;\; g_1};
  (-4,-12)*{}="t1";
  (4,-12)*{}="t2";
  "t2";"t1" **\crv{(4,-5) & (-4,-5)}; ?(.12)*\dir{>} ?(.97)*\dir{>}
  ?(.8)*\dir{}+(0,-.2)*{\bullet}+(1,4)*{\scs g_3};
  (16,-4)*{\cbub{\scs \la i,\lambda \ra-1+g_4}{i}};
  (24,6)*{\lambda};
  \endxy
 \label{eq_r3_extra}
\end{equation}
where the first sum is over all $f_1, f_2, f_3, f_4 \geq 0$ with
$f_1+f_2+f_3+f_4=\la i,\lambda \ra$ and the second sum is over all $g_1, g_2,
g_3, g_4 \geq 0$ with $g_1+g_2+g_3+g_4=\la i,\lambda \ra -2$.  Recall that all
summations in this paper are increasing, so that the first summation is zero if
$\la i,\lambda \ra<0$ and the second is zero when $\la i,\lambda \ra<2$.

Reidemeister 3 like relations for all other orientations are determined from
\eqref{eq_r3_easy-gen}, \eqref{eq_r3_hard-gen}, and the above relations using
duality.
\end{prop}

\begin{proof}
For \eqref{eq_other_r3_1} if $i \neq j$ post-compose both sides with the
isomorphism
 \[\vcenter{
 \xy 0;/r.15pc/:
    (-4,-4)*{};(4,4)*{} **\crv{(-4,-1) & (4,1)}?(1)*\dir{>};
    (4,-4)*{};(-4,4)*{} **\crv{(4,-1) & (-4,1)}
    ?(0)*\dir{<};
    (12,-4)*{}; (12,4) **\dir{-};
  (18,3)*{\lambda};
  (-4,-7)*{\scs j};
  (4,-7)*{\scs k};
  (12,-7)*{\scs i};
\endxy} ,
\]
then use that $i=j \neq k$ to apply \eqref{eq_r3_easy-gen} on the right term and
establish the equality.  If $i=j$ then we may assume that $j \neq k$.  In this
case we pre-compose with the isomorphism
 \[\vcenter{
 \xy 0;/r.15pc/:
    (4,-4)*{};(-4,4)*{} **\crv{(4,-1) & (-4,1)}?(0)*\dir{<};
    (-4,-4)*{};(4,4)*{} **\crv{(-4,-1) & (4,1)}
    ?(1)*\dir{>};
    (-12,-4)*{}; (-12,4) **\dir{-};
  (10,3)*{\lambda};
  (-12,-7)*{\scs i};
  (-4,-7)*{\scs k};
  (4,-7)*{\scs j};
\endxy} ,
\]
then use \eqref{eq_r3_easy-gen} on the left side to establish the identity.

The case $i=j=k$ appears in \cite[Section 5.4]{Lau1}.
\end{proof}

%
\subsection{Spanning sets of HOMs in $\Ucat$} \label{subsec-spansets}
%

Given two Laurent power series $f(q)= \sum_{k=a}^{+\infty}f_kq^k$ and
$h(q)=\sum_{k=a}^{+\infty}$ with $f_k,h_k\in \Z$ we say that $f(q) \leq h(q)$ if
$f_k \leq h_k$ for all $k$.

%
\subsubsection{Endomorphisms of $\onel$}
%

For any root datum and $\lambda \in X$ define a graded commutative ring
$\Pi_{\lambda}$ freely generated by symbols
\begin{equation} \label{eq_bub_rules}
\xy 0;/r.18pc/:
 (-12,0)*{\cbub{\la i, \lambda\ra-1+\alpha}{i}};
 (-8,8)*{\lambda};
 \endxy
 \quad \text{for $\la i, \lambda\ra\geq 0$} \qquad {\rm and} \qquad
 \xy 0;/r.18pc/:
 (-12,0)*{\ccbub{-\la i, \lambda\ra-1+\alpha}{i}};
 (-8,8)*{\lambda};
 \endxy
 \quad \text{for $\la i, \lambda\ra< 0$}
\end{equation}
of degree $\alpha i \cdot i$ over all $i \in I$ and $\alpha> 0$.

\begin{prop} \label{prop_bubbles_same_orient}
Interpreting these generators of $\Pi_{\lambda}$ as elements of
$\HOMU(\onel,\onel)$ induces a surjective graded $\Bbbk$-algebra homomorphism
\begin{equation} \label{homo_Pi_Uone}
\Pi_{\lambda} \to \HOMU(\onel,\onel).
\end{equation}
\end{prop}

\begin{proof}
By induction on the number of crossings of a closed diagram $D$ representing an
endomorphism of $\onel$ one can reduce $D$ to a linear combination of
crossingless diagrams following the methods of \cite[Section 8]{Lau1}.
Crossingless diagrams that contain nested bubbles can be written as linear
combinations of crossingless nonnested diagrams using the bubble slide equations
in Propositions~\ref{prop_bubble_slide1} and \ref{prop_bubble_slide2}. Using the
Grassmannian relations \eqref{eq_infinite_Grass} all dotted bubbles with the same
label $i$ can be made to have the same orientation given by \eqref{eq_bub_rules}.
\end{proof}

The image of the monomial basis of $\Pi_{\lambda}$ under surjective homomorphism
\eqref{homo_Pi_Uone} is a homogeneous  spanning set of the graded vector space
$\HOMU(\onel,\onel)$. Denote this spanning set by
$B_{\emptyset,\emptyset,\lambda}$.

Let
\begin{equation}
\pi = \prod_{i \in I}\prod_{a=1}^{\infty}\frac{1}{1-q_i^{2a}}.
\end{equation}
$\pi$ depends only on the root datum (on the values of $i \cdot i$ over all $i\in
I$).  The graded dimension of $\Pi_{\lambda}$ is $\pi$.

\begin{cor}
$\gdim \;\HOMU(\onel,\onel) \leq \pi$ and $\HOMU(\onel,\onel)$ is a local graded
ring.
\end{cor}

\begin{rem}
$\Pi_{\lambda}$ is not Noetherian.
\end{rem}

We call monomials in this basis of $\Pi_{\lambda}$ and their images in
$\HOMU(\onel,\onel)$ {\em bubble monomials}.

%
\subsubsection{Homs between $\cal{E}_{\ii}\onel$ and $\cal{E}_{\jj}\onel$ for positive $\ii$ and $\jj$}
%

Recall rings $R(\nu)$, for $\nu\in \N[I]$, from \cite{KL,KL2} and decomposition
\begin{equation}
 R(\nu) = \bigoplus_{\ii,\jj \in \seq(\nu)} {}_{\jj}R(\nu)_{\ii},
\end{equation}
where ${}_{\jj}R(\nu)_{\ii}$ is spanned by braid-like diagrams $D$ with $\ii$,
$\jj$ being the lower and upper sequences of $D$.  Adding upward orientations to
a diagram $D$ in $R(\nu)$, placing it to the left of a collection of bubbles
representing a monomial in $\Pi_{\lambda}$
\[
\xy 0;/r.18pc/: (-14,8)*{\xybox{
 (0,-10)*{}; (-16,10)*{} **\crv{(0,-6) & (-16,6)}?(.5)*\dir{};
 (-16,-10)*{}; (-8,10)*{} **\crv{(-16,-6) & (-8,6)}?(1)*\dir{}+(.1,0)*{\bullet};
  (-8,-10)*{}; (0,10)*{} **\crv{(-8,-6) & (-0,6)}?(.6)*\dir{}+(.2,0)*{\bullet}?
  (1)*\dir{}+(.1,0)*{\bullet};
  (0,10)*{}; (-16,30)*{} **\crv{(0,14) & (-16,26)}?(1)*\dir{>};
 (-16,10)*{}; (-8,30)*{} **\crv{(-16,14) & (-8,26)}?(1)*\dir{>};
  (-8,10)*{}; (0,30)*{} **\crv{(-8,14) & (-0,26)}?(1)*\dir{>}?(.6)*\dir{}+(.25,0)*{\bullet};
  (-8,-13)*{\underbrace{\hspace{.6in}}};(-8,33)*{\overbrace{\hspace{.6in}}};
  (-8,-18)*{\scs \ii};(-8,38)*{\scs \jj};
   }};
 (14,-5)*{\cbub{\la i, \lambda\ra-1+\alpha_2}{i}};
 (36,-5)*{\ccbub{ \;\;\;-\la j, \lambda\ra-1+\alpha_4}{j}};
 (14,15)*{\cbub{\la i, \lambda\ra-1+\alpha_1}{i}};
 (37,15)*{\ccbub{\;\;\;-\la j, \lambda\ra-1+\alpha_3}{j}};
 (65,8)*{\cbub{\la k, \lambda\ra-1+\alpha_5}{k}};
 (22,24)*{\lambda};
 \endxy
\]
and viewing the result as a 2-morphism from $\cal{E}_{\ii}\onel$ to
$\cal{E}_{\jj}\onel$ induces a grading-preserving $\Bbbk$-linear map
\begin{equation}
  \varphi_{\ii,\jj,\lambda} \maps {}_{\jj}R(\nu)_{\ii} \otimes_{\Bbbk} \Pi_{\lambda}
 \longrightarrow \HOMU(\cal{E}_{\ii}\onel,\cal{E}_{\jj}\onel).
\end{equation}

\begin{lem} \label{lem_surjective}
$\varphi_{\ii,\jj,\lambda}$ is surjective.
\end{lem}

\begin{proof}
Start with a diagram $D$ that represents an element in $\HOMU(\cal{E}_{\ii}\onel,
\cal{E}_{\jj}\onel)$. The relations in our graphical calculus allow us to
inductively simplify $D$, by reducing the number of crossings if a strand or a
circle of $D$ has a self-intersection or if $D$ contains two strands that
intersect more than once. Bubble sliding rules allow moving bubbles to the far
right of the diagram. Eventually, $D$ reduces to a linear combination of diagrams
which are products of diagrams representing elements of ${}_{\jj}R(\nu)_{\ii}$
and monomials in $\Pi_{\lambda}$.
\end{proof}

The elements in $\HOMU(\cal{E}_{\ii}\onel, \cal{E}_{\jj}\onel)$ given by diagrams
without circles, with no two strands intersecting more than once, and with all
dots at the bottom are precisely the image under $\varphi_{\ii,\jj,\lambda}$ of
the basis ${}_{\jj}B_{\ii}$ of ${_{\jj}R(\nu)_{\ii}}$ described in~\cite{KL}.
Denote by $B_{\ii,\jj,\lambda}$ the image under $\varphi_{\ii,\jj,\lambda}$ of
the product basis ${}_{\jj}B_{\ii}\times \{ \text{monomials in $\Pi_{\lambda}$}
\}$ in ${}_{\jj}R(\nu)_{\ii}\times \Pi_{\lambda}$. The lemma implies that
$B_{\ii,\jj,\lambda}$ is a spanning set in $\HOMU(\cal{E}_{\ii}\onel,
\cal{E}_{\jj}\onel)$.

Let
\begin{equation}
 \cal{E}_{\nu}\onel := \bigoplus_{\ii \in \seq(\nu)}\cal{E}_{\ii}\onel, \qquad
 \quad \cal{E}_{-\nu}\onel := \bigoplus_{\ii \in \seq(\nu)} \cal{E}_{-\ii}\onel .
\end{equation}
$\cal{E}_{\nu}\onel$ and $\cal{E}_{-\nu}\onel$ are 1-morphisms in $\Ucat$.
Summing $\varphi_{\ii,\jj,\lambda}$ over all $\ii,\jj \in \seq(\nu)$, we obtain a
homomorphism
\begin{equation} \label{eq_phi_nu_lambda}
  \phi_{\nu,\lambda} \maps R(\nu) \otimes_{\Bbbk} \Pi_{\lambda} \longrightarrow
   \ENDU(\cal{E}_{\nu}\onel)=\End_{\Ucatq}(\cal{E}_{\nu}\onel).
\end{equation}

\begin{prop}
Homomorphism $\phi_{\nu,\lambda}$ of graded $\Bbbk$-algebras is surjective.
\end{prop}

$\square$

Adding a downward orientation to a diagram $D$ in $R(\nu)$, multiplying it by
$(-1)^a$ where $a$ is the number of crossings of identically colored lines, and
placing it to the left of a collection of bubbles representing a monomial in
$\Pi_{\lambda}$ induces a surjective homomorphism
\begin{equation}
\phi_{-\nu,\lambda} \maps R(\nu) \otimes_{\Bbbk} \Pi_{\lambda} \longrightarrow
\ENDU(\cal{E}_{-\nu}\onel).
\end{equation}

%
\subsubsection{Spanning sets for general $\ii$ and $\jj$}
%

We now describe a spanning set $B_{\ii,\jj,\lambda}$ in $\HOMU \left(
\cal{E}_{\ii}\onel, \cal{E}_{\jj}\onel\right) $ for any $\ii,\jj,\lambda$. Recall
that $p'(\ii,\jj)$ denotes the set of $(\ii,\jj)$-pairings, and $p(\ii,\jj)$ is a
set of minimal diagrammatic representatives of these pairings. For each diagram
$D\in p(\ii,\jj)$ choose an interval on each of the arcs, away from the
intersections. The basis $B_{\ii,\jj,\lambda}$ consists of the union, over all
$D$, of diagrams built out of $D$ by putting an arbitrary number of dots on each
of the intervals and placing any diagram representing a monomial in
$\Pi_{\lambda}$ to the right of $D$ decorated by these dots. Notice that
$B_{\ii,\jj,\lambda}$ depends on extra choices, which we assume are made once and
for all. For example, for $\ii=(-j,+i,+j,-i)$ and $\jj=(+i,-i)$ the choice of
minimal diagrams is unique, see example following Theorem~\ref{thm_form_formula}.
The choice of intervals for the dots is not unique, though. Choosing these
intervals as shown below
\[
  \vcenter{\xy 0;/r.18pc/:
    (-6,10)*{};(6,10)*{} **\crv{(-6,-5) & (6,-5)}?(0)*\dir{<};
    ?(.35)*\dir{|}?(.6)*\dir{|};;
  (-4,-20)*{};(12,-20)*{} **\crv{(-4,-10) & (12,-10)}?(1)*\dir{>};
   ?(.65)*\dir{|}?(.8)*\dir{|};
   (4,-20)*{};(-12,-20)*{} **\crv{(4,-10) &
(-12,-10)}?(1)*\dir{>};
  ?(.65)*\dir{|}?(.8)*\dir{|};
  (10,-5)*{\lambda};(-6,13)*{\scs +i};(6,13)*{\scs -i};
  (-13,-23)*{\scs -j};(-4.5,-23)*{\scs +i};(3.5,-23)*{\scs +j};(11,-23)*{\scs -i};
 \endxy}
 \qquad \qquad
   \vcenter{\xy 0;/r.18pc/:
    (-6,10)*{};(-4,-20)*{} **\crv{(-6,-5) & (-4,-5)}?(0)*\dir{<};
    ?(.4)*\dir{|}?(.6)*\dir{|};;
  (12,-20)*{};(6,10)*{} **\crv{(12,0) & (6,-10)}?(0)*\dir{<};
   ?(.65)*\dir{|}?(.8)*\dir{|};
   (4,-20)*{};(-12,-20)*{} **\crv{(4,-10) &
(-12,-10)}?(1)*\dir{>};
  ?(.65)*\dir{|}?(.8)*\dir{|};
  (16,-5)*{\lambda};(-6,13)*{\scs +i};(6,13)*{\scs -i};
  (-13,-23)*{\scs -j};(-4.5,-23)*{\scs +i};(3.5,-23)*{\scs +j};(11,-23)*{\scs -i};
 \endxy}
\]
results in the spanning set $B_{(-j,+i,+j,-i),(+i,-i),\lambda}$ whose elements
are
\[
  \vcenter{\xy 0;/r.18pc/:
    (-6,10)*{};(6,10)*{} **\crv{(-6,-5) & (6,-5)}?(0)*\dir{<};
    ?(.5)*\dir{}+(0,0)*{\bullet}+(0,-3)*{\scs a_1};
  (-4,-20)*{};(12,-20)*{} **\crv{(-4,-10) & (12,-10)}?(1)*\dir{>};
   ?(.73)*\dir{}+(0,0)*{\bullet}+(1,3)*{\scs a_3};
   (4,-20)*{};(-12,-20)*{} **\crv{(4,-10) &
(-12,-10)}?(1)*\dir{>};
  ?(.73)*\dir{}+(0,0)*{\bullet}+(-1,3)*{\scs a_2};;
  (-6,13)*{\scs +i};(6,13)*{\scs -i};
  (-13,-23)*{\scs -j};(-4.5,-23)*{\scs +i};(3.5,-23)*{\scs +j};(11,-23)*{\scs -i};
 (58,6)*{\xybox{
 (14,-2)*{\cbub{\la i, \lambda\ra-1+\alpha_2}{i}};
 (36,-2)*{\ccbub{ \;\;\;-\la j, \lambda\ra-1+\alpha_4}{j}};
 (14,13)*{\cbub{\la i, \lambda\ra-1+\alpha_1}{i}};
 (37,13)*{\ccbub{\;\;\;-\la j, \lambda\ra-1+\alpha_3}{j}};
 (65,8)*{\cbub{\la k, \lambda\ra-1+\alpha_5}{k}};
 (-2,0)*{\lambda}}};
 (62,-15)*{\underbrace{\hspace{1.9in}}};
 (62,-20)*{\text{bubble monomial in $\Pi_{\lambda}$}};
 \endxy}
 \]
\[
   \vcenter{\xy 0;/r.18pc/:
    (-6,10)*{};(-4,-20)*{} **\crv{(-6,-5) & (-4,-5)}?(0)*\dir{<};
    ?(.5)*\dir{}+(0,0)*{\bullet}+(-3,0)*{\scs a_1};
  (12,-20)*{};(6,10)*{} **\crv{(12,0) & (6,-10)}?(0)*\dir{<};
   ?(.73)*\dir{}+(0,0)*{\bullet}+(3,0)*{\scs a_3};
   (4,-20)*{};(-12,-20)*{} **\crv{(4,-10) &
(-12,-10)}?(1)*\dir{>};
  ?(.73)*\dir{}+(0,0)*{\bullet}+(1,3)*{\scs a_3};
  (-6,13)*{\scs +i};(6,13)*{\scs -i};
  (-13,-23)*{\scs -j};(-4.5,-23)*{\scs +i};(3.5,-23)*{\scs +j};(11,-23)*{\scs -i};
 (58,6)*{\xybox{
 (14,-2)*{\cbub{\la i, \lambda\ra-1+\alpha_2}{i}};
 (36,-2)*{\ccbub{ \;\;\;-\la j, \lambda\ra-1+\alpha_4}{j}};
 (14,13)*{\cbub{\la i, \lambda\ra-1+\alpha_1}{i}};
 (37,13)*{\ccbub{\;\;\;-\la j, \lambda\ra-1+\alpha_3}{j}};
 (65,8)*{\cbub{\la k, \lambda\ra-1+\alpha_5}{k}};
 (-2,0)*{\lambda}}};
 (62,-15)*{\underbrace{\hspace{1.9in}}};
 (62,-20)*{\text{bubble monomial in $\Pi_{\lambda}$}};
 \endxy}
\]
over all nonnegative integers $a_1, a_2, a_3$ and over all diagrammatic monomials
in $\Pi_{\lambda}$ (bubble orientations are for the case $\la i,\lambda \ra \geq
0$, $\la j,\lambda\ra <0$, $\la k,\lambda \ra \geq 0$).

\begin{prop} For any intermediate choices made, $B_{\ii,\jj,\lambda}$
spans the $\Bbbk$-vector space $\HOMU \left( \cal{E}_{\ii}\onel,
 \cal{E}_{\jj}\onel\right) $.
\end{prop}

\begin{proof} Relations on 2-morphisms in $\Ucat(\cal{E}_{\nu}\onel)$ allow arbitrary homotopies
of colored dotted diagrams modulo lower order terms, i.e., terms with fewer
crossings, fewer circles, etc. Detailed discussion in~\cite[Section 8]{Lau1}
generalizes to the present situation without difficulty.
\end{proof}

For each $s\in B_{\ii,\jj,\lambda}$ its degree $\deg(s)$ is an integer,
determined by the rules in Section~\ref{subsec_Uprime}.

\begin{prop} For any $\ii,\jj$, and $\lambda$
 \begin{equation} \label{eq_pi_bi_s}
 \pi (E_{\ii}1_{\lambda},E_{\jj}1_{\lambda})=  \sum_{s\in B_{\ii,\jj,\lambda}} q^{\deg(s)}.   \end{equation}
\end{prop}

\begin{proof} The LHS of the equation \eqref{eq_pi_bi_s} equals $\pi$ times the
RHS of the formula \eqref{eq_thm_pairing}. This $\pi$ is matched in the RHS of
\eqref{eq_pi_bi_s} by the summation over all monomials in $\Pi_{\lambda}$, since
$q$ to the degree of these monomials add up to $\pi$. The product term in the RHS
of \eqref{eq_thm_pairing} is matched by the contribution to the RHS of
\eqref{eq_pi_bi_s} by all possible placements of dots. For each $i$-labelled
strand dots contribute
$$\sum_{a=0}^{\infty} q^{a (i\cdot i)}= \frac{1}{1-q^{i\cdot i}}=
\frac{1}{1-q_i^2}$$ to the product, since the degree of a dot is $i\cdot i$, and
the sum is over all ways to put some number $a$ of dots on this strand. Finally,
the sums over all minimal diagrams $D\in p(\ii,\jj)$ give equal contributions to
the two sides of \eqref{eq_pi_bi_s}.
\end{proof}

\begin{rem} In view of the first equality in \eqref{eq_thm_pairing}, we
can restate the above proposition via $\sla,\sra$ in place of $(,)$ on the LHS.
\end{rem}

Notice that
\begin{equation} \label{eq_another_equality}
{\rm gdim}_{\Bbbk}\left( \HOMU \left( \cal{E}_{\ii}\onel,
 \cal{E}_{\jj}\onel\right)  \right)\quad = \sum_{s\in B_{\ii,\jj,\lambda}} q^{\deg(s)}
\end{equation}
if and only if $B_{\ii,\jj,\lambda}$ is a basis of $\HOMU \left(
\cal{E}_{\ii}\onel,\cal{E}_{\jj}\onel\right)$.

\begin{cor} \label{cor-ineq} For any sequences $\ii,\jj$ and $\lambda \in X$
\begin{equation}
{\rm gdim}_{\Bbbk}\left( \HOMU \left( \cal{E}_{\ii}\onel,
 \cal{E}_{\jj}\onel\right)  \right)\quad \leq \quad  \pi \la E_{\ii}1_{\lambda},
E_{\jj}1_{\lambda} \ra.
\end{equation}
\end{cor}

\begin{defn} We say that our graphical calculus is \emph{nondegenerate}
for a given root system and field $\Bbbk$ if for all $\ii,\jj,\lambda$ the set
$B_{\ii,\jj,\lambda}$ is a basis of $\HOMU \left( \cal{E}_{\ii}\onel,
 \cal{E}_{\jj}\onel\right)$.
\end{defn}

Thus, a calculus is nondegenerate if the equality holds in
Corollary~\ref{cor-ineq} for all $\ii,\jj,\lambda$.

\begin{rem} Nondegeneracy holds if the above condition is true for
all $\lambda\in X$ and all pairs of positive sequences $\ii,\jj$.
\end{rem}

%
\subsubsection{Endomorphisms of  $\cal{E}_{\nu,-\nu'}\onel$ }
%

For $\nu,\nu' \in \N[I]$ let
\begin{equation}
  \cal{E}_{\nu,-\nu'}\onel:= \bigoplus_{
   \xy
  (0,3)*{\scs \ii \in \seq(\nu)};(0,-1)*{\scs \jj \in \seq(\nu')};
   \endxy
  }
   \cal{E}_{\ii -\jj}\onel.
\end{equation}
Consider the graded ring $\ENDU(\cal{E}_{\nu,-\nu'}\onel)$.  A spanning set for
this ring is given by dotted minimal diagrams of
$(\ii(-\ii'),\jj(-\jj'))$-pairings over all $\ii,\jj \in \seq(\nu)$, $\ii',\jj'
\in \seq(\nu')$ times bubble monomials in $\Pi_{\lambda}$.
\[
\xy 0;/r.18pc/:
 (-4,-15)*{}; (-20,25) **\crv{(-3,-6) & (-20,4)}?(0)*\dir{<}?(.6)*\dir{}+(0,0)*{\bullet};
 (-12,-15)*{}; (-4,25) **\crv{(-12,-6) & (-4,0)}?(0)*\dir{<}?(.6)*\dir{}+(.2,0)*{\bullet};
 ?(0)*\dir{<}?(.75)*\dir{}+(.2,0)*{\bullet};?(0)*\dir{<}?(.9)*\dir{}+(0,0)*{\bullet};
 (-28,25)*{}; (-12,25) **\crv{(-28,10) & (-12,10)}?(0)*\dir{<};
  ?(.2)*\dir{}+(0,0)*{\bullet}?(.35)*\dir{}+(0,0)*{\bullet};
  (-50,-15)*{}; (-50,25) **\crv{(-51,10) & (-49,10)}?(1)*\dir{>}?(.35)*\dir{}+(.2,0)*{\bullet};;
 (-36,-15)*{}; (-36,25) **\crv{(-34,-6) & (-35,4)}?(1)*\dir{>};
 (-28,-15)*{}; (-42,25) **\crv{(-28,-6) & (-42,4)}?(1)*\dir{>};
 (-42,-15)*{}; (-20,-15) **\crv{(-42,-5) & (-20,-5)}?(1)*\dir{>};
 (14,-5)*{\cbub{\la i, \lambda\ra-1+\alpha_2}{i}};
 (36,-5)*{\ccbub{ \;\;\;-\la j, \lambda\ra-1+\alpha_4}{j}};
 (14,15)*{\cbub{\la i, \lambda\ra-1+\alpha_1}{i}};
 (37,15)*{\ccbub{\;\;\;-\la j, \lambda\ra-1+\alpha_3}{j}};
 (65,8)*{\cbub{\la k, \lambda\ra-1+\alpha_5}{k}};
 (22,24)*{\lambda};
 (-39,-18)*{\underbrace{\hspace{.75in}}};(-39,-22)*{\scs \ii};
 (-12,-18)*{\underbrace{\hspace{.55in}}};(-12,-22)*{\scs \ii'};
 (-39,28)*{\overbrace{\hspace{.75in}}};(-39,32)*{\scs \jj};
 (-12,28)*{\overbrace{\hspace{.55in}}};(-12,32)*{\scs \jj'};
 \endxy
\]
Let $\cal{I}_{\nu, -\nu',\lambda}$ be the subspace spanned by diagrams which
contain a $U$-turn, i.e. an arc with both endpoints on $\R \times \{1\}$ or on
$\R\times\{0\}$.

\begin{prop}
$\cal{I}_{\nu, -\nu',\lambda}$ is a 2-sided homogeneous ideal of
$\ENDU(\cal{E}_{\nu,-\nu'}\onel)$ which does not depend on choices of minimal
diagrams for pairings.
\end{prop}

\begin{proof}
Left to the reader.
\end{proof}

Denote by $R_{\nu,-\nu',\lambda}:= \ENDU(\cal{E}_{\nu,-\nu'}\onel)/\cal{I}_{\nu,
-\nu',\lambda}$ the graded quotient ring, and by $\beta$ the quotient map.  There
is a homomorphism
\begin{equation}
 \alpha \maps  R(\nu) \otimes_{\Bbbk} R(\nu') \otimes_{\Bbbk} \Pi_{\lambda} \longrightarrow
\ENDU(\cal{E}_{\nu,-\nu'}\onel)
\end{equation}
given by placing diagrams representing elements in $R(\nu)$, $R(\nu')$,
$\Pi_{\lambda}$ in parallel
\[
(-1)^{a} \;\;\; \vcenter{\xy 0;/r.18pc/: (-48,6)*{\xybox{
 (0,20)*{}; (20,20) **\dir{-};(20,0)*{}; (20,20) **\dir{-};(0,0)*{}; (20,0) **\dir{-};
 (0,0)*{}; (0,20) **\dir{-}; (10,10)*{D};
 (2,20)*{}; (2,28)*{}**\dir{-}?(1)*\dir{>};(2,0)*{}; (2,-8)*{}**\dir{-};
 (10,20)*{}; (10,28)*{}**\dir{-}?(1)*\dir{>};(10,0)*{}; (10,-8)*{}**\dir{-};
 (18,20)*{}; (18,28)*{}**\dir{-}?(1)*\dir{>}; (18,0)*{}; (18,-8)*{}**\dir{-}; }};
(-16,6)*{\xybox{
 (0,20)*{}; (28,20) **\dir{-};(28,0)*{}; (28,20) **\dir{-};(0,0)*{}; (28,0) **\dir{-};
 (0,0)*{}; (0,20) **\dir{-}; (14,10)*{D'};
 (2,20)*{}; (2,28)*{}**\dir{-};(2,0)*{}; (2,-8)*{}**\dir{-}?(1)*\dir{>};
 (10,20)*{}; (10,28)*{}**\dir{-};(10,0)*{}; (10,-8)*{}**\dir{-}?(1)*\dir{>};
 (18,20)*{}; (18,28)*{}**\dir{-}; (18,0)*{}; (18,-8)*{}**\dir{-}?(1)*\dir{>};
 (26,20)*{}; (26,28)*{}**\dir{-};(26,0)*{}; (26,-8)*{}**\dir{-}?(1)*\dir{>}; }};
(14,-5)*{\cbub{\la i, \lambda\ra-1+\alpha_2}{i}};
 (36,-5)*{\ccbub{ \;\;\;-\la j, \lambda\ra-1+\alpha_4}{j}};
 (14,15)*{\cbub{\la i, \lambda\ra-1+\alpha_1}{i}};
 (37,15)*{\ccbub{\;\;\;-\la j, \lambda\ra-1+\alpha_3}{j}};
 (65,8)*{\cbub{\la k, \lambda\ra-1+\alpha_5}{k}};
 (22,24)*{\lambda};
 \endxy} ~,
\]
orienting diagrams for $R(\nu)$ upwards and diagrams for $R(\nu')$ downwards, and
multiplying by $(-1)^a$, where $a$ is the number of crossings in $D'$ of equally
labelled strands. The composition $\beta\alpha$ of this homomorphism with the
quotient map is surjective. The diagram below contains an exact sequence of a
ring, its 2-sided ideal, and the quotient ring:
\begin{equation}
\label{eq_diagram_K}
 \xy
 (-55,0)*+{0}="1";
  (-40,0)*+{\cal{I}_{\nu,-\nu',\lambda}}="2";
 (0,0)*+{\ENDU(\cal{E}_{\nu,-\nu'}\onel)}="3";
 (40,0)*+{R_{\nu,-\nu',\lambda}}="4";
 (55,0)*+{0}="5";
 (0,20)*+{R(\nu) \otimes R(\nu') \otimes \Pi_{\lambda}}="6";
 {\ar "1";"2"}; {\ar "2";"3"};{\ar_{\beta} "3";"4"};{\ar "4";"5"};
 {\ar_{\alpha} "6";"3"};{\ar@{->>}^{\beta\alpha} "6";"4"};
 \endxy .
\end{equation}

\begin{rem} If the graphical calculus is nondegenerate, $\beta\alpha$
is an isomorphism, and the sequence splits
\begin{equation} \label{eq_diagram_Ksplit}
\xymatrix{
 0 \ar[r] & \cal{I}_{\nu,-\nu',\lambda} \ar[r] &
\ENDU(\cal{E}_{\nu,-\nu'}\onel) \ar@<.5ex>[r]^{\beta} & \ar@<.5ex>[l]^{\alpha}
 R(\nu) \otimes R(\nu') \otimes \Pi_{\lambda} \ar[r] & 0 .
}
\end{equation}
A split ring homomorphism induces a split exact sequence of Grothendieck
groups~\cite[Section 1.5]{Ros}
\begin{equation}\label{eq_diagram_KspG}
\xymatrix{ 0 \ar[r] & K_0(\cal{I}_{\nu,-\nu',\lambda}) \ar[r]&
K_0\big(\ENDU(\cal{E}_{\nu,-\nu'}\onel) \big) \ar@<.5ex>[rr]^-{K_0(\beta)} &&
\ar@<.5ex>[ll]^-{K_0(\alpha)} K_0\big( R(\nu) \otimes R(\nu') \otimes
\Pi_{\lambda}\big) \ar[r] & 0} ,
\end{equation}
leading to a canonical decomposition of the middle term as the sum of its two
neighbors.
\end{rem}

%
\subsection{Properties and symmetries of 2-category $\Ucat$}
%

%
\subsubsection{Almost biadjoints}
%

The 1-morphism $\cal{E}_{+i}\onel$ does not have a simultaneous left and right
adjoint $\cal{E}_{-i}\mathbf{1}_{\lambda+i_X}$ because the units and counits
which realize these biadjoints in $\Ucatq$ are not degree preserving. However, if
we shift $\cal{E}_{-i}\mathbf{1}_{\lambda+i_X}$ by $\{-c_{+i,\lambda}\}$, then
the unit and counit for the adjunction $\cal{E}_{+i}\onel \dashv
\cal{E}_{-i}\mathbf{1}_{\lambda+i_X}\{-c_{+i,\lambda} \}$ become degree
preserving.  More generally, we have $\cal{E}_{+}i\onel \{t\} \dashv
\cal{E}_{-i}\mathbf{1}_{\lambda+i_X}\{-c_{+i,\lambda}-t\}$ in $\Ucat$ since the
units and counits have degree:
\begin{eqnarray}
  \deg\left(  \xy
    (0,-3)*{\bbpef{i}};
    (8,-5)*{ \lambda};
    (-4,3)*{\scs -i};
    (14,3)*{\scs +i\;\{t-c_{+i,\lambda}-t\}};
    (-8,0)*{};(12,0)*{};
    \endxy \right)
    & = &
    c_{+i,\lambda}+(-c_{+i,\lambda})
    \quad = \quad 0 \\
    \deg\left(\xy
    (0,0)*{\bbcfe{i}};
    (8,5)*{ \lambda+i_X};
    (-4,-6)*{\scs +i};
    (14,-6)*{\scs -i\;\{t-c_{+i,\lambda}-t\}};
    (-8,0)*{};(12,0)*{};
    \endxy\right)
    & = &
    (c_{-i,\lambda+i_X}))-(-c_{+i,\lambda})
    \quad = \quad 0
\end{eqnarray}
and still satisfy the zigzag identities.  Similarly, $\cal{E}_{+i}\onel \{t\}$
has a left adjoint $\cal{E}_{-i}\mathbf{1}_{\lambda+i_X} \{c_{+i,\lambda}-t\}$ in
$\Ucat$. One can check that with these shifts the units and counits of the
adjunction $\cal{E}_{-i}\mathbf{1}_{\lambda+i_X} \{c_{+i,\lambda}-t\}\dashv
\cal{E}_{+i}\onel \{t\}$ become degree zero and are compatible with the zigzag
identities.

The left adjoint $\cal{E}_{-i}\mathbf{1}_{\lambda+i_X}\{c_{+i,\lambda}-t\}$ and
the right adjoint $\cal{E}_{-i}\mathbf{1}_{\lambda+i_X}\{-c_{+i,\lambda}-t\}$ of
$\cal{E}_{+i}\onel \{t\}$ only differ by a shift. We call morphisms with this
property {\em almost biadjoint}. This situation is familiar to those studying
derived categories of coherent sheaves on Calabi-Yau manifolds. Functors with
these properties are called `almost Frobenius functors' in \cite{Kh3} where
several other examples of this phenomenon are also given.

It is then clear that $\cal{E}_{+i}\onel \{t\}$ and $\cal{E}_{-i}\onel \{t\}$
have almost biadjoints in $\Ucat$ for all $t \in \Z$ and $\lambda \in X$ with
\begin{equation} \label{eq_almost_biadjoints}
\begin{array}{ccccc}
 \onel \cal{E}_{-i}\mathbf{1}_{\lambda+i_X} \{c_{+i,\lambda}-t \} & \;\dashv\;
 & \mathbf{1}_{\lambda+i_X}\cal{E}_{+i}\onel \{t\} & \;\dashv\; & \onel \cal{E}_{-i}\mathbf{1}_{\lambda+i_X} \{-c_{+i,\lambda}-t \} \\
  \onel \cal{E}_{+i}\mathbf{1}_{\lambda-i_X} \{c_{-i,\lambda}-t \} & \;\dashv\; &
  \mathbf{1}_{\lambda-i_X}\cal{E}_{-i}\onel \{t\} & \;\dashv\;
  & \onel \cal{E}_{+i}\mathbf{1}_{\lambda-i_X} \{-c_{-i,\lambda}-t \}.
\end{array}
\end{equation}
Every 1-morphism in $\Ucat$ is a direct sum of composites of $\cal{E}_{+i}\onel
\{t\}$'s and $\cal{E}_{-i}\onel \{t\}$'s together with identities; by composing
adjunctions as explained in \cite[Section 5.5]{Lau1}, the right and left adjoints
of $\cal{E}_{\ii}\onel \{t\}$ can be computed. Thus, it is clear that all
1-morphisms in $\Ucat$ have almost biadjoints.

\paragraph{Positivity of bubbles:}
$\Ucat(\onel ,\onel \{t\}) = 0$ if $t >0$, and $\Ucat(\onel,\onel)$ is at most
1-dimensional (isomorphic to $\Bbbk$ if the calculus is nondegenerate).

%
\subsubsection{Symmetries of $\Ucat$ } \label{sec_symm}
%

We denote by $\Ucat^{\op}$ the 2-category with the same objects as $\Ucat$ but
the 1-morphisms reversed.  The direction of the 2-morphisms remain fixed. The
2-category $\Ucat^{\co}$ has the same objects and 1-morphism as $\Ucat$, but the
directions of the 2-morphisms is reversed. That is, $\Ucat^{\co}(x,y)=\Ucat(y,x)$
for 1-morphisms $x$ and $y$. Finally, $\Ucat^{\co\op}$ denotes the 2-category
with the same objects as $\Ucat$, but the directions of the 1-morphisms and
2-morphisms have been reversed.

Using the symmetries of the diagrammatic relations imposed on $\Ucat$ we
construct 2-functors on the various versions of $\Ucat$.  In
Proposition~\ref{prop_tilde_lifts} we relate these 2-functors to various
$\Z[q,q^{-1}]$-(anti)linear (anti)automorphisms of the algebra $\U$. The various
forms of contravariant behaviour for 2-functors on $\Ucat$ translate into
properties of the corresponding homomorphism in $\U$ as the following table
summarizes:
\begin{center}
\begin{tabular}{|l|l|}
  \hline
  {\bf 2-functors} & {\bf Algebra maps} \\ \hline \hline
  $\Ucat \to \Ucat$ &  $\Z[q,q^{-1}]$-linear
 homomorphisms\\
  $\Ucat \to \Ucat^{\op}$ & $\Z[q,q^{-1}]$-linear
antihomomorphisms \\
  $\Ucat \to \Ucat^{\co}$ & $\Z[q,q^{-1}]$-antilinear
 homomorphisms \\
  $\Ucat \to \Ucat^{\co\op}$ & $\Z[q,q^{-1}]$-antilinear
antihomomorphisms \\
  \hline
\end{tabular}
\end{center}

\paragraph{Rescale, invert the orientation, and send $\lambda \mapsto -\lambda$:}

Consider the operation on the diagrammatic calculus that rescales the
$ii$-crossing $\Ucross_{i,i,\lambda} \mapsto -\Ucross_{i,i,\lambda}$ for all $i
\in I$ and $\lambda \in X$, inverts the orientation of each strand and sends
$\lambda \mapsto -\lambda$:
\[
\tilde{\omega}\left(\;\;\vcenter{\xy 0;/r.16pc/:
 (-4,-15)*{}; (-20,25) **\crv{(-3,-6) & (-20,4)}?(0)*\dir{<}?(.6)*\dir{}+(0,0)*{\bullet};
 (-12,-15)*{}; (-4,25) **\crv{(-12,-6) & (-4,0)}?(0)*\dir{<}?(.6)*\dir{}+(.2,0)*{\bullet};
 ?(0)*\dir{<}?(.75)*\dir{}+(.2,0)*{\bullet};?(0)*\dir{<}?(.9)*\dir{}+(0,0)*{\bullet};
 (-28,25)*{}; (-12,25) **\crv{(-28,10) & (-12,10)}?(0)*\dir{<};
  ?(.2)*\dir{}+(0,0)*{\bullet}?(.35)*\dir{}+(0,0)*{\bullet};
 (-36,-15)*{}; (-36,25) **\crv{(-34,-6) & (-35,4)}?(1)*\dir{>};
 (-28,-15)*{}; (-42,25) **\crv{(-28,-6) & (-42,4)}?(1)*\dir{>};
 (-42,-15)*{}; (-20,-15) **\crv{(-42,-5) & (-20,-5)}?(1)*\dir{>};
 (6,10)*{\cbub{}{i}};
 (-23,0)*{\cbub{}{j}};
 (8,-4)*{\lambda};(-44,-4)*{\mu};
 (-44,-19)*{\scs +i};(-36,-19)*{\scs +k};(-29,-19)*{\scs +j};(-21,-19)*{\scs -i};(-13,-19)*{\scs -j};(-4,-19)*{ \scs -j};
 (-44,29)*{\scs +j};(-36,29)*{\scs+k};  (-29,29)*{\scs+k};(-21,29)*{\scs-j};(-13,29)*{\scs-k};
 (-4,29)*{\scs-j};
 \endxy}\;\;\right) \quad = \quad - \;\;
 \vcenter{\xy 0;/r.16pc/:
 (-4,-15)*{}; (-20,25) **\crv{(-3,-6) & (-20,4)}?(1)*\dir{>}?(.6)*\dir{}+(0,0)*{\bullet};
 (-12,-15)*{}; (-4,25) **\crv{(-12,-6) & (-4,0)}?(1)*\dir{>}?(.6)*\dir{}+(.2,0)*{\bullet};
 ?(1)*\dir{>}?(.75)*\dir{}+(.2,0)*{\bullet};?(.9)*\dir{}+(0,0)*{\bullet};
 (-28,25)*{}; (-12,25) **\crv{(-28,10) & (-12,10)}?(1)*\dir{>};
  ?(.2)*\dir{}+(0,0)*{\bullet}?(.35)*\dir{}+(0,0)*{\bullet};
 (-36,-15)*{}; (-36,25) **\crv{(-34,-6) & (-35,4)}?(0)*\dir{<};
 (-28,-15)*{}; (-42,25) **\crv{(-28,-6) & (-42,4)}?(0)*\dir{<};
 (-42,-15)*{}; (-20,-15) **\crv{(-42,-5) & (-20,-5)}?(0)*\dir{<};
 (6,10)*{\ccbub{}{i}};
 (-23,0)*{\ccbub{}{j}};
 (8,-4)*{-\lambda};(-44,-4)*{-\mu};
 (-46,-19)*{\scs -i};(-36,-19)*{\scs -k};(-29,-19)*{\scs -j};(-21,-19)*{\scs +i};(-13,-19)*{\scs +j};
 (-4,-19)*{\scs +j};
 (-46,29)*{\scs  -j};(-36,29)*{\scs  -k};  (-29,29)*{\scs  -k};(-21,29)*{\scs  +j};(-13,29)*{\scs  +k};
 (-4,29)*{\scs  +j};
 \endxy}
\]
This transformation preserves the degree of a diagram, so by extending to sums of
diagrams we get a 2-functor $\tilde{\omega}\maps \Ucat \to \Ucat$ given by
\begin{eqnarray}
  \tilde{\omega} \maps \Ucat &\to& \Ucat \nn \\
  \lambda &\mapsto&  -\lambda \nn \\
  \mathbf{1}_{\mu} \cal{E}_{\ii}\onel \{t\}
 &\mapsto &
  \mathbf{1}_{-\mu} \cal{E}_{-\ii} \mathbf{1}_{-\lambda}  \{t\}
\end{eqnarray}
It is straight forward to check that $\tilde{\omega}$ is a strict 2-functor. In
fact, it is a 2-isomorphism since its square is the identity.

\paragraph{Rescale, reflect across the $y$-axis, and send $\lambda \mapsto -\lambda$: }

The operation on diagrams that rescales the $ii$-crossing $\Ucross_{i,i,\lambda}
\mapsto -\Ucross_{i,i,\lambda}$ for all $i \in I$ and $\lambda \in X$, reflects a
diagram across the y-axis, and sends $\lambda$ to $-\lambda$ leaves invariant the
relations on the 2-morphisms of $\Ucat$. This operation
\[
\tilde{\sigma}\left(\;\;\vcenter{\xy 0;/r.16pc/:
 (-4,-15)*{}; (-20,25) **\crv{(-3,-6) & (-20,4)}?(0)*\dir{<}?(.6)*\dir{}+(0,0)*{\bullet};
 (-12,-15)*{}; (-4,25) **\crv{(-12,-6) & (-4,0)}?(0)*\dir{<}?(.6)*\dir{}+(.2,0)*{\bullet};
 ?(0)*\dir{<}?(.75)*\dir{}+(.2,0)*{\bullet};?(0)*\dir{<}?(.9)*\dir{}+(0,0)*{\bullet};
 (-28,25)*{}; (-12,25) **\crv{(-28,10) & (-12,10)}?(0)*\dir{<};
  ?(.2)*\dir{}+(0,0)*{\bullet}?(.35)*\dir{}+(0,0)*{\bullet};
 (-36,-15)*{}; (-36,25) **\crv{(-34,-6) & (-35,4)}?(1)*\dir{>};
 (-28,-15)*{}; (-42,25) **\crv{(-28,-6) & (-42,4)}?(1)*\dir{>};
 (-42,-15)*{}; (-20,-15) **\crv{(-42,-5) & (-20,-5)}?(1)*\dir{>};
 (6,10)*{\cbub{}{i}};
 (-23,0)*{\cbub{}{j}};
 (8,-4)*{\lambda};(-44,-4)*{\mu};
 (-44,-19)*{\scs +i};(-36,-19)*{\scs +k};(-29,-19)*{\scs +j};(-21,-19)*{\scs -i};(-13,-19)*{\scs -j};(-4,-19)*{ \scs -j};
 (-44,29)*{\scs +j};(-36,29)*{\scs+k};  (-29,29)*{\scs+k};(-21,29)*{\scs-j};(-13,29)*{\scs-k};
 (-4,29)*{\scs-j};
 \endxy}\;\;\right) \quad = \quad -\;\;
 \vcenter{\xy 0;/r.16pc/:
 (4,-15)*{}; (20,25) **\crv{(3,-6) & (20,4)}?(0)*\dir{<}?(.6)*\dir{}+(0,0)*{\bullet};
 (12,-15)*{}; (4,25) **\crv{(12,-6) & (4,0)}?(0)*\dir{<}?(.6)*\dir{}+(.2,0)*{\bullet};
 ?(0)*\dir{<}?(.75)*\dir{}+(.2,0)*{\bullet};?(0)*\dir{<}?(.9)*\dir{}+(0,0)*{\bullet};
 (28,25)*{}; (12,25) **\crv{(28,10) & (12,10)}?(0)*\dir{<};
  ?(.2)*\dir{}+(0,0)*{\bullet}?(.35)*\dir{}+(0,0)*{\bullet};
 (36,-15)*{}; (36,25) **\crv{(34,-6) & (35,4)}?(1)*\dir{>};
 (28,-15)*{}; (42,25) **\crv{(28,-6) & (42,4)}?(1)*\dir{>};
 (42,-15)*{}; (20,-15) **\crv{(42,-5) & (20,-5)}?(1)*\dir{>};
 (-6,10)*{\ccbub{}{i}};
 (23,0)*{\ccbub{}{j}};
 (-8,-4)*{-\lambda};(44,-4)*{-\mu};
 (23,5)*{\xybox{
 (-44,-19)*{\scs -j};(-36,-19)*{\scs -j};(-29,-19)*{\scs -i};(-21,-19)*{\scs +j};(-13,-19)*{\scs +k};
 (-4,-19)*{ \scs +i};
 (-44,29)*{\scs -j};(-36,29)*{\scs-k};  (-29,29)*{\scs-j};(-21,29)*{\scs+k};(-13,29)*{\scs+k};
 (-4,29)*{\scs+j};}};
 \endxy}
\]
is contravariant for composition of 1-morphisms, covariant for composition of
2-morphisms, and preserves the degree of a diagram.  Hence, this symmetry gives a
2-isomorphism
\begin{eqnarray}
  \tilde{\sigma} \maps \Ucat &\to& \Ucat^{\op} \nn \\
  \lambda &\mapsto&  -\lambda \nn \\
   \mathbf{1}_{\mu}  \cal{E}_{s_1} \cal{E}_{s_2} \cdots
    \cal{E}_{s_{m-1}}\cal{E}_{s_m}\onel \{t\}
 &\mapsto &
  \mathbf{1}_{-\lambda} \cal{E}_{s_m}  \cal{E}_{s_{m-1}} \cdots
    \cal{E}_{s_2}\cal{E}_{s_1}\mathbf{1}_{-\mu} \{t\} \nn
\end{eqnarray}
and on 2-morphisms $\tilde{\sigma}$ maps linear combinations of diagrams to the
linear combination of the diagrams obtained by applying the above transformation
to each summand. The relations on $\Ucat$ are symmetric under this
transformation, and $\tilde{\sigma}$ is a 2-functor.  The square of
$\tilde{\sigma}$ is the identity.

\paragraph{Reflect across the x-axis and invert orientation:}
Here we are careful to keep track of what happens to the shifts of sources and
targets
\[
\tilde{\psi}\left(\;\;\vcenter{\xy 0;/r.16pc/:
 (-4,-15)*{}; (-20,25) **\crv{(-3,-6) & (-20,4)}?(0)*\dir{<}?(.6)*\dir{}+(0,0)*{\bullet};
 (-12,-15)*{}; (-4,25) **\crv{(-12,-6) & (-4,0)}?(0)*\dir{<}?(.6)*\dir{}+(.2,0)*{\bullet};
 ?(0)*\dir{<}?(.75)*\dir{}+(.2,0)*{\bullet};?(0)*\dir{<}?(.9)*\dir{}+(0,0)*{\bullet};
 (-28,25)*{}; (-12,25) **\crv{(-28,10) & (-12,10)}?(0)*\dir{<};
  ?(.2)*\dir{}+(0,0)*{\bullet}?(.35)*\dir{}+(0,0)*{\bullet};
 (-36,-15)*{}; (-36,25) **\crv{(-34,-6) & (-35,4)}?(1)*\dir{>};
 (-28,-15)*{}; (-42,25) **\crv{(-28,-6) & (-42,4)}?(1)*\dir{>};
 (-42,-15)*{}; (-20,-15) **\crv{(-42,-5) & (-20,-5)}?(1)*\dir{>};
 (6,10)*{\cbub{}{i}};
 (-23,0)*{\cbub{}{j}};
 (8,-4)*{\lambda};(-44,-4)*{\mu};
 (-44,-19)*{\scs +i};(-36,-19)*{\scs +k};(-29,-19)*{\scs +j};(-21,-19)*{\scs -i};(-13,-19)*{\scs -j};(-4,-19)*{ \scs -j};
 (-44,29)*{\scs +j};(-36,29)*{\scs+k};  (-29,29)*{\scs+k};(-21,29)*{\scs-j};(-13,29)*{\scs-k};
 (-4,29)*{\scs-j}; (3,-19)*{\scs \{t\}};(3,29)*{\scs \{t'\}};
 \endxy}\;\;\right) \quad = \quad
  \vcenter{\xy 0;/r.16pc/:
 (-4,15)*{}; (-20,-25) **\crv{(-3,6) & (-20,-4)}?(1)*\dir{>}?(.6)*\dir{}+(0,0)*{\bullet};
 (-12,15)*{}; (-4,-25) **\crv{(-12,6) & (-4,0)}?(1)*\dir{>}?(.6)*\dir{}+(.2,0)*{\bullet};
 ?(1)*\dir{>}?(.75)*\dir{}+(.2,0)*{\bullet};?(.9)*\dir{}+(0,0)*{\bullet};
 (-28,-25)*{}; (-12,-25) **\crv{(-28,-10) & (-12,-10)}?(1)*\dir{>};
  ?(.2)*\dir{}+(0,0)*{\bullet}?(.35)*\dir{}+(0,0)*{\bullet};
 (-36,15)*{}; (-36,-25) **\crv{(-34,6) & (-35,-4)}?(0)*\dir{<};
 (-28,15)*{}; (-42,-25) **\crv{(-28,6) & (-42,-4)}?(0)*\dir{<};
 (-42,15)*{}; (-20,15) **\crv{(-42,5) & (-20,5)}?(0)*\dir{<};
 (6,-10)*{\cbub{}{i}};
 (-23,0)*{\cbub{}{j}};
 (8,4)*{\lambda};(-44,4)*{\mu};
 (-46,19)*{\scs +i};(-38,19)*{\scs+k};(-29,19)*{\scs+j};(-21,19)*{\scs-i};(-13,19)*{\scs-j};(-4,19)*{\scs-j};
 (-46,-29)*{\scs+j};(-38,-29)*{\scs+k};  (-29,-29)*{\scs+k};(-21,-29)*{\scs-j};(-13,-29)*{\scs-k}; (-4,-29)*{\scs-j};
 (3,-29)*{\scs \;\;\{-t'\}};(3,19)*{\scs\;\; \{-t\}};
 \endxy}
\]
The degree shifts on the right hand side are required for this transformation to
preserve the degree of a diagram.  This transformation preserves the order of
composition of 1-morphisms, but is contravariant with respect to composition of
2-morphisms.  Hence, by extending this transformation to sums of diagrams we get
a 2-isomorphism given by
\begin{eqnarray}
  \tilde{\psi} \maps \Ucat &\to& \Ucat^{\co} \nn \\
  \lambda &\mapsto&  \lambda \nn \\
  \mathbf{1}_{\mu} \cal{E}_{\ii}\onel \{t\}
 &\mapsto &
 \mathbf{1}_{\mu} \cal{E}_{\ii}\onel \{-t\}
\end{eqnarray}
and on 2-morphisms $\tilde{\psi}$ reflects the diagrams representing summands
across the $x$-axis and inverts the orientation.  Again, the relations on $\Ucat$
possess this symmetry so it is not difficult to check that $\tilde{\psi}$ is a
2-functor. Furthermore, it is clear that $\tilde{\psi}$ is invertible since its
square is the identity.

\bigskip

It is easy to see that these 2-functors commute with each other `on-the-nose'.
That is, we have equalities
\begin{equation}
  \tilde{\omega} \tilde{\sigma} = \tilde{\omega} \tilde{\sigma}, \qquad \tilde{\sigma} \tilde{\psi} = \tilde{\psi} \tilde{\sigma}, \qquad
  \tilde{\omega} \tilde{\psi} = \tilde{\psi} \tilde{\omega}
\end{equation}
The composite 2-functor $\tilde{\psi} \tilde{\omega} \tilde{\sigma}$ is given by
\begin{eqnarray}
  \tilde{\psi} \tilde{\omega} \tilde{\sigma} \maps \Ucat &\to& \Ucat^{\co\op} \nn \\
  \lambda &\mapsto&  \lambda \nn \\
   \mathbf{1}_{\mu}  \cal{E}_{s_1} \cal{E}_{s_2} \cdots
    \cal{E}_{s_{m-1}}\cal{E}_{s_m}\onel \{t\}
 &\mapsto &
  \mathbf{1}_{\lambda} \cal{E}_{-s_m}  \cal{E}_{-s_{m-1}} \cdots
    \cal{E}_{-s_2}\cal{E}_{-s_1}\mathbf{1}_{\mu} \{-t\} \nn
\end{eqnarray}
and is given on 2-morphisms by rotating diagrams by $180^{\circ}$.

The following transformation only differs from $\tilde{\psi} \tilde{\omega}
\tilde{\sigma}$ by a shift and is given by taking adjoints.

\paragraph{Rotation by $180^{\circ}$ (taking right adjoints):}
This transformation is a bit more subtle because it uses the almost biadjoint
structure of $\Ucat$, in particular, the calculus of mates (see \cite[Section
4.3]{Lau1}). For each $\mathbf{1}_{\mu}x\onel  \in \Ucat$ denote its right
adjoint by $\onel y\mathbf{1}_{\mu}$. The symmetry of rotation by $180^{\circ}$
can also be realized by the 2-functor that sends a 1-morphism
$\mathbf{1}_{\mu}x\onel$ to its right adjoint $\onel y\mathbf{1}_{\mu}$ and each
2-morphism $\zeta \maps \mathbf{1}_{\mu}x\onel  \To \mathbf{1}_{\mu}x'\onel$ to
its mate under the adjunctions $\mathbf{1}_{\mu}x\onel  \dashv \onel
y\mathbf{1}_{\mu}$ and $\mathbf{1}_{\mu}x'\onel \dashv \onel y'\mathbf{1}_{\mu}$.
That is, $\zeta$ is mapped to its right dual $\zeta^*$. Pictorially,
\[
 \xy 0;/r.16pc/:
 (0,0)*{\bullet}+(3,1)*{\scs \zeta};
 (0,-16);(0,16) **\dir{-};
 (11,-14)*{\lambda}; (-11,14)*{\lambda'};
 (0,-18)*{\scs x};(-0,18)*{\scs x'};
 \endxy
 \quad \rightsquigarrow \quad
  \xy 0;/r.16pc/:
 (8,4)*{}="1";
 (0,4)*{}="2";
 (0,-4)*{}="2'";
 (-8,-4)*{}="3";
 (0,0)*{\bullet}+(3,1)*{\scs \zeta};
 (8,-16);"1" **\dir{-};
 "2";"2'" **\dir{-};
 "1";"2" **\crv{(8,12) & (0,12)};
 "2'";"3" **\crv{(0,-12) & (-8,-12)};
 "3";(-8,16) **\dir{-};
 (-11,-14)*{\lambda}; (11,14)*{\lambda'};
 (8,-18)*{\scs y'};(-8,18)*{\scs y};
 (-3.5,-12)*{\scs};
 (4,12)*{\scs };
 \endxy
 \quad = \quad
  \xy 0;/r.16pc/:
 (0,0)*{\bullet}+(5,1)*{\scs \zeta^*};
 (0,-16);(0,16) **\dir{-};
 (11,-14)*{\lambda'}; (-11,14)*{\lambda};
 (0,-18)*{\scs y'};(-0,18)*{\scs y};
 \endxy
\]
This transformation is contravariant with respect to composition of 1-morphisms
and 2-morphisms. We get a 2-functor
\begin{eqnarray}
  \tilde{\tau} \maps \Ucat &\to& \Ucat^{\co\op} \nn \\
  \lambda &\mapsto&  \lambda \nn \\
   \mathbf{1}_{\mu}  \cal{E}_{s_1} \cal{E}_{s_2} \cdots
    \cal{E}_{s_{m-1}}\cal{E}_{s_m}\onel \{t\}
 &\mapsto &
  \mathbf{1}_{\lambda} \cal{E}_{-s_m}  \cal{E}_{-s_{m-1}} \cdots
    \cal{E}_{-s_2}\cal{E}_{-s_1}\mathbf{1}_{\mu} \{-t+t'\} \nn \\
 \zeta & \mapsto & \zeta^*
\end{eqnarray}
where the degree shift $t'$ for the right adjoint $\mathbf{1}_{\lambda}
\cal{E}_{-s_m}  \cal{E}_{-s_{m-1}} \cdots
\cal{E}_{-s_2}\cal{E}_{-s_1}\mathbf{1}_{\mu} \{-t+t''\}$, determined by
\eqref{eq_almost_biadjoints},  ensures that $\tilde{\tau}$ is degree preserving.
Inspection of the relations for $\Ucat$ will reveal that they are invariant under
this transformation so that $\tilde{\tau}$ is a 2-functor.

We can define an inverse for $\tilde{\tau}$ given by taking left adjoints.  We
record this 2-morphism here.
\begin{eqnarray}
  \tilde{\tau}^{-1} \maps \Ucat &\to& \Ucat^{\co\op} \nn \\
  \lambda &\mapsto&  \lambda \nn \\
   \mathbf{1}_{\mu}  \cal{E}_{s_1} \cal{E}_{s_2} \cdots
    \cal{E}_{s_{m-1}}\cal{E}_{s_m}\onel \{t\}
 &\mapsto &
  \mathbf{1}_{\lambda} \cal{E}_{-s_m}  \cal{E}_{-s_{m-1}} \cdots
    \cal{E}_{-s_2}\cal{E}_{-s_1}\mathbf{1}_{\mu} \{-t+t''\} \nn \\
 \zeta & \mapsto & ^*\zeta.
\end{eqnarray}
with degree shift $t''$ determined from \eqref{eq_almost_biadjoints} and the left
dual $^*\zeta$ of the 2-morphism $\zeta$ defined in \cite[Section 4.3]{Lau1}.

\begin{rem}
The composition $\tilde{\tau} \tilde{\psi} \tilde{\omega} \tilde{\sigma} \maps
\Ucat \to \Ucat$ gives 2-isomorphism that fixes all diagrams and only effects the
grading shifts.
\end{rem}

\begin{rem} \label{lem_right_adjoints}
There are degree zero isomorphisms of graded $\Bbbk$-vector spaces
\begin{eqnarray}
 \Ucatq(f x,y) &\to& \Ucatq(x,\tilde{\tau}(f)y), \\
  \Ucatq(x, g y) &\to& \Ucatq(\tilde{\tau}^{-1}(g)x,y),
\end{eqnarray}
for all 1-morphisms $f,g,x,y$ in $\Ucatq$, defined at the end of
Section~\ref{subsubsec_definition}.
\end{rem}

%
\subsection{Karoubi envelope, $\UcatD$, and 2-representations}
\label{subsec_Karoubi}
%

The Karoubi envelope $Kar(\cal{C})$ of a category $\cal{C}$ is an enlargement of
the category $\cal{C}$ in which all idempotents split (see \cite[Section 9]{Lau1}
and references therein). There is a fully faithful functor $\cal{C} \to
Kar(\cal{C})$ that is universal with respect to functors which split idempotents
in $\cal{C}$. This means that if $F\maps \cal{C} \to \cal{D}$ is any functor
where all idempotents split in $\cal{D}$, then $F$ extends uniquely (up to
isomorphism) to a functor $\tilde{F} \maps Kar(\cal{C}) \to \cal{D}$ (see for
example \cite{Bor}, Proposition 6.5.9). Furthermore, for any functor $G \maps
\cal{C} \to \cal{D}$ and a natural transformation $\alpha \maps F \To G$,
$\alpha$ extends uniquely to a natural transformation $\tilde{\alpha} \maps
\tilde{F}\To\tilde{G}$.  When $\cal{C}$ is additive the inclusion $\cal{C} \to
Kar(\cal{C})$ is an additive functor.

\begin{defn}
Define the additive $\Bbbk$-linear 2-category $\UcatD$ to have the same objects
as $\Ucat$ and hom additive $\Bbbk$-linear categories given by
$\UcatD(\lambda,\lambda') = Kar\left(\Ucat(\lambda,\lambda')\right)$. The
fully-faithful additive $\Bbbk$-linear functors $\Ucat(\lambda,\lambda') \to
\UcatD(\lambda,\lambda')$ combine to form an additive $\Bbbk$-linear 2-functor
$\Ucat \to \UcatD$ universal with respect to splitting idempotents in the hom
categories $\UcatD(\lambda,\lambda')$.  The composition functor
$\UcatD(\lambda,\lambda') \times \UcatD(\lambda',\lambda'') \to
\UcatD(\lambda,\lambda'')$ is induced by the universal property of the Karoubi
envelope from the composition functor for $\Ucat$. The 2-category $\UcatD$ has
graded 2-homs given by
\begin{equation}
\HOM_{\UcatD}(x,y) := \bigoplus_{t\in \Z}\Hom_{\UcatD}(x\{t\},y).
\end{equation}
\end{defn}

\begin{defn} A 2-representation of $\Ucatq$ is a (weak) graded additive $\Bbbk$-linear
2-functor $\Psi^* \maps \Ucatq \to \cal{M}^*$, where $\cal{M}^*$ is a graded
additive $\Bbbk$-linear 2-category with a translation.

 A 2-representation of $\Ucat$ is an additive $\Bbbk$-linear 2-functor $\Psi \maps
\Ucat \to \cal{M}$ that respects the grading.  This happens when there is an
additive $\Bbbk$-linear 2-functor $\cal{M} \to \cal{M}^*$, with $\cal{M}^*$ a
graded additive $\Bbbk$-linear 2-category, making the diagram
\begin{equation}
 \xymatrix{ \Ucat \ar[d]_-{\Psi} \ar[r]& \Ucatq \ar[d]^-{\Psi^*} \\
  \cal{M} \ar[r] & \cal{M}^* }
\end{equation}
weakly commutative.  Thus, to study 2-representation of $\Ucat$ it suffices to
study 2-representation of $\Ucatq$ and then restrict to degree--preserving
2-morphisms.

A 2-representation of $\UcatD$ is a additive $\Bbbk$-linear 2-functor $\dot{\Psi}
\maps \UcatD \to \cal{M}$ that respects the grading.
\end{defn}

Let $\cal{M}$ be a 2-category as above in which idempotents split. Any
2-representation $\Psi^*\maps\Ucatq \to \cal{M}^*$ gives a unique (up to
isomorphism) 2-representation $\dot{\Psi}\maps \UcatD \to \cal{M}$.  The
2-functor $\dot{\Psi}$ is obtained from $\Psi^*$ by restricting to the degree
preserving 2-morphisms of $\Ucatq$ and using the universal property of the
Karoubi envelope. This is illustrated schematically below:

\[
  \xy
 (-40,10)*+{\Ucatq}="q"; (0,10)*++{\Ucat}="0"; (50,10)*+{\UcatD}="d";
 (-40,-15)*+{\cal{M}^* }="kq"; (0,-15)*++{\cal{M}}="k";
  {\ar_{\Psi^*} "q";"kq"};
  {\ar^{\Psi} "0";"k"};
  {\ar "0";"d"};{\ar^{\dot{\Psi}} "d";"k"};
  {\ar@<.6em> "q";"0"}; {\ar@<.6em>@{_{(}->} "0";"q"};
  {\ar@{_{(}->} "k";"kq"};
  (-20,20)*{\txt\small{restrict to \\ degree 0 \\ 2-morphisms}};
  (25,15)*{\txt\small{Karoubian envelope}};
 \endxy
\]

\begin{rem}
The 2-functors $\tilde{\omega}$, $\tilde{\sigma}$, $\tilde{\psi}$, $\tilde{\tau}$
on $\Ucat$ extend to 2-functors on $\UcatD$, for which we use the same notations.
For example,
\begin{eqnarray}
  \tilde{\omega} \maps \UcatD &\to& \UcatD \nn \\
  \lambda &\mapsto&  -\lambda \nn \\
  \Big( \cal{E}_{\ii}\onel \{t\}, e \Big)
 &\mapsto &
 \Big( \tilde{\omega}\big(\cal{E}_{\ii}\onel \{t\}\big),\tilde{\omega}(e) \Big) \nn\\
 \zeta &\mapsto& \tilde{\omega}(\zeta)
\end{eqnarray}
and the other 2-morphisms $\tilde{\sigma}$, $\tilde{\psi}$, and $\tilde{\tau}$
are defined analogously.  In particular, each 1-morphism in $\UcatD$ has left and
right adjoints.
\end{rem}

%
\subsection{Direct sum decompositions} \label{subsec_dirsumdecs}
%

Recall that $\cal{E}_{\nu}\onel$ is the direct sum of $\cal{E}_{\ii}\onel$ over
all $\ii \in \seq(\nu)$:
\begin{equation}
\cal{E}_{\nu}\onel := \bigoplus_{\ii \in \seq(\nu)} \cal{E}_{\ii}\onel.
\end{equation}
Due to the existence of the homomorphism $\phi_{\nu,\lambda}$ in the formula
\eqref{eq_phi_nu_lambda} any degree $0$ idempotent $e$ of $R(\nu)$ gives rise to
the idempotent $\phi_{\nu,\lambda}(e)$ of $\cal{E}_{\nu}\onel$ and to the
1-morphism $\left(\cal{E}_{\nu}\onel,\phi_{\nu,\lambda}\left(e\right)\right)$ of
$\UcatD$.

Introduce idempotents
\[
e_{+i,m} = \xy 0;/r.15pc/:
 (-12,-20)*{}; (12,20) **\crv{(-12,-8) & (12,8)}?(1)*\dir{>};
 (-4,-20)*{}; (4,20) **\crv{(-4,-13) & (12,2) & (12,8)&(4,13)}?(1)*\dir{>};?(.88)*\dir{}+(0.1,0)*{\bullet};
 (4,-20)*{}; (-4,20) **\crv{(4,-13) & (12,-8) & (12,-2)&(-4,13)}?(1)*\dir{>}?(.86)*\dir{}+(0.1,0)*{\bullet};
 ?(.92)*\dir{}+(0.1,0)*{\bullet};
 (12,-20)*{}; (-12,20) **\crv{(12,-8) & (-12,8)}?(1)*\dir{>}?(.70)*\dir{}+(0.1,0)*{\bullet};
 ?(.90)*\dir{}+(0.1,0)*{\bullet};?(.80)*\dir{}+(0.1,0)*{\bullet};
 \endxy
 \in \Ucat(\cal{E}_{mi}\onel,\cal{E}_{mi}\onel),
\qquad  e_{-i,m} = (-1)^{\frac{m(m-1)}{2}}\;\xy 0;/r.15pc/:
 (-12,-20)*{}; (12,20) **\crv{(-12,-8) & (12,8)}?(0)*\dir{<};
 (-4,-20)*{}; (4,20) **\crv{(-4,-13) & (12,2) & (12,8)&(4,13)}?(0)*\dir{<};?(.88)*\dir{}+(0.1,0)*{\bullet};
 (4,-20)*{}; (-4,20) **\crv{(4,-13) & (12,-8) & (12,-2)&(-4,13)}?(0)*\dir{<}?(.86)*\dir{}+(0.1,0)*{\bullet};
 ?(.92)*\dir{}+(0.1,0)*{\bullet};
 (12,-20)*{}; (-12,20) **\crv{(12,-8) & (-12,8)}?(0)*\dir{<}?(.70)*\dir{}+(0.1,0)*{\bullet};
 ?(.90)*\dir{}+(0.1,0)*{\bullet};?(.80)*\dir{}+(0.1,0)*{\bullet};
 \endxy
\in \Ucat(\cal{E}_{-mi}\onel,\cal{E}_{-mi}\onel)
\]
similar to the idempotent $e_{i,m}$ in \cite[Section 2.2]{KL} and \cite{KL2}.
Define 1-morphisms $\cal{E}_{+i^{(m)}}\onel$ and $\cal{E}_{-i^{(m)}}\onel$ in
$\UcatD$ by
\begin{equation}
\cal{E}_{+i^{(m)}}\onel : = \left( \cal{E}_{+i^m}\onel, e_{+i,m}\right)\left\{
\frac{m(1-m)}{2} \frac{i \cdot i}{2} \right\},
\end{equation}
\begin{equation}
\cal{E}_{-i^{(m)}}\onel : = \left( \cal{E}_{-i^m}\onel, e_{-i,m}\right)\left\{
\frac{m(1-m)}{2} \frac{i \cdot i}{2} \right\}.
\end{equation}
As in \cite{KL,KL2}, we have direct sum decompositions
\begin{equation}
  \cal{E}_{+i^m}\onel \cong \left(\cal{E}_{+i^{(m)}}\onel\right)^{\oplus[m]_i!}, \qquad \quad
  \cal{E}_{-i^m}\onel \cong \left(\cal{E}_{-i^{(m)}}\onel\right)^{\oplus[m]_i!}. \nn
\end{equation}

For any divided power sequence $\ii =
(\epsilon_1i_1^{(a_1)},\epsilon_2i_2^{(a_2)}, \dots, \epsilon_mi_m^{(a_m)})$
define
\begin{equation}
  \cal{E}_{\ii}\onel := (\cal{E}_{\hat{\ii}}\onel,e_{\ii}), \nn
\end{equation}
where $\hat{\ii}$ is the sequence
\begin{equation}
  \big(\epsilon_1 i_1, \dots,\epsilon_1 i_1, \epsilon_2 i_2, \dots \epsilon_2 i_2, \dots
  , \epsilon_m i_m \dots, \epsilon_m i_m\big) = \big((\epsilon_1 i_1)^{a_1} (\epsilon_2
  i_2)^{a_2} \dots (\epsilon_m i_m)^{a_m}\big), \nn
\end{equation}
with term $\epsilon_1i_1$ repeating $a_1$ times, term $\epsilon_2i_2$ repeating
$a_2$ times, etc., and
\begin{equation}
  e_{\ii} = e_{\epsilon_1i_1,a_1} \cdot e_{\epsilon_2i_2,a_2}\cdot \ldots \cdot
   e_{\epsilon_mi_m,a_m}
\end{equation}
is the horizontal product of idempotents.

When interested in only one part of a sequence $\ii$, we write $\dots \ii''
\dots$ instead of $\ii = \ii' \ii''\ii'''$ and $\cal{E}_{\dots\ii''\dots}\onel$
instead of $\cal{E}_{\ii}\onel =\cal{E}_{\ii'\ii''\ii'''}\onel$.

\begin{prop} \label{serre-isoms}
For each $i,j\in I$, $i \neq j$, and $\lambda \in X$ there are 2-isomorphisms of
1-morphisms in $\UcatD$
\begin{eqnarray}
  \bigoplus_{a=0}^{\lfloor \frac{d+1}{2} \rfloor}
 {\cal{E}}_{\ldots +i^{(2a)}+j +i^{(d+1-2a)} \ldots} \onel &\cong&
   \bigoplus_{a=0}^{\lfloor \frac{d}{2} \rfloor}
 {\cal{E}}_{\ldots +i^{(2a+1)}+j +i^{(d-2a)} \ldots}\onel , \nn\\
 \bigoplus_{a=0}^{\lfloor \frac{d+1}{2} \rfloor}
 {\cal{E}}_{\ldots -i^{(2a)}-j -i^{(d+1-2a)} \ldots} \onel &\cong&
   \bigoplus_{a=0}^{\lfloor \frac{d}{2} \rfloor}
 {\cal{E}}_{\ldots -i^{(2a+1)}-j -i^{(d-2a)} \ldots}\onel ,
\end{eqnarray}
where $d=d_{ij}=-\la i,j_X\ra$.
\end{prop}

\begin{proof}
These isomorphisms follow from categorified quantum Serre relations
\cite[Proposition 2.13]{KL} and \cite[Proposition 6]{KL2} between idempotents in
rings $R(\nu)$, via homomorphisms $\phi_{\nu,\lambda}$ and $\phi_{-\nu,\lambda}$.
\end{proof}

\begin{prop} \label{pmii-isoms}
For each $i\in I$, $\lambda \in X$ there are 2-isomorphisms in $\UcatD$
\[
 \begin{array}{ccl}
   \text{$\cal{E}_{\ii'+i-i\ii''}\onel \cong \cal{E}_{\ii'-i+i\ii''}\onel \oplus_{
  [\la i, \mu \ra]_i} \cal{E}_{\ii'\ii''} \onel$}
  & \quad & \text{if $\la i,\lambda +\ii''_X\ra \geq 0$},
   \\
  \text{$ \cal{E}_{\ii'-i+i\ii''}\onel \cong \cal{E}_{\ii'+i-i\ii''}\onel
    \oplus_{-[\la i, \mu \ra]_i} \cal{E}_{\ii'\ii''} \onel$}& \quad&
     \text{if $\la i, \lambda +\ii''_X\ra \leq 0$,}
 \end{array}
\]
where $\mu=\lambda+\ii''_X$.
\end{prop}

\begin{proof}
Set $\lambda+\ii''_X= \mu$.  The decomposition $\cal{E}_{\ii'+i-i\ii''}\onel
\cong \cal{E}_{\ii'-i+i\ii''}\onel \oplus_{
  [\la i, \mu \ra]_i} \cal{E}_{\ii'\ii''} \onel$  for $\la i, \mu\ra \geq 0$ is given by 2-morphisms in $\Ucat$
\begin{eqnarray}
 \alpha &\maps& \cal{E}_{\ii'+i-i\ii''}\onel   \to
\cal{E}_{\ii'-i+i\ii''}\onel \oplus_{
  [\la i, \mu \ra]_i} \cal{E}_{\ii'\ii''} \onel\\
 \alpha^{-1} &\maps& \cal{E}_{\ii'-i+i\ii''}\onel \oplus_{
  [\la i, \mu \ra]_i} \cal{E}_{\ii'\ii''} \onel
\to  \cal{E}_{\ii'+i-i\ii''}\onel
\end{eqnarray}
where $\alpha$ and $\alpha^{-1}$ consist of matrices of diagrams
\begin{eqnarray}
 \alpha &=&
 \left(
   \begin{array}{c}
    - \;
  \xy 0;/r.15pc/:
  (0,0)*{\xybox{
    (-4,-4)*{};(4,4)*{} **\crv{(-4,-1) & (4,1)}?(1)*\dir{>} ;
    (4,-4)*{};(-4,4)*{} **\crv{(4,-1) & (-4,1)}?(0)*\dir{<};
    (-5,-3)*{\scs i};
     (7,-3)*{\scs i};
     (14,2)*{ \mu};
     }};
  \endxy  \\
  \alpha_0   \\
     \vdots \\
    \alpha_{\la i,\mu\ra-1}  \\
   \end{array}
 \right),
  \quad
 \alpha_s \;\; :=
  \sum_{j=0}^{s}
 \vcenter{\xy 0;/r.16pc/:
           (-4,-2)*{}="t1";
            (4,-2)*{}="t2";
            "t2";"t1" **\crv{(4,5) & (-4,5)}; ?(.03)*\dir{<} ?(.9)*\dir{<}
            ?(.25)*\dir{}+(0,-.1)*{\bullet}+(5,1)*{\scs s-j};;
            (2,13)*{\ccbub{\;\; -\la i,\mu\ra-1+j}{i}};
            (18,18)*{ \mu};
        \endxy}
 \quad \text{for $0 \leq s \leq \la i,\mu\ra-1$},
 \nn\\ \nn\\
 \alpha^{-1} &=&
 \left(
  \begin{array}{cccccc}
 \xy 0;/r.15pc/:
  (0,0)*{\xybox{
    (-4,-4)*{};(4,4)*{} **\crv{(-4,-1) & (4,1)}?(0)*\dir{<} ;
    (4,-4)*{};(-4,4)*{} **\crv{(4,-1) & (-4,1)}?(1)*\dir{>};
    (-7,-3)*{\scs i};
     (6,-3)*{\scs i};
     (10,2)*{ \mu};
     }};
  \endxy
  &
  \vcenter{\xy 0;/r.15pc/:
            (-8,0)*{};
           (-4,2)*{}="t1";
            (4,2)*{}="t2";
            "t2";"t1" **\crv{(4,-5) & (-4,-5)}; ?(.1)*\dir{>} ?(.95)*\dir{>}
            ?(.2)*\dir{}+(0,-.1)*{\bullet}+(4,-2)*{\scs \quad \la i,\mu\ra-1};;
        \endxy}
  & \dots
   &  \vcenter{\xy 0;/r.15pc/:
            (-7,0)*{};
           (-4,2)*{}="t1";
            (4,2)*{}="t2";
            "t2";"t1" **\crv{(4,-5) & (-4,-5)}; ?(.1)*\dir{>} ?(.95)*\dir{>}
            ?(.2)*\dir{}+(0,-.1)*{\bullet}+(4,-2)*{\scs \qquad \la i,\mu\ra-1-s};;
        \endxy}
   & \dots &
   \vcenter{\xy 0;/r.15pc/:
            (0,0)*{};
           (-4,2)*{}="t1";
            (4,2)*{}="t2";
            "t2";"t1" **\crv{(4,-5) & (-4,-5)}; ?(.1)*\dir{>} ?(.95)*\dir{>}
            ?(.2)*\dir{};
        \endxy} \nn \\
    \end{array}
    \right) .
\end{eqnarray}
Note that all bubbles that appear in $\alpha_s$ above are fake bubbles.  One can
check that $\alpha^{-1}\alpha ={\rm Id}_{\cal{E}_{\ii'+i-i\ii''}\onel}$ and that
$\alpha\alpha^{-1} ={\rm Id}_{\cal{E}_{\ii'-i+i\ii''}\onel\oplus_{[\la
i,\mu\ra]}\onel}$ using the $\mathfrak{sl}_2$-relations (for details see
\cite{Lau1}). Here we have taken
\begin{equation}
 \oplus_{[\la i,\mu\ra]} \onel = \onel\{1-\la i,\mu\ra\} \oplus
 \cdots \oplus \onel\{2s+1-\la i,\mu\ra\} \oplus \cdots \onel\{\la i,\mu\ra-1\}
\end{equation}
so that $\alpha$ and $\alpha^{-1}$ have degree zero.  The isomorphism $
\cal{E}_{\ii'-i+i\ii''}\onel  \cong \cal{E}_{\ii'+i-i\ii''}\onel
    \oplus_{-[\la i, \mu \ra]_i} \cal{E}_{\ii'\ii''} \onel$ for $\la i, \lambda +\ii''_X\ra \leq 0$ is given similarly (see
\cite{Lau1}).
\end{proof}

\begin{prop} \label{pmij-isoms}
For each $i,j \in I$, $i \neq j$, $\lambda\in X$ there are 2-isomorphisms
\begin{eqnarray}
  \cal{E}_{\ldots+i-j\ldots}\onel \cong \cal{E}_{\ldots-j+i\ldots}\onel
\end{eqnarray}
\end{prop}

\begin{proof}
The degree zero 2-isomorphism $\cal{E}_{\ldots+i-j\ldots}\onel \cong
\cal{E}_{\ldots-j+i\ldots}\onel$ for $i \neq j$ is given by maps
\begin{eqnarray}
 \xy 0;/r.15pc/:
  (0,0)*{\xybox{
    (-4,-4)*{};(4,4)*{} **\crv{(-4,-1) & (4,1)}?(1)*\dir{>} ;
    (4,-4)*{};(-4,4)*{} **\crv{(4,-1) & (-4,1)}?(0)*\dir{<};
    (-5,-3)*{\scs i};
     (7,-3)*{\scs j};
     (10,2)*{ \lambda};
     }};
  \endxy &\maps & \cal{E}_{\ldots+i-j\ldots}\onel \to
  \cal{E}_{\ldots-j+i\ldots}\onel \nn \\
   \xy 0;/r.15pc/:
  (0,0)*{\xybox{
    (-4,-4)*{};(4,4)*{} **\crv{(-4,-1) & (4,1)}?(0)*\dir{<} ;
    (4,-4)*{};(-4,4)*{} **\crv{(4,-1) & (-4,1)}?(1)*\dir{>};
    (-7,-3)*{\scs j};
     (6,-3)*{\scs i};
     (10,2)*{ \lambda};
     }};
  \endxy &\maps &  \cal{E}_{\ldots-j+i\ldots}\onel\to
   \cal{E}_{\ldots+i-j\ldots}\onel\nn
\end{eqnarray}
To see that these maps are isomorphisms use \eqref{eq_downup_ij-gen}.
\end{proof}

%
\subsection{$K_0(\UcatD)$ and homomorphism $\gamma$} \label{subsec_KzeroU}
%

$K_0(\UcatD)$ can be viewed as a pre-additive category or, alternatively, as an
idempotented ring.  When thought of as a category, it has objects $\lambda$, over
all $\lambda \in X$.  The abelian group of morphisms
$K_0\big(\UcatD(\lambda,\mu)\big)$ is defined as the (split) Grothendieck group
of the additive category $\UcatD(\lambda,\mu)$. The split Grothendieck group
$K_0(\mc{A})$ of an additive category $\mc{A}$ has generators $[P]$, over all
objects $P$ of $\mc{A}$, and relations $[P]=[P']+[P'']$ whenever $P\cong P'\oplus
P''$. In the case of $\UcatD(\lambda,\mu)$ the generators are
$[\cal{E}_{\ii}\onel\{t\},e]$, where $\mu=\lambda+\ii_X$, $t\in \Z$, and
\begin{equation}
  e \in \End_{\UcatD}(\cal{E}_{\ii}\onel\{t\}) \cong
  \End_{\UcatD}(\cal{E}_{\ii}\onel) = \End_{\Ucat}(\cal{E}_{\ii}\onel)
\end{equation}
is an idempotent (degree zero idempotent when viewed as an element of the larger
ring $\ENDU(\cal{E}_{\ii}\onel) $).  The defining relations are
\begin{equation}
  [\cal{E}_{\ii}\onel\{t\},e] = [\cal{E}_{\ii'}\mathbf{1}_{\lambda'}\{t'\},e']
 + [\cal{E}_{\ii''}\mathbf{1}_{\lambda''}\{t''\},e'']
\end{equation}
whenever there is an isomorphism in $\UcatD(\lambda,\mu)$
\begin{equation}
(\cal{E}_{\ii}\onel\{t\},e) \cong (\cal{E}_{\ii'}\mathbf{1}_{\lambda'}\{t'\},e')
\oplus (\cal{E}_{\ii''}\mathbf{1}_{\lambda''}\{t''\},e'').
\end{equation}
Moreover, $K_0\big(\UcatD(\lambda,\mu)\big)$ is a $\Z[q,q^{-1}]$-module, with
multiplication by $q$ coming from the grading shift
\begin{equation}
  [\cal{E}_{\ii}\onel \{t+1\},e\{1\}] = q[\cal{E}_{\ii}\onel\{t\},e].
\end{equation}
We write $[\cal{E}_{\ii}\onel]$ instead of $[\cal{E}_{\ii}\onel,1]$, where $1$ is
the identity 2-morphism of $\cal{E}_{\ii}\onel$.

The space of homs between any two objects in $\UcatD(\lambda,\mu)$ is a
finite-dimensional $\Bbbk$-vector space.  In particular, the Krull-Schmidt
decomposition theorem holds, and an indecomposable object of
$\UcatD(\lambda,\mu)$ has the form $(\cal{E}_{\ii}\onel\{t\},e)$ for some
minimal/primitive idempotent $e$.  Any presentation of $1=e_1+\dots+e_k$ into the
sum of minimal mutually-orthogonal idempotents gives rise to a decomposition
\begin{equation}
  \cal{E}_{\ii}\onel\{t\} \cong \bigoplus_{r=1}^{k}(\cal{E}_{\ii}\onel\{t\},e_r)
\end{equation}
into a direct sum of indecomposable objects of $\UcatD(\lambda,\mu)$. Any object
of $\UcatD(\lambda,\mu)$ has a unique presentation, up to permutation of factors
and isomorphisms, as a direct sum of indecomposables. Choose one representative
${b}$ for each isomorphism class of indecomposables, up to grading shifts, and
denote by $\dot{\mc{B}}(\lambda,\mu)$ the set of these representatives. Then
$\{[{b}]\}_b$ is a basis of $K_0\big(\UcatD(\lambda,\mu)\big)$, viewed as a free
$\Z[q,q^{-1}]$-module. Composition bifunctors
\begin{equation}
  \UcatD(\lambda,\lambda') \times \UcatD(\lambda',\lambda'') \longrightarrow
  \UcatD(\lambda,\lambda'')
\end{equation}
induce $\Z[q,q^{-1}]$-bilinear maps
\begin{equation}
  K_0\big(\UcatD(\lambda,\lambda')\big) \otimes K_0\big(\UcatD(\lambda',\lambda'')\big) \longrightarrow
  K_0\big(\UcatD(\lambda,\lambda'')\big)
\end{equation}
turning $K_0(\UcatD)$ into a $\Z[q,q^{-1}]$-linear additive category with objects
$\lambda\in X$. Alternatively, we may view $K_0(\UcatD)$ as a non-unital
$\Z[q,q^{-1}]$-algebra
\begin{equation}
  \bigoplus_{\lambda,\mu \in X} K_0(\Ucat)(\lambda,\mu)
\end{equation}
with a family of idempotents $[\onel]$.  The set $\dot{\mc{B}} :=
\bigoplus_{\lambda,\mu \in X} \dot{\mc{B}}(\lambda,\mu)$ gives rise to a basis
$[\dot{\mc{B}}]:=\{[{b}] \}_{b\in \dot{\mc{B}}}$ of idempotented
$\Z[q,q^{-1}]$-algebra $K_0(\UcatD)$.  Notice that basis elements are defined up
to multiplication by powers of $q$;  we will not try to choose a canonical
grading normalization here.  Multiplication in this basis has coefficients in
$\N[q,q^{-1}]$.

Both $\UA$ and $K_0(\UcatD)$ are idempotented $\Z[q,q^{-1}]$-algebras, with the
idempotents $1_{\lambda}$ and $[\onel]$ labelled by $\lambda\in X$. To relate the
two algebras, send $1_{\lambda}$ to $[\onel]$ and, more generally,
$E_{\ii}1_{\lambda} $ to $[\Eol]$ for all $\ii \in \sseqd$.

\begin{prop}
The assignment $E_{\ii}1_{\lambda} \longrightarrow [\Eol]$ extends to
$\Z[q,q^{-1}]$-algebra homomorphism
\begin{equation}
  \gamma \maps \UA \longrightarrow K_0(\UcatD).
\end{equation}
\end{prop}

Multiplication by $q$ corresponds to the grading shift $\{ 1\}.$

\begin{proof}  $K_0(\UcatD)$ is a  free $\Z[q,q^{-1}]$-module,
so it is enough to check that the assignment above extends to a homomorphism of
$\Q(q)$-algebras
\begin{equation}\label{gamma-field}
\gamma_{\Q(q)} \maps \U \lra K_0(\UcatD)\otimes_{\Z[q,q^{-1}]}\Q(q)
\end{equation}
($\UA$ is also a free $\Z[q,q^{-1}]$-module, but this fact is not needed in the
proof). Propositions \ref{serre-isoms},~\ref{pmii-isoms}, and~\ref{pmij-isoms}
show that defining relations of $\U$ lift to 2-isomorphisms of 1-morphisms in
$\UcatD$ and, therefore, descend to relations in the Grothendieck group
$K_0(\UcatD)$. Restricting $\gamma_{\Q(q)}$ to $\UA$ gives a homomorphism of
$\Z[q,q^{-1}]$-algebras with the image of the homomorphism lying in
$K_0(\UcatD)$.
\end{proof}

For each $\lambda, \mu\in X$ homomorphism $\gamma$ restricts to a homomorphism of
$\Z[q,q^{-1}]$-modules
\begin{equation}
   1_{\mu}(\UA)1_{\lambda} \longrightarrow
  K_0\left(\UcatD(\lambda,\mu)\right).
\end{equation}

\bigskip

\begin{prop} \label{prop_tilde_lifts}
Homomorphism $\gamma$ intertwines (anti)automorphisms $\psi$, $\omega$, $\sigma$,
$\tau$ of $\UA$ with (anti) automorphisms $[\tilde{\psi}]$, $[\tilde{\omega}]$,
$[\tilde{\sigma}]$, and $[\tilde{\tau}]$ of $K_0(\UcatD)$, respectively, \ie, the
following diagrams commute.
\[
 \xymatrix{\UA \ar[r]^-{\gamma} \ar[d]_-{\psi} & K_0(\UcatD) \ar[d]^-{[\tilde{\psi}]} \\
 \UA \ar[r]_-{\gamma}  & K_0(\UcatD) }
 \qquad \quad
  \xymatrix{\UA \ar[r]^-{\gamma} \ar[d]_-{\omega} & K_0(\UcatD) \ar[d]^-{[\tilde{\omega}]} \\
 \UA \ar[r]_-{\gamma}  & K_0(\UcatD) }
 \qquad \quad
  \xymatrix{\UA \ar[r]^-{\gamma} \ar[d]_-{\sigma} & K_0(\UcatD) \ar[d]^-{[\tilde{\sigma}]} \\
 \UA \ar[r]_-{\gamma}  & K_0(\UcatD) }
 \qquad \quad
  \xymatrix{\UA \ar[r]^-{\gamma} \ar[d]_-{\tau} & K_0(\UcatD) \ar[d]^-{[\tilde{\tau}]} \\
 \UA \ar[r]_-{\gamma}  & K_0(\UcatD) }
\]
$[\tilde{\psi}]$ denotes the induced action of $\tilde{\psi}$ on the Grothendieck
group, etc.
\end{prop}

\begin{proof}
The proof follows from definitions and our construction of $\gamma$.
\end{proof}

The 2-isomorphisms $\tilde{\omega}$, $\tilde{\sigma}$, $\tilde{\psi}$ on $\Ucatq$
give isomorphisms of graded $\Bbbk$-vector spaces
\begin{eqnarray}
  \Ucatq\big(x,y\big) & \cong &
  \Ucatq\big(\tilde{\omega}(x),\tilde{\omega}(y)\big), \\
  \Ucatq\big(x,y\big) & \cong &
  (\Ucatq)^{\op}\big(\tilde{\sigma}(x),\tilde{\sigma}(y)\big)
  \;\; = \;\;  \Ucatq\big(\tilde{\sigma}(x),\tilde{\sigma}(y)\big),\\
    \Ucatq\big(x,y\big) & \cong &
  (\Ucatq)^{\co}\big(\tilde{\psi}(x),\tilde{\psi}(y)\big)
  \;\; = \;\;  \Ucatq\big(\tilde{\psi}(y),\tilde{\psi}(x)\big).
\end{eqnarray}
On Grothendieck rings these isomorphisms give equalities
\begin{eqnarray}
  \sla x,y \sra & =& \sla \omega(x),\omega(y) \sra ,\\
    \sla x,y \sra & =& \sla \sigma(x),\sigma(y) \sra, \\
      \sla x,y \sra & =& \sla \psi(y),\psi(x) \sra, \\
\end{eqnarray}
which should be compared with Propositions 26.1.4 and 26.1.6 in \cite{Lus4} and
property (v) of the semilinear form.   The last equality expressed in terms of
the bilinear form $(,)$, with $x$ replaced by $\psi(x)$, gives the identity
$(x,y) = (y,x)$.

%
\subsection{Idempotented rings} \label{subsec-idemp}
%

An idempotented ring $A$ is an associative ring, not necessarily unital, equipped
with a system of idempotents $\{1_x\}$, over elements $x$ of some set $Z$.   We
require orthogonality $1_{x}1_{y}=\delta_{x,y}1_{x}$ and decomposition
\begin{equation}
  A = \bigoplus_{x, y \in Z} 1_{y}A1_{x}.
\end{equation}
By a (left) $A$-module we mean an $A$-module $M$ such that
\begin{equation}
  M = \bigoplus_{x\in Z}1_{x}M.
\end{equation}

In this paper three collections of idempotented rings appear:
\begin{itemize}
  \item Lusztig's $\U$ and its integral form $\UA$.  Here $Z=X$, the weight
  lattice,
  \begin{equation}
    \U = \bigoplus_{\lambda,\mu \in X} 1_{\mu}\U1_{\lambda}, \qquad
    \UA = \bigoplus_{\lambda,\mu \in X} 1_{\mu}(\UA)1_{\lambda}.
  \end{equation}

  \item The Grothendieck groups $K_0(Kar(\Ucatq))$ and $K_0(\UcatD)$, the latter
defined in Section~\ref{subsec_KzeroU}.  Again, the parameterizing set $Z=X$.  We
only study
  \begin{equation}
K_0(\UcatD) = \bigoplus_{\lambda,\mu \in X} [{\bf 1}_{\mu}]K_0(\UcatD)[\onel],
  \end{equation}
 with $\left\{[\onel] \right\}_{\lambda \in X}$ being the system of idempotents
 in $K_0(\UcatD)$.  Map $\gamma \maps \UA \to K_0(\UcatD)$
is a homomorphism of idempotented rings.

\item For each $\lambda, \mu \in X$ the $\Z$-graded ring
\begin{equation}
  {}_{\mu}\Ucatq_{\lambda} := \bigoplus_{\ii,\jj} \HOMU(\cal{E}_{\ii}\onel,
  \cal{E}_{\jj}\onel),
\end{equation}
where the sum is over all $\ii, \jj \in \sseq$ with $\ii_X,\jj_X=\lambda-\mu$.
Thus, the sum is over all sequences such that $E_{\ii}1_{\lambda}$,
$E_{\jj}1_{\lambda}$ have left weight $\mu$.  The parameterizing set $Z=\left\{
\ii \in \sseq \mid \ii_X = \lambda -\mu\right\}$.
\end{itemize}

The category $\UcatD(\lambda,\mu)$ is equivalent to the category of right
finitely-generated graded projective ${}_{\mu}\Ucatq_{\lambda}$-modules and
grading preserving homomorphisms. The equivalence functor
\begin{equation}
  \UcatD(\lambda,\mu) \longrightarrow {\rm pmod-}{}_{\mu}\Ucatq_{\lambda}
\end{equation}
takes $\cal{E}_{\ii}\onel$ to
\[
 {}_{\ii,\lambda}P := \bigoplus_{\jj \in Z} \HOMU(\cal{E}_{\ii}\onel, \cal{E}_{\jj}\onel),
\]
and, more generally, an object $(\cal{E}_{\ii}\onel,e)$ to
\begin{equation}
  {}_{\ii,\lambda,e}P :=
  \bigoplus_{\jj \in Z}
  \HOM_{\UcatD}\big((\cal{E}_{\ii}\onel,e),\cal{E}_{\jj}\onel\big). \nn
\end{equation}
The Grothendieck group $K_0\big(\UcatD(\lambda,\mu)\big)$ is isomorphic to the
Grothendieck group of ${\rm pmod-}{}_{\mu}\Ucatq_{\lambda}$.

\bigskip

\noindent Notice that we get idempotented rings from the 2-category $\UcatD$ in
various ways:
\begin{enumerate}[1)]
 \item as the Grothendieck ring/pre-additive category $K_0(\UcatD)$ of $\UcatD$,
 \item as rings associated to categories $\UcatD(\lambda,\mu)$.
\end{enumerate}

The 2-category $\UcatD$ can itself be viewed as an idempotented monoidal
category.  We encode these observations into a diagram
\[
 \xy
 (-35,20)*+++{\txt\bf{(small) pre-additive \\ 2-categories}}="t1";
 (35,20)*+++{\txt\bf{idempotented additive \\ monoidal categories}}="t2";
 (-35,-15)*+++{\txt\bf{(small) pre-additive categories}}="m1";
  (35,-15)*+++{\txt\bf{idempotented  rings }}="m2";
  {\ar@{<->} "t1";"t2"};{\ar@{<->} "m1";"m2"};
  {\ar@/_1pc/@{=>}_{\txt{Grothendieck \\ category/ring}\;\;} (-5,15)*+{} ;(-5,-10)*+{}};
  {\ar@/^1pc/@{=>}^{\;\;\txt{Categories/rings of \\homs between objects}} (5,15)*+{} ;(5,-10)*+{}};
 \endxy
\]

%
\subsection{Surjectivity of $\gamma$} \label{subsec_surjectivity}
%

To prove surjectivity of $\gamma$ we will analyze the diagram of $K$-groups and
their homomorphisms induced by the diagram \eqref{eq_diagram_K} and trace minimal
idempotents there, but first recall some basics of Grothendieck groups of
finite-dimensional algebras (as a model example) and graded algebras.

%
\subsubsection{$K_0$ of finite-dimensional algebras}
%

A homomorphism of rings $\alpha:A \to B$ induces a homomorphism of $K_0$-groups
$$K_0(\alpha)\maps K_0(A) \to K_0(B)$$ of finitely-generated projective modules.
For definition and properties of $K_0$ we refer the reader
to~\cite{Ros},~\cite[Chapter II]{Weibel}.

Assume that $A$ and $B$ are finite-dimensional $\Bbbk$-algebras, for a field
$\Bbbk$, and $\alpha$ is a $\Bbbk$-algebra homomorphism. If $\alpha$ is
surjective then $K_0(\alpha)$ is surjective as well. On the level of idempotents,
if $1 = e_1 + \dots +e_k$ is a decomposition of $1\in A$ into a sum of
mutually-orthogonal minimal idempotents, then $Ae_s$ is an indecomposable
projective $A$-module, $K_0(A)$ is a free abelian group with a basis
$\{[Ae_r]\}_{r\in S}$, for a subset $S\subset\{1,\dots,k\}$. $S$ is any maximal
subset with the property that $Ae_s \ncong Ae_t$ as left $A$-modules for any $s,t
\in S$, $s\neq t$. Applying $\alpha$ to the above decomposition results in the
equation
\begin{equation}
  B \ni 1= \alpha(e_1) + \dots + \alpha(e_k).
\end{equation}
where each $\alpha(e_s)$ is either $0$ or a minimal idempotent in $B$. Relabel
minimal idempotents so that $\alpha(e_1),\dots, \alpha(e_m)\neq 0$,
$\alpha(e_{m+1})= \dots = \alpha(e_k)=0$ (elements of $S$ get permuted as well).
Then $1=\alpha(e_1)+\dots + \alpha(e_m)$ is a decomposition of $1\in B$ into a
sum of mutually orthogonal minimal idempotents, $B\alpha(e_r)$ is an
indecomposable projective $B$-module, $1 \leq r \leq m$, and
\begin{equation}
  \{[B\alpha(e_r)] \}_{r \in S\cap \{1,\dots, m\}}
\end{equation}
is a basis of $K_0(B)$.

\begin{rem}
If A is a finite-dimensional $\Bbbk$-algebra, the quotient map $A \to A/J(A)$,
where $J(A)$ is the Jacobson radical of $A$, induces an isomorphism of
$K_0$-groups
\begin{equation} \label{eq_KJA}
   K_0(A) \cong K_0(A/J(A)).
\end{equation}
Indeed, $J(A)$ is a nilpotent ideal, $J(A)^N=0$ for sufficiently large $N$, and
the quotient by a nilpotent ideal induces an isomorphism of $K_0$'s,
see~\cite[Chapter II, Lemma 2.2]{Weibel}.
\end{rem}

\begin{prop}
If $A, B$ are finite-dimensional $\Bbbk$-algebras such that all simple $A$- and
$B$-modules are absolutely irreducible over $\Bbbk$, then
\begin{equation}
  K_0(A) \otimes_{\Z} K_0(B) \cong K_0(A \otimes_{\Bbbk} B)
\end{equation}
via an isomorphism which takes $[P]\otimes[Q]$ for projective $A$, respectively
$B$, modules $P$ and $Q$ to $[P \otimes_{\Bbbk} Q]$.
\end{prop}

\begin{proof}
By passing to $A/J(A)$, $B/J(B)$ and using the above remark, we reduce to the
case of semisimple $A$ and $B$.  Then both $A$ and $B$ are finite products of the
field $\Bbbk$ and the proposition follows.
\end{proof}

%
\subsubsection{$K_0$ of graded algebras}
%

From here on we only consider $\Z$-graded $\Bbbk$-algebras, for a field $\Bbbk$.
For a $\Z$-graded $\Bbbk$-algebra $A=\oplus_{a\in\Z} A_a$ denote by $K_0(A)$ the
Grothendieck group of finitely-generated graded left projective $A$-modules.
$K_0(A)$ is a $\Z[q,q^{-1}]$-module.

Throughout this subsection we assume that all weight spaces $A_a$ are
finite-dimensional, and the grading is bounded below: $A_a=0$ for all $a\ll0$.

Let $PI(A)$ be the set of isomorphism classes of indecomposable graded projective
$A$-modules, up to a grading shift. We can normalize the grading and choose one
representative $Q$ for each element of $PI(A)$ so that $0$ is the lowest
nontrivial degree of $Q$. We write $Q\in PI(A)$.

\begin{prop} \label{prop-freemod}
For $A$ as above, $K_0(A)$ is a free $\Z[q,q^{-1}]$-module with the basis $\{
[Q]\}_{Q\in PI(A)}$.
\end{prop}

\begin{proof} Since each weight space of $A$ is finite-dimensional, the
Krull-Schmidt property holds for graded projective finitely-generated
$A$-modules. Any such module has a unique, up to isomorphism, decomposition as a
direct sum of indecomposables, and $K_0(A)$ is a free abelian group with a basis
labelled by isomorphism classes of indecomposable projectives. Boundedness of $A$
from below ensures that an indecomposable projective is not isomorphic to itself
with a shifted grading, implying that $K_0(A)$ is a free $\Z[q,q^{-1}]$-module
and the rest of the proposition.
\end{proof}

We say that a 2-sided homogeneous ideal $J$ of $A$ is \emph{virtually nilpotent}
if for any $a\in \Z$ the weight space $(J^N)_a=0$ for sufficiently large $N$.

\begin{prop} \label{prop-virtuallynil}
For $A$ as above and $J$ a virtually nilpotent ideal of $A$ the quotient map
$A\lra A/J$ induces an isomorphism $K_0(A)\cong K_0(A/J)$.
\end{prop}

\begin{proof} Proposition follows from the lifting idempotents property.
This is the graded version of Lemma 2.2 in~\cite[Chapter 2]{Weibel}.
\end{proof}

\begin{prop} \label{prop_induced_surjective}
Let $\alpha: A\lra B$ be a surjective homomorphism of
finite-dimensional graded $\Bbbk$-algebras. Then the induced map $K_0(\alpha)$ is
surjective.
\end{prop}

\begin{proof} The argument is essentially the same as in the nongraded case
discussed earlier. In the bases of $K_0(A)$ and $K_0(B)$ given by indecomposable
projective modules, the map $K_0(\alpha)$ sends some basis elements of $K_0(A)$
to 0 and the rest go bijectively to the basis of $K_0(B)$ (possibly after grading
shifts).
\end{proof}

Finally, we discuss $K_0$ of graded idempotented algebras. Let $A$ be an
associative graded $\Bbbk$-algebra, possibly nonunital, with a family of
mutually-orthogonal degree 0 idempotents $1_{x}\in A, x\in Z$, such that
$$ A = \bigoplus_{x,y\in Z} 1_y A 1_x $$
(compare with the definition of idempotented ring in Section~\ref{subsec-idemp}).
We say that $A$ is a graded idempotented $\Bbbk$-algebra. By a graded
finitely-generated projective $A$-module we mean a homogeneous direct summand of
a finite direct sum (with finite multiplicities) of graded left $A$-modules $A1_x
\{t\}$, over $x\in Z$ and $t\in \Z$. By $K_0(A)$ we denote the corresponding
Grothendieck group, which is again a $\Z[q,q^{-1}]$-module.

We assume that for each $x,y\in Z$ the graded $\Bbbk$-vector space $1_y A1_x$ is
bounded below and has finite-dimensional weight spaces.

\begin{prop}  For $A$ as above, $K_0(A)$ is a free $\Z[q,q^{-1}]$-module
with a basis given by isomorphism classes of indecomposables, up to grading
shifts.
\end{prop}

\begin{proof} The proof is essentially the same as that of Proposition~\ref{prop-freemod}.
The difference is in the absence of a canonical grading normalization for a
representative $Q$ of an isomorphism class of indecomposables up to grading
shifts. This normalization can be chosen ad hoc, of course.
\end{proof}

%
\subsubsection{A triangle of $K_0$'s}\label{subsubsec-triangle}
%

We will work in the graded case, so that the rings are $\Z$-graded and $K_0$ are
$\Z[q,q^{-1}]$-modules.  Consider the diagram of $\Z[q,q^{-1}]$-modules
\begin{equation}\label{eq-KO-triangle}
 \xy
 (0,0)*+{K_0\big(\ENDU(\cal{E}_{\nu,-\nu'}\onel)\big)}="3";
 (45,0)*+{K_0(R_{\nu,-\nu',\lambda})}="4";
 (0,20)*+{K_0\big(R(\nu) \otimes R(\nu') \otimes \Pi_{\lambda}\big)}="6";
{\ar_-{K_0(\beta)} "3";"4"};
 {\ar_{K_0(\alpha)} "6";"3"};
{\ar^{K_0(\beta \alpha)} "6";"4"};
 \endxy
\end{equation}
given by applying the $K_0$ functor to the commutative triangle
in~\eqref{eq_diagram_K}.

Recall that in \cite{KL,KL2} we constructed an isomorphism of
$\Z[q,q^{-1}]$-modules
\begin{equation}
  K_0(R(\nu)) \cong {}_{\cal{A}}{\bf f}_{\nu},
\end{equation}
where ${}_{\cal{A}}{\bf f}_{\nu}$ is the weight $\nu$ summand of the
$\Z[q,q^{-1}]$-algebra ${}_{\cal{A}}{\bf f}$. Likewise,
\begin{equation}
  K_0(R(\nu')) \cong {}_{\cal{A}}{\bf f}_{\nu'}.
\end{equation}

$R(\nu)$, respectively $R(\nu')$,
 is a free finite rank graded module over its center $Z(R(\nu)) \cong
{\rm Sym}(\nu)$, respectively ${\rm Sym}(\nu')$, isomorphic to the
$\Bbbk$-algebra of polynomials in several homogeneous generators, all of positive
degree.  Hence, $R(\nu) \otimes R(\nu') \otimes \Pi_{\lambda}$ is a free finite
rank graded module over the central graded polynomial algebra ${\rm
Sym}(\nu)\otimes {\rm Sym}(\nu') \otimes \Pi_{\lambda}$.  This algebra contains a
homogeneous augmentation ideal ${\rm Sym}^+$ of codimension 1.  Let
\begin{equation}
J=\big(R(\nu) \otimes R(\nu') \otimes \Pi_{\lambda}\big){\rm Sym}^+
\end{equation}
be the corresponding 2-sided ideal of $R(\nu) \otimes R(\nu') \otimes
\Pi_{\lambda}$, and consider the quotient algebra
\begin{equation}
  R := R(\nu) \otimes R(\nu') \otimes \Pi_{\lambda}/ J
  \cong \Big(R(\nu)/\left(R(\nu) \cdot {\rm Sym}^+(\nu)\right)\Big) \otimes
  \Big(R(\nu')/\left(R(\nu') \cdot {\rm Sym}^+(\nu') \right)\Big).
\end{equation}
$R$ is a finite-dimensional $\Bbbk$-algebra, and the quotient map
\begin{equation}
  \alpha' \maps R(\nu) \otimes R(\nu') \otimes \Pi_{\lambda} \longrightarrow R
\end{equation}
induces an isomorphism of $K_0$-groups
\begin{equation}
  K_0(\alpha') \maps K_0(R(\nu) \otimes R(\nu') \otimes \Pi_{\lambda})
  \longrightarrow K_0(R),
\end{equation}
since the ideal $J$ is virtually nilpotent, see
Proposition~\ref{prop-virtuallynil}.

We proved in \cite{KL,KL2} that any simple graded $R(\nu)$-module is absolutely
irreducible for any field $\Bbbk$, same for simple graded $R(\nu)/\big(R(\nu){\rm
Sym}^+\big)$-modules.  Also note that  $K_0(\Pi_{\lambda}) \cong \Z[q,q^{-1}]$,
since $\Pi_{\lambda}$ is a graded local ring.  The chain of isomorphisms
\begin{equation}
 K_0 \big(R(\nu) \otimes R(\nu') \otimes
  \Pi_{\lambda}\big) \cong K_0(R)  \cong  K_0(R(\nu)) \otimes K_0(R(\nu'))
 \cong {}_{\cal{A}}{\bf f}_{\nu} \otimes {}_{\cal{A}}{\bf f}_{\nu'}
\end{equation}
establishes the following result:
\begin{prop} \label{prop-can-koiso}
There is a canonical isomorphism
\begin{equation}
K_0\big(R(\nu) \otimes R(\nu')\otimes\Pi_{\lambda}\big) \cong {}_{\cal{A}}{\bf
f}_{\nu} \otimes {}_{\cal{A}}{\bf f}_{\nu'}
\end{equation}
induced by isomorphisms $K_0(R(\nu)) \cong {}_{\cal{A}}{\bf f}_{\nu}$ and
$K_0(R(\nu'))\cong {}_{\cal{A}}{\bf f}_{\nu'}$ constructed in \cite{KL,KL2}.
\end{prop}

This proposition gives us a grip on the top term in the diagram
\eqref{eq-KO-triangle}.

\vspace{0.1in}

\begin{prop}\label{prop_betaalpha_surj} $K_0(\beta\alpha)$ is surjective. \end{prop}

\begin{proof} Since $\beta\alpha$ is surjective, $\beta\alpha(J)$ is a 2-sided
ideal of $R_{\nu,-\nu',\lambda}$. Start with a commutative square of surjective
algebra homomorphisms
\[
 \xy
 (-25,10)*++{R(\nu) \otimes R(\nu')\otimes \Pi_{\lambda} }="tl";
  (25,10)*+++{R_{\nu,-\nu',\lambda}}="tr";
 (-25,-10)*++{R}="bl"; (25,-10)*++{R_{\nu,-\nu',\lambda}/(\beta\alpha(J))}="br";
 {\ar@{->>}^-{\beta\alpha} "tl";"tr"};{\ar@{->>}_{{/J}} "tl";"bl"};
 {\ar@{->>} "bl";"br"};
 {\ar@{->>}^{/\beta\alpha(J)} "tr";"br"};
 \endxy
\]
and apply functor $K_0$ to obtain a commutative diagram
\[
 \xy
 (-30,10)*++{K_0\big(R(\nu) \otimes R(\nu')\otimes \Pi_{\lambda}\big)}="tl";
  (30,10)*+++{K_0\big(R_{\nu,-\nu',\lambda}\big)}="tr";
 (-30,-10)*++{K_0(R) }="bl"; (30,-10)*++{K_0\big(R_{\nu,-\nu',\lambda}/(\beta\alpha(J))\big)}="br";
 {\ar@{->}^-{K_0(\beta\alpha)} "tl";"tr"};{\ar@{->}_{{/J}} "tl";"bl"};
 {\ar@{->} "bl";"br"};
 {\ar@{->}^{{/\beta\alpha(J)}} "tr";"br"};
 \endxy
\]
The vertical arrows are isomorphisms since $J$ and $\beta\alpha(J)$ are virtually
nilpotent ideals. The bottom arrow is surjective, by
Proposition~\ref{prop_induced_surjective}, since $R$ and $R_{\nu,-\nu',\lambda}$
are finite-dimensional over $\Bbbk$. Surjectivity of the top arrow follows.
\end{proof}

\begin{cor} $K_0(\beta)$ is surjective. \end{cor}

These observations are summarized in the following enhancement of
\eqref{eq-KO-triangle}
\begin{equation}\label{eq-KO-triangleII}
 \xy
 (0,0)*+{K_0(\ENDU(\cal{E}_{\nu,-\nu'}\onel))}="3";
 (45,0)*+{K_0(R_{\nu,-\nu',\lambda})}="4";
 (0,20)*+{K_0(R(\nu) \otimes R(\nu') \otimes \Pi_{\lambda})
 \cong {}_{\cal{A}}{\bf f}_{\nu}\otimes {}_{\cal{A}}{\bf f}_{\nu'}}="6";
{\ar@{->>}_-{K_0(\beta)} "3";"4"};
 {\ar_{K_0(\alpha)} "6";"3"};
{\ar@{->>}^{K_0(\beta \alpha)} "6";"4"};
 \endxy .
\end{equation}

%
\subsubsection{Idempotents in $\UcatD$}
%

Let $1=e_1+ \dots + e_k$, respectively $1=e'_1 + \dots+ e_{k'}'$, be a
decomposition of $1 \in \End_{\UcatD}(\cal{E}_{\nu}\onel) \cong R(\nu)_0$,
respectively $1 \in \End_{\UcatD}(\cal{E}_{-\nu'}\onel) \cong R(\nu')_0$, into
the sum of minimal mutually-orthogonal idempotents. Here $R(\nu)_0$ denotes the
degree $0$ subalgebra of $R(\nu)$.  Each term in the decomposition
\begin{equation}
  1 = \sum_{r=1}^{k} \sum_{r'=1}^{k'} e_r \otimes e'_{r'} \otimes 1
\end{equation}
of $1 \in R(\nu) \otimes R(\nu') \otimes \Pi_{\lambda}$ is a minimal degree 0
idempotent, in view of discussion preceding Proposition~\ref{prop-can-koiso}.

Let
\begin{equation}
  e_{r,r'}:= \alpha\left( e_r \otimes e'_{r'} \otimes 1 \right) \in
  \End_{\UcatD}(\cal{E}_{\nu,-\nu'}\onel)
\end{equation}
be the corresponding idempotent in the endomorphism algebra of
$\cal{E}_{\nu,-\nu'}\onel$ which may not be minimal.  We can decompose
\begin{equation} \label{eq_err}
  e_{r,r'} = \sum_{r''=1}^{k(r,r')} e_{r,r',r''}
\end{equation}
into a sum of minimal mutually-orthogonal degree zero idempotents $e_{r,r',r''}
\in \End_{\UcatD}(\cal{E}_{\nu,-\nu'}\onel)$.

Homomorphism $\beta\alpha$ induces a surjection of Grothendieck groups and maps
each minimal idempotent $e_r \otimes e'_{r'} \otimes 1$ either to 0 or to a
minimal degree 0 idempotent in $R_{\nu,-\nu',\lambda}$
(Proposition~\ref{prop_betaalpha_surj}). Consequently, for each $(r,r',r'')$ the
image $\beta(e_{r,r',r''})$ is either 0 or a minimal idempotent in
$R_{\nu,-\nu',\lambda}$.  Moreover, for each $(r,r')$ at most one of
$\beta(e_{r,r',r''})\neq 0$ in $R_{\nu,-\nu',\lambda}$.  We can relabel
idempotents so that $\beta(e_{r,r',1}) \neq 0$ and $\beta(e_{r,r',r''})=0$ for
$r''>1$ whenever $\beta\alpha(e_r \otimes e'_{r'} \otimes 1) \neq 0$.
Necessarily, $\beta(e_{r,r',r''})=0$ for all $r''$ if $\beta\alpha(e_r \otimes
e'_{r'} \otimes 1) =0$.

If $\beta(e_r\otimes e'_{r'}\otimes 1)=0$ then $e_{r,r',r''} \in
\cal{I}_{\nu,-\nu',\lambda}$. A homogeneous element $a \in
\cal{I}_{\nu,-\nu',\lambda}$ can be written as a finite sum $a = \sum_{s=1}^u
a'_s a_s$, where $a_s,a'_s$ are homogeneous,
\begin{equation}
  a_s \in \Ucatq(\cal{E}_{\nu,-\nu}\onel,\cal{E}_{\textbf{\textit{i(s)}}}\onel),
  \qquad \quad
  a'_s \in \Ucatq(\cal{E}_{\textbf{\textit{i(s)}}}\onel,\cal{E}_{\nu,-\nu}\onel),
\end{equation}
and $\textbf{\textit{i(s)}} \in \sseq$ with
$\parallel\textbf{\textit{i(s)}}\parallel
  \;< \;\parallel\nu\parallel +\parallel\nu'\parallel$.
Indeed, an element of $\cal{I}_{\nu,-\nu',\lambda}$ can be written as  a linear
combination of diagrams with $U$-turns.  Cutting each diagram in the middle
allows us to view it as composition
\[
 \cal{E}_{\nu,-\nu'}\onel \longrightarrow \cal{E}_{\textbf{\textit{i(s)}}}\onel
 \longrightarrow  \cal{E}_{\nu,-\nu'}\onel,
 \]
 with the length $\parallel\textbf{\textit{i(s)}}\parallel$ of the sequence
$\textbf{\textit{i(s)}}$ strictly less than the sum of lengths
$\parallel\nu\parallel +\parallel\nu'\parallel$.
\[
\xy 0;/r.18pc/:
 (-4,-15)*{}; (-20,25) **\crv{(-3,-6) & (-20,4)}?(0)*\dir{<}?(.6)*\dir{}+(0,0)*{\bullet};
 (-12,-15)*{}; (-4,25) **\crv{(-12,-6) & (-4,0)}?(0)*\dir{<}?(.6)*\dir{}+(.2,0)*{\bullet};
 ?(0)*\dir{<}?(.75)*\dir{}+(.2,0)*{\bullet};?(0)*\dir{<}?(.9)*\dir{}+(0,0)*{\bullet};
 (-28,25)*{}; (-12,25) **\crv{(-28,10) & (-12,10)}?(0)*\dir{<};
  ?(.2)*\dir{}+(0,0)*{\bullet}?(.35)*\dir{}+(0,0)*{\bullet};
  (-50,-15)*{}; (-50,25) **\crv{(-51,10) & (-49,10)}?(1)*\dir{>}?(.35)*\dir{}+(.2,0)*{\bullet};;
 (-36,-15)*{}; (-36,25) **\crv{(-34,-6) & (-35,4)}?(1)*\dir{>};
 (-28,-15)*{}; (-42,25) **\crv{(-28,-6) & (-42,4)}?(1)*\dir{>};
 (-42,-15)*{}; (-20,-15) **\crv{(-42,-5) & (-20,-5)}?(1)*\dir{>};
 (-60,0)*{}; (8,0)*{} **\dir{.}; (15,0)*{\textbf{\textit{i(s)}}};
 (22,24)*{\lambda};
 \endxy
\]

Choose such a decomposition for each
\begin{equation}
 e_{r,r',r''} \in \cal{I}_{\nu,-\nu',\lambda}, \qquad \quad e_{r,r',r''} =
 \sum_{s=1}^{u}a_s'a_s,
\end{equation}
where in the notations we suppress dependence of $u(r,r',r'')$, $a_s(r,r',r'')$
and $a_s'(r,r',r'')$ on the three parameters. Multiplication by $e_{r,r',r''}$ is
the identity endomorphism of $(\cal{E}_{\nu,-\nu'}\onel,e_{r,r',r''})$.  We can
view this indecomposable 1-morphism of $\UcatD$ as an indecomposable projective
module (call it $P$) over the graded idempotented ring
${}_{\mu}\Ucatq_{\lambda}$, $\mu=\lambda+\nu_X-\nu'_X$.  Then the identity
endomorphism of $P$ factors through projective module $Q$ corresponding to the
direct sum $\bigoplus_{s=1}^{u}\cal{E}_{\textbf{\textit{i(s)}}}\onel\{t_s\}$
\[
 \xy
 (-30,0)*+{P}="1";(0,0)*+{Q}="2";(30,0)*+{P}="3";
 {\ar "1";"2"}; {\ar "2";"3"};  {\ar@/^1.9pc/^{\Id} "1";"3"};
 \endxy
\]
for some $t_s\in \Z$.  Therefore, $P$ is isomorphic to a direct summand of $Q$,
and the 1-morphism $(\cal{E}_{\nu,-\nu'}\onel,e_{r,r',r''})$ of $\UcatD$ is
isomorphic to a direct summand of
$\bigoplus_{s=1}^{u}\cal{E}_{\textbf{\textit{i(s)}}}\onel\{t_s\}$.

Define the width $\parallel P\parallel$ of an indecomposable 1-morphism $P \in
{\rm Ob}( \UcatD(\lambda,\mu))$ as the smallest $m$ such that $P$ is isomorphic
to a direct summand of $\cal{E}_{\ii}\onel\{t\}$ for some $\ii \in\sseq$,
$\parallel\ii\parallel=m$ and $t \in \Z$.

For example, if $P$ has width 0, then $P$ is isomorphic to a direct summand of
$\onel\{t\}$ for some $\lambda$ and $t$. One-morphism $\onel\{t\}$ is
indecomposable, since its endomorphism ring
$\UcatD(\onel\{t\},\onel\{t\})=\Bbbk$, or 0 (a possibility if the calculus is
degenerate). This implies that any width zero 1-morphism is isomorphic to
$\onel\{t\}$.

\begin{lem}
If $P$ has width $m$ then $P$ is isomorphic to a direct summand of
$\cal{E}_{\nu,-\nu'}\onel\{t\}$ for some $\nu,\nu' \in \N[I]$,
$\parallel\nu\parallel +\parallel \nu'\parallel=m$, $\lambda \in X$, and $t \in
\Z$.
\end{lem}

\begin{proof}
If $i,j \in I$ and $\ii=\ii'-i+j\ii''$ has length $m$, then $P$ is direct summand
of $\cal{E}_{\ii}\onel\{t\}$ if an only if it is a direct summand of
$\cal{E}_{\ii'+j-i\ii''}\onel\{t\}$.  Indeed, these two 1-morphisms are either
isomorphic (if $i\neq j$) or differ by direct summands
$\cal{E}_{\ii'\ii''}\onel\{t'\}$, all whose indecomposable summands have width at
most $m-2$, and thus cannot be isomorphic to $P$.  By assumption, $P$ is
isomorphic to a direct summand of $\cal{E}_{\ii}\onel\{t\}$ with
$\parallel\ii\parallel=m$. Moving all positive terms of $\ii$ to the left of all
negative terms produces a sequence $\jj(-\jj')$ with $\jj,\jj'$ positive,
$\parallel\jj\parallel+\parallel\jj'\parallel=m$ and $P$ being a summand of
$\cal{E}_{\jj(-\jj)}\onel\{t\}$.  But this 1-morphism is a direct summand of
$\cal{E}_{\nu,-\nu'}\onel\{t\}$ with $\nu$, respectively $\nu'$, being the weight
of $\jj$, respectively $\jj'$.
\end{proof}

\begin{proof}[Proof of Theorem~\ref{thm-surjective}]
We show that $[P]$ is in the image of $\gamma \maps \UA \to K_0(\UcatD)$ by
induction on the length of $P$.  Let $P$ have length $m$.  Then $P$ is a direct
summand of $\cal{E}_{\nu,-\nu'}\onel \{t\}$ for $\nu,\nu'$ as above.  By shifting
the degree of $P$ down by $t$, $P \cong (\cal{E}_{\nu,-\nu'}\onel,e_{r,r',r''})$
for at least one minimal idempotent $e_{r,r',r''}$. We must have $r''=1$ and
$\beta(e_{r,r',r''})\neq 0$, for otherwise $P$ is isomorphic to a direct summand
of $\bigoplus_s \cal{E}_{\textbf{\textit{i(s)}}}\onel\{t_s\}$, and since
$\parallel\textbf{\textit{i(s)}}\parallel=m-2$, the width of $P$ is at most
$m-2$, a contradiction. Thus $r''=1$ and $\beta(e_{r,r',r''})\neq 0$.

We have
\begin{equation}
  [P] =[\cal{E}_{\nu,-\nu'}\onel,e_{r,r',1}],
\end{equation}
and
\begin{equation}
  \sum_{r''=1}^{k(r,r')}[\cal{E}_{\nu,-\nu'}\onel,e_{r,r',r''}] =
  [\cal{E}_{\nu,-\nu'}\onel,e_{r,r'}],
\end{equation}
from \eqref{eq_err}.  For $r''>1$ we have $\beta(e_{r,r',r''})=0$, and
$(\cal{E}_{\nu,-\nu'}\onel,e_{r,r',r''})$ is isomorphic to a direct summand of a
finite sum of $\cal{E}_{\ii}\onel\{t\}$, for sequences $\ii$ of length $m-2$.
Each indecomposable summand of $(\cal{E}_{\nu,-\nu'}\onel,e_{r,r',r''})$ has
length at most $m-2$.  By induction hypothesis,
$[\cal{E}_{\nu,-\nu'}\onel,e_{r,r',r''}]$ belongs to the image of $\gamma$ for
all $2 \leq r'' \leq k(r,r')$.  Thus,
\begin{eqnarray}
  [P] = [\cal{E}_{\nu,-\nu'}\onel,e_{r,r',1}] =
  [\cal{E}_{\nu,-\nu'}\onel,e_{r,r'}]-\sum_{r''=2}^{k(r,r')}[\cal{E}_{\nu,-\nu'}
  \onel,e_{r,r',r''}] \in [\cal{E}_{\nu,-\nu'}\onel,e_{r,r'}]+\gamma(\UA). \nn
\end{eqnarray}

It now suffices to show that $[\cal{E}_{\nu,-\nu'}\onel,e_{r,r'}]$ belongs to
image of $\gamma$.  But the idempotent $e_{r,r'}$ is the image of $e_r \otimes
e_{r'} \otimes 1$ in $R(\nu) \otimes R(\nu') \otimes \Pi_{\lambda}$, and the
Grothendieck group of the latter is isomorphic to ${}_{\cal A}{\bf f}(\nu)
\otimes {}_{\cal A}{\bf f}(\nu')$.  Therefore, $[\cal{E}_{\nu,-\nu'},e_{r,r'}]$
is in the image of ${}_{\cal A}{\bf f}(\nu) \otimes {}_{\cal A}{\bf f}(\nu')$
under the composition map
\[
 \xymatrix{
{}_{\cal A}{\bf f}(\nu) \otimes {}_{\cal A}{\bf f}(\nu') \ar[rr] &&
 1_{\mu}(\UA)1_{\lambda} \ar[rr]^{\gamma} && K_0\big(\UcatD(\lambda,\mu)\big)
 }
\]
\[
 \xymatrix{
x\otimes y \quad \ar[rr] && \;\; x^+y^-1_{\lambda}, \ar@{}[rr] && \qquad}
\]
\begin{eqnarray}
  x \mapsto x^+ &\maps& {}_{\cal A}{\bf f}(\nu) \longrightarrow {\bf U}^+,\nn\\
  y \mapsto y^-& \maps& {}_{\cal A}{\bf f}(\nu') \longrightarrow {\bf U}^-.\nn
\end{eqnarray}
This completes the proof of surjectivity of $\gamma$.
\end{proof}

\bigskip

%
\subsection{Injectivity of $\gamma$ in the nondegenerate case}
\label{subsec_injectivity}
%

Assume that our graphical calculus is nondegenerate for a given root datum and
field $\Bbbk$, so that $B_{\ii,\jj,\lambda}$ is a basis of
$\HOMU(\cal{E}_{\ii}\onel,\cal{E}_{\jj}\onel)$ for all $\ii$, $\jj$ and
$\lambda$. Then
\begin{equation}
 \gdim \;\HOMU\left( \cal{E}_{\ii}\onel,\cal{E}_{\jj}\onel \right)=
 \sum_{t \in \Z}q^t\dim_{\Bbbk}\UcatD(\cal{E}_{\ii}\onel\{t\},\cal{E}_{\jj}\onel
 ).
\end{equation}
Since the calculus is nondegenerate,
\begin{equation} \label{eq_gdim_U_semi}
  \gdim \; \HOMU\left( \cal{E}_{\ii}\onel,\cal{E}_{\jj}\onel \right)
  = \pi \la E_{\ii}1_{\lambda}, E_{\jj}1_{\lambda}\ra ,
\end{equation}
and $E_{\ii}1_{\lambda}$, over all $\ii$, $\lambda$, span $\U$, the
$\Q(q)$-algebra homomorphism
\begin{equation}
  \gamma_{\Q(q)} \maps \U \to K_0(\UcatD) \otimes_{\Z[q,q^{-1}]} \Q(q)
\end{equation}
intertwines the $\Q(q)$-semilinear forms $\pi\sla, \sra$ on $\U$ and $\gdim \;
\HOMU(,)$ on $K_0$.  The latter form, which we denote $\sla,\sra_{\pi}$, is given
by
\begin{equation} \label{eq_semilinear_form_pi}
  \sla [P],[Q] \sra_{\pi} := \sum_{t\in\Z}q^t\gdim \big( \UcatD(P\{t\},Q) \big)
\end{equation}
for any two 1-morphisms $P,Q \in \UcatD(\lambda,\mu)$, and extends to the entire
$K_0(\UcatD) \otimes_{\Z[q,q^{-1}]} \Q(q)$ via $\Q(q)$-semilinearity and the
orthogonality condition $\sla x,y \sra_{\pi}=0$ for $x \in
K_0\big(\UcatD(\lambda,\mu)\big)$, $y\in K_0\big(\UcatD(\lambda',\mu')\big)$
unless $\lambda=\lambda'$ and $\mu=\mu'$.

By Proposition~\ref{prop_nondeg}, $\sla ,\sra$ is nondegenerate on $\U$.
Therefore, $\gamma_{\Q(q)}$ is injective, implying that $\gamma$ is injective.

%
\section{Categorification of $\U$ for $\mathfrak{sl}_n$} \label{sec_Usln}
%

%
\subsection{Forms of quantum $\mf{sl}_n$}
%

We consider various forms of the quantized enveloping algebra of $\mf{sl}_n$
corresponding to the root datum of the Dynkin graph
\begin{equation} \label{eq_dynkin_sln}
    \xy
  (-15,0)*{\circ}="1";
  (-5, 0)*{\circ}="2";
  (5,  0)*{\circ}="3";
  (35,  0)*{\circ}="4";
  "1";"2" **\dir{-}?(.55)*\dir{};
  "2";"3" **\dir{-}?(.55)*\dir{};
  "3";(15,0) **\dir{-}?(.55)*\dir{};
  (25,0);"4" **\dir{-}?(.55)*\dir{};
  (-15,2.2)*{\scs 1};
  (-5,2.2)*{\scs 2};
  (5,2.2)*{\scs 3};
  (35,2.2)*{\scs n-1};
  (20,0)*{\cdots };
  \endxy
  \nn
\end{equation}
For this root datum,  any weight $\lambda \in X$ can be written as $\lambda =
(\lambda_1,\lambda_2, \dots, \lambda_{n-1})$, where $\lambda_i=\la i, \lambda
\ra$.

The algebra $\Uq$ is the $\Q(q)$-algebra with 1 generated by the elements $E_i$,
$F_i$ and $K_i^{\pm 1}$ for $i = 1, 2, \dots , n-1$, with  the defining relations
\begin{equation}
K_iK_i^{-1} = K_i^{-1}K_i = 1, \quad  K_iK_j = K_jK_i,
\end{equation}
\begin{equation}
K_iE_jK_i^{-1} = q^{i \cdot j} E_j, \quad  K_iF_jK_i^{-1} = q^{-i \cdot j} F_j,
\end{equation}where $i \cdot i =2$, $i \cdot j =-1$ if $j=i\pm 1$, and $i\cdot j =0$
otherwise,
\begin{equation}
E_iF_j - F_jE_i = \delta_{ij} \frac{K_i-K_{i}^{-1}}{q-q^{-1}},
\end{equation}
\begin{equation}
E_i^2E_j-(q+q^{-1})E_iE_jE_i+E_jE_i^2 =0\;\; \text{if $j=i\pm 1$},
\end{equation}\begin{equation}
F_i^2F_j-(q+q^{-1})F_iF_jF_i+F_jF_i^2=0 \;\; \text{if $j=i\pm 1$},
\end{equation}\begin{equation}
E_iE_j=E_jE_i, \quad F_iF_j = F_jF_i \;\; \text{if $|i-j|>1$.}
\end{equation}

Recall that $\U(\mf{sl}_n)$ is obtained from $\Uq$ by adjoining a collection of
orthogonal idempotents $1_{\lambda}$ indexed by the weight lattice $X$ of
$\mathfrak{sl}_n$,
\begin{equation}
  1_{\lambda}1_{\lambda'} = \delta_{\lambda\lambda'} 1_{\lambda},
\end{equation}
such that if $\lambda = (\lambda_1,\lambda_2, \dots, \lambda_{n-1})$, then
\begin{equation} \label{eq_onesubn}
K_i1_{\lambda} =1_{\lambda}K_i= q^{\lambda_i} 1_{\lambda}, \quad
E_i^{}1_{\lambda} = 1_{\lambda+i_X}E_i, \quad F_i1_{\lambda} = 1_{\lambda-i_X}F_i
,
\end{equation}
where
\begin{eqnarray} \label{eq_weight_action1}
 \lambda +i_X &=& \quad \left\{
\begin{array}{ccl}
   (\lambda_1+2, \lambda_2-1,\lambda_3,\dots,\lambda_{n-2}, \lambda_{n-1}) & \quad & \text{if $i=1$} \\
   (\lambda_1, \lambda_2,\dots,\lambda_{n-2},\lambda_{n-1}-1, \lambda_{n-1}+2) & \quad & \text{if $i=n-1$} \\
  (\lambda_1, \dots, \lambda_{i-1}-1, \lambda_i+2, \lambda_{i+1}-1, \dots,
\lambda_{n-1}) & \quad & \text{otherwise.}
\end{array}
 \right.
\end{eqnarray}
The $\Z[q,q^{-1}]$-algebra $\UA(\mf{sl}_n)$ is the integral form of
$\U(\mf{sl}_n)$ generated by products of divided powers $E^{(a)}_i1_{\lambda}:=
\frac{E^{a}_i}{[a]!}1_{\lambda}$, $F^{(a)}_i1_{\lambda}:=
\frac{F^{a}_i}{[a]!}1_{\lambda}$ for $\lambda \in X$ and $i = 1, 2, \dots , n-1$.
The relationships are collected below:
\[
  \xy
 (-50,0)*+{\Uq}="q"; (50,0)*++{\UA(\mathfrak{sl}_n)}="0";
 (0,0)*+{\U(\mf{sl}_n)}="kq";
  {\ar^{\txt{add \\idempotents}} "q";"kq"};
  {\ar@{_{(}->}_{\txt{integral \\form}}  "0";"kq"};
 \endxy
\]

%
\subsection{The 2-category $\Ucats(\mathfrak{sl}_n)$}
%

We introduce a 2-category $\Ucats$ that is defined analogously to $\Ucat$ in the
$\mf{sl}_n$-case, but the $R(\nu)$-relations have been modified to the signed
$R(\nu)$-relations given in \cite{KL2}. Namely, the $R(\nu)$-relations in $\Ucat$
are replaced in $\Ucats$ by the signed $R(\nu)$-relations obtained from the
oriented graph
\begin{equation} \label{eq_oriented_graph}
    \xy
  (-15,0)*{\circ}="1";
  (-5, 0)*{\circ}="2";
  (5,  0)*{\circ}="3";
  (35,  0)*{\circ }="4";
  "1";"2" **\dir{-}?(.55)*\dir{>};
  "2";"3" **\dir{-}?(.55)*\dir{>};
  "3";(15,0) **\dir{-}?(.55)*\dir{>};
  (25,0);"4" **\dir{-}?(.55)*\dir{>};
  (-15,2.2)*{\scs 1};
  (-5,2.2)*{\scs 2};
  (5,2.2)*{\scs 3};
  (35,2.2)*{\scs n-1};
  (20,0)*{\cdots };
  \endxy ~,
\end{equation}
with vertices enumerated by the set $\{1,2,\dots,n-1\}$, using signs
$\tau_{ij}=\tau_{ji}=-1$ for all edges.  It was observed in \cite{KL2} that the
resulting ring $R_{\tau}(\nu)$ is isomorphic to $R(\nu)$. In
Section~\ref{subsubsec_isom}, following the definition of $\Ucats$, we extend
this isomorphism to an isomorphism $\Ucat \to \Ucats$ of 2-categories. The
2-category $\Ucats$ is more convenient for constructing a representation on
iterated flag varieties in Section~\ref{sec_rep}.

In general it is a poor practice to set up an isomorphism rather than an
equivalence of categories, not to mention 2-categories.  However, having an
isomorphism $ \Ucat\to \Ucats$ is justified, since  $\Ucat$ and $\Ucats$ have the
same objects, morphisms, and generating 2-morphisms.

\begin{defn} \label{def_Ucatq-sln}
$\Ucats(\mathfrak{sl}_n)$ is a additive $\Bbbk$-linear 2-category with
translation.  The 2-category $\Ucats(\mathfrak{sl}_n)$ has objects, morphisms,
and generating 2-morphisms as defined in \eqref{def_Ucat}, but some of the
relations on 2-morphisms are modified.
\begin{itemize}
  \item The $\mf{sl}_2$ relations and the shift isomorphism relations
are the same as before, see equations
\eqref{eq_biadjoint1}--\eqref{eq_nil_dotslide}.

  \item All 2-morphisms are cyclic with respect to the biadjoint
   structure as before, see \eqref{eq_cyclic_dot} and \eqref{eq_cyclic_cross-gen}.

 \item The relations \eqref{eq_downup_ij-gen} hold.

\item The signed $R(\nu)$ relations are
\begin{enumerate}[(a)]
\item For $i \neq j$, the relations
\begin{eqnarray}
  \vcenter{\xy 0;/r.18pc/:
    (-4,-4)*{};(4,4)*{} **\crv{(-4,-1) & (4,1)}?(1)*\dir{>};
    (4,-4)*{};(-4,4)*{} **\crv{(4,-1) & (-4,1)}?(1)*\dir{>};
    (-4,4)*{};(4,12)*{} **\crv{(-4,7) & (4,9)}?(1)*\dir{>};
    (4,4)*{};(-4,12)*{} **\crv{(4,7) & (-4,9)}?(1)*\dir{>};
  (8,8)*{\lambda};(-5,-3)*{\scs i};
     (5.1,-3)*{\scs j};
 \endxy}
 \qquad = \qquad
 \left\{
 \begin{array}{ccc}
     \xy 0;/r.18pc/:
  (3,9);(3,-9) **\dir{-}?(.5)*\dir{<}+(2.3,0)*{};
  (-3,9);(-3,-9) **\dir{-}?(.5)*\dir{<}+(2.3,0)*{};
  (8,2)*{\lambda};(-5,-6)*{\scs i};     (5.1,-6)*{\scs j};
 \endxy &  &  \text{if $i \cdot j=0$,}\\ \\
 (i-j)\left(\;\; \xy 0;/r.18pc/:
  (3,9);(3,-9) **\dir{-}?(.5)*\dir{<}+(2.3,0)*{};
  (-3,9);(-3,-9) **\dir{-}?(.5)*\dir{<}+(2.3,0)*{};
  (8,2)*{\lambda}; (-3,4)*{\bullet};
  (-5,-6)*{\scs i};     (5.1,-6)*{\scs j};
 \endxy \quad
 - \quad
 \xy 0;/r.18pc/:
  (3,9);(3,-9) **\dir{-}?(.5)*\dir{<}+(2.3,0)*{};
  (-3,9);(-3,-9) **\dir{-}?(.5)*\dir{<}+(2.3,0)*{};
  (8,2)*{\lambda}; (3,4)*{\bullet};
  (-5,-6)*{\scs i};     (5.1,-6)*{\scs j};
 \endxy\;\;\right)
   &  & \text{if $i \cdot j =-1$.}
 \end{array}
 \right. \nn \\\label{eq_r2_ij}
\end{eqnarray}

\item For $i \neq j$, the relations
\begin{eqnarray} \label{eq_dot_slide_ij}
\xy
  (0,0)*{\xybox{
    (-4,-4)*{};(4,4)*{} **\crv{(-4,-1) & (4,1)}?(1)*\dir{>}?(.75)*{\bullet};
    (4,-4)*{};(-4,4)*{} **\crv{(4,-1) & (-4,1)}?(1)*\dir{>};
    (-5,-3)*{\scs i};
     (5.1,-3)*{\scs j};
     (8,1)*{ \lambda};
     (-10,0)*{};(10,0)*{};
     }};
  \endxy
 \;\; =
\xy
  (0,0)*{\xybox{
    (-4,-4)*{};(4,4)*{} **\crv{(-4,-1) & (4,1)}?(1)*\dir{>}?(.25)*{\bullet};
    (4,-4)*{};(-4,4)*{} **\crv{(4,-1) & (-4,1)}?(1)*\dir{>};
    (-5,-3)*{\scs i};
     (5.1,-3)*{\scs j};
     (8,1)*{ \lambda};
     (-10,0)*{};(10,0)*{};
     }};
  \endxy
\qquad  \xy
  (0,0)*{\xybox{
    (-4,-4)*{};(4,4)*{} **\crv{(-4,-1) & (4,1)}?(1)*\dir{>};
    (4,-4)*{};(-4,4)*{} **\crv{(4,-1) & (-4,1)}?(1)*\dir{>}?(.75)*{\bullet};
    (-5,-3)*{\scs i};
     (5.1,-3)*{\scs j};
     (8,1)*{ \lambda};
     (-10,0)*{};(10,0)*{};
     }};
  \endxy
\;\;  =
  \xy
  (0,0)*{\xybox{
    (-4,-4)*{};(4,4)*{} **\crv{(-4,-1) & (4,1)}?(1)*\dir{>} ;
    (4,-4)*{};(-4,4)*{} **\crv{(4,-1) & (-4,1)}?(1)*\dir{>}?(.25)*{\bullet};
    (-5,-3)*{\scs i};
     (5.1,-3)*{\scs j};
     (8,1)*{ \lambda};
     (-10,0)*{};(12,0)*{};
     }};
  \endxy
\end{eqnarray}
for all $\lambda$.

\item Unless $i = k$ and $j =i \pm 1$
\begin{equation}
 \vcenter{
 \xy 0;/r.18pc/:
    (-4,-4)*{};(4,4)*{} **\crv{(-4,-1) & (4,1)}?(1)*\dir{>};
    (4,-4)*{};(-4,4)*{} **\crv{(4,-1) & (-4,1)}?(1)*\dir{>};
    (4,4)*{};(12,12)*{} **\crv{(4,7) & (12,9)}?(1)*\dir{>};
    (12,4)*{};(4,12)*{} **\crv{(12,7) & (4,9)}?(1)*\dir{>};
    (-4,12)*{};(4,20)*{} **\crv{(-4,15) & (4,17)}?(1)*\dir{>};
    (4,12)*{};(-4,20)*{} **\crv{(4,15) & (-4,17)}?(1)*\dir{>};
    (-4,4)*{}; (-4,12) **\dir{-};
    (12,-4)*{}; (12,4) **\dir{-};
    (12,12)*{}; (12,20) **\dir{-};
  (18,8)*{\lambda};
  (-6,-3)*{\scs i};
  (6,-3)*{\scs j};
  (15,-3)*{\scs k};
\endxy}
 \;\; =\;\;
 \vcenter{
 \xy 0;/r.18pc/:
    (4,-4)*{};(-4,4)*{} **\crv{(4,-1) & (-4,1)}?(1)*\dir{>};
    (-4,-4)*{};(4,4)*{} **\crv{(-4,-1) & (4,1)}?(1)*\dir{>};
    (-4,4)*{};(-12,12)*{} **\crv{(-4,7) & (-12,9)}?(1)*\dir{>};
    (-12,4)*{};(-4,12)*{} **\crv{(-12,7) & (-4,9)}?(1)*\dir{>};
    (4,12)*{};(-4,20)*{} **\crv{(4,15) & (-4,17)}?(1)*\dir{>};
    (-4,12)*{};(4,20)*{} **\crv{(-4,15) & (4,17)}?(1)*\dir{>};
    (4,4)*{}; (4,12) **\dir{-};
    (-12,-4)*{}; (-12,4) **\dir{-};
    (-12,12)*{}; (-12,20) **\dir{-};
  (10,8)*{\lambda};
  (7,-3)*{\scs k};
  (-6,-3)*{\scs j};
  (-14,-3)*{\scs i};
\endxy} \label{eq_r3_easy}
\end{equation}

For $i \cdot j =i \pm 1$
\begin{equation}
\xy 0;/r.18pc/:
  (4,12);(4,-12) **\dir{-}?(.5)*\dir{<}+(2.3,0)*{};
  (-4,12);(-4,-12) **\dir{-}?(.5)*\dir{<}+(2.3,0)*{};
  (12,12);(12,-12) **\dir{-}?(.5)*\dir{<}+(2.3,0)*{};
  (17,2)*{\lambda}; (-6,-9)*{\scs i};     (6.1,-9)*{\scs j};
  (14,-9)*{\scs i};
 \endxy
 \;\; =\;\;
(i-j) \left( \vcenter{
 \xy 0;/r.18pc/:
    (-4,-4)*{};(4,4)*{} **\crv{(-4,-1) & (4,1)}?(1)*\dir{>};
    (4,-4)*{};(-4,4)*{} **\crv{(4,-1) & (-4,1)}?(1)*\dir{>};
    (4,4)*{};(12,12)*{} **\crv{(4,7) & (12,9)}?(1)*\dir{>};
    (12,4)*{};(4,12)*{} **\crv{(12,7) & (4,9)}?(1)*\dir{>};
    (-4,12)*{};(4,20)*{} **\crv{(-4,15) & (4,17)}?(1)*\dir{>};
    (4,12)*{};(-4,20)*{} **\crv{(4,15) & (-4,17)}?(1)*\dir{>};
    (-4,4)*{}; (-4,12) **\dir{-};
    (12,-4)*{}; (12,4) **\dir{-};
    (12,12)*{}; (12,20) **\dir{-};
  (18,8)*{\lambda};
  (-6,-3)*{\scs i};
  (6,-3)*{\scs j};
  (14,-3)*{\scs i};
\endxy}
\quad - \quad
 \vcenter{
 \xy 0;/r.18pc/:
    (4,-4)*{};(-4,4)*{} **\crv{(4,-1) & (-4,1)}?(1)*\dir{>};
    (-4,-4)*{};(4,4)*{} **\crv{(-4,-1) & (4,1)}?(1)*\dir{>};
    (-4,4)*{};(-12,12)*{} **\crv{(-4,7) & (-12,9)}?(1)*\dir{>};
    (-12,4)*{};(-4,12)*{} **\crv{(-12,7) & (-4,9)}?(1)*\dir{>};
    (4,12)*{};(-4,20)*{} **\crv{(4,15) & (-4,17)}?(1)*\dir{>};
    (-4,12)*{};(4,20)*{} **\crv{(-4,15) & (4,17)}?(1)*\dir{>};
    (4,4)*{}; (4,12) **\dir{-};
    (-12,-4)*{}; (-12,4) **\dir{-};
    (-12,12)*{}; (-12,20) **\dir{-};
  (10,8)*{\lambda};
  (6,-3)*{\scs i};
  (-6,-3)*{\scs j};
  (-14,-3)*{\scs i};
\endxy}
\right) \label{eq_r3_hard}
\end{equation}
\end{enumerate}

\end{itemize}

\end{defn}

%
\subsubsection{The 2-isomorphism $\Sigma \maps \Ucat \to \Ucats$}
\label{subsubsec_isom}
%

Define an isomorphism of 2-categories $\Sigma \maps \Ucat \to \Ucats$  on objects
by mapping $\lambda \mapsto \lambda$, and on hom categories by graded additive
$\Bbbk$-linear functors
\begin{eqnarray}
\Sigma \maps \Ucat(\lambda, \mu)  &\longrightarrow&
\Ucats(\lambda, \mu) \\
\cal{E}_{\ii}\onel & \mapsto & \cal{E}_{\ii}\onel \\
   \xy 0;/r.18pc/:
  (6,6);(6,-6) **\dir{-}?(.5)*\dir{<}+(2.3,0)*{};
  (18,6);(18,-6) **\dir{-}?(.5)*\dir{}+(2.3,0)*{};
  (-6,6);(-6,-6) **\dir{-}?(.5)*\dir{}+(2.3,0)*{};
  (24,2)*{\lambda};(-6,-9)*{\scs i_1};     (6,-9)*{\scs i_{\alpha}};
  (19,-9)*{\scs i_m};
  (6.1,2)*{\bullet};
  (0,0)*{\cdots};(12,0)*{\cdots};
 \endxy & \mapsto &(-1)^{i_{\alpha}}
    \;
   \xy 0;/r.18pc/:
  (6,6);(6,-6) **\dir{-}?(.5)*\dir{<}+(2.3,0)*{};
  (18,6);(18,-6) **\dir{-}?(.5)*\dir{}+(2.3,0)*{};
  (-6,6);(-6,-6) **\dir{-}?(.5)*\dir{}+(2.3,0)*{};
  (24,2)*{\lambda};(-6,-9)*{\scs i_1};     (6,-9)*{\scs i_{\alpha}};
  (19,-9)*{\scs i_m};
  (6.1,2)*{\bullet};
  (0,0)*{\cdots};(12,0)*{\cdots};
 \endxy \nn \\ \nn\\
 \xy 0;/r.18pc/:
  (0,0)*{\xybox{
    (-4,-6)*{};(4,6)*{} **\crv{(-4,-1) & (4,1)}?(1)*\dir{>} ;
    (4,-6)*{};(-4,6)*{} **\crv{(4,-1) & (-4,1)}?(1)*\dir{>};
    (15,6);(15,-6) **\dir{-}?(.5)*\dir{}+(2.3,0)*{};
    (-15,6);(-15,-6) **\dir{-}?(.5)*\dir{}+(2.3,0)*{};
     (24,1)*{ \lambda}; (-9,0)*{\cdots};(9,0)*{\cdots};
     (-12,0)*{};(12,0)*{};(-15,-9)*{\scs i_1};     (-4,-9)*{\scs i_{\alpha}};(5,-9)*{\scs i_{\alpha+1}};
       (16,-9)*{\scs i_m};
     }};
  \endxy & \mapsto &
 \left\{\begin{array}{cc}
 (-1)^{i_{\alpha+1}} \;  \xy 0;/r.18pc/:
  (0,0)*{\xybox{
    (-4,-6)*{};(4,6)*{} **\crv{(-4,-1) & (4,1)}?(1)*\dir{>} ;
    (4,-6)*{};(-4,6)*{} **\crv{(4,-1) & (-4,1)}?(1)*\dir{>};
    (15,6);(15,-6) **\dir{-}?(.5)*\dir{}+(2.3,0)*{};
    (-15,6);(-15,-6) **\dir{-}?(.5)*\dir{}+(2.3,0)*{};
     (24,1)*{ \lambda}; (-9,0)*{\cdots};(9,0)*{\cdots};
     (-12,0)*{};(12,0)*{};(-15,-9)*{\scs i_1};     (-4,-9)*{\scs i_{\alpha}};(5,-9)*{\scs i_{\alpha+1}};
       (16,-9)*{\scs i_m};
     }};
  \endxy & \text{if $i_{\alpha}=i_{\alpha+1}$, or $i_{\alpha} \longrightarrow i_{\alpha+1}$}
  \\ \\
   \xy 0;/r.18pc/:
  (0,0)*{\xybox{
    (-4,-6)*{};(4,6)*{} **\crv{(-4,-1) & (4,1)}?(1)*\dir{>} ;
    (4,-6)*{};(-4,6)*{} **\crv{(4,-1) & (-4,1)}?(1)*\dir{>};
    (15,6);(15,-6) **\dir{-}?(.5)*\dir{}+(2.3,0)*{};
    (-15,6);(-15,-6) **\dir{-}?(.5)*\dir{}+(2.3,0)*{};
     (24,1)*{ \lambda}; (-9,0)*{\cdots};(9,0)*{\cdots};
     (-12,0)*{};(12,0)*{};(-15,-9)*{\scs i_1};     (-4,-9)*{\scs i_{\alpha}};(5,-9)*{\scs i_{\alpha+1}};
       (16,-9)*{\scs i_m};
     }};
  \endxy & \text{otherwise}
 \end{array}\right. \nn \\
  \xy 0;/r.18pc/:
  (0,0)*{\xybox{
    (-4,6)*{};(4,6)*{} **\crv{(-4,-1) & (4,-1)}?(0)*\dir{} ;
    (15,6);(15,-6) **\dir{-}?(.5)*\dir{}+(2.3,0)*{};
    (-15,6);(-15,-6) **\dir{-}?(.5)*\dir{}+(2.3,0)*{};
     (24,1)*{ \lambda}; (-9,0)*{\cdots};(9,0)*{\cdots};
     (-12,0)*{};(12,0)*{};(-15,-9)*{\scs i_1};     (-4,-9)*{\scs i_{\alpha}};(5,-9)*{\scs i_{\alpha+1}};
       (16,-9)*{\scs i_m};
     }}; \endxy & \mapsto &
\xy 0;/r.18pc/:
  (0,0)*{\xybox{
    (-4,6)*{};(4,6)*{} **\crv{(-4,-1) & (4,-1)}?(0)*\dir{} ;
    (15,6);(15,-6) **\dir{-}?(.5)*\dir{}+(2.3,0)*{};
    (-15,6);(-15,-6) **\dir{-}?(.5)*\dir{}+(2.3,0)*{};
     (24,1)*{ \lambda}; (-9,0)*{\cdots};(9,0)*{\cdots};
     (-12,0)*{};(12,0)*{};(-15,-9)*{\scs i_1};     (-4,-9)*{\scs i_{\alpha}};(5,-9)*{\scs i_{\alpha+1}};
       (16,-9)*{\scs i_m};
     }}; \endxy \qquad \text{for all orientations}
\nn\\
 \xy 0;/r.18pc/:
  (0,0)*{\xybox{
    (-4,-6)*{};(4,-6)*{} **\crv{(-4,1) & (4,1)}?(0)*\dir{} ;
    (15,6);(15,-6) **\dir{-}?(.5)*\dir{}+(2.3,0)*{};
    (-15,6);(-15,-6) **\dir{-}?(.5)*\dir{}+(2.3,0)*{};
     (24,1)*{ \lambda}; (-9,0)*{\cdots};(9,0)*{\cdots};
     (-12,0)*{};(12,0)*{};(-15,-9)*{\scs i_1};     (-4,-9)*{\scs i_{\alpha}};(5,-9)*{\scs i_{\alpha+1}};
       (16,-9)*{\scs i_m};
     }}; \endxy & \mapsto &
 \xy 0;/r.18pc/:
  (0,0)*{\xybox{
    (-4,-6)*{};(4,-6)*{} **\crv{(-4,1) & (4,1)}?(0)*\dir{} ;
    (15,6);(15,-6) **\dir{-}?(.5)*\dir{}+(2.3,0)*{};
    (-15,6);(-15,-6) **\dir{-}?(.5)*\dir{}+(2.3,0)*{};
     (24,1)*{ \lambda}; (-9,0)*{\cdots};(9,0)*{\cdots};
     (-12,0)*{};(12,0)*{};(-15,-9)*{\scs i_1};     (-4,-9)*{\scs i_{\alpha}};(5,-9)*{\scs i_{\alpha+1}};
       (16,-9)*{\scs i_m};
     }}; \endxy \qquad \text{for all orientations}
\nn
\end{eqnarray}
Above, the $i_{\alpha}$ in $(-1)^{i_{\alpha}}$ refers to the enumeration of the
vertex $i_{\alpha}$ in \eqref{eq_oriented_graph}.  One can check that the above
transformations respects the $\mf{sl}_2$-relations and the cyclic condition.
Furthermore, $\Sigma$ maps the $R(\nu)$-relations to their signed analogs by
rescaling the generators as above.  $\Sigma$ is a 2-functor and an isomorphism of
2-categories.

\begin{rem}
Since the 2-categories $\Ucat$ and $\Ucats$ are isomorphic, by the universal
property of the Karoubi envelope,  their Karoubi envelopes are isomorphic as
well.
\end{rem}

Define $\Ucats^*$ to be the graded additive $\Bbbk$-linear category which has the
same objects and 1-morphisms as $\Ucats$ and 2-morphisms
\begin{equation}
 \Ucats^*(x,y) := \bigoplus_{t\in \Z} \Ucats(x\{t\},y).
\end{equation}

%
\subsubsection{Relation to rings $R(\nu)$}
%

Regard the graded $\Bbbk$-algebra $R(\nu)$ with system of idempotents
$\{1_{\ii}\}$  as a pre-additive $\Bbbk$-linear category whose objects are $\{
\ii \mid \ii \in \seq(\nu)\}$. The $\Bbbk$-vector space of morphisms from $\ii$
to $\ii'$ is ${_{\ii'}R(\nu)_{\ii}}$. The composition ${_{\ii''}R(\nu)_{\ii'}}
\otimes {_{\ii'}R(\nu)_{\ii}} \to {_{\ii''}R(\nu)_{\ii}}$ is given by
multiplication in $R(\nu)$.

For any weight $\lambda$ there is a graded additive $\Bbbk$-linear functor
\begin{eqnarray} \label{eq_inclusion_iota}
  \iota_{\lambda} \maps R(\nu) &\longrightarrow& \Ucats^*(\lambda, \lambda+\nu_X)
\end{eqnarray}
that takes object $\ii$ to $\cal{E}_{{\ii}}\onel$, and is given on generators of
homs by
\begin{eqnarray} \nn
{_{\jj}R(\nu)_{\ii}}\qquad  &\to& \Ucats^*(\cal{E}_{\ii}\onel,
\cal{E}_{\jj}\onel) \\
  \xy
  (-12,0)*{\sup{}};  (-12,-6)*{\scs i_1};
  (-6,0)*{  \dots};
  (0,0)*{\supdot{}};  (0,-6)*{\scs i_{\alpha}};
  (6,0)*{ \dots};
  (12,0)*{\sup{}};  (13,-6)*{\scs i_m}; \endxy & \mapsto &(-1)^{i_{\alpha}}
    \;
   \xy 0;/r.18pc/:
  (6,6);(6,-6) **\dir{-}?(.5)*\dir{<}+(2.3,0)*{};
  (18,6);(18,-6) **\dir{-}?(.5)*\dir{<}+(2.3,0)*{};
  (-6,6);(-6,-6) **\dir{-}?(.5)*\dir{<}+(2.3,0)*{};
  (24,2)*{\lambda};(-6,-9)*{\scs i_1};     (6,-9)*{\scs i_{\alpha}};
  (19,-9)*{\scs i_m};
  (6.1,2)*{\bullet};
  (0,0)*{\cdots};(12,0)*{\cdots};
 \endxy \nn \\ \nn\\
 \xy  (-12,0)*{\sup{}};  (-12,-6)*{\scs i_1};
 (-8,0)*{\dots};
 (0,0)*{\dcross{i_\alpha}{\; \; i_{\alpha+1}}};
 (8,0)*{\dots};
 (12,0)*{\sup{}}; (13,-6)*{\scs i_m}; \endxy & \mapsto &
 \left\{\begin{array}{cc}
 (-1)^{i_{\alpha+1}} \;  \xy 0;/r.18pc/:
  (0,0)*{\xybox{
    (-4,-6)*{};(4,6)*{} **\crv{(-4,-1) & (4,1)}?(1)*\dir{>} ;
    (4,-6)*{};(-4,6)*{} **\crv{(4,-1) & (-4,1)}?(1)*\dir{>};
    (15,6);(15,-6) **\dir{-}?(.5)*\dir{<}+(2.3,0)*{};
    (-15,6);(-15,-6) **\dir{-}?(.5)*\dir{<}+(2.3,0)*{};
     (24,1)*{ \lambda}; (-9,0)*{\cdots};(9,0)*{\cdots};
     (-12,0)*{};(12,0)*{};(-15,-9)*{\scs i_1};     (-4,-9)*{\scs i_{\alpha}};(5,-9)*{\scs i_{\alpha+1}};
       (16,-9)*{\scs i_m};
     }};
  \endxy & \text{if $i_{\alpha}=i_{\alpha+1}$, or $i_{\alpha} \longrightarrow i_{\alpha+1}$}
  \\ \\
   \xy 0;/r.18pc/:
  (0,0)*{\xybox{
    (-4,-6)*{};(4,6)*{} **\crv{(-4,-1) & (4,1)}?(1)*\dir{>} ;
    (4,-6)*{};(-4,6)*{} **\crv{(4,-1) & (-4,1)}?(1)*\dir{>};
    (15,6);(15,-6) **\dir{-}?(.5)*\dir{<}+(2.3,0)*{};
    (-15,6);(-15,-6) **\dir{-}?(.5)*\dir{<}+(2.3,0)*{};
     (24,1)*{ \lambda}; (-9,0)*{\cdots};(9,0)*{\cdots};
     (-12,0)*{};(12,0)*{};(-15,-9)*{\scs i_1};     (-4,-9)*{\scs i_{\alpha}};(5,-9)*{\scs i_{\alpha+1}};
       (16,-9)*{\scs i_m};
     }};
  \endxy & \text{otherwise.}
 \end{array}\right. \nn
\end{eqnarray}
Linear combination of diagrams in $R(\nu)$ get sent to the corresponding rescaled
linear combination in $\Ucats^*(\lambda, \lambda+\nu_X)$ with the weight
$\lambda$ labelling the far right region.   In what follows we will refer to
those diagrams in the image of $\iota_{\lambda}$ as $R(\nu)$ generators.  The
image of $R(\nu)$ is spanned by diagrams with all strands upward pointing and no
caps or cups.

%
\section{Iterated flag varieties} \label{sec_Flag}
%

%
\subsection{Cohomology of $n$-step varieties}
%

The material in this section generalizes that of \cite[Section 6]{Lau1}.  The
reader is encouraged to start there for more examples and greater detail in the
$\mathfrak{sl}_2$ case. We enumerate by $I=\{1,2,\dots,n-1\}$ the vertex set of
the Dynkin diagram of $\mf{sl}_n$
\begin{equation}
    \xy
  (-15,0)*{\circ}="1";
  (-5, 0)*{\circ}="2";
  (5,  0)*{\circ}="3";
  (35,  0)*{\circ}="4";
  "1";"2" **\dir{-}?(.55)*\dir{};
  "2";"3" **\dir{-}?(.55)*\dir{};
  "3";(15,0) **\dir{-}?(.55)*\dir{};
  (25,0);"4" **\dir{-}?(.55)*\dir{};
  (-15,2.2)*{\scs 1};
  (-5,2.2)*{\scs 2};
  (5,2.2)*{\scs 3};
  (35,2.2)*{\scs n-1};
  (20,0)*{\cdots };
  \endxy
  \nn
\end{equation}

Fix $N \geq 0$,  and consider the variety $Fl(n)$ of $n$-step partial flags $F$
\[
 F = ( 0=F_0 \subseteq F_1 \subseteq \ldots \subseteq F_n =\C^{N} )
\]
in $\C^N$.  The dimensions of the subspaces $F_i$ are conveniently expressed as a
vector,
\[
 \underline{\dim} F = (\dim F_0,\dim F_1,\dim F_2, \ldots, \dim F_n ).
\]
The connected components of $Fl(n)$ are parameterized by non-negative integers
$\uk = (k_0,k_1,k_2,\ldots,k_{n})$ such that $0=k_0\leq k_1 \leq k_2 \leq \ldots
\leq k_{n} = N$.  The connected component $Fl(\uk)$ corresponding to $\uk$
consists of all flags $F$ such that $\underline{\dim}F = \uk$.  Throughout this
section we refer to the terms $k_{\alpha}$ of $\uk$ with the convention that
\begin{equation}
  k_{\alpha} = \left\{ \begin{array}{ccc}
                  0 &  & \text{if $\alpha\leq0$,} \\
                  N &  & \text{if $\alpha\geq n$.}
                \end{array}
  \right.
\end{equation}

The cohomology algebra of $Fl(\uk)$ is $\Z_{+}$-graded,
\[
 H_{\uk}:= H^*(Fl(\uk), \Bbbk) =
 \bigoplus_{0 \leq \ell \leq k_1(k_2-k_1)\cdots(N-k_{n-1})}H^{2\ell}(Fl(\uk),
 \Bbbk).
\]
For $1 \leq j\leq n$, $0 < \alpha \leq k_{j}-k_{j-1}$, let $x(\uk)_{j,\alpha}$ be
a formal variable of degree $2\alpha$.  The ring $H_{\uk}$ is isomorphic to the
quotient ring
\begin{equation}
 \left( \bigotimes_{j=1}^{n}\Bbbk[x(\uk)_{j,1},x(\uk)_{j,2},\dots,
  x(\uk)_{j,k_j-k_{j-1}}]\right)/I_{\uk,N},
\end{equation}
where $I_{\uk,N}$ is the ideal generated by the homogeneous terms in the equation
\begin{equation} \label{eq_Hk}
 \prod_{j=1}^{n} (1+x(\uk)_{j,1}t+ x(\uk)_{j,2}t^2+\cdots+
 x(\uk)_{j,k_j-k_{j-1}}t^{k_{j}-k_{j-1}})=1.
\end{equation}
Above, $t$ is a formal variable used to keep track of the degrees.  For
notational convenience we add variables $x(\uk)_{j,0}$ and set $x(\uk)_{j,0}=1$.
Furthermore, we set
\begin{equation} \label{eq_zero_variables}
x(\uk)_{j,\alpha} = 0 \qquad \text{if $\alpha > k_{j}-k_{j-1}$}.
\end{equation}

It is helpful to express the above relation in an alternative form.  Let
$\overline{x(\uk)_{i,\alpha}}$ denote the homogeneous term of degree $2\alpha$ in
the product
\begin{equation} \label{eq_overlinea}
 \prod_{
\xy
 (0,1)*{ \scs  j=1,};
 (-0.7,-1.7)*{\scs  j\neq i};
\endxy
  }^{n} (1+x(\uk)_{j,1}t+ x(\uk)_{j,2}t^2+\cdots+
 x(\uk)_{j,k_j-k_{j-1}}t^{k_{j}-k_{j-1}}).
\end{equation}
For example, if $n=4$ and $\uk=(1,3,4,7)$, then equation \eqref{eq_Hk} becomes
\begin{eqnarray}
 \big(1+x(\uk)_{1,1}t\big)
 \big(1+x(\uk)_{2,1}t+x(\uk)_{2,2}t^2\big)
 \big(1+x(\uk)_{3,1}t\big)
 \big(1+x(\uk)_{4,1}t+x(\uk)_{4,2}t^2+x(\uk)_{4,3}t^3\big)=1 \nn
\end{eqnarray}
and the terms $\overline{x(\uk)_{2,\alpha}}$ are given by omitting the second
term and multiplying out the rest
\begin{eqnarray}\nn
 \big(1+x(\uk)_{1,1}t\big)
 \big(1+x(\uk)_{3,1}t\big)
 \big(1+x(\uk)_{4,1}t+x(\uk)_{4,2}t^2+x(\uk)_{4,3}t^3\big)
\end{eqnarray}
so that
\begin{eqnarray}
  \overline{x(\uk)_{2,0}} &=& 1 ,\nn \\
  \overline{x(\uk)_{2,1}} &=& x(\uk)_{1,1}+x(\uk)_{3,1}+x(\uk)_{4,1} ,\nn \\
  \overline{x(\uk)_{2,2}} &=& x(\uk)_{1,1}x(\uk)_{3,1}+x(\uk)_{1,1}x(\uk)_{4,1}+x(\uk)_{3,1}x(\uk)_{4,1}+x(\uk)_{4,2} ,\nn \\
  \overline{x(\uk)_{2,3}} &=& x(\uk)_{1,1}x(\uk)_{3,1}x(\uk)_{4,1}+x(\uk)_{1,1}x(\uk)_{4,2}+x(\uk)_{3,1}x(\uk)_{4,2}
  +x(\uk)_{4,3}, \nn \\
 \overline{x(\uk)_{2,4}} &=& x(\uk)_{1,1}x(\uk)_{3,1}x(\uk)_{4,2}+
 x(\uk)_{1,1}x(\uk)_{4,3}+x(\uk)_{3,1}x(\uk)_{4,3}, \nn\\
 \overline{x(\uk)_{2,5}} &=& x(\uk)_{1,1}x(\uk)_{3,1}x(\uk)_{4,3},
\end{eqnarray}
and $\overline{x(\uk)_{2,\alpha}}=0$ for $\alpha>5$.  It is clear that
\eqref{eq_Hk} can be written
\begin{equation} \label{eq_sum_equal_delta}
 \sum_{f=0}^{\alpha} x(\uk)_{j,f} \;\overline{x(\uk)_{j,\alpha-f}} = \delta_{\alpha,0}
\end{equation}
for any $1 \leq j \leq n$, where $\delta_{\alpha,0}$ is the Kronecker delta.  We
call the elements $\overline{x(\uk)_{j,\alpha}}$ {\em dual generators} in light
of \eqref{eq_sum_equal_delta}.

For $1 \leq i \leq n-1$ let
\begin{equation}
\ukep=\big(({}_{+i}k)_0,({}_{+i}k)_1,({}_{+i}k)_2,\ldots, ({}_{+i}k)_{n}\big ),
\qquad 0 =({}_{+i}k)_0 \leq ({}_{+i}k)_1 \leq \cdots \leq ({}_{+i}k)_n=N \nn,
\end{equation}
be the sequence obtained from the sequence $\uk$ by increasing the $i$th term by
one
\begin{equation}
\ukep:=(k_0,k_1,k_2,\ldots, k_{i-1},k_i+1,k_{i+1},\ldots,k_n)
\end{equation}
if $k_i+1\leq k_{i+1}$, or by setting the sequence to the empty sequence
$\emptyset$ if $k_{i}=k_{i+1}$.  Namely, $({}_{+i}k)_j=k_j$ if $j\neq i$ and
$({}_{+i}k)_{i}=k_i+1$ if $k_i+1\leq k_{i+1}$.

When $k_i+1\leq k_{i+1}$ then $H_{\ukep}$ is the cohomology ring of the partial
flag variety consisting of flags $F$ with $\underline{\dim}F = \ukep$. The ring
$H_{\ukep}$ is given by
\begin{eqnarray}
 H_{\ukep}&=& \left(\bigotimes_{j =1}^{n}
 \Bbbk[x(\ukep)_{j,1},x(\ukep)_{j,2},\dots, x(\ukep)_{j,({}_{+i}k)_j-({}_{+i}k)_{j-1}}] \right)
 /I_{\ukep,N} \nn \\
 &=& \bigotimes_{j \neq i,i+1}
 \Bbbk[x(\ukep)_{j,1},\dots, x(\ukep)_{j,k_j-k_{j-1}}]
 \otimes \Bbbk[x(\ukep)_{i,1},\dots,x(\ukep)_{k_{i}-k_{i-1}+1}]
 \nn \\ && \quad
 \otimes  \Bbbk[x(\ukep)_{i+1,1},\dots, x(\ukep)_{k_{i+1}-k_{i}-1}] /I_{\ukep,N}
\end{eqnarray}
where $I_{\ukep,N}$ is the ideal generated by the homogeneous terms in
\begin{equation}
 \prod_{j=1}^{n} \left( \sum_{\alpha\geq 0}x(\ukep)_{j,\alpha}\;t^{\alpha} \right)
 =1.
\end{equation}
We define $H_{\emptyset}=0$.

Going back to the example of $n=4$ and $\uk=(1,3,4,7)$, then ${}_{+3}\uk =
(1,3,5,7)$, so that
\begin{equation}
H_{{}_{+3}\uk} =
\Bbbk[x({}_{+3}\uk)_{1,1},x({}_{+3}\uk)_{2,1},x({}_{+3}\uk)_{2,2},
x({}_{+3}\uk)_{3,1},x({}_{+3}\uk)_{3,2},x({}_{+3}\uk)_{4,1},x({}_{+3}\uk)_{4,2}]/I_{{}_{+3}\uk,7}
\end{equation}
where $I_{{}_{+3}\uk,7}$ is the ideal generated by the homogeneous terms in
\begin{eqnarray} \nn
(1+x({}_{+3}\uk)_{1,1})(1+x({}_{+3}\uk)_{2,1}+x({}_{+3}\uk)_{2,2})(1+x({}_{+3}\uk)_{3,1}+
x({}_{+3}\uk)_{3,2}) \\ \hspace{1in} \times
(1+x({}_{+3}\uk)_{4,1}+x({}_{+3}\uk)_{4,2})=1.
\end{eqnarray}

Similarly, we write $\ukem=(k_0, k_1,k_2,\ldots,
k_{i-1},k_i-1,k_{i+1},\ldots,k_n)$ for the sequence $\uk$ where we have
subtracted one from the $i$th position whenever $k_{i-1}\leq k_{i}-1$.  The
cohomology ring of the flag variety $Fl(\ukem)$ is $H_{\ukem}$, which can be
expressed explicitly in terms of generators as above.  When $k_{i-1}=k_{i}$ then
$\ukem= \emptyset$ and $H_{\ukem}:=H_\emptyset=0$.

%
\subsubsection{Flag varieties for the action of $E_i$ and $F_i$}
%

For $1\leq i \leq n-1$, define
\begin{eqnarray}
 \ukp &=& \left\{
 \begin{array}{ccc}
   (k_0,k_1, k_2, \ldots k_i, k_i+1, k_{i+1}, \ldots, k_{n-1},k_n)
   &  & \text{if $k_{i+1} \geq k_i+1$,} \\
   \emptyset &  & \text{otherwise.}
 \end{array} \right. \nn \\
 \ukm &=& \left\{
 \begin{array}{ccc}
   (k_0,k_1, k_2, \ldots  k_{i-1},k_i-1, k_i, \ldots, k_{n-1},k_n)
   &  & \text{if $k_{i}-1 \geq k_{i-1}$,} \\
  \emptyset &  & \text{otherwise.}
 \end{array} \right. \nn
\end{eqnarray}

For $\uk^{\pm i} \neq \emptyset$  the variety $Fl(\uk^{\pm i})$ is the component
of $Fl(n+1)$ consisting of flags $F$ such that $\underline{\dim}F= \uk^{\pm i}$.
The cohomology ring of $Fl(\uk^{\pm i})$ will be denoted by $H_{\uk^{\pm i}}$.
The cohomology ring $H_{\ukp}$ is
\begin{equation}
  H_{\ukp} =
  \bigotimes_{j\neq i+1} \Bbbk[x(\ukp)_{j,1},\dots,x(\ukp)_{j,k_j-k_{j-1}}]
  \otimes \Bbbk[\xi_i] \otimes \Bbbk[x(\ukp)_{i+1,1},\dots,x(\ukp)_{i+1,k_{i+1}-k_{i}-1}]
  /I_{\ukp,N}
\end{equation}
where $I_{\ukp,N}$ is the ideal generated by the homogeneous terms in
\begin{eqnarray}
 \big(1+\xi_{i} t \big)
 \big(1+x(\ukp)_{i+1,1}t+ \cdots +x(\ukp)_{i+1,k_{i+1}-k_i-1}t^{k_{i+1}-k_i-1}\big)
 \hspace{0.5in}\\
 \times
 \prod_{j\neq  i+1} \big(\sum_{f=0}^{k_j-k_{j-1}} x(\ukp)_{j,f}\;t^{f}\big) =1.
 \nn
\end{eqnarray}

The forgetful maps
  \[
 \xymatrix{Fl(\uk) & Fl(\ukp )\ar[l]_{p_1} \ar[r]^{p_2} & Fl(\ukep})
\]
induce maps of cohomology rings
\[
 \xymatrix{ H_{\uk} \ar[r]^-{p_1^*} & H_{\ukp}   & H_{\ukep} \ar[l]_-{p_2^*}}
\]
that make $H_{\ukp}$ a right $H_{\uk} \otimes H_{\ukep}$-module.  Since the
algebra $H_{\ukep}$ is commutative, we can turn a right $ H_{\ukep}$-module into
a left $H_{\ukep}$-module.   Hence, we can make $ H_{\ukp}$ into a $( H_{\ukep},
H_{\uk})$-bimodule.  In fact, $ H_{\ukp}$ is free as a graded $ H_{\uk}$-module
and as a graded $ H_{\ukep}$-module.

These inclusions making $H_{\ukp}$ a $(H_{\ukep},H_{\uk})$-bimodule are given
explicitly as follows:
\begin{eqnarray}\label{eq_inclusion1}
 H_{\uk} & \xymatrix@1{\ar@{^{(}->}[r] & } & H_{\ukp} \\
 x(\uk)_{j,\alpha} & \mapsto & x(\ukp)_{j,\alpha} \qquad \text{for $j\neq i+1$}~,\nn \\
 x(\uk)_{i+1,\alpha} & \mapsto & \xi_i \cdot x(\ukp)_{i+1,\alpha-1}+x(\ukp)_{i+1,\alpha}~,
 \nn
\end{eqnarray}
and
\begin{eqnarray} \label{eq_inclusion2}
 H_{\ukep} & \xymatrix@1{\ar@{^{(}->}[r] & } & H_{\ukp} \\
 x(\ukep)_{j,\alpha} & \mapsto & x(\ukp)_{j,\alpha} \qquad \text{for $j\neq i$}~, \nn \\
 x(\ukep)_{i,\alpha} & \mapsto & \xi_i \cdot x(\ukp)_{i,\alpha-1}+x(\ukp)_{i,\alpha}~.
 \nn
\end{eqnarray}
Notice that $x(\uk)_{j,\alpha}$ and $x(\ukep)_{j,\alpha}$ for $j\neq i,i+1$ are
mapped to the same element of $H_{\ukp}$.  Using these inclusions we identify
these elements of $H_{\uk}$ and $H_{\ukep}$ with their images in the bimodule
$H_{\ukp}$. Furthermore, we can also express the generators $x(\ukp)_{i,\alpha}$
and $x(\ukp)_{i+1,\beta}$ of $H_{\ukp}$ as the images of certain generators in
$H_{\uk}$ or $H_{\ukep}$.  Thus we can write $ H_{\ukp}$ as
\begin{eqnarray} \label{eq_generators_Heik1} \nn
 H_{\ukp} = \bigotimes_{j\neq i+1}^{} \Bbbk\left[ x(\uk)_{j,1}, \dots, x(\uk)_{j,k_{j}-k_{j-1}}  \right]
 \otimes \Bbbk[x(\ukep)_{i+1,1}, \dots, x(\ukep)_{i+1,k_{i+1}-k_{i}-1}] \otimes
 \Bbbk[\xi_i]/I_{\ukp,N}
\end{eqnarray}
or equivalently,
\begin{eqnarray}\label{eq_generators_Heik2} \nn
 H_{\ukp} = \bigotimes_{j\neq i}^{} \Bbbk\left[ x(\ukep)_{j,1}, \dots, x(\ukep)_{j,k_{j}-k_{j-1}}  \right]
 \otimes \Bbbk[x(\uk)_{i,1}, \dots, x(\uk)_{i,k_{i}-k_{i-1}}] \otimes
 \Bbbk[\xi_i]/I_{\ukp,N},
\end{eqnarray}
where $I_{\ukp,N}$ is the ideal described above.  Therefore, $\xi_i$ is the only
generator of $H_{\ukp}$ that is not identified with a generator of $H_{\uk}$ or
$H_{\ukep}$ under the above inclusions.

\begin{rem}
The generators $x(\uk)_{j,\alpha}$, $x(\ukep)_{j,\alpha}$, $\xi_i$ of $H_{\ukp}$
correspond to Chern classes of tautological bundles over the variety $Fl(\ukp)$.
The generator $\xi_i$ corresponds to the line bundle $F_{k_{i}+1}/F_{k_i}$
associated to the subspaces $F_{k_i} \subset F_{k_{i}+1}$ created by the subspace
insertion $\ukp$.
\end{rem}

\begin{defn}
The set of multiplicative generators
\[
\left\{ \begin{array}{ccl}
  \xi_i & & \\
    x(\uk)_{i,\alpha_i} & \qquad & \text{for $0 < \alpha_i\leq k_{i}-k_{i-1}$,} \\
   x(\ukep)_{i+1,\alpha_{i+1}} & \qquad & \text{for $0 <\alpha_{i+1}\leq k_{i+1}-k_{i}-1$,} \\
   x(\uk)_{j,\alpha_j}=x(\ukep)_{j,\alpha_j}\in H_{\ukp} & \qquad & \text{for $j \neq i, i+1$, and
   $ 0 < \alpha_j\leq k_{j}-k_{j-1}$},
 \end{array}\right\}
\]
for the ring $H_{\ukp}$, corresponding to the Chern class of the tautological
line bundle, and to the canonical inclusions \eqref{eq_inclusion1} and
\eqref{eq_inclusion2} of generators in $H_{\ukep}$ and $H_{\uk}$ into $H_{\ukp}$,
are called the {\em canonical generators} of $H_{\ukp}$.
\end{defn}

Using commutativity we can regard $H_{\ukp}$ as an
$(H_{\uk},H_{\ukep})$-bimodule. The generators of $H_{\uk}$ and $H_{\ukep}$ that
are not mapped to canonical generators can be expressed in terms of canonical
generators as follows:
\begin{eqnarray}
  x(\uk)_{i+1,\alpha} & =&  \xi_i \cdot x(\ukep)_{i+1,\alpha-1}
  + x(\ukep)_{i+1,\alpha}, \nn\\
  x(\ukep)_{i,\alpha} & = & \xi_i \cdot x(\uk)_{i,\alpha-1}  +x(\uk)_{i,\alpha}, \label{eq_noncan}
\end{eqnarray}
for all values of $\alpha$.

%
\subsection{Graphical calculus for iterated flag varieties}
%

%
\subsubsection{Rings $H_{\uk}$}
%

All of what has been described can be easily visualized using a shorthand
notation in which the generators of $H_{\uk}$ are drawn as labelled bubbles
floating in a region carrying a label $\lambda$ called the weight.  The label
$\lambda$ will be important when we relate partial flag varieties to the
2-category $\UcatD$ categorifying $\UA(\mathfrak{sl}_n)$. More precisely,
$\lambda$ will correspond to a weight of the irreducible representation of
$\UA(\mathfrak{sl}_n)$ with highest weight $(N,0,0,\ldots,0)$.

To a sequence $\uk=(k_0, k_1,k_2,\ldots,k_n)$ as above associate
$\lambda=\lambda(\uk)=(\lambda_1,\lambda_2, \ldots , \lambda_{n-1}) \in \Z^{n-1}$
where
\begin{equation}
\lambda_{\alpha} = -k_{\alpha+1}+2k_{\alpha}-k_{\alpha-1}.
\end{equation}
The weight corresponding to the sequence $\ukep$ is defined analogously, where
$k_i$ is replaced by $k_i+1$.  Comparing with \eqref{eq_weight_action1} it is
clear that $\lambda(\ukep) = \lambda+i_X$.  Similarly, $\lambda(\ukem)= \lambda -
i_X$.

With this convention, the generators of $H_{\uk}$ and $H_{\ukep}$ are depicted as
follows:
\begin{eqnarray}
 x(\uk)_{j,\alpha}:= \begin{pspicture}[.5](2.4,1.7)
  \rput(2.3,1.9){$\lambda$}
  \rput(1,1.1){\chern{j,\alpha}}
  \end{pspicture} \\ \nn\\
 x(\ukep)_{j,\alpha}:= \begin{pspicture}[.5](2.4,1.7)
  \rput(2.3,1.9){$\lambda+i_X$}
  \rput(1,1.1){\chern{j,\alpha}}
  \end{pspicture}
\end{eqnarray}
where the identity is depicted by the empty region of the appropriate weight.
Products of generators are depicted by a bubble in the plane for each generator
present in the product. Diagrams are only considered up to planar isotopy.  A
generic element in $H_{\uk}$ can be depicted as a formal linear combination of
such diagrams. For example, if $n=4$, $\uk=(1,3,4,7)$ then the element
$x(\uk)_{1,1}x(\uk)_{4,3}+5 \cdot x(\uk)_{3,1} \in H_{\uk}$ is represented as
\[
 \begin{pspicture}[.5](3.4,1.9)
  \rput(1.7,1.9){$\lambda$}
  \rput(1,.8){\chern{1,1}}
  \rput(2.4,.8){\chern{4,3}}
  \end{pspicture}
  + 5
  \begin{pspicture}[.5](3.4,1.9)
  \rput(1.7,1.9){$\lambda$}
  \rput(1,.8){\chern{3,1}}
  \end{pspicture}
\]

If we depict the dual generators $\overline{x(\uk)_{j,\alpha}}$ of $H_{\uk}$
defined in \eqref{eq_overlinea} as
\begin{equation} \nn
\overline{x(\uk)_{j,\alpha}} =
 \begin{pspicture}[.5](2.4,1.7)
  \rput(2.3,1.9){$\lambda$}
  \rput(1,1.1){\dchern{j,\alpha}}
  \end{pspicture}
\end{equation}
then the defining relations \eqref{eq_sum_equal_delta} for $H_{\uk}$ become the
equations
\begin{equation} \label{eq_graph_def_relations}
\sum_{f=0}^{\alpha}
 \begin{pspicture}[.5](3.5,1.7)
  \rput(1.7,1.9){$\lambda$}
  \rput(1,1){\chern{j,f}}
  \rput(2.7,1){\dchern{j,{\alpha-f}}}
  \end{pspicture}
 \quad = \quad
 \sum_{g=0}^{\alpha}
 \begin{pspicture}[.5](3.5,1.7)
  \rput(1.7,1.9){$\lambda$}
  \rput(1,1){\chern{j,{\alpha-g}}}
  \rput(2.7,1){\dchern{j,g}}
  \end{pspicture}
  \quad = \quad
 \begin{pspicture}[.5](3.5,1.7)
  \rput(.5,.85){$  \delta_{\alpha,0}$}
  \rput(1.7,1.9){$\lambda$}
  \end{pspicture}
\end{equation}
for $\alpha\geq 0$ and $1 \leq j \leq n$.  Notice that
\begin{eqnarray}
\begin{pspicture}[.5](2.4,1.7)
  \rput(2.3,1.9){$\lambda$}
  \rput(1,1.1){\chern{j,0}}
  \end{pspicture} \quad = \quad \begin{pspicture}[.5](2.4,1.7)
  \rput(2.3,1.9){$\lambda$}
  \rput(1,1.1){\dchern{j,0}}
  \end{pspicture}\;\; =1, \qquad\begin{pspicture}[.5](2.4,1.7)
  \rput(2.3,1.9){$\lambda$}
  \rput(1,1.1){\chern{j,\alpha}}
  \end{pspicture}\;\; = 0 \qquad \text{if $\alpha < 0$, or $k_j-k_{j-1}<\alpha$,}
  \nn
\end{eqnarray}
see \eqref{eq_zero_variables} and the comments preceding
\eqref{eq_zero_variables}.

%
\subsubsection{Bimodules $H_{\ukp}$}
%

The identity element in $H_{\ukp}$  is represented by a vertical line labelled
$i$
\[
  H_{\ukp}\ni 1 \quad := \qquad
 \begin{pspicture}[.5](2,1.5)
  \rput(1.5,0){\Eline{i}}
  \rput(2.6,1.1){$\lambda$}
  \rput(.4,1.1){$\lambda+i_X$}
\end{pspicture}
\]
where the orientation indicates that we are regarding $H_{\ukp}$ as an
$(H_{\ukep},H_{\uk})$-bimodule.  The $\lambda$ on the right hand side is the
weight corresponding to $\uk$.  Hence, having $\lambda$ on the right hand side of
the diagram indicates the right action of $H_{\uk}$ on $H_{\ukp}$. Similarly, the
$\lambda+i_X$ on the left indicates the left action of $H_{\ukep}$ on $H_{\ukp}$.

When regarding $H_{\ukp}$ as an $(H_{\uk},H_{\ukep})$-bimodule we depict it in
the graphical calculus with the opposite orientation (a downward pointing arrow).
\[
H_{\ukp} \ni 1 \quad := \qquad
 \begin{pspicture}[.5](2,1.5)
  \rput(1.5,0){\Fline{i}}
  \rput(2.6,1.1){$\lambda+i_X$}
  \rput(.4,1.1){$\lambda
  $}
\end{pspicture}
\]
We will often omit the weights from all regions but one, with it understood that
crossing an upward pointing arrow with label $i$ from right to left changes the
weight by $i_X$, and crossing a downward pointing arrow from right to left
changes the weight by $-i_X$.

Equations \eqref{eq_generators_Heik2} show that all of the generators from
$H_{\ukp}$ except for $\xi_i$ can be interpreted as either generators of
$H_{\uk}$ or $H_{\ukep}$ under the natural inclusions. This fact is represented
in the graphical calculus as follows:
\begin{eqnarray}
 H_{\ukp} \ni x(\uk)_{j,\alpha} & := &
 \begin{pspicture}[.5](4,1.5)
  \rput(.5,0){\Eline{i}}
  \rput(1.5,1){\chern{j,\alpha}}
  \rput(2.6,1.5){$\lambda$}
\end{pspicture} \qquad j\neq i+1 \label{eq_c_calc}\\
H_{\ukp} \ni x(\ukep)_{j,\beta} & := &
 \begin{pspicture}[.5](4,2)
  \rput(2,0){\Eline{i}}
  \rput(1,1){\chern{j,\beta}}
  \rput(2.6,1.5){$\lambda$}
\end{pspicture} \qquad j\neq i \\
 H_{\ukp} \ni \xi_i & :=&
\begin{pspicture}[.5](3,2)
  \rput(.5,0){\Elinedot{i}{}}
  \rput(1.6,1.5){$\lambda$}
\end{pspicture}
\end{eqnarray}
where each diagram inherits a grading from the Chern class it represents ($\deg
x(\uk)_{j,\alpha} =2\alpha$, $\deg x(\ukep)_{j,\beta}=2\beta$, and $\deg \xi_i
=2$). Equation \eqref{eq_c_calc} is meant to depicts the generator
$x(\uk)_{j,\alpha}\in H_{\ukp}$ as the element $x(\uk)_{j,\alpha}\in H_{\uk}$
acting on the identity of $H_{\ukp}$. Likewise, the generator $x(\ukep)_{j,\beta}
\in H_{\ukp}$ is depicted as the element $x(\ukep)_{j,\beta} \in H_{\ukep}$
acting on the identity of $H_{\ukp}$. The generator $\xi_i$ is represented by a
dotted line so that $\xi_i^{\alpha}$ is represented by $\alpha$ dots on a line,
but for simplicity we write this using a single dot and a label to indicate the
power.

The identification \eqref{eq_inclusion1} and \eqref{eq_inclusion2} of
$x(\uk)_{j,\alpha}$ with $x(\ukep)_{j,\alpha}$ in $H_{\ukp}$ for $j\neq, i, i+1$
leads to the graphical identity:
\begin{equation} \label{eq_graph_easy_slide}
 \begin{pspicture}[.5](3,1.7)
  \rput(.5,0){\Eline{i}}
  \rput(1.5,1){\chern{j,\alpha}}
  \rput(2.6,1.5){$\lambda$}
\end{pspicture}
\quad = \quad
 \begin{pspicture}[.5](3,1.7)
  \rput(2,0){\Eline{i}}
  \rput(1,1){\chern{j,\alpha}}
  \rput(2.6,1.5){$\lambda$}
\end{pspicture}
\quad \text{for $j\neq i,i+1$.}
\end{equation}
Similarly, \eqref{eq_noncan} provides the identities
\begin{eqnarray} \label{eq_noncan1}
 \begin{pspicture}[.5](3.1,1.7)
  \rput(.5,0){\Eline{i}}
  \rput(1.7,1){\chern{i+1,\alpha}}
  \rput(2.8,1.5){$\lambda$}
\end{pspicture}
\quad &=& \quad
 \begin{pspicture}[.5](3.3,1.7)
  \rput(2.4,0){\Elinedot{i}{}}
  \rput(1,1){\chern{i+1,\alpha-1}}
  \rput(3,1.5){$\lambda$}
\end{pspicture}
 \;\;+ \;\;
 \begin{pspicture}[.5](3,1.7)
  \rput(2.2,0){\Eline{i}}
  \rput(1,1){\chern{i+1,\alpha}}
  \rput(2.8,1.5){$\lambda$}
\end{pspicture} \\ \label{eq_noncan2}
\begin{pspicture}[.5](3,2)
  \rput(2,0){\Eline{i}}
  \rput(1,1){\chern{i,\alpha}}
  \rput(2.6,1.5){$\lambda$}
\end{pspicture}
\quad &=& \quad
 \begin{pspicture}[.5](3,1.7)
  \rput(.5,0){\Elinedot{i}{}}
  \rput(1.5,1){\chern{i,\alpha-1}}
  \rput(2.6,1.5){$\lambda$}
\end{pspicture}
+
 \begin{pspicture}[.5](3,1.7)
  \rput(.5,0){\Eline{i}}
  \rput(1.5,1){\chern{i,\alpha}}
  \rput(2.6,1.5){$\lambda$}
\end{pspicture}
\end{eqnarray}
expressing non-canonical generators in terms of canonical generators.  It is
sometimes helpful to express the canonical generators in terms of non-canonical
generators:
\begin{eqnarray} \label{eq_can_to_noncan1}
 \begin{pspicture}[.5](3.1,1.5)
  \rput(2,0){\Eline{i}}
  \rput(.8,1){\chern{i+1,\alpha}}
  \rput(2.8,1.5){$\lambda$}
\end{pspicture}
\quad &=& \quad
 \sum_{f=0}^{\alpha} \; (-1)^{f}
 \begin{pspicture}[.5](3.3,2)
  \rput(.5,0){\Elinedot{i}{f}}
  \rput(2.1,1){\chern{i+1,\alpha-f}}
  \rput(3.6,1.5){$\lambda$}
\end{pspicture}
 \\ \label{eq_can_to_noncan2}
 \begin{pspicture}[.5](3,2)
  \rput(.5,0){\Eline{i}}
  \rput(1.5,1){\chern{i,\alpha}}
  \rput(2.6,1.5){$\lambda$}
\end{pspicture}
\quad &=& \quad \sum_{f=0}^{\alpha} \; (-1)^{f}
 \begin{pspicture}[.5](3.1,2)
  \rput(2.3,0){\Elinedot{i}{f}}
  \rput(1.1,1){\chern{i,\alpha-f}}
  \rput(3.6,1.5){$\lambda$}
\end{pspicture}
\end{eqnarray}
which can be verified using \eqref{eq_noncan1} and \eqref{eq_noncan2}. From these
equations we can derive other useful identities:
\begin{eqnarray}  \label{eq_alphadot1}
 \begin{pspicture}[.5](2.5,1.8)
  \rput(.8,0){\Elinedot{i}{\alpha}}
  \rput(2,1.5){$\lambda$}
\end{pspicture}
&=& (-1)^{\alpha}\sum_{f=0}^{\alpha}
 \begin{pspicture}[.5](5,1.8)
  \rput(2.7,0){\Eline{i}}
  \rput(1.3,1){\chern{i+1,\alpha-f}}
  \rput(3.8,1){\dchern{i+1,f}}
  \rput(4.9,1.5){$\lambda$}
\end{pspicture} \\ \label{eq_alphadot2}
\begin{pspicture}[.5](2.5,1.8)
  \rput(.8,0){\Elinedot{i}{\alpha}}
  \rput(2,1.5){$\lambda$}
\end{pspicture}
&=& (-1)^{\alpha}\sum_{g=0}^{\alpha}
 \begin{pspicture}[.5](5,1.8)
  \rput(2.7,0){\Eline{i}}
  \rput(1.3,1){\dchern{i,\alpha-g}}
  \rput(3.8,1){\chern{i,g}}
  \rput(4.9,1.5){$\lambda$}
\end{pspicture}
\end{eqnarray}
\begin{prop}
\begin{eqnarray}
 \begin{pspicture}[.5](3,1.8)
  \rput(.5,0){\Eline{i}}
  \rput(1.5,1){\dchern{j,\alpha}}
  \rput(2.6,1.5){$\lambda$}
\end{pspicture}
\quad &=& \quad \left\{
\begin{array}{ccc}
 \begin{pspicture}[.5](3.1,1.8)
  \rput(2.3,0){\Eline{i}}
  \rput(1.1,1){\dchern{j,\alpha}}
  \rput(3.6,1.5){$\lambda$}
\end{pspicture}
 & \quad & \text{if $j \neq i, i+1$} \\
 \begin{pspicture}[.5](3.3,2)
  \rput(2.4,0){\Elinedot{i}{}}
  \rput(1,1){\dchern{i,\alpha-1}}
  \rput(3,1.5){$\lambda$}
\end{pspicture}
 \;\;+ \;\;
 \begin{pspicture}[.5](3,2)
  \rput(2.2,0){\Eline{i}}
  \rput(1,1){\dchern{i,\alpha}}
  \rput(2.8,1.5){$\lambda$}
\end{pspicture}
  &  \quad & \text{if $j=i$} \\ \\
 \sum_{f=0}^{\alpha} \; (-1)^{f}
 \begin{pspicture}[.5](4.1,1.8)
  \rput(2.6,0){\Elinedot{i}{f}}
  \rput(1.2,1){\dchern{i+1,\alpha-f}}
  \rput(3.9,1.5){$\lambda$}
\end{pspicture}
  & \quad  & \text{if $j=i+1$}
\end{array}
\right. \nn \\
\label{eq_dark_slide1}\end{eqnarray} 
\begin{eqnarray}
 \begin{pspicture}[.5](3.1,1.8)
  \rput(2.3,0){\Eline{i}}
  \rput(1.1,1){\dchern{j,\alpha}}
  \rput(3.6,1.5){$\lambda$}
\end{pspicture}
\quad &=& \quad \left\{
\begin{array}{ccc}
\begin{pspicture}[.5](3,1.8)
  \rput(.5,0){\Eline{i}}
  \rput(1.5,1){\dchern{j,\alpha}}
  \rput(2.6,1.5){$\lambda$}
\end{pspicture}
 & \quad & \text{if $j \neq i, i+1$} \\
 \sum_{f=0}^{\alpha} \; (-1)^{f}
 \begin{pspicture}[.5](3.3,2)
  \rput(.5,0){\Elinedot{i}{f}}
  \rput(2.1,1){\dchern{i,\alpha-f}}
  \rput(3.6,1.5){$\lambda$}
\end{pspicture}
  &  \quad & \text{if $j=i$} \\ \\
 \begin{pspicture}[.5](3.7,1.5)
  \rput(.5,0){\Elinedot{i}{}}
  \rput(2,1){\dchern{i+1,\alpha-1}}
  \rput(3.4,1.5){$\lambda$}
\end{pspicture}
+
 \begin{pspicture}[.5](3.3,1.5)
  \rput(.5,0){\Eline{i}}
  \rput(1.9,1){\dchern{i+1,\alpha}}
  \rput(3,1.5){$\lambda$}
\end{pspicture}
  & \quad  & \text{if $j=i+1$}
\end{array}
\right. \nn \\
\label{eq_dark_slide2}
\end{eqnarray}
\end{prop}

\begin{proof}
Recall that the elements $\overline{x(\uk)_{j,\alpha}}$ are sums of homogeneous
symmetric terms in all variables except for $x(\uk)_{j,\beta}$. By
\eqref{eq_graph_easy_slide} all terms $x(\uk)_{\ell,\alpha}$ for $\ell \neq i,
i+1$ can be slid across the line labelled $i$.  The case when $j \neq i, i+1$
then reduces to the problem of sliding symmetric combinations $\sum_{f=0}^{\beta}
x(\uk)_{i,f} x(\uk)_{i+1,\beta-f}$ across a line labelled $i$. Such slides are
determined by the following calculation in $H_{\ukp}$:
\begin{eqnarray}
  \sum_{f=0}^{\beta} x(\uk)_{i,f} \;x(\uk)_{i+1,\beta-f}
  & \refequal{\eqref{eq_can_to_noncan2}} &
  \sum_{f=0}^{\beta} \sum_{g=0}^{f} (-1)^{g} x(_{+i}\uk)_{i,f-g}\;
  \xi_i^{g} \;x(\uk)_{i+1,\beta-f} \nn
\end{eqnarray}
\begin{eqnarray}
  \refequal{\eqref{eq_noncan1}}
 \sum_{f=0}^{\beta} \sum_{g=0}^{f} (-1)^{g} x(_{+i}\uk)_{i,f-g}\;
 x(_{+i}\uk)_{i+1,\beta-f-1}\;\xi_i^{g+1} \nn \\
+
 \sum_{f=0}^{\beta} \sum_{g=0}^{f} (-1)^{g} x(_{+i}\uk)_{i,f-g}
 \;x(_{+i}\uk)_{i+1,\beta-f}\;
 \xi_i^{g}
\end{eqnarray}
Change variables to $f'=f+1$, and $g'=g+1$ in the first summation, so that all
terms cancel except for $g=0$ term in the second summation; this term is equal to
\begin{equation}
  \sum_{f=0}^{\beta} x(_{+i}\uk)_{i,f}\;
  x(_{+i}\uk)_{i+1,\beta-f}.
\end{equation}
Hence, when $j \neq i,i+1$ the elements $\overline{x(\uk)_{j,\alpha}}$ and
$\overline{x(_{+i}\uk)_{j,\alpha}}$ can be slid across a line labelled $i$.

The dual generator $\overline{x(\uk)_{j,\alpha}}$ contains symmetric homogeneous
combinations of variables $x(\uk)_{\ell,\beta}$ for $\ell \neq j$. When $j=i$ in
\eqref{eq_dark_slide1} all terms in $\overline{x(\uk)_{i,\alpha}}$ slide across
lines labelled $i$ except for the variables $x(\uk)_{i+1,\beta}$. Using
\eqref{eq_noncan1} to slide these across establishes \eqref{eq_dark_slide1} for
the case $j=i$. Similarly, when $j=i+1$ all terms in
$\overline{x(\uk)_{i+1,\alpha}}$ slide from right to left across lines labelled
$i$ except for the variables $x(\uk)_{i,\beta}$. Using \eqref{eq_can_to_noncan2}
to slide these completes the proof of \eqref{eq_dark_slide1}.  Equation
\eqref{eq_dark_slide2} is proven similarly.
\end{proof}

By duality, analogous equations as those above hold for downward pointing arrows.
For example, equation \eqref{eq_alphadot1} implies
\begin{equation}
\begin{pspicture}[.5](2.5,1.8)
  \rput(2,0){\Flinedot{i}{\alpha}}
  \rput(.4,1.5){$\lambda$}
\end{pspicture}
= (-1)^{\alpha}\sum_{f=0}^{\alpha}
 \begin{pspicture}[.5](5,1.8)
  \rput(3.1,0){\Fline{i}}
  \rput(1.7,1){\dchern{i,\alpha-f}}
  \rput(4.1,1){\chern{i,f}}
  \rput(.4,1.5){$\lambda$}
\end{pspicture}
\end{equation}

%
\subsubsection{Bimodules $H_{\uk^{\ii}}$}
%

For the remainder of this paper we write a signed sequence $\ii = \epsilon_1 i_1
\epsilon_2 i_2 \dots \epsilon_m i_m$ as $\ii = s_1s_2 \dots s_m$ with $s_{\alpha}
= \epsilon_{\alpha}{i_{\alpha}}$. For $s_{\alpha}=\epsilon_{\alpha} i_{\alpha}$
write ${}_{s_{\alpha}}\uk$ for the sequence obtained from $\uk =
(k_0,k_1,k_2,\ldots,k_{n})$ by increasing the $k_{i_{\alpha}}$  by 1 if
$\epsilon_{\alpha} =+$ and $i_{\alpha}<i_{\alpha+1}$, decreasing the sequence by
1 if $\epsilon_{\alpha}=-$ and $i_{\alpha-1}<i_{\alpha}$, and setting the
sequence to $\emptyset$ otherwise. Then ${}_{\ii}\uk$ is either set to
$\emptyset$, or else it denotes the sequence obtained from $\uk$ with the
$k_{i_{\alpha}}$th term increased by one if $s_{\alpha}=+{i_{\alpha}}$, or
decreased by one if $s_{\alpha}=-{i_{\alpha}}$, sequentially for each
$s_{\alpha}$ in the signed sequence $\ii$, reading from the right. It is clear
that the sequence ${}_{\ii}\uk$ is equal to the sequence ${}_{\jj}\uk$ whenever
$\ii,\jj\in\sseq$ with $\ii_X=\jj_X\in X$ and ${}_{\ii}\uk\neq \emptyset$ and
${}_{\jj}\uk\neq\emptyset$.

We write $H_{\uk^{\ii}}$ for the $(H_{{}_{\ii}\uk},H_{\uk})$-bimodule
\begin{eqnarray} \label{eq_Hes}
  H_{\uk^{\ii}} :=
    H_{{}_{s_2s_3 \cdots s_m}\uk^{s_1}}
    \otimes_{H_{{}_{s_2s_3 \cdots s_m}\uk}}
\cdots \otimes_{H_{{}_{s_{m-1}s_m}\uk}} H_{{}_{s_m}\uk^{s_{m-1}}}
\otimes_{H_{{}_{s_m}\uk}} H_{\uk^{{s_m}}}    .
\end{eqnarray}
This bimodule can also be described as the cohomology ring of the variety of
$m+n$ step iterated partial flags corresponding to the sequence obtained from
$\uk$ by the ordered insertion of subspaces determined from the signed sequence
$\ii$.

The $(H_{{}_{\ii}\uk},H_{\uk})$-bimodule $H_{\uk^{\ii}}$ can be understood using
the graphical calculus. A general signed sequence $\ii = s_1 s_2 \cdots s_m$ is
represented by a sequence of lines coloured by the sequence $i_1i_2\cdots i_m$,
where the line coloured by $i_{\alpha}$ is oriented upward if
$s_{\alpha}=+{i_{\alpha}}$ and oriented downward if $s_{\alpha}=-{i_{\alpha}}$.

\paragraph{Examples:}
\begin{enumerate}
\item
For the signed sequence $+j+i \in \sseq$ consider the
$(H_{{}_{+j+i}\uk},H_{\uk})$-bimodule
\[
 H_{\uk^{+j+i}} = H_{{}_{+i}\uk^{+j}} \otimes_{H_{{}_{+i}\uk}} H_{\ukp}.
\]
As explained in the previous section, the identity elements of the
$(H_{\ukep},H_{\uk})$-bimodule $H_{\ukp}$ and the
$(H_{{}_{+i+j}\uk},H_{\ukep})$-bimodule $H_{\ukep^{j}}$ are depicted as
\begin{eqnarray}
 \begin{pspicture}[.5](3,1.5)
  \rput(1.5,0){\Eline{i}}
  \rput(2.6,1.1){$\lambda$}
  \rput(.4,1.1){$\lambda+i_X$}
\end{pspicture}\qquad \quad {\rm and}
\qquad \qquad
 \begin{pspicture}[.5](2,1.5)
  \rput(2.8,0){\Eline{j}}
  \rput(4.1,1.1){$\lambda+i_X$}
  \rput(.8,1.1){$\lambda+i_X+j_X$}
\end{pspicture}
\end{eqnarray}
respectively.

The identity element of $H_{\uk^{+j+i}}$ is represented by the diagram
\begin{eqnarray}
 \begin{pspicture}[.5](3,1.5)
  \rput(2.7,0){\Eline{j}}
  \rput(5,0){\Eline{i}}
  \rput(6,1.1){$\lambda$}
  \rput(3.9,1.1){$\lambda+i_X$}
  \rput(.8,1.1){$\lambda+i_X+j_X$}
\end{pspicture}
\end{eqnarray}
The region in the middle of the two lines is labelled by the weight $\lambda+i_X$
corresponding to $\ukep$.  The tensor product over the action of $H_{\ukep}$ is
represented diagrammatically by the fact that a labelled bubble in the region
with weight $\lambda+i_X$ can be equivalently regarded as an element of
$H_{\ukep}$ acting on the line corresponding to $H_{\ukp}$, or the line
corresponding to $H_{\ukep^{j}}$.
\[
 \begin{array}{ccc}
   \text{Action of $H_{\uk}$} & \text{Action of $H_{\ukep}$} & \text{Action of $H_{{}_{+j+i}\uk}$} \\
    &  &  \\
    \begin{pspicture}[.5](5,1.5)
  \rput(1,0){\Eline{j}}
  \rput(2.5,0){\Eline{i}}
  \rput(3.5,.8){\chern{\ell,\alpha}}
  \rput(4.5,1.3){$\lambda$}
\end{pspicture} &
\begin{pspicture}[.5](4.3,1.5)
  \rput(1,0){\Eline{j}}
  \rput(3,0){\Eline{i}}
  \rput(2,.8){\chern{\ell,\alpha}}
  \rput(4,1.3){$\lambda$}
\end{pspicture}
 &
\begin{pspicture}[.5](4.6,1.5)
  \rput(2,0){\Eline{j}}
  \rput(3.5,0){\Eline{i}}
  \rput(1,.8){\chern{\ell,\alpha}}
  \rput(4.2,1.3){$\lambda$}
\end{pspicture}
 \end{array}
\]
The weight $\lambda$ on the far right, and the weight of $\lambda+i_X+j_X$ on the
far left, indicate that this diagram is describing an
$(H_{{}_{+j+i}\uk},H_{\uk})$-bimodule where the various actions are depicted as
above.

  \item For $\ii = +{i_1}+{i_2}+{i_3}\cdots +{i_r}$ the identity element of the bimodule
$H_{\uk^{\ii}}$ is depicted by a sequence of upward oriented labelled lines
\begin{equation}
H_{\uk^{\ii}} \ni 1 \quad= \quad
\begin{pspicture}[.5](7,1.7)
  \rput(1,0){\Eline{i_1}}
  \rput(2,0){\Eline{i_{2}}}
  \rput(3,0){\Eline{i_{3}}}
  \rput(4,1.3){$\cdots$}
  \rput(5,0){\Eline{i_{m-1}}}
  \rput(6,0){\Eline{i_m}}
  \rput(6.8,1.3){$\lambda$}
\end{pspicture}
\end{equation}

The canonical generators $\xi_{i_{\alpha}}$ of each term in the tensor product
$\eqref{eq_Hes}$ are represented graphically by a dot on the line labelled
$i_{\alpha}$.
\begin{equation}
H_{\uk^{\ii}} \ni 1 \otimes 1 \otimes \cdots  1 \otimes \xi_{i_{\alpha}} \otimes
1 \cdots \otimes 1 \quad= \quad
\begin{pspicture}[.5](7,1.5)
  \rput(1,0){\Eline{i_1}}
  \rput(2,0){\Eline{i_{2}}}
  \rput(3,1.3){$\cdots$}
  \rput(4,0){\Elinedot{i_{\alpha}}{}}
  \rput(5,1.3){$\cdots$}
  \rput(6,0){\Eline{i_m}}
  \rput(6.8,1.3){$\lambda$}
\end{pspicture} \nn
\end{equation}
Likewise, the tensor product over the rings $H_{{}_{s_{\alpha}s_{\alpha+1}\cdots
s_m}\uk}$ is represented by the regions between lines. Again, the weights on the
far left and right of the diagram indicate the bimodule structure.
\[
 \begin{array}{ccccc}
   \text{Action of $H_{\uk}$} &\quad &
    \text{Action of $H_{{}_{i_{\alpha}\cdots i_m}\uk}$} &\quad&
    \text{Action of $H_{{}_{+\ii}\uk}$} \\
    & & & &  \\
    \begin{pspicture}[.5](5,1.5)
  \rput(.5,0){\Eline{i_1}}
  \rput(1.25,0){\Eline{i_2}}
  \rput(2,1.3){$\cdots$}
  \rput(2.75,0){\Eline{i_m}}
  \rput(3.75,1){\chern{j,\beta}}
  \rput(4.75,1.5){$\lambda$}
\end{pspicture} & &
\begin{pspicture}[.5](6,1.5)
  \rput(.5,0){\Eline{i_1}}
  \rput(1.25,1.3){$\cdots$}
  \rput(2,0){\Eline{}}
  \rput(2.75,1){\chern{j,\beta}}
  \rput(3.5,0){\Eline{i_{\alpha}}}
  \rput(4.25,1.3){$\cdots$}
  \rput(5,0){\Eline{i_m}}
  \rput(5.75,1.5){$\lambda$}
\end{pspicture}
 & &
\quad  \begin{pspicture}[.5](5,1.5)
  \rput(1.25,0){\Eline{i_1}}
  \rput(2,0){\Eline{i_2}}
  \rput(2.75,1.3){$\cdots$}
  \rput(3.5,0){\Eline{i_m}}
  \rput(.4,1){\chern{j,\beta}}
  \rput(4.75,1.5){$\lambda$}
\end{pspicture}
 \end{array}
\]

\item Consider the $(H_{\uk},H_{\uk})$-bimodule corresponding to the tensor product
\begin{equation}
  H_{\ukp} \otimes_{H_{\ukep}} H_{\ukp}
\end{equation}
where in the first factor we are regarding $H_{\ukp}$ as a
$(H_{\uk},H_{\ukep})$-bimodule.  The identity element of this bimodule is
represented by the diagram
\begin{equation}
   \begin{pspicture}[.5](3,1.7)
  \rput(1.1,0){\Fline{i}}
  \rput(3.3,0){\Eline{i}}
  \rput(4,1.1){$\lambda$}
  \rput(2.25,1.1){$\lambda+i_X$}
  \rput(.5,1.1){$\lambda$}
\end{pspicture}
\end{equation}
\end{enumerate}

%
\subsubsection{Identities arising from tensor products}
%

Bubbles with a given label $(j,\alpha)$ can pass from right to left, or left to
right, through a line coloured by $i$ as long as $j\neq i,i+1$.  If $j=i$ or
$j=i+1$ then a bubble can move through a line subject to the rules
\eqref{eq_noncan1}--\eqref{eq_can_to_noncan2}.  Furthermore, dots on a line can
be exchanged for bubbles in the neighboring regions using \eqref{eq_alphadot1}
and \eqref{eq_alphadot2}.  Dual bubbles corresponding to dual generators can be
slid across lines using the rules \eqref{eq_dark_slide1} and
\eqref{eq_dark_slide2}.

The following Lemma is needed for the definition of the 2-representation $\Gamma$
given in the next section.  In particular, parts (i) and (ii) are used to provide
two equivalent definitions of
\[
 \Gamma \left(\vcenter{\xy   (-6,0)*{};(6,2)*{};
  (-4,0)*{}="t1"; (4,0)*{}="t2";
  "t1";"t2" **\crv{(-4,-6) & (4,-6)}; ?(.15)*\dir{>} ?(.9)*\dir{>}
   ?(.5)*\dir{}+(0,-2)*{\scriptstyle{i}};
    (8,-5)*{ \lambda};
    \endxy }\right),
\qquad  \qquad \Gamma \left(\vcenter{\xy(-6,0)*{};(6,2)*{};
  (-4,0)*{}="t1"; (4,0)*{}="t2";
  "t1";"t2" **\crv{(-4,-6) & (4,-6)}; ?(.1)*\dir{<} ?(.85)*\dir{<}
   ?(.5)*\dir{}+(0,-2)*{\scriptstyle{i}};
    (8,-5)*{ \lambda};
    \endxy}\right) ,
\]
and (iii) and (iv) are used to show that
\[
 \Gamma \left(\vcenter{\xy   (-6,0)*{};(6,2)*{};
  (-4,0)*{}="t1"; (4,0)*{}="t2";
  "t1";"t2" **\crv{(-4,-6) & (4,-6)}; ?(.15)*\dir{}+(0,0)*{\bullet} ?(.9)*\dir{>}
   ?(.5)*\dir{}+(0,-2)*{\scriptstyle{i}};
    (8,-5)*{ \lambda};
    \endxy }\right)
\quad = \quad
 \Gamma \left(\vcenter{\xy(-6,0)*{};(6,2)*{};
  (-4,0)*{}="t1"; (4,0)*{}="t2";
  "t1";"t2" **\crv{(-4,-6) & (4,-6)}; ?(.15)*\dir{>} ?(.9)*\dir{}+(0,0)*{\bullet}
   ?(.5)*\dir{}+(0,-2)*{\scriptstyle{i}};
    (8,-5)*{ \lambda};
    \endxy}\right),
\qquad  \Gamma \left(\vcenter{\xy(-6,0)*{};(6,2)*{};
  (-4,0)*{}="t1"; (4,0)*{}="t2";
  "t1";"t2" **\crv{(-4,-6) & (4,-6)}; ?(.15)*\dir{}+(0,0)*{\bullet} ?(.85)*\dir{<}
   ?(.5)*\dir{}+(0,-2)*{\scriptstyle{i}};
    (8,-5)*{ \lambda};
    \endxy}\right)
\quad = \quad
 \Gamma \left(\vcenter{\xy(-6,0)*{};(6,2)*{};
  (-4,0)*{}="t1"; (4,0)*{}="t2";
  "t1";"t2" **\crv{(-4,-6) & (4,-6)}; ?(.1)*\dir{<} ?(.9)*\dir{}+(0,0)*{\bullet}
   ?(.5)*\dir{}+(0,-2)*{\scriptstyle{i}};
    (8,-5)*{ \lambda};
    \endxy}\right).
\]
\begin{lem}\label{lem_relations}
  The following identities hold:
\begin{enumerate}[i)]
\item Equivalent definitions of cups labelled $i$:
 in the ring $H_{\ukp} \otimes_{H_{\ukep}} H_{\ukp}$ we have
\begin{eqnarray}
  \xsum{f=0}{\alpha}(-1)^{\alpha-f} \; \begin{pspicture}[.5](4.5,1.5)
  \rput(.5,0){\Flinedot{i}{f}}
  \rput(1.8,0){\Eline{i}}
  \rput(3,1){\chern{i,\alpha-f}}
  \rput(4.2,1.5){$\lambda$}
\end{pspicture}
&= & \xsum{g=0}{\alpha}(-1)^{\alpha-g} \;
 \begin{pspicture}[.5](4.5,1.5)
  \rput(2,0){\Fline{i}}
  \rput(3,0){\Elinedot{i}{g}}
  \rput(.8,1){\chern{i,\alpha-g}}
  \rput(4.2,1.5){$\lambda$}
\end{pspicture} \nn
\\ \label{eq_lem_cup_du} \\
    \xsum{f=0}{\alpha}(-1)^{\alpha-f} \; \xi_i^{f} \otimes x(\uk)_{i,\alpha-f}
    \quad
  &= &
  \xsum{g=0}{\alpha}(-1)^{\alpha-g} \; x(\uk)_{i,\alpha-g} \otimes
  \xi_i^{g} \nn
\end{eqnarray}
for all $\alpha\in \N$.
\item
In the ring $H_{{}_{-i}\ukp} \otimes_{H_{\ukem}} H_{{}_{-i}\ukp}$  we have
\begin{eqnarray}
  \xsum{f=0}{\alpha}(-1)^{\alpha-f} \; \begin{pspicture}[.5](4.5,1.5)
  \rput(.5,0){\Elinedot{i}{f}}
  \rput(1.7,0){\Fline{i}}
  \rput(3,1){\chern{i+1,\alpha-f}}
  \rput(4.2,1.5){$\lambda$}
\end{pspicture}
&= & \xsum{g=0}{\alpha}(-1)^{\alpha-g} \;
 \begin{pspicture}[.5](4.9,1.5)
  \rput(2.4,0){\Eline{i}}
  \rput(3.6,0){\Flinedot{i}{g}}
  \rput(1,1){\chern{i+1,\alpha-g}}
  \rput(4.6,1.5){$\lambda$}
\end{pspicture} \nn
\\ \label{eq_lem_cup_ud} \\
    \xsum{f=0}{\alpha}(-1)^{\alpha-f} \; \xi_i^f \otimes x(\uk)_{i+1,\alpha-f}
  &=&
  \xsum{g=0}{\alpha}(-1)^{\alpha-g} \; x(\uk)_{i+1,\alpha-g} \otimes \xi_i^{g} \nn
\end{eqnarray}
for all $\alpha\in \N$.
\item Dot slide formulas for cups:
in the ring $H_{\ukp} \otimes_{H_{\ukep}} H_{\ukp}$ we have
\begin{eqnarray}
    \xsum{f=0}{k_i-k_{i-1}}(-1)^{k_i-k_{i-1}-f} \; \xi_i^{f+1} \otimes x(\uk)_{i+1,k_i-k_{i-1}-f}
  = \hspace{2in}\\
  \hspace{1in}
\xsum{g=0}{k_i-k_{i-1}}(-1)^{k_i-k_{i-1}-g} \; \xi_i^{g} \otimes
x(\uk)_{i+1,k_i-k_{i-1}-g} \cdot \xi_i .\nn
\end{eqnarray}
\item Dot slide formulas for cups:
In the ring $H_{{}_{-i}\ukp} \otimes_{H_{\ukem}} H_{{}_{-i}\ukp}$ we have
\begin{eqnarray}
    \xsum{f=0}{k_{i+1}-k_i}(-1)^{k_{i+1}-k_i-f} \; \xi_i^{f+1} \otimes x(\uk)_{i+1,k_{i+1}-k_i-f}
  = \hspace{2in}\\
  \hspace{1in}
\xsum{g=0}{k_{i+1}-k_i}(-1)^{k_{i+1}-k_i-g} \; \xi_i^{g} \otimes
x(\uk)_{i+1,k_{i+1}-k_i-g} \cdot \xi_i .\nn
\end{eqnarray}
\end{enumerate}
\end{lem}

\begin{proof}
Part i) follows from the chain of equalities:
\begin{eqnarray}
  \sum_{f=0}^{\alpha}(-1)^{\alpha-f} \;
  \begin{pspicture}[.5](4.5,1.5)
  \rput(.5,0){\Flinedot{i}{f}}
  \rput(1.8,0){\Eline{i}}
  \rput(3,1){\chern{i,\alpha-f}}
  \rput(4.2,1.5){$\lambda$}
 \end{pspicture}
  & \refequal{\eqref{eq_alphadot2}}&
  \sum_{f=0}^{\alpha} \sum_{g=0}^{f}(-1)^{\alpha
  } \;
 \begin{pspicture}[.5](6.5,1.5)
 \rput(.8,1){\chern{i,g}}
  \rput(1.8,0){\Fline{i}}
  \rput(3,1){\dchern{i, f-g}}
  \rput(4,0){\Eline{i}}
  \rput(5.1,1){\chern{i,\alpha-f}}
  \rput(6.2,1.5){$\lambda$}
 \end{pspicture} \nn
\end{eqnarray}
If we re-index by letting $f'=\alpha-f$ and switch the order of summation we have
\begin{eqnarray}
  =
     \sum_{g=0}^{\alpha}\sum_{f'=0}^{\alpha-g}(-1)^{\alpha} \;
 \begin{pspicture}[.5](7,1.5)
 \rput(.5,1){\chern{i,g}}
  \rput(1.4,0){\Fline{i}}
  \rput(3,1){\dchern{i, (\alpha-g)-f'}}
  \rput(4.6,0){\Eline{i}}
  \rput(5.5,1){\chern{i,f'}}
  \rput(6.4,1.5){$\lambda$}
 \end{pspicture}
 \refequal{\eqref{eq_alphadot2}}
     \sum_{g=0}^{\alpha}(-1)^{g} \;
 \begin{pspicture}[.5](4.5,1.5)
  \rput(1.5,0){\Fline{i}}
  \rput(2.5,0){\Elinedot{i}{\quad\alpha-g}}
  \rput(.5,1){\chern{i,g}}
  \rput(4.2,1.5){$\lambda$}
\end{pspicture} \nn
\end{eqnarray}
which after re-indexing completes the proof of Part i).  Part ii) is proven
similarly.  Part iii) follows from Part i) by letting $f'=f+1$ and adding and
subtracting the term $(-1)^{\alpha+1} 1 \otimes x(\uk)_{i, \alpha+1}$, so that
\begin{eqnarray}
\sum_{f=0}^{\alpha}(-1)^{\alpha-f} \; \xi_i^{f+1} \otimes x(\uk)_{i,\alpha-f}
\hspace{3in} \nn \\
=
  \sum_{f'=0}^{\alpha+1}(-1)^{\alpha+1-f'} \; \xi_i^{f'} \otimes
  x(\uk)_{i,\alpha+1-f'} - (-1)^{\alpha+1} 1 \otimes x(\uk)_{i, \alpha+1}
  \nn \\
 \refequal{\text{Part i)}}
 \sum_{g=0}^{\alpha+1}(-1)^{\alpha+1-g} x(\uk)_{i,\alpha+1-g} \otimes
 \xi_i^{g} - (-1)^{\alpha+1} 1 \otimes x(\uk)_{i, \alpha+1}. \nn \\
\end{eqnarray}
Pulling off the $g =0$ term and re-indexing $g'=g-1$ we have
\begin{eqnarray}
\sum_{g'=0}^{\alpha}(-1)^{\alpha-g'} x(\uk)_{i,\alpha-g'} \otimes
 \xi_i^{g'+1} + (-1)^{\alpha+1} \big(x(\uk)_{i, \alpha+1} \otimes 1  -
 1 \otimes x(\uk)_{i, \alpha+1} \big). \nn
\end{eqnarray}
But the term with the summation is equal to
\begin{equation}
 \xsum{g'=0}{\alpha}(-1)^{\alpha-g'} \;
   \xi_i^{g'} \otimes x(\uk)_{i,\alpha-g'} \cdot \xi_i
\end{equation}
by Part i), and the remaining terms $ (-1)^{\alpha+1} \big(x(\uk)_{i, \alpha+1}
\otimes 1  - 1 \otimes x(\uk)_{i, \alpha+1} \big)$, are zero when
$\alpha=k_{i}-k_{i-1}$ since $x(\uk)_{i,\alpha+1}$ is zero for $\alpha \geq
k_i-k_{i+1}$, see \eqref{eq_zero_variables}.  Part iv) is proven similarly using
Part ii).
\end{proof}

\begin{cor} \label{cor_bimodule_maps}
The assignments (see Definition~\ref{def_biadjoint})
\[
\begin{array}{ccl}
    \Gamma\left( \text{$\Ucupr_{i,\lambda}$} \right) \maps H_{\uk}  & \longrightarrow &
    \left(H_{\ukp} \otimes_{H_{\ukep}} H_{\ukp}\right) \{
    1+k_{i-1}-k_{i+1}\}
    \\
     1 &\mapsto&
    \xsum{f=0}{k_i-k_{i-1}}(-1)^{k_i-{k_{i-1}}-f} \; \xi_i^{f}
    \otimes x(\uk)_{i,k_i-k_{i-1}-f}
 \end{array}
 \]
 and
\[
 \begin{array}{ccl}
   \Gamma\left( \text{$\Ucupl_{i,\lambda}$} \right) \maps H_{\uk} \ & \longrightarrow &
    \left(H_{{}_{-i}\ukp} \otimes_{H_{\ukem}} H_{{}_{-i}\ukp}\right)\{
    1+k_{i-1}-k_{i+1}\}
    \\
     1 &\mapsto&
     \xsum{g=0}{k_{i+1}-k_i}(-1)^{k_{i+1}-k_i-g} \; \xi_i^g
     \otimes x(\uk)_{i+1,k_{i+1}-k_i-g}
 \end{array}
 \]
define morphisms of graded bimodules of degree $1+\lambda_i =
1-k_{i-1}+2k_i-k_{i+1}$ and $1-\lambda_i = 1+k_{i-1}-2k_i+k_{i+1}$, respectively.
\end{cor}

\begin{proof}
For the first claim it suffices to check that the left action of each generator
of $H_{\uk}$ on $\Gamma\left( \text{$\Ucupr_{i,\lambda}$} \right) (1) \in
\left(H_{\ukp} \otimes_{H_{\ukep}} H_{\ukp}\right)$ is equal to right action of
this generator. The Corollary follows since bubbles can slide across lines by
\eqref{eq_noncan1}--\eqref{eq_can_to_noncan2} at the cost of introducing powers
of $\xi_i$ on one of the tensor factors; by Lemma~\ref{lem_relations} (iii) and
(iv), factors of $\xi_i$ can be slid across tensor factors in the above sums. The
second claim is proven similarly and the degrees of these bimodule maps are
easily computed.
\end{proof}

%
\subsection{The 2-category $\cat{Flag}_{N}$}
%

$\cat{Bim}$ is a 2-category whose objects are graded rings, whose 1-morphisms are
graded bimodules, and 2-morphisms are degree-preserving bimodule homomorphisms.
Idempotent bimodule homomorphisms split in $\cat{Bim}$, so that any
2-representation $\Psi \maps \Ucat \to \cat{Bim}$ extends uniquely (up to
isomorphisms) to a 2-representation $\UcatD \to \cat{Bim}$, see
Section~\ref{subsec_Karoubi}.

Graded 2-homs between graded bimodules $M_1$ and $M_2$ are given by
\begin{equation}
  \HOM_{\cat{Bim}}(M_1,M_2) := \bigoplus_{t\in\Z}\Hom_{\cat{Bim}}(M_1\{t\},M_2).
\end{equation}
Let $\cat{Bim}^*$ be the 2-category with the same objects and 1-morphisms as
$\cat{Bim}$ and 2-morphisms given by
\begin{equation}
  \cat{Bim}^*(M_1,M_2): = \HOM_{\cat{Bim}}(M_1,M_2).
\end{equation}

We now define a sub 2-category $\cat{Flag}_{N}$ of the 2-category $\cat{Bim}$ for
each integer $N\in\Z_+$.

\begin{defn}
The additive $\Bbbk$-linear 2-category $\cat{Flag}_{N}$ is the idempotent
completion inside of $\cat{Bim}$ of the 2-category consisting of
\begin{itemize}
  \item objects: the graded rings $H_{\uk}$ for all $\uk = (k_0, k_1,k_2, \dots, k_n)$ with
$0 \leq k_1 \leq k_2 \dots \leq k_n = N$.

  \item morphisms: generated by the graded ($H_{\uk}$,$H_{\uk}$)-bimodule $H_{\uk}$, the
graded ($H_{\ukep}$,$H_{\uk}$)-bimodules $H_{\ukp}$ and the graded
($H_{\uk}$,$H_{\ukep}$)-bimodule $H_{\ukp}$ for all $i \in I$, together with
their shifts $H_{\uk}\{t\}$, $H_{\ukp}\{t\}$, and $H_{\ukp}\{t\}$ for $t \in \Z$.
The bimodules $H_{\uk}=H_{\uk}\{0\}$ are the identity 1-morphisms.  Thus, a
morphism from $H_{\uk}$ to $H_{{}_{\ii}\uk}$ is a finite direct sum of graded
$(H_{{}_{\jj}\uk},H_{\uk})$-bimodules of the form
\[
   H_{\uk^{\jj}} \{t\}:=
    H_{{}_{s_2s_3 \cdots s_m}\uk^{s_1}}
    \otimes_{H_{{}_{s_2s_3 \cdots s_m}\uk}}
\cdots \otimes_{H_{{}_{s_{m-1}s_m}\uk}} H_{{}_{s_m}\uk^{s_{m-1}}}
\otimes_{H_{{}_{s_m}\uk}} H_{\uk^{{s_m}}}\{t\}
\]
 for signed sequences $\jj=s_1s_2\dots s_m$ with $\ii_X=\jj_X \in X$.
  \item 2-morphisms:  degree-preserving bimodule maps.
\end{itemize}
\end{defn}
There is a graded additive subcategory $\Gr$ of $\cat{Bim}^*$ with the same
objects and 1-morphisms as $\cat{Flag}_{N}$, but with
\begin{equation}
  \Gr(M_1,M_2) := \bigoplus_{t\in \Z} \cat{Flag}_{N}(M_1\{t\},M_2).
\end{equation}
In Section~\ref{sec_rep} we show that $\Gr$ provides a 2-representation of
$\Ucats^*$; using the isomorphism $\Sigma\maps \Ucatq \to \Ucats^*$ and
restricting to degree zero 2-morphisms, the subcategory $\cat{Flag}_{N}$ provides
a 2-representation of $\Ucat$.

%
\section{Representing $\Ucats^*$ on the flag 2-category}\label{sec_rep}
%

In this section we define for each positive integer $N$ a 2-representation
$\Gamma_N\maps \Ucats^* \to \Gr$.  The 2-functor $\Gamma_N$ is degree preserving
so that it restricts to a weak 2-functor $\Gamma\maps \Ucat \to \cat{Flag}_{N}$.
We will sometimes shorten $\Gamma_N$ to $\Gamma$ for simplicity.

%
\subsection{Defining the 2-functor $\Gamma_N$} \label{subsec_define_gamma}
%
On objects the 2-representation $\Gamma_N\maps \Ucats^* \to \Gr$ sends
$\lambda=(\lambda_1, \lambda_2, \ldots, \lambda_{n-1})$ to the ring $H_{\uk}$
when $\lambda=\lambda(\uk)$, \ie, when $\lambda_{\alpha} =
-k_{\alpha+1}+2k_{\alpha}-k_{\alpha-1}$. 

\begin{eqnarray}
 \Gamma_N\maps \Ucats & \to & \Gr \nn \\
  \lambda & \mapsto &
  \left\{\begin{array}{ccl}
    H_{\uk} & & \text{if $\lambda=\lambda(\uk)$} \\
    0  & & \text{otherwise.}
  \end{array} \right. \nn
\end{eqnarray}
Morphisms of $\Ucats$ get mapped by $\Gamma_N$ to graded bimodules:
\begin{eqnarray}
 \Gamma_N\maps \Ucats & \to & \Gr \nn \\
  \onel\{t\} & \mapsto &
  \left\{\begin{array}{ccl}
    H_{\uk}\{t\} & & \text{if $\lambda_{\alpha}
= -k_{\alpha+1}+2k_{\alpha}-k_{\alpha-1}$} \nn \\
    0  & & \text{otherwise.}
  \end{array} \right. \\
  \cal{E}_{+i}\onel\{t\} & \mapsto &
  \left\{\begin{array}{ccl}
    H_{\ukp}\{t+1+k_{i-1}+k_{i}-k_{i+1}\} & & \text{if $\lambda_{\alpha}
= -k_{\alpha+1}+2k_{\alpha}-k_{\alpha-1}$} \\
    0  & & \text{otherwise.}
  \end{array} \right. \nn\\
  \cal{E}_{-i}\onel\{t\} & \mapsto &
  \left\{\begin{array}{ccl}
    H_{\ukm}\{t+1-k_{i}\} & &\text{if $\lambda_{\alpha}
= -k_{\alpha+1}+2k_{\alpha}-k_{\alpha-1}$} \\
    0  & & \text{otherwise.}
  \end{array} \right. \nn
\end{eqnarray}

Here $H_{\ukm}\{t+1-k_{i}\}$ is the bimodule $H_{\ukm}$ with the grading shifted
by $t+1-k_{i}$ so that
$$(H_{\ukp}\{t+1-k_{i}\})_j = (H_{\ukp})_{j+t+1-k_{i}}.$$ More generally,
the 1-morphism
\begin{equation} \label{eq_Es}
  \cal{E}_{\ii}\onel\{t\}\quad  = \quad
  \cal{E}_{s_1}\mathbf{1}_{\lambda+(s_2)_X+\dots+(s_{m-1})_X+(s_{m})_X }
  \circ \cdots\circ \cal{E}_{s_{m-1}}\mathbf{1}_{\lambda+(s_m)_X} \circ
  \cal{E}_{s_m}\onel\{t\}
\end{equation}
is mapped by $\Gamma_N$ to the graded bimodule $H_{e_{\ii \uk}}$ with grading
shift $\{t+t'\}$, where $t'$ is the sum of the grading shifts for each terms of
the composition in \eqref{eq_Es}.  Formal direct sums of morphisms of the above
form are mapped to direct sums of the corresponding bimodules.

%
\subsubsection{Biadjointness}
%

\begin{defn} \label{def_biadjoint}
The 2-morphisms generating biadjointness in $\Ucats^*$ are mapped by $\Gamma$ to
the following bimodule maps.
\begin{eqnarray}
    \Gamma\left(\;\xy
    (0,-3)*{\bbpef{i}};
    (8,-5)*{ \lambda};
    \endxy \; \right)
    & : &
 \left\{
  \begin{array}{ccl}
    H_{\uk} \ & \longrightarrow &
    \left(H_{\ukp} \otimes_{H_{\ukep}} H_{\ukp}\right)\{1+k_{i-1}-k_{i+1}\}
    \\
 \begin{pspicture}[.5](1,1.7)
  \rput(1,1.3){$\lambda$}
\end{pspicture}
 &\mapsto&
  \xsum{f=0}{k_i-k_{i-1}}(-1)^{k_i-{k_{i-1}}-f} \;
  \begin{pspicture}[.5](5.3,1.9)
  \rput(.5,0){\Flinedot{i}{f}}
  \rput(1.5,0){\Eline{i}}
  \rput(3.4,1){\chern{i,k_i-k_{i-1}-f}}
  \rput(5.1,1.5){$\lambda$}
\end{pspicture}
    \\ \\
     1 &\mapsto&
     \xsum{f=0}{k_i-k_{i-1}}(-1)^{k_i-{k_{i-1}}-f} \; \xi_i^{f} \otimes x(\uk)_{i,k_i-k_{i-1}-f}
 \end{array}
 \right.
     \label{def_eq_FE_G}\\
    \Gamma\left(\;\xy
    (0,-3)*{\bbpfe{i}};
    (8,-5)*{ \lambda};
    \endxy\;\right)
     & : &
     \left\{
  \begin{array}{ccl}
    H_{\uk} \ & \longrightarrow &
    \left(H_{{}_{-i}\ukp} \otimes_{H_{\ukem}} H_{{}_{-i}\ukp}\right)\{1+k_{i-1}-k_{i+1}\}
    \\
    \begin{pspicture}[.5](1,1.7)
  \rput(1,1.3){$\lambda$}
\end{pspicture}
 &\mapsto&
  \xsum{f=0}{k_{i+1}-k_i}(-1)^{k_{i+1}-k_i-f} \;
 \begin{pspicture}[.5](5.3,1.9)
  \rput(.5,0){\Elinedot{i}{f}}
  \rput(1.5,0){\Fline{i}}
  \rput(3.4,1){\chern{i+1,k_{i+1}-k_i-f}}
  \rput(5.4,1.5){$\lambda$}
\end{pspicture}
\\ \\
     1 &\mapsto&
    \xsum{f=0}{k_{i+1}-k_i}(-1)^{k_{i+1}-k_i-f} \; \xi_i^f \otimes x(\uk)_{i+1,k_{i+1}-k_i-f}
 \end{array}
 \right.
     \label{def_eq_EF_G}\\
   \Gamma\left(\;\xy
    (0,0)*{\bbcef{i}};
    (8,2)*{ \lambda};
    \endxy\;\right)
& : &
 \left\{
  \begin{array}{ccl}
     \left(H_{\ukp} \otimes_{H_{\ukep}} H_{\ukp}\right)\{1+k_{i-1}-k_{i+1}\}
     & \rightarrow &
     H_{\uk}
    \\
  \begin{pspicture}[.5](4,1.9)
  \rput(2,0){\Elinedot{i}{\alpha_2}}
  \rput(.5,0){\Flinedot{i}{\alpha_1}}
  \rput(3,1.5){$\lambda$}
\end{pspicture}
\mapsto(-1)^{\alpha_1+\alpha_2+1+k_i-k_{i+1}} & \times &
 \begin{pspicture}[.5](4,1.9)
  \rput(2.1,1){\dchern{i+1,\alpha_1+\alpha_2+1+k_i-k_{i+1}}}
  \rput(3.5,1.8){$\lambda$}
\end{pspicture}
    \\ \\
\xi_i^{\alpha_1} \otimes \xi_i^{\alpha_2}
\mapsto(-1)^{\alpha_1+\alpha_2+1+k_i-k_{i+1}}
    & \times &
     \overline{x(\uk)_{i+1,\alpha_1+\alpha_2+1+k_i-k_{i+1}}}
 \end{array}
 \right.
   \label{def_FE_cap} \\
  \Gamma\left(\;\xy
    (0,0)*{\bbcfe{i}};
    (8,2)*{ \lambda};
    \endxy \right)
& : &
 \left\{
  \begin{array}{ccl}
     \left(H_{{}_{-i}\ukp} \otimes_{H_{\ukem}} H_{{}_{-i}\ukp}\right)\{1+k_{i-1}-k_{i+1}\}
     & \rightarrow &
     H_{\uk}
    \\
  \begin{pspicture}[.5](4,1.9)
  \rput(2,0){\Flinedot{i}{\alpha_2}}
  \rput(.5,0){\Elinedot{i}{\alpha_1}}
  \rput(3,1.5){$\lambda$}
\end{pspicture} \mapsto (-1)^{\alpha_1+\alpha_2+1+k_{i-1}-k_{i}}
& \times  &
 \begin{pspicture}[.5](4,1.9)
  \rput(2.1,1){\dchern{i,\alpha_1+\alpha_2+1+k_{i-1}-k_{i}}}
  \rput(3.5,1.8){$\lambda$}
\end{pspicture}
    \\ \\
\xi_i^{\alpha_1} \otimes \xi_i^{\alpha_2}\mapsto
(-1)^{\alpha_1+\alpha_2+1+k_{i-1}-k_{i}}
    & \times &
    \overline{x(\uk)_{i,\alpha_1+\alpha_2+1+k_{i-1}-k_{i}}}
 \end{array}
 \right.
\label{def_EF_cap}
\end{eqnarray}
\end{defn}
Corollary~\ref{cor_bimodule_maps} shows that the  cups above are bimodule maps.
It is clear that the caps are bimodule maps. These definitions preserve the
degree of the 2-morphisms of $\Ucats^*$ defined in Section~\ref{def_Ucat}.  By
Lemma~\ref{lem_relations}, the clockwise oriented cap and cup have degree
$1-\lambda_i$ and the counter-clockwise oriented cap and cup have degree
$1+\lambda_i$ so that these assignments are degree preserving.

%
\subsubsection{$R(\nu)$ generators} \label{subsubsec_Rnu}
%

\begin{defn}
The 2-morphisms $\Uup_{i,\lambda}$ and $\Udown_{i,\lambda}$ in $\Ucats^*$ are
mapped by $\Gamma_N$ to the graded bimodule maps:
\begin{eqnarray} \label{eq_gamma_updot}
  \Gamma\left(
 \xy
  (0,0)*{\lineu{i}};
  (0,4)*{\txt\large{$\bullet$}};
 (6,0)*{ \lambda};
 (-8,0)*{ \lambda+i_X};
 (-10,0)*{};(10,0)*{};
 \endxy\right)
\quad &\maps&
 \left\{
\begin{array}{ccc}
  H_{\ukp}\{1+k_{i-1}-k_{i+1}\} & \to  & H_{\ukp}
  \{1+k_{i-1}-k_{i+1}\} \\ \\
  \xi_i^\alpha & \mapsto &  \quad \xi_i^{\alpha+1}
\end{array}
\right. \\
 \Gamma\left(
 \xy
  (0,0)*{\lined{i}};
  (0,4)*{\txt\large{$\bullet$}};
 (-6,0)*{ \lambda};
 (8,0)*{ \lambda+i_X};
 (-10,0)*{};(10,0)*{};
 \endxy\right)
\quad &\maps&
 \left\{
\begin{array}{ccc}
  H_{\ukp}\{1+k_{i-1}-k_{i+1}\} & \to  & H_{\ukp}\{1+k_{i-1}-k_{i+1}\} \\ \\
  \xi_i^{\alpha} & \mapsto &  \quad \xi_i^{\alpha+1}
\end{array}
\right.
\end{eqnarray}
Note that these assignment are degree preserving since these bimodule maps are
degree 2.

The 2-morphisms $\Ucross_{i,j,\lambda}$ and $\Ucrossd_{i,j,\lambda}$ are mapped
by $\Gamma$ to the graded bimodule maps:
\begin{equation} \label{eq_gamma_dcross}
    \Gamma\left(
 \xy
  (0,0)*{\xybox{
    (-4,-4)*{};(4,4)*{} **\crv{(-4,-1) & (4,1)}?(1)*\dir{>} ;
    (4,-4)*{};(-4,4)*{} **\crv{(4,-1) & (-4,1)}?(1)*\dir{>};
    (-5,-3)*{\scs i};
     (5.1,-3)*{\scs j};
     (8,1)*{ \lambda};
     (-7,0)*{};(9,0)*{};
     }};
  \endxy
  \right)
  \quad \maps \quad
H_{_{+j}\uk^{+i}} \otimes_{H_{_{+j}\uk}} H_{\uk^{+j}}
 \to H_{\ukep^{+j}} \otimes_{H_{\ukep}} H_{\ukp}
     \hspace{0.5in}
\end{equation}
\begin{eqnarray}
    \xi_i^{\alpha_1} \otimes \xi_j^{\alpha_2}& \mapsto &
    \left\{
    \begin{array}{ccl}
      \xi_j^{\alpha_2} \otimes \xi_i^{\alpha_1} &  & \text{if $i \cdot j =0$} \\
      \xsum{f=0}{\alpha_1-1}\xi_i^{\alpha_1+\alpha_2-1-f} \otimes \xi_i^{f} -
      \xsum{g=0}{\alpha_2-1}\xi_i^{\alpha_1+\alpha_2-1-g} \otimes \xi_i^{g}&  &
      \text{if $i =j$}\\
    \left( \xi_j^{\alpha_2} \otimes \xi_i^{\alpha_1+1}-\xi_j^{\alpha_2+1} \otimes
    \xi_i^{\alpha_1} \right) \{ -1\}
       &  &  \text{if $ \xy
  (-5,0)*{\circ};
  (5,0)*{\circ}
  **\dir{-}?(.55)*\dir{>};
  (-5,2.2)*{\scs j};
  (5,2.2)*{\scs i};
  \endxy$} \\
    \left( \xi_j^{\alpha_2} \otimes \xi_i^{\alpha_1}\right)\{1\}  &  &
  \text{if $ \xy
  (-5,0)*{\circ};
  (5,0)*{\circ}
  **\dir{-}?(.55)*\dir{>};
  (-5,2.2)*{\scs i};
  (5,2.2)*{\scs j};
  \endxy$}
    \end{array} \right.
     \nn
\end{eqnarray}
\begin{equation} \label{eq_gamma_dcross_down}
    \Gamma\left(
 \xy
  (0,0)*{\xybox{
    (-4,-4)*{};(4,4)*{} **\crv{(-4,-1) & (4,1)}?(0)*\dir{<} ;
    (4,-4)*{};(-4,4)*{} **\crv{(4,-1) & (-4,1)}?(0)*\dir{<};
    (-6,-3)*{\scs i};
     (6,-3)*{\scs j};
     (8,1)*{ \lambda};
     (-7,0)*{};(9,0)*{};
     }};
  \endxy
  \right)
  \quad \maps \quad
H_{_{-j}\uk^{-i}} \otimes_{H_{_{-j}\uk}} H_{\uk^{-j}}
 \to H_{\ukem^{-j}} \otimes_{H_{\ukem}} H_{\uk^{-i}}
     \hspace{0.5in}
\end{equation}
\begin{eqnarray}
    \xi_i^{\alpha_1} \otimes \xi_j^{\alpha_2}& \mapsto &
    \left\{
    \begin{array}{ccl}
      \xi_j^{\alpha_2} \otimes \xi_i^{\alpha_1} &  & \text{if $i \cdot j =0$} \\
      \xsum{f=0}{\alpha_2-1}\xi_i^{\alpha_1+\alpha_2-1-f} \otimes \xi_i^{f} -
      \xsum{g=0}{\alpha_1-1}\xi_i^{\alpha_1+\alpha_2-1-g} \otimes \xi_i^{g}&  &
      \text{if $i =j$}\\
 \left( \xi_j^{\alpha_2} \otimes \xi_i^{\alpha_1}\right)\{-1\}
       &  &  \text{if $ \xy
  (-5,0)*{\circ};
  (5,0)*{\circ}
  **\dir{-}?(.55)*\dir{>};
  (-5,2.2)*{\scs j};
  (5,2.2)*{\scs i};
  \endxy$} \\
  \left( \xi_j^{\alpha_2+1} \otimes \xi_i^{\alpha_1}-\xi_j^{\alpha_2} \otimes
    \xi_i^{\alpha_1+1} \right) \{ 1\}    &  &
  \text{if $ \xy
  (-5,0)*{\circ};
  (5,0)*{\circ}
  **\dir{-}?(.55)*\dir{>};
  (-5,2.2)*{\scs i};
  (5,2.2)*{\scs j};
  \endxy$}
    \end{array} \right.
     \nn
\end{eqnarray}
It is straightforward to see that these assignments define bimodule maps of
degree $- i \cdot j$.  In the case when $i=j$ the bimodule map is just the
divided difference operator acting on $\xi_i$ and $\xi_j$.
\end{defn}

%
\subsection{Checking the relations of $\Ucats$ } \label{subsec_check_relations}
%

In this section we show that the relations on the 2-morphisms of $\Ucats^*$  are
satisfied in $\Gr$, thus establishing that $\Gamma$ is a graded additive
$\Bbbk$-linear 2-functor. From the definitions in the previous section it is
clear that $\Gamma$ preserves the degrees associated to generators. For this
reason, we often simplify our notation in this section by omitting the grading
shifts $\{t\}$ when no confusion is likely to arise.

\begin{prop}
$\Gamma$ preserves the $\mathfrak{\mf{sl}_2}$-relations of $\Ucats^*$.
\end{prop}

\begin{proof}
The proof in \cite{Lau1} that the $\mathfrak{\mf{sl}_2}$-relations of $\Ucats^*$
are preserved by $\Gamma$ generalizes immediately with only minor refinements to
grading shifts and summation indices.  The proof only makes use of
\eqref{eq_graph_def_relations} and
\eqref{eq_graph_easy_slide}--\eqref{eq_dark_slide2}. For example, to prove
biadjointness we must show
\[
  \Gamma\left(\;  \xy   0;/r.18pc/:
    (-8,0)*{}="1";
    (0,0)*{}="2";
    (8,0)*{}="3";
    (-8,-10);"1" **\dir{-};
    "1";"2" **\crv{(-8,8) & (0,8)} ?(0)*\dir{>} ?(1)*\dir{>};
    "2";"3" **\crv{(0,-8) & (8,-8)}?(1)*\dir{>};
    "3"; (8,10) **\dir{-};
    (14,9)*{\lambda};
    (-6,9)*{\lambda+i_X};
    \endxy \;\right)
    \; =
    \;
       \Gamma\left(\;\xy   0;/r.18pc/:
    (-8,0)*{}="1";
    (0,0)*{}="2";
    (8,0)*{}="3";
    (0,-10);(0,10)**\dir{-} ?(.5)*\dir{>};
    (5,8)*{\lambda};
    (-9,8)*{\lambda+i_X};
    \endxy\;\right)
\qquad
    \Gamma\left(\;\xy   0;/r.18pc/:
    (-8,0)*{}="1";
    (0,0)*{}="2";
    (8,0)*{}="3";
    (-8,-10);"1" **\dir{-};
    "1";"2" **\crv{(-8,8) & (0,8)} ?(0)*\dir{<} ?(1)*\dir{<};
    "2";"3" **\crv{(0,-8) & (8,-8)}?(1)*\dir{<};
    "3"; (8,10) **\dir{-};
    (12,9)*{\lambda};
    (-6,9)*{ \lambda-i_X};
    \endxy\;\right)
    \; =
    \;
       \Gamma\left(\;\xy   0;/r.18pc/:
    (-8,0)*{}="1";
    (0,0)*{}="2";
    (8,0)*{}="3";
    (0,-10);(0,10)**\dir{-} ?(.5)*\dir{<};
    (6,8)*{\lambda};
    (-9,8)*{ \lambda-i_X};
    \endxy\;\right)
\]
\[
  \Gamma\left(\;  \xy   0;/r.18pc/:
    (8,0)*{}="1";
    (0,0)*{}="2";
    (-8,0)*{}="3";
    (8,-10);"1" **\dir{-};
    "1";"2" **\crv{(8,8) & (0,8)} ?(0)*\dir{>} ?(1)*\dir{>};
    "2";"3" **\crv{(0,-8) & (-8,-8)}?(1)*\dir{>};
    "3"; (-8,10) **\dir{-};
    (14,-9)*{\lambda};
    (-5,-9)*{\lambda+i_X};
    \endxy \;\right)
    \; =
    \;
       \Gamma\left(\;\xy 0;/r.18pc/:
    (8,0)*{}="1";
    (0,0)*{}="2";
    (-8,0)*{}="3";
    (0,-10);(0,10)**\dir{-} ?(.5)*\dir{>};
    (5,-8)*{\lambda};
    (-9,-8)*{\lambda+i_X};
    \endxy\;\right)
\qquad
    \Gamma\left(\;\xy  0;/r.18pc/:
    (8,0)*{}="1";
    (0,0)*{}="2";
    (-8,0)*{}="3";
    (8,-10);"1" **\dir{-};
    "1";"2" **\crv{(8,8) & (0,8)} ?(0)*\dir{<} ?(1)*\dir{<};
    "2";"3" **\crv{(0,-8) & (-8,-8)}?(1)*\dir{<};
    "3"; (-8,10) **\dir{-};
    (12,-9)*{\lambda};
    (-6,-9)*{ \lambda-i_X};
    \endxy\;\right)
    \; =
    \;
       \Gamma\left(\;\xy  0;/r.18pc/:
    (8,0)*{}="1";
    (0,0)*{}="2";
    (-8,0)*{}="3";
    (0,-10);(0,10)**\dir{-} ?(.5)*\dir{<};
    (6,-8)*{\lambda};
    (-9,-8)*{ \lambda-i_X};
    \endxy\;\right)
\]
for all strings labelled by $i$. We will prove the first equality by computing
the bimodules maps on elements $\xi_{i}^{\alpha} \in H_{\ukp}$.
\begin{eqnarray}
 \text{$ \Gamma\left(\;  \xy   0;/r.18pc/:
    (-8,0)*{}="1";
    (0,0)*{}="2";
    (8,0)*{}="3";
    (-8,-10);"1" **\dir{-};
    "1";"2" **\crv{(-8,8) & (0,8)} ?(0)*\dir{>} ?(1)*\dir{>};
    "2";"3" **\crv{(0,-8) & (8,-8)}?(1)*\dir{>};
    "3"; (8,10) **\dir{-};
    (14,9)*{\lambda};
    (-6,9)*{\lambda+i_X};
    \endxy \;\right)$}
    \maps \xi_i^{\alpha} & \mapsto & \xsum{f=0}{k_i-k_{i-1}}(-1)^{\alpha} \;
  \begin{pspicture}[.5](9,1.9)
  \rput(2,1){\dchern{i, \alpha-(k_i-k_{i-1}-f)}}
  \rput(4.5,0){\Eline{i}}
  \rput(6.5,1){\chern{i,k_i-k_{i-1}-f}}
  \rput(8.6,1.5){$\lambda$}
\end{pspicture} \nn
\end{eqnarray}
But,
\begin{equation}
\xsum{f=0}{k_i-k_{i-1}}(-1)^{\alpha} \;
  \begin{pspicture}[.5](9,1.9)
  \rput(2,1){\dchern{i, \alpha-(k_i-k_{i-1}-f)}}
  \rput(4.5,0){\Eline{i}}
  \rput(6.5,1){\chern{i,k_i-k_{i-1}-f}}
  \rput(8.6,1.5){$\lambda$}
\end{pspicture} \refequal{\eqref{eq_alphadot2}} \xi_i^{\alpha}
\end{equation}
since the sum can be taken to $\alpha$ by removing terms that are equal to zero
when $\alpha < k_i-k_{i-1}$, or by adding terms that are equal to zero if $\alpha
> k_i-k_{i-1}$. The others are proven similarly.
\end{proof}

\begin{lem} \label{lem_other_crossings}
\begin{eqnarray}
\Gamma\left( \xy 0;/r.19pc/:
  (0,0)*{\xybox{
    (4,-4)*{};(-4,4)*{} **\crv{(4,-1) & (-4,1)}?(1)*\dir{>};
    (-4,-4)*{};(4,4)*{} **\crv{(-4,-1) & (4,1)};
     (-4,4);(-4,12) **\dir{-};
     (-12,-4);(-12,12) **\dir{-};
     (4,-4);(4,-12) **\dir{-};(12,4);(12,-12) **\dir{-};
     (22,1)*{\lambda+i_X};
     (-10,0)*{};(10,0)*{};
     (-4,-4)*{};(-12,-4)*{} **\crv{(-4,-10) & (-12,-10)}?(1)*\dir{<}?(0)*\dir{<};
      (4,4)*{};(12,4)*{} **\crv{(4,10) & (12,10)}?(1)*\dir{>}?(0)*\dir{>};
      (-14,11)*{\scs i};(-2,11)*{\scs j};
      (14,-11)*{\scs i};(2,-11)*{\scs j};
     }};
  \endxy \right) &=& \left\{
\begin{array}{ccc}
  H_{\uk^{+j}} \otimes_{H_{\uk}}  H_{\ukp}&  \to &
   H_{{}_{+i+j}\ukp} \otimes_{H_{{}_{+i+j}\uk}} H_{{}_{+i}\uk^{+j}} \\
    \xi_j^{\alpha_1} \otimes \xi_i^{\alpha_2} & \mapsto &
   \xi_i^{\alpha_2} \otimes \xi_j^{\alpha_1}
\end{array}
  \right. \nn
\\ \label{eq_lem_crossr}
\\
\Gamma\left(
 \xy 0;/r.19pc/:
  (0,0)*{\xybox{
    (-4,-4)*{};(4,4)*{} **\crv{(-4,-1) & (4,1)}?(1)*\dir{>};
    (4,-4)*{};(-4,4)*{} **\crv{(4,-1) & (-4,1)};
     (4,4);(4,12) **\dir{-};
     (12,-4);(12,12) **\dir{-};
     (-4,-4);(-4,-12) **\dir{-};(-12,4);(-12,-12) **\dir{-};
     (-22,1)*{\lambda+i_X};
     (10,0)*{};(-10,0)*{};
     (4,-4)*{};(12,-4)*{} **\crv{(4,-10) & (12,-10)}?(1)*\dir{<}?(0)*\dir{<};
      (-4,4)*{};(-12,4)*{} **\crv{(-4,10) & (-12,10)}?(1)*\dir{>}?(0)*\dir{>};
      (14,11)*{\scs j};(2,11)*{\scs i};
      (-14,-11)*{\scs j};(-2,-11)*{\scs i};
     }};
  \endxy
\right) &=& \left\{
\begin{array}{ccc}
   H_{_{+j+i}\ukep}\otimes_{H_{_{+j+i}\uk}} H_{_{+j}\ukp}&  \to &
  H_{\ukp} \otimes_{H_{\uk}} H_{\uk^{+j}}  \\
    \xi_j^{\alpha_1} \otimes \xi_i^{\alpha_2} & \mapsto &
   \xi_i^{\alpha_2} \otimes \xi_j^{\alpha_1}
\end{array}
  \right. \nn \\ \label{eq_lem_crossl}
\end{eqnarray}
These bimodule maps have degree zero for all $i,j \in I$ and all weights
$\lambda$.
\end{lem}

\begin{proof}
We compute the bimodule maps directly using the definitions in the previous
section.  The case when $i=j$ appears in \cite{Lau1} so we will omit this case
here. The map in \eqref{eq_lem_crossr}, using \eqref{eq_lem_cup_du} for the cup,
is given by
\begin{equation}
\Gamma\left( \xy 0;/r.19pc/:
  (0,0)*{\xybox{
    (4,-4)*{};(-4,4)*{} **\crv{(4,-1) & (-4,1)}?(1)*\dir{>};
    (-4,-4)*{};(4,4)*{} **\crv{(-4,-1) & (4,1)};
     (-4,4);(-4,12) **\dir{-};
     (-12,-4);(-12,12) **\dir{-};
     (4,-4);(4,-12) **\dir{-};(12,4);(12,-12) **\dir{-};
     (22,1)*{\lambda+i_X};
     (-10,0)*{};(10,0)*{};
     (-4,-4)*{};(-12,-4)*{} **\crv{(-4,-10) & (-12,-10)}?(1)*\dir{<}?(0)*\dir{<};
      (4,4)*{};(12,4)*{} **\crv{(4,10) & (12,10)}?(1)*\dir{>}?(0)*\dir{>};
      (-14,11)*{\scs i};(-2,11)*{\scs j};
      (14,-11)*{\scs i};(2,-11)*{\scs j};
     }};
  \endxy \right) \maps \xi_j^{\alpha_1} \otimes \xi_i^{\alpha_2}\;  \mapsto \hspace{2in}
 \nn
\end{equation}
\begin{eqnarray} \label{lem_eq_ieqj}
 \sum_{f=0}^{k_i-k_{i-1}} (-1)^{\alpha_2}
\begin{pspicture}[.5](10,1.5)
  \rput(1.5,1){\chern{i,k_i-k_{i-1}-f}}
  \rput(3.3,0){\Fline{i}}
  \rput(4,0){\Elinedot{j}{\alpha_1}}
  \rput(6.8,1){\dchern{ i,\alpha_2-(k_i-k_{i+1}-f)}}
  \rput(9.7,1.5){$\lambda+i_X$}
\end{pspicture}
\end{eqnarray}
if $i \cdot j =0$ or $ \xy
  (-5,0)*{\circ};
  (5,0)*{\circ}
  **\dir{-}?(.55)*\dir{>};
  (-5,2.2)*{\scs i};
  (5,2.2)*{\scs j};
  \endxy$, and
\begin{eqnarray}
\sum_{f=0}^{k_i-k_{i-1}-1} (-1)^{\alpha_2}
\begin{pspicture}[.5](12,1.5)
  \rput(1.7,1){\chern{i,k_i-k_{i-1}-1-f}}
  \rput(3.7,0){\Fline{i}}
  \rput(4.4,0){\Elinedot{j}{\alpha_1}}
  \rput(7.3,1){\dchern{ i,\alpha_2-(k_i-k_{i+1}-1-f)}}
  \rput(10.5,1.5){$\lambda+i_X$}
\end{pspicture}  \nn \\
-\; \sum_{f=0}^{k_i-k_{i-1}-1} (-1)^{\alpha_2-1}
\begin{pspicture}[.5](11,1.5)
  \rput(1.7,1){\chern{i,k_i-k_{i-1}-1-f}}
  \rput(3.7,0){\Fline{i}}
  \rput(4.4,0){\Elinedot{j}{\quad \;\;\alpha_1+1}}
  \rput(7.8,1){\dchern{ i,\alpha_2-(k_i-k_{i+1}-f)}}
  \rput(10.7,1.5){$\lambda+i_X$}
\end{pspicture} \nn \\ \label{lem_eq_ij_minus_one}
\end{eqnarray}
if $ \xy
  (-5,0)*{\circ};
  (5,0)*{\circ}
  **\dir{-}?(.55)*\dir{>};
  (-5,2.2)*{\scs j};
  (5,2.2)*{\scs i};
  \endxy$.
Careful calculation, keeping in mind the weights of each region, will show that
these maps have degree zero.  In both cases, the dual generators can slide across
the line labelled $j$ via \eqref{eq_dark_slide2}.  After changing indices,
equations \eqref{lem_eq_ieqj} and \eqref{lem_eq_ij_minus_one} both become
\begin{eqnarray}
 \sum_{f'=0}^{({}_{+j}k)_i-({}_{+j}k)_{i-1}} (-1)^{\alpha_2}
\begin{pspicture}[.5](7,1.5)
  \rput(1,1){\chern{i,f'}}
  \rput(2,0){\Fline{i}}
  \rput(5,0){\Elinedot{j}{\alpha_1}}
  \rput(3.5,1){\dchern{ i,\alpha_2-f'}}
  \rput(7,1.5){$\lambda+i'$}
\end{pspicture}
\end{eqnarray}
where $({}_{+j}k)_i-({}_{+j}k)_{i-1}$ are the $(i-1)$-st and $i$-th terms in
${}_{+j}\uk$. By adding or removing terms that are equal to zero, depending on
whether $\alpha_2$ is greater than or less than $({}_{+j}k)_i-({}_{+j}k)_{i-1}$,
the above summation can be taken to $\alpha_2$, so that \eqref{eq_alphadot2}
completes the proof.
\end{proof}

\begin{prop} The equality
\begin{equation}
\Gamma\left(\xy 0;/r.19pc/:
  (0,0)*{\xybox{
    (-4,-4)*{};(4,4)*{} **\crv{(-4,-1) & (4,1)}?(1)*\dir{>};
    (4,-4)*{};(-4,4)*{} **\crv{(4,-1) & (-4,1)};
     (4,4)*{};(-18,4)*{} **\crv{(4,16) & (-18,16)} ?(1)*\dir{>};
     (-4,-4)*{};(18,-4)*{} **\crv{(-4,-16) & (18,-16)} ?(1)*\dir{<}?(0)*\dir{<};
     (18,-4);(18,12) **\dir{-};(12,-4);(12,12) **\dir{-};
     (-18,4);(-18,-12) **\dir{-};(-12,4);(-12,-12) **\dir{-};
     (8,1)*{ \lambda};
     (-10,0)*{};(10,0)*{};
      (4,-4)*{};(12,-4)*{} **\crv{(4,-10) & (12,-10)}?(1)*\dir{<}?(0)*\dir{<};
      (-4,4)*{};(-12,4)*{} **\crv{(-4,10) & (-12,10)}?(1)*\dir{>}?(0)*\dir{>};
      (20,11)*{\scs i};(10,11)*{\scs j};
      (-20,-11)*{\scs i};(-10,-11)*{\scs j};
     }};
  \endxy \right)
\quad =     \Gamma\left(
 \xy
  (0,0)*{\xybox{
    (-4,-4)*{};(4,4)*{} **\crv{(-4,-1) & (4,1)}?(0)*\dir{<} ;
    (4,-4)*{};(-4,4)*{} **\crv{(4,-1) & (-4,1)}?(0)*\dir{<};
    (-6,-3)*{\scs i};
     (6,-3)*{\scs j};
     (-8,1)*{ \lambda};
     (-7,0)*{};(9,0)*{};
     }};
  \endxy
  \right) \quad =\quad \Gamma\left(
 \xy 0;/r.19pc/:
  (0,0)*{\xybox{
    (4,-4)*{};(-4,4)*{} **\crv{(4,-1) & (-4,1)}?(1)*\dir{>};
    (-4,-4)*{};(4,4)*{} **\crv{(-4,-1) & (4,1)};
     (-4,4)*{};(18,4)*{} **\crv{(-4,16) & (18,16)} ?(1)*\dir{>};
     (4,-4)*{};(-18,-4)*{} **\crv{(4,-16) & (-18,-16)} ?(1)*\dir{<}?(0)*\dir{<};
     (-18,-4);(-18,12) **\dir{-};(-12,-4);(-12,12) **\dir{-};
     (18,4);(18,-12) **\dir{-};(12,4);(12,-12) **\dir{-};
     (8,1)*{ \lambda};
     (-10,0)*{};(10,0)*{};
     (-4,-4)*{};(-12,-4)*{} **\crv{(-4,-10) & (-12,-10)}?(1)*\dir{<}?(0)*\dir{<};
      (4,4)*{};(12,4)*{} **\crv{(4,10) & (12,10)}?(1)*\dir{>}?(0)*\dir{>};
      (-20,11)*{\scs j};(-10,11)*{\scs i};
      (20,-11)*{\scs j};(10,-11)*{\scs i};
     }};
  \endxy \right)
\end{equation}
of graded bimodule maps holds in $\Gr$ for all weights $\lambda$.
\end{prop}

\begin{proof}
By direct calculation, using Lemma~\ref{lem_other_crossings}, we have that both
the far left and right bimodule maps are given by
\begin{eqnarray}
    \xi_i^{\alpha_1} \otimes \xi_j^{\alpha_2}& \mapsto &
    \left\{
    \begin{array}{ccl}
      \xi_j^{\alpha_2} \otimes \xi_i^{\alpha_1} &  & \text{if $i \cdot j =0$} \\
      \xsum{f=0}{\alpha_2-1}\xi_i^{\alpha_1+\alpha_2-1-f} \otimes \xi_i^{f} -
      \xsum{g=0}{\alpha_1-1}\xi_i^{\alpha_1+\alpha_2-1-g} \otimes \xi_i^{g}&  &
      \text{if $i =j$}\\
 \left( \xi_j^{\alpha_2} \otimes \xi_i^{\alpha_1}\right)\{-1\}
       &  &  \text{if $ \xy
  (-5,0)*{\circ};
  (5,0)*{\circ}
  **\dir{-}?(.55)*\dir{>};
  (-5,2.2)*{\scs j};
  (5,2.2)*{\scs i};
  \endxy$} \\
  \left( \xi_j^{\alpha_2+1} \otimes \xi_i^{\alpha_1}-\xi_j^{\alpha_2} \otimes
    \xi_i^{\alpha_1+1} \right) \{ 1\}    &  &
  \text{if $ \xy
  (-5,0)*{\circ};
  (5,0)*{\circ}
  **\dir{-}?(.55)*\dir{>};
  (-5,2.2)*{\scs i};
  (5,2.2)*{\scs j};
  \endxy$}
    \end{array} \right.
     \nn
\end{eqnarray}
which agrees with $\Gamma\left(\Ucrossd_{i,j,\lambda+i_X+j_X} \right)$.
\end{proof}

\begin{prop} The equalities
\begin{equation}
   \Gamma\left(\vcenter{   \xy 0;/r.18pc/:
    (-4,-4)*{};(4,4)*{} **\crv{(-4,-1) & (4,1)}?(1)*\dir{>};
    (4,-4)*{};(-4,4)*{} **\crv{(4,-1) & (-4,1)}?(1)*\dir{<};?(0)*\dir{<};
    (-4,4)*{};(4,12)*{} **\crv{(-4,7) & (4,9)};
    (4,4)*{};(-4,12)*{} **\crv{(4,7) & (-4,9)}?(1)*\dir{>};
  (8,8)*{\lambda};(-6,-3)*{\scs i};
     (6,-3)*{\scs j};
 \endxy} \right)
 \;\;= \;\;
   \Gamma\left(
\xy 0;/r.18pc/:
  (3,9);(3,-9) **\dir{-}?(.55)*\dir{>}+(2.3,0)*{};
  (-3,9);(-3,-9) **\dir{-}?(.5)*\dir{<}+(2.3,0)*{};
  (8,2)*{\lambda};(-5,-6)*{\scs i};     (5.1,-6)*{\scs j};
 \endxy \right) , \quad
 \qquad
   \Gamma\left(
    \vcenter{\xy 0;/r.18pc/:
    (-4,-4)*{};(4,4)*{} **\crv{(-4,-1) & (4,1)}?(1)*\dir{<};?(0)*\dir{<};
    (4,-4)*{};(-4,4)*{} **\crv{(4,-1) & (-4,1)}?(1)*\dir{>};
    (-4,4)*{};(4,12)*{} **\crv{(-4,7) & (4,9)}?(1)*\dir{>};
    (4,4)*{};(-4,12)*{} **\crv{(4,7) & (-4,9)};
  (8,8)*{\lambda};(-6,-3)*{\scs i};
     (6,-3)*{\scs j};
 \endxy} \right)
 \;\;=\;\; \Gamma\left(
\xy 0;/r.18pc/:
  (3,9);(3,-9) **\dir{-}?(.5)*\dir{<}+(2.3,0)*{};
  (-3,9);(-3,-9) **\dir{-}?(.55)*\dir{>}+(2.3,0)*{};
  (8,2)*{\lambda};(-5,-6)*{\scs i};     (5.1,-6)*{\scs j};
 \endxy \right)
\end{equation}
of graded bimodule maps hold in $\Gr$ for all weights $\lambda$.

\begin{proof}
This follows immediately from Lemma~\ref{lem_other_crossings}.
\end{proof}
\end{prop}

\begin{prop}
The equalities
\begin{eqnarray}
  \Gamma\left( \xy
  (0,0)*{\xybox{
    (-4,-4)*{};(4,4)*{} **\crv{(-4,-1) & (4,1)}?(1)*\dir{>}?(.75)*{\bullet};
    (4,-4)*{};(-4,4)*{} **\crv{(4,-1) & (-4,1)}?(1)*\dir{>};
    (-5,-3)*{\scs i};
     (5.1,-3)*{\scs j};
     (8,1)*{ \lambda};
     (-10,0)*{};(10,0)*{};
     }};
  \endxy\right)
 \;\; =
   \Gamma\left(
\xy
  (0,0)*{\xybox{
    (-4,-4)*{};(4,4)*{} **\crv{(-4,-1) & (4,1)}?(1)*\dir{>}?(.25)*{\bullet};
    (4,-4)*{};(-4,4)*{} **\crv{(4,-1) & (-4,1)}?(1)*\dir{>};
    (-5,-3)*{\scs i};
     (5.1,-3)*{\scs j};
     (8,1)*{ \lambda};
     (-10,0)*{};(10,0)*{};
     }};
  \endxy\right),
\qquad    \Gamma\left(\xy
  (0,0)*{\xybox{
    (-4,-4)*{};(4,4)*{} **\crv{(-4,-1) & (4,1)}?(1)*\dir{>};
    (4,-4)*{};(-4,4)*{} **\crv{(4,-1) & (-4,1)}?(1)*\dir{>}?(.75)*{\bullet};
    (-5,-3)*{\scs i};
     (5.1,-3)*{\scs j};
     (8,1)*{ \lambda};
     (-10,0)*{};(10,0)*{};
     }};
  \endxy\right)
\;\;  =
  \Gamma\left(\xy
  (0,0)*{\xybox{
    (-4,-4)*{};(4,4)*{} **\crv{(-4,-1) & (4,1)}?(1)*\dir{>} ;
    (4,-4)*{};(-4,4)*{} **\crv{(4,-1) & (-4,1)}?(1)*\dir{>}?(.25)*{\bullet};
    (-5,-3)*{\scs i};
     (5.1,-3)*{\scs j};
     (8,1)*{ \lambda};
     (-10,0)*{};(12,0)*{};
     }};
  \endxy \right)
\end{eqnarray}
hold in $\Gr$ for all weights $\lambda$.
\end{prop}

\begin{proof}
Using the definition of the bimodule map $\Gamma(\Ucross_{i,j,\lambda})$ in
\eqref{eq_gamma_dcross} the proposition is easily verified.
\end{proof}

\begin{prop} The equalities
\begin{equation} \label{eq_prop_gamma1}
    \Gamma\left(\vcenter{\xy 0;/r.18pc/:
    (-4,-4)*{};(4,4)*{} **\crv{(-4,-1) & (4,1)}?(1)*\dir{>};
    (4,-4)*{};(-4,4)*{} **\crv{(4,-1) & (-4,1)}?(1)*\dir{>};
    (-4,4)*{};(4,12)*{} **\crv{(-4,7) & (4,9)}?(1)*\dir{>};
    (4,4)*{};(-4,12)*{} **\crv{(4,7) & (-4,9)}?(1)*\dir{>};
  (8,8)*{\lambda};(-5,-3)*{\scs i};
     (5.1,-3)*{\scs j};
 \endxy} \right)
 \qquad = \qquad
 \left\{
 \begin{array}{ccc}
       \Gamma\left(\xy 0;/r.18pc/:
  (3,9);(3,-9) **\dir{-}?(.5)*\dir{<}+(2.3,0)*{};
  (-3,9);(-3,-9) **\dir{-}?(.5)*\dir{<}+(2.3,0)*{};
  (8,2)*{\lambda};(-5,-6)*{\scs i};     (5.1,-6)*{\scs j};
 \endxy\right) &  &  \text{if $i \cdot j=0$}\\ \\
 (i-j)\Gamma\left(\;\; \xy 0;/r.18pc/:
  (3,9);(3,-9) **\dir{-}?(.5)*\dir{<}+(2.3,0)*{};
  (-3,9);(-3,-9) **\dir{-}?(.5)*\dir{<}+(2.3,0)*{};
  (8,2)*{\lambda}; (-3,4)*{\bullet};
  (-5,-6)*{\scs i};     (5.1,-6)*{\scs j};
 \endxy \quad
 - \quad
 \xy 0;/r.18pc/:
  (3,9);(3,-9) **\dir{-}?(.5)*\dir{<}+(2.3,0)*{};
  (-3,9);(-3,-9) **\dir{-}?(.5)*\dir{<}+(2.3,0)*{};
  (8,2)*{\lambda}; (3,4)*{\bullet};
  (-5,-6)*{\scs i};     (5.1,-6)*{\scs j};
 \endxy\;\;\right)
   &  & \text{if $i \cdot j =-1$}
 \end{array}
 \right.
\end{equation}
\begin{equation} \label{eq_prop_gamma2}
\Gamma \left( \vcenter{
 \xy 0;/r.18pc/:
    (-4,-4)*{};(4,4)*{} **\crv{(-4,-1) & (4,1)}?(1)*\dir{>};
    (4,-4)*{};(-4,4)*{} **\crv{(4,-1) & (-4,1)}?(1)*\dir{>};
    (4,4)*{};(12,12)*{} **\crv{(4,7) & (12,9)}?(1)*\dir{>};
    (12,4)*{};(4,12)*{} **\crv{(12,7) & (4,9)}?(1)*\dir{>};
    (-4,12)*{};(4,20)*{} **\crv{(-4,15) & (4,17)}?(1)*\dir{>};
    (4,12)*{};(-4,20)*{} **\crv{(4,15) & (-4,17)}?(1)*\dir{>};
    (-4,4)*{}; (-4,12) **\dir{-};
    (12,-4)*{}; (12,4) **\dir{-};
    (12,12)*{}; (12,20) **\dir{-};
  (18,8)*{\lambda};
  (-6,-3)*{\scs i};
  (6,-3)*{\scs j};
  (15,-3)*{\scs k};
\endxy} \right)
 \;\; =\;\;
\Gamma\left( \vcenter{
 \xy 0;/r.18pc/:
    (4,-4)*{};(-4,4)*{} **\crv{(4,-1) & (-4,1)}?(1)*\dir{>};
    (-4,-4)*{};(4,4)*{} **\crv{(-4,-1) & (4,1)}?(1)*\dir{>};
    (-4,4)*{};(-12,12)*{} **\crv{(-4,7) & (-12,9)}?(1)*\dir{>};
    (-12,4)*{};(-4,12)*{} **\crv{(-12,7) & (-4,9)}?(1)*\dir{>};
    (4,12)*{};(-4,20)*{} **\crv{(4,15) & (-4,17)}?(1)*\dir{>};
    (-4,12)*{};(4,20)*{} **\crv{(-4,15) & (4,17)}?(1)*\dir{>};
    (4,4)*{}; (4,12) **\dir{-};
    (-12,-4)*{}; (-12,4) **\dir{-};
    (-12,12)*{}; (-12,20) **\dir{-};
  (10,8)*{\lambda};
  (7,-3)*{\scs k};
  (-6,-3)*{\scs j};
  (-14,-3)*{\scs i};
\endxy} \right) \qquad \text{unless $i=k$ and $i \cdot j \neq 0$}
\end{equation}
\begin{equation} \label{eq_prop_gamma3}
 \Gamma\left(\xy 0;/r.18pc/:
  (4,12);(4,-12) **\dir{-}?(.5)*\dir{<}+(2.3,0)*{};
  (-4,12);(-4,-12) **\dir{-}?(.5)*\dir{<}+(2.3,0)*{};
  (12,12);(12,-12) **\dir{-}?(.5)*\dir{<}+(2.3,0)*{};
  (17,2)*{\lambda}; (-6,-9)*{\scs i};     (6.1,-9)*{\scs j};
  (14,-9)*{\scs i};
 \endxy\right)
 \;\; =\;\;
(i-j) \Gamma \left( \vcenter{
 \xy 0;/r.18pc/:
    (-4,-4)*{};(4,4)*{} **\crv{(-4,-1) & (4,1)}?(1)*\dir{>};
    (4,-4)*{};(-4,4)*{} **\crv{(4,-1) & (-4,1)}?(1)*\dir{>};
    (4,4)*{};(12,12)*{} **\crv{(4,7) & (12,9)}?(1)*\dir{>};
    (12,4)*{};(4,12)*{} **\crv{(12,7) & (4,9)}?(1)*\dir{>};
    (-4,12)*{};(4,20)*{} **\crv{(-4,15) & (4,17)}?(1)*\dir{>};
    (4,12)*{};(-4,20)*{} **\crv{(4,15) & (-4,17)}?(1)*\dir{>};
    (-4,4)*{}; (-4,12) **\dir{-};
    (12,-4)*{}; (12,4) **\dir{-};
    (12,12)*{}; (12,20) **\dir{-};
  (18,8)*{\lambda};
  (-6,-3)*{\scs i};
  (6,-3)*{\scs j};
  (14,-3)*{\scs i};
\endxy}
\quad - \quad
 \vcenter{
 \xy 0;/r.18pc/:
    (4,-4)*{};(-4,4)*{} **\crv{(4,-1) & (-4,1)}?(1)*\dir{>};
    (-4,-4)*{};(4,4)*{} **\crv{(-4,-1) & (4,1)}?(1)*\dir{>};
    (-4,4)*{};(-12,12)*{} **\crv{(-4,7) & (-12,9)}?(1)*\dir{>};
    (-12,4)*{};(-4,12)*{} **\crv{(-12,7) & (-4,9)}?(1)*\dir{>};
    (4,12)*{};(-4,20)*{} **\crv{(4,15) & (-4,17)}?(1)*\dir{>};
    (-4,12)*{};(4,20)*{} **\crv{(-4,15) & (4,17)}?(1)*\dir{>};
    (4,4)*{}; (4,12) **\dir{-};
    (-12,-4)*{}; (-12,4) **\dir{-};
    (-12,12)*{}; (-12,20) **\dir{-};
  (10,8)*{\lambda};
  (6,-3)*{\scs i};
  (-6,-3)*{\scs j};
  (-14,-3)*{\scs i};
\endxy}
\right) \qquad \text{if $i \cdot j =-1$}
\end{equation}
of bimodule maps hold in $\Gr$ for all weights $\lambda$.
\end{prop}

\begin{proof}
To prove the proposition we compute the bimodule maps in \eqref{eq_prop_gamma1}
on the elements of the form $\xi_i^{\alpha_1} \otimes \xi_j^{\alpha_2}$. Bimodule
maps in \eqref{eq_prop_gamma2} and \eqref{eq_prop_gamma3} are computed on
elements $\xi_i^{\alpha_1} \otimes \xi_j^{\alpha_2} \otimes \xi_k^{\alpha_3}$ and
$\xi_i^{\alpha_1} \otimes \xi_j^{\alpha_2} \otimes \xi_i^{\alpha_3}$,
respectively.  The action on other elements in the cohomology rings is determined
by the fact that the maps are bimodules morphisms.

The proof of these remaining relations is the same as the proof that rings
$R(\nu)$ act on $\Pol_{\nu}$ (notation as in \cite{KL}). Replacing the variables
$x_{k}(\ii) \in \Pol_{\ii}$ with the Chern classes of line bundles $\xi_k$ in the
corresponding cohomology rings turns formulas in \cite[Section 2.3]{KL} for the
action of dots and crossings into formulas \eqref{eq_gamma_updot} and
\eqref{eq_gamma_dcross}, with the signs taken into account.
\end{proof}

Thus, we have proven the following Theorem:

\begin{thm} \label{thm_2rep}
$\Gamma_N \maps \Ucats^* \to \Gr$ is a 2-functor and a 2-representation.
\end{thm}

%
\subsection{Equivariant representation} \label{subsec_equivariant}
%

%
\subsubsection{Reminders on equivariant cohomology}
%

The $GL(N)$-equivariant cohomology of a point \cite{Fulton2} is given by
\begin{eqnarray}
  H_{GL(N)}^*(pt) = H^*\big(Gr(N, \infty)\big) = \Bbbk[x_1,x_2,\ldots,x_N,y_1,y_2,
  \ldots ]/I_{N,\infty}
\end{eqnarray}
where $I_{N,\infty}$ is the ideal generated by the homogeneous terms in
\[
 \left(
 1 + x_1 t + x_2 t^2 + \cdots x_N t^N
 \right)
 \left(
 1 + y_1 t + y_2 t^2 + \cdots y_j t^j + \cdots
 \right)
 =1.
\]
Thus, $H_{GL(N)}^*(pt)$ is isomorphic to the polynomial ring
\begin{eqnarray} \label{eq_EQpoint}
   H_{GL(N)}^*(pt) \cong \Bbbk[x_1, x_2, \ldots, x_{N-1},x_N]
\end{eqnarray}
with $x_i$ in degree $2i$.

Given a sequence $\uk = (k_0, k_1,k_2, \dots, k_n)$ with $0 \leq k_1 \leq k_2
\leq \dots \leq k_n =N$, $GL(N)$ acts transitively on $Fl(\uk)$, so the
equivariant cohomology of $Fl(\uk)$ is
\begin{eqnarray}
  H^*_{GL(N)}(Fl(\uk)) = H^*_{{\rm Stab}(pt)}\big(pt \in Fl(\uk)\big)
\end{eqnarray}
where the stabilizer of a point $\left( 0 \subset \C^{k_1} \subset \dots \subset
\C^{k_n} =C^N\right)$ in $FL(\uk)$ is the group of invertible block $\big(k_1
\times (k_2-k_1) \times \dots \times (N-k_{n-1})\big)$ upper--triangular
matrices. This group is contractible onto its subgroup $GL(k_1) \times
GL(k_2-k_1) \times \dots \times GL(N-k_{n-1})$. Hence,
\begin{eqnarray}
 \HG_{\uk} & \cong & H^*_{GL(k_1) \times GL(k_2-k_1) \times \dots
\times GL(N-k_{n-1})}(pt) \nn \\
& \cong & \bigotimes_{j=1}^{n} H^*_{GL(k_j-k_{j-1})}(pt)  \nn \\
 & \cong & \bigotimes_{j=1}^{n} \Bbbk[x(\uk)_{j,1}, x(\uk)_{j,2}, \dots, x(\uk)_{j,k_{j}-k_{j-1}}]
\end{eqnarray}
with $\deg x(\uk)_{j,\alpha}=2\alpha$.  Thus, the equivariant cohomology of
$Fl(\uk)$ has the same generators as the ordinary cohomology ring, but we do not
mod out by the ideal $I_{\uk,N}$.

The equivariant cohomology rings $\HG_{\ukep}$ and $\HG_{\ukem}$ of $Fl(\ukep)$
and $Fl(\ukem)$ can be similarly computed. They have the same generators as the
ordinary cohomology rings $H_{\ukep}$ and $H_{\ukem}$, but no relations.

The equivariant cohomology $\HG_{\ukp}$ of $Fl(\ukp)$ also has the same
generators as the ordinary cohomology ring, with no relations.  Using the
forgetful maps $Fl(\uk) \leftarrow Fl(\ukp) \rightarrow Fl(\ukep)$ we get
inclusions
\begin{eqnarray}\label{eq_Ginclusion1}
 \HG_{\uk} & \xymatrix@1{\ar@{^{(}->}[r] & } & \HG_{\ukp} \\
 x(\uk)_{j,\alpha} & \mapsto & x(\ukp)_{j,\alpha} \qquad \text{for $j\neq i+1$}~,\nn \\
 x(\uk)_{i+1,\alpha} & \mapsto & \xi_i \cdot x(\ukp)_{i+1,\alpha-1}+x(\ukp)_{i+1,\alpha}~,
 \nn
\end{eqnarray}
and
\begin{eqnarray} \label{eq_Ginclusion2}
 \HG_{\ukep} & \xymatrix@1{\ar@{^{(}->}[r] & } & \HG_{\ukp} \\
 x(\ukep)_{j,\alpha} & \mapsto & x(\ukp)_{j,\alpha} \qquad \text{for $j\neq i$}~, \nn \\
 x(\ukep)_{i,\alpha} & \mapsto & \xi_i \cdot x(\ukp)_{i,\alpha-1}+x(\ukp)_{i,\alpha}~.
 \nn
\end{eqnarray}
making $\HG_{\ukp}$ a graded $(\HG_{\ukep},H_{\uk})$-bimodule, just as in the
non-equivariant case. Using these inclusions we introduce canonical generators of
$\HG_{\ukp}$ given by identifying certain generators of $\HG_{\uk}$ and
$\HG_{\ukep}$ with their images in $\HG_{\ukp}$. Thus, we can identify
$x(\uk)_{j,\alpha}$ and $x(\ukep)_{j,\alpha}$ in $\HG_{\ukp}$ when $j\neq i,
i+1$, and
\begin{eqnarray} \nn
 \HG_{\ukp} = \bigotimes_{j\neq i+1}^{} \Bbbk\left[ x(\uk)_{j,1} \dots x(\uk)_{j,k_{j}-k_{j-1}}  \right]
 \otimes \Bbbk[x(\ukep)_{i+1,1}, \dots, x(\ukep)_{i+1,k_{i+1}-k_{i}-1}] \otimes \Bbbk[\xi_i]
\end{eqnarray}
or equivalently,
\begin{eqnarray} \nn
 \HG_{\ukp} = \bigotimes_{j\neq i}^{} \Bbbk\left[ x(\ukep)_{j,1} \dots x(\ukep)_{j,k_{j}-k_{j-1}}  \right]
 \otimes \Bbbk[x(\uk)_{i,1}, \dots, x(\uk)_{i,k_{i}-k_{i-1}}] \otimes \Bbbk[\xi_i].
\end{eqnarray}

The generators of $\HG_{\uk}$ and $\HG_{\ukep}$ that are not mapped to canonical
generators in $\HG_{\ukp}$ can be expressed in terms of canonical generators as
follows:
\begin{eqnarray}
  x(\uk)_{i+1,\alpha} & =&  \xi_i \cdot x(\ukep)_{i+1,\alpha-1} + x(\ukep)_{i+1,\alpha} \nn\\
  x(\ukep)_{i,\alpha} & = & \xi_i \cdot x(\uk)_{i,\alpha-1}  +x(\uk)_{i,\alpha} \label{eq_Gnoncan}
\end{eqnarray}
for all values of $\alpha$.

%
\subsubsection{The 2-category $\cat{EqFlag}_N$}
%

The 2-category $\cat{EqFlag}_N$ is the equivariant analog of $\cat{Flag}_N$.

\begin{defn}
The additive $\Bbbk$-linear 2-category $\cat{EqFlag}_N$ is the idempotent
completion inside of $\cat{Bim}$ of the 2-category consisting of
\begin{itemize}
  \item objects: the graded rings $\HG_{\uk}$ for all $\uk = (k_0, k_1,k_2, \dots, k_n)$ with
$0 \leq k_1 \leq k_2 \dots \leq k_n = N$.

  \item morphisms: generated by the graded ($\HG_{\uk}$,$\HG_{\uk}$)-bimodule $\HG_{\uk}$, the
graded ($\HG_{\ukep}$,$\HG_{\uk}$)-bimodules $\HG_{\ukp}$ and the graded
($\HG_{\uk}$,$\HG_{\ukep}$)-bimodule $\HG_{\ukp}$ for all $i \in I$, together
with their shifts $\HG_{\uk}\{t\}$, $\HG_{\ukp}\{t\}$, and $\HG_{\ukp}\{t\}$ for
$t \in \Z$. The bimodules $\HG_{\uk}=\HG_{\uk}\{0\}$ are the identity
1-morphisms. Thus, a morphism from $\HG_{\uk}$ to $\HG_{{}_{\ii}\uk}$ is a finite
direct sum of graded $(\HG_{{}_{\jj}\uk},\HG_{\uk})$-bimodules of the form
\[
  \HG_{\uk^{\jj}} \{t\}:=
   \HG_{{}_{s_2s_3 \cdots s_m}\uk^{s_1}}
    \otimes_{\HG_{{}_{s_2s_3 \cdots s_m}\uk}}
\cdots \otimes_{\HG_{{}_{s_{m-1}s_m}\uk}} \HG_{{}_{s_m}\uk^{s_{m-1}}}
\otimes_{\HG_{{}_{s_m}\uk}} \HG_{\uk^{{s_m}}} \{t\}
\]
 for signed sequences $\jj=s_1s_2\dots s_m$ with $\ii_X = \jj_X \in X$.
  \item 2-morphisms: degree-preserving bimodule maps.
\end{itemize}
\end{defn}
There is a graded additive subcategory $\GrG$ of $\cat{Bim}^*$ with the same
objects and 1-morphisms as $\cat{EqFlag}_{N}$, and
\begin{equation}
  \GrG(M_1,M_2) := \bigoplus_{t\in \Z} \cat{EqFlag}_{N}(M_1\{t\},M_2).
\end{equation}

%
\subsubsection{Equivariant representation $\Gamma^G_N$}
%

A 2-representation of the 2-category $\Ucatq(\mf{sl}_2)=\Ucats^*(\mf{sl}_2)$ was
constructed in \cite{Lau2} using equivariant cohomology of partial flag
varieties.  Here we extend that construction to the $\mathfrak{sl}_n$-case and
define a 2-representation $\Gamma_N^{G} \maps \Ucats^* \to \GrG$.

The verification that $\Gamma_N$ is 2-representation with the assignments given
in Section~\ref{subsec_define_gamma} used only the relations
\eqref{eq_graph_def_relations} and
\eqref{eq_graph_easy_slide}--\eqref{eq_dark_slide2}, together with
Lemma~\ref{lem_relations} and Corollary~\ref{cor_bimodule_maps} which both follow
from these relations.

To define the equivariant 2-representation $\Gamma_N^G$ ordinary cohomology rings
are replaced by equivariant cohomology rings and bimodule maps associated to
2-morphisms are defined the same way, except that the dual generators
$\overline{x(\uk)_{j,\alpha}}$ must be redefined.   Set
$\overline{x(\uk)_{j,0}}=1$ and inductively define
\begin{equation}
\label{eq_new_overline_a} \overline{x(\uk)_{j,\alpha}} = -\sum_{f=1}^{\alpha}
x(\uk)_{j,f} \overline{x(\uk)_{j,\alpha-f}}.
\end{equation}
For example,
\begin{eqnarray}
\overline{x(\uk)_{j,1}} &=& - x(\uk)_{j,1} \nn \\
\overline{x(\uk)_{j,2}} &=& - x(\uk)_{j,1} \overline{x(\uk)_{j,1}} - x(\uk)_{j,2}
= x(\uk)_{j,1}^2- x(\uk)_{j,2} \nn \\
 \overline{x(\uk)_{j,3}} &=&
 -x(\uk)_{j,1}^3+2 \;x(\uk)_{j,1}\;x(\uk)_{j,2}-x(\uk)_{j,3} \nn
\end{eqnarray}
In the nonequivariant cohomology rings the two definitions \eqref{eq_overlinea}
and \eqref{eq_new_overline_a} of dual generators agree, but not in the
equivariant cohomology ring.

\begin{rem}
We could have started with the definition of $\overline{x(\uk)_{j,\alpha}} $
given in \eqref{eq_new_overline_a} and used this definition to construct the
non-equivariant representation in the previous section.  After all, the two
definitions \eqref{eq_overlinea} and \eqref{eq_new_overline_a} are equivalent in
the non-equivariant cohomology ring.  However, the recursive definition
\eqref{eq_new_overline_a} makes it more cumbersome to calculate with in practice
and that is why we used \eqref{eq_overlinea} to construct the non-equivariant
representation.
\end{rem}

\begin{lem} \label{lem_equivariant_rel}
With $\overline{x(\uk)_{j,\alpha}}$ redefined as in \eqref{eq_new_overline_a} the
relations \eqref{eq_graph_def_relations} and
\eqref{eq_graph_easy_slide}--\eqref{eq_dark_slide2} hold in the equivariant
cohomology rings:
\begin{eqnarray}
\sum_{f=0}^{\alpha} x(\uk)_{j,f} \overline{x(\uk)_{j,\alpha-f}} &=& \delta_{\alpha,0} \label{eq_G_1}\\
 x(\uk)_{j,\alpha} &=& x(\ukep)_{j,\alpha} \qquad \text{for $j \neq i, i+1$} \label{eq_G_2}\\
 x(\uk)_{i+1,\alpha}  &=&  \xi_i \cdot x(\ukep)_{i+1,\alpha-1} + x(\ukep)_{i+1,\alpha} \label{eq_G_3}\\
  x(\ukep)_{i,\alpha}  &= & \xi_i \cdot x(\uk)_{i,\alpha-1}  +x(\uk)_{i,\alpha} \label{eq_G_4}\\
  x(\ukep)_{i+1,\alpha} &=& \sum_{f=0}^{\alpha} (-1)^{f} \xi_i^f
  x(\uk)_{i+1,\alpha-f} \label{eq_G_5}\\
  x(\uk)_{i,\alpha} &= &\sum_{f=0}^{\alpha} x(\ukep)_{i,\alpha-f} \xi_i^f \label{eq_G_6}\\
  (-1)^{\alpha}\xi_i^{\alpha} &=&  \sum_{f=0}^{\alpha} x(\ukep)_{i+1,\alpha-f} \;
    \overline{x(\uk)_{i+1,f}}
  =
  \sum_{g=0}^{\alpha} \overline{x(\ukep)_{i\alpha-g}} \;
  x(\uk)_{i,g}   \nn \\\label{eq_G_7}\\
  \overline{x(\uk)_{j, \alpha}} &=&
    \overline{x(\ukep)_{j, \alpha}} \qquad \text{if $j \neq i, i+1$}
    \label{eq_G_8}\\ \nn \\
   \overline{x(\uk)_{j, \alpha}} &=& \left\{
  \begin{array}{cl}
    \overline{x(\ukep)_{i, \alpha-1}}\; \xi_i + \overline{x(\ukep)_{i,\alpha}}   & \text{if $j =i$}\\ \label{eq_G_9}\\
     \xsum{f=0}{\alpha}(-1)^{\alpha} \overline{x(\ukep)_{i+1,\alpha-f}}\; \xi_i^f&  \text{if $j =i+1$}
  \end{array}
   \right.\\
   \overline{x(\ukep)_{j, \alpha}} &=& \left\{
  \begin{array}{cl}
   \xsum{f=0}{\alpha}(-1)^{\alpha} \overline{x(\uk)_{i,\alpha-f}} \;\xi_i^f   & \text{if $j =i$} \\ \label{eq_G_10}\\
    \overline{x(\uk)_{i+1, \alpha-1}}\; \xi_i + \overline{x(\uk)_{i+1,\alpha}}  &  \text{if $j =i+1$}
  \end{array}
   \right.
\end{eqnarray}
\end{lem}

\begin{proof}
Equation \eqref{eq_G_1} follows from the definition \eqref{eq_new_overline_a} of
the dual elements $\overline{x(\uk)_{j,\alpha}}$. Equations
\eqref{eq_G_2}--\eqref{eq_G_4} equate two images of noncanonical generators under
the inclusions \eqref{eq_Ginclusion1} and \eqref{eq_Ginclusion2}. Equations
\eqref{eq_G_5} and \eqref{eq_G_6} follow from \eqref{eq_G_3} and \eqref{eq_G_4}.
The first equality in \eqref{eq_G_7} is proven as follows:
\begin{eqnarray}
  (-1)^{\alpha} \sum_{f=0}^{\alpha} x(\ukep)_{i+1,\alpha-f} \;
  \overline{x(\uk)_{i+1,f}} &\refequal{\eqref{eq_G_5}}&
   \sum_{f=0}^{\alpha} \sum_{g=0}^{\alpha-f}(-1)^{\alpha+g} \xi_i^{g} x(\uk)_{i+1,\alpha-f-g} \;
  \overline{x(\uk)_{i+1,f}} \nn \\
  &=& \sum_{g=0}^{\alpha}(-1)^{\alpha+g} \xi_i^{g} \sum_{f=0}^{\alpha-g} x(\uk)_{i+1,(\alpha-g)-f} \;
  \overline{x(\uk)_{i+1,f}} \nn \\
  &\refequal{\eqref{eq_G_1}}&
  \sum_{g=0}^{\alpha}(-1)^{\alpha+g} \xi_i^{g} \delta_{\alpha-g,0} =
  \xi_i^{\alpha}.
\end{eqnarray}
The second equality above is just a re-indexing of the summation. The second
equation in \eqref{eq_G_7} is proven similarly.

Equation \eqref{eq_G_8} follows from \eqref{eq_G_2} and the definition of
$\overline{x(\uk)_{j,\alpha}}$.  Equations \eqref{eq_G_9} and \eqref{eq_G_10}
follow from \eqref{eq_G_3} and \eqref{eq_G_4} and the definition of
$\overline{x(\uk)_{j,\alpha}}$.
\end{proof}

\begin{thm}
$\Gamma_N^G \maps \Ucats^* \to \GrG$ is a 2-representation.
\end{thm}

\begin{proof}
It is clear that $\Gamma_N^G$ preserves the degree associated to 2-morphisms
since the 2-representation $\Gamma_N$ does.  The proof that $\Gamma_N^G$
preserves the relations in $\Ucats^*$ can be copied line by line from the proof
the $\Gamma_N$ preserves the relations. By Lemma~\ref{lem_equivariant_rel} all
identities used in the proof of Theorem~\ref{thm_2rep} hold in $\GrG$ with the
dual elements redefined according to \eqref{eq_new_overline_a}.
\end{proof}

\begin{thm}
The 2-representations $\Gamma_N$ and $\Gamma_N^G$ categorify the irreducible
representation $V_N$ of $\U(\mf{sl}_{n})$ with highest weight $(N,0,\dots,0)$.
\end{thm}

\begin{proof}
Idempotent bimodule maps split in $\cat{Bim}$, so by the universal property of
the Karoubi envelope the additive 2-functors $\Gamma_N \maps \Ucat \to
\cat{Flag}_N$ and $\Gamma_N^G \maps \Ucat \to \cat{EqFlag}_N$, obtained using the
isomorphism $\Sigma\maps \Ucatq \to \Ucats^*$ and restricting to degree
preserving 2-morphisms, extend to 2-representations of $\UcatD$.

The rings $H_{\uk}$ and $\HG_{\uk}$  are graded local rings so that every
projective module is free, and they both have (up to isomorphism and grading
shift) a unique graded indecomposable projective module.  The Grothendieck group
of the category $\bigoplus_{\uk}H_{\uk} {\rm -gmod}$ (respectively
$\bigoplus_{\uk}\HG_{\uk} {\rm -gmod}$) is a free $\Z[q,q^{-1}]$-module with
basis elements $[H_{\uk}]$ (respectively $[\HG_{\uk}]$) over all $\uk$, where
$q^i$ acts by shifting the grading degree by $i$. Thus, we have
\begin{eqnarray}
K_0\big(\bigoplus_{\uk}\HG_{\uk}{\rm -gmod}\big) \cong
K_0\big(\bigoplus_{\uk}H_{\uk}{\rm -gmod}\big) \cong  {}_{\cal{A}}V_N
\end{eqnarray}
as $\Z[q,q^{-1}]$-modules, where the sums are over all sequences $0 \leq k_1 \leq
\dots \leq k_n =N$, and ${}_{\cal{A}}V_N$ is a representation of
$\UA(\mf{sl}_{n})$, an integral form of the representation $V_N$ of
$\U(\mf{sl}_{n})$.

The bimodules $\Gamma(\onel)$, $\Gamma(\cal{E}_{+i}\onel)$ and
$\Gamma(\cal{E}_{-i}\onel)$ (or equivalently $\Gamma^G(\onel)$,
$\Gamma^G(\cal{E}_{+i}\onel)$ and $\Gamma^G(\cal{E}_{-i}\onel)$) induce functors
on the graded module categories given by tensoring with these bimodules. The
functors
\begin{eqnarray}
 \onel &:=& H_k \otimes_{H_{\uk}}- \maps H_{\uk} {\rm -gmod} \to H_{\uk}{\rm -gmod} \\
 \cal{E}_i\onel &:=& H_{\ukp} \otimes_{H_{\uk}} - \;\{1+k_{i-1}-k_{i+1}\} \maps
 H_{\uk}{\rm -gmod} \to H_{\ukep}{\rm -gmod} \\
 \cal{F}_i\mathbf{1}_{\lambda+i_X} &:=&  H_{\ukp} \otimes_{H_{\ukep}} -\; \{-k_i\} \maps
H_{\ukep} {\rm -gmod} \to H_{\uk} {\rm -gmod}
\end{eqnarray}
have both left and right adjoints and commute with the shift functor, so they
induce $\Z[q,q^{-1}]$-module maps on Grothendieck groups. Furthermore, the
2-functor $\Gamma_N$ respects the relations of $\UcatD$, so by
Propositions~\ref{serre-isoms}--\ref{pmij-isoms} these functors satisfy relations
lifting those of $\U$.
\end{proof}

%
\subsection{Nondegeneracy of $\Ucats(\mf{sl}_n)$}
%

\begin{lem} \label{lem_bubbles_same_orient_I}
The surjective graded $\Bbbk$-algebra homomorphism
\begin{equation}
\Pi_{\lambda} \to \Ucatq(\onel,\onel) \cong \Ucats^*(\onel,\onel)
\end{equation}
of Proposition~\ref{prop_bubbles_same_orient} is an isomorphism.
\end{lem}

\begin{proof}
Injectivity is established by showing that for each $M \in \N$ there exists some
large $N$ such that the images of bubble monomials in $\Ucats^*(\onel,\onel)$ in
variables of degree less than $M$ act by linearly independent operators on
$\HG_{\uk}$ for $\lambda=\lambda(\uk)$, under the 2-functor $\Gamma^G_N$.  By a
direct calculation
\begin{eqnarray} \label{eq_gamma_bubbles}
    \Gamma^G\left(\;\xy
  (4,8)*{\lambda};
  (0,-2)*{\cbub{\lambda_i-1+\alpha}{i}};
 \endxy \; \right) & \maps & 1 \to
 (-1)^{\alpha} \sum_{f=0}^{\min(\alpha,k_{i+1}-k_i)}
  \overline{x(\uk)_{i,\alpha-f}}\;x(\uk)_{i+1,f}, \nn \\ \\
     \Gamma^G\left(\;\xy
  (4,8)*{\lambda};
  (0,-2)*{\ccbub{-\lambda_i-1+\beta}{i}};
 \endxy \; \right) & \maps & 1 \to
 (-1)^{\beta} \sum_{g=0}^{\min(\beta,k_{i}-k_{i-1})}
  \overline{x(\uk)_{i+1,\alpha-g}}\;x(\uk)_{i,g}, \nn
\end{eqnarray}
as bimodule endomorphisms of $\HG_{\uk}$. After expanding the
$\overline{x(\uk)_{j,\alpha}}$ using \eqref{eq_new_overline_a} we have
\begin{eqnarray}
    \Gamma^G\left(\;\xy
  (4,8)*{\lambda};
  (0,-2)*{\cbub{\lambda_i-1+\alpha}{i}};
 \endxy \; \right) & \maps & 1 \to
 (-1)^{\alpha} (x(\uk)_{i+1,\alpha} - x(\uk)_{i,\alpha}) +
  \left(\txt{products of lower \\order terms} \right) ,\nn \\
      \Gamma^G\left(\;\xy
  (4,8)*{\lambda};
  (0,-2)*{\ccbub{-\lambda_i-1+\beta}{i}};
 \endxy \; \right) & \maps & 1 \to
 (-1)^{\alpha} (x(\uk)_{i,\beta} - x(\uk)_{i+1,\beta}) +
  \left(\txt{products of lower \\order terms} \right). \nn
\end{eqnarray}
Only one orientation for dotted bubbles labelled by vertex $i$ is allowed in a
bubble monomial in $\Pi_{\lambda}$, see \eqref{eq_bub_rules}. The image under
$\Gamma^G$ of a bubble monomial is composed of products of bimodule maps of the
above form where, in the equivariant cohomology ring, sums of  products of
element of the form $(x(\uk)_{i+1,\alpha} - x(\uk)_{i,\alpha})$, respectively
$(x(\uk)_{i,\beta} - x(\uk)_{i+1,\beta})$ are independent provided $\alpha,\beta<
k_{i}-k_{i-1}$ and $\alpha,\beta\leq k_{i+1}-k_{i}$ so that $(x(\uk)_{i+1,\alpha}
-x(\uk)_{i+1,\alpha})$, respectively $(x(\uk)_{i,\beta} - x(\uk)_{i+1,\beta})$,
is nonzero.  By taking $N$ large, we can ensure this condition is satisfied for
any fixed $M$. Hence, images of bubble monomials in $\Ucats^*(\onel,\onel)$ are
independent.
\end{proof}


\begin{lem}
There is an isomorphism of graded $\Bbbk$-algebras
\begin{eqnarray}
  \iota'_{\lambda} \maps R(\nu) \otimes
  \Pi_{\lambda}
 &\longrightarrow& \Ucats^*(\cal{E}_{\nu}\onel, \cal{E}_{\nu}\onel) \\
 D \otimes D_{\pi} &\mapsto& \iota_{\lambda}(D) \; . \; D_{\pi} , \nn
\end{eqnarray}
with $\iota_{\lambda}$ given by \eqref{eq_inclusion_iota} and $D_{\pi}$ a bubble
monomial in $\Pi_{\lambda}$.
\end{lem}

\begin{proof}
Surjectivity of $\iota'_{\lambda}$ follows from Lemma~\ref{lem_surjective} and
the isomorphism $\Sigma$ from Section~\ref{subsubsec_isom}. Injectivity of $
\iota'_{\lambda}$ is established by showing that for each $M \in \N$ there exists
some large $N$ such that degree $M$ elements $
  \iota_{\lambda}(D) \; . \; D_{\pi},
$, as $D$ and $D_{\pi}$ run over a basis of $R(\nu)$, respectively
$\Pi_{\lambda}$, act by linear independent operators under the 2-representation
$\Gamma^G_N$.  The 2-functor $\Gamma^G$ must map horizontal composites to
horizontal composites (given by tensor products) in $\GrG$, so that
\begin{equation}
  \Gamma^G\left(\iota_{\lambda}(D) \; . \;D_{\pi} \right) \;\; = \;\;
 \Gamma^G\big(\iota_{\lambda}(D) \big) \otimes_{\HG_{\uk}}\Gamma^G\left( D_{\pi} \right)
 \;\; := \;\; f_{D} \otimes_{\HG_{\uk}} g_{D_{\pi}}
\end{equation}
for bimodule maps $g_{D_{\pi}} \maps \HG_{\uk} \to \HG_{\uk}$ and $f_D \maps
\HG_{\uk^{\ii}} \to \HG_{\uk^{\jj}}$ for some $\ii,\jj \in \seq(\nu)$.

Let
\begin{equation}
  \Pol_{\nu}(\xi) := \bigoplus_{\ii \in \seq(\nu)} \Pol_{\ii}, \qquad
  \Pol_{\ii} = \Bbbk[\xi_{\ii_1},\xi_{\ii_2}, \dots, \xi_{\ii_m}], \qquad   m=|\nu| \nn.
\end{equation}
For large enough $N$, the bimodule $\Gamma^G(\cal{E}_{\nu}\onel)$ contains
$\Pol_{\nu}(\xi)$ as a subspace.  From the definition of $\iota_{\lambda}$ (see
\eqref{eq_inclusion_iota}), together with the definitions of the bimodule maps
associated to $R(\nu)$ generators (see Section~\ref{subsubsec_Rnu}), it is clear
that the action on $\Pol_{\nu}(\xi)$ given by the 2-functor $\Gamma^G$ coincides
with the action of $R(\nu)$ on $\Pol_{\nu}$ defined in \cite{KL}.  In particular,
bimodule maps $f_D$ corresponding to basis elements of $R(\nu)$ must act by
linear independent operators (see  \cite[proof of Theorem 2.5]{KL}) on the
subspace $\Pol_{\nu}(\xi)$ of $\Gamma^G(\cal{E}_{\nu}\onel)$.  Furthermore, from
the definitions of bimodule maps associated to $R(\nu)$ generators it is also
clear that $f_D$ fixes all other generators of $\Gamma^G(\cal{E}_{\nu}\onel)$.

By Lemma~\ref{lem_bubbles_same_orient_I}, for large $N$ the bimodule maps
$g_{D_{\pi}}$ act by linearly independent operators on $\HG_{\uk}$.  Write
\begin{equation}
\pi_{i,\alpha} \quad :=\quad \left\{
\begin{array}{ccl}
  \xy 0;/r.18pc/:
 (-12,0)*{\cbub{\la i, \lambda\ra-1+\alpha}{i}};
 (-8,8)*{\lambda};
 \endxy & \quad & \text{for $\la i, \lambda\ra\geq 0$,} \\
   \xy 0;/r.18pc/:
 (-12,0)*{\ccbub{-\la i, \lambda\ra-1+\alpha}{i}};
 (-8,8)*{\lambda};
 \endxy & \quad & \text{for $\la i, \lambda\ra< 0$.}
\end{array}\right.
\end{equation}
and let
$D_{\pi}=\pi_{\ell_1,\alpha_1}\pi_{\ell_2,\alpha_2}\cdots\pi_{\ell_r,\alpha_r}$.
Then $D_{\pi}$ acts on $\HG_{\uk^{\ii}} \cong \HG_{\uk^{\ii}} \otimes_{\HG_{\uk}}
\HG_{\uk}$ via $\Gamma^G(\Id_{\onel}\; . \;D_{\pi})=1 \otimes_{\HG_{\uk}}
g_{D_{\pi}}$.  From \eqref{eq_gamma_bubbles} we have
\begin{equation} \label{eq_one_otimes_gDpi}
  1 \otimes_{\HG_{\uk}} g_{D_{\pi}} \maps 1 \otimes 1 \mapsto \left(\sum_{f_1=0}^{\alpha_1}
  x(\uk)_{\ell_1+1,f_1} \overline{x(\uk)_{\ell_1,\alpha_1-f_1}}\right) \cdots \left(\sum_{f_r=0}^{\alpha_r}
  x(\uk)_{\ell_1+1,f_r} \overline{x(\uk)_{\ell_r,\alpha_r-f_r}}\right)
\end{equation}
in $\HG_{\uk^{\ii}}$. After expanding the $\overline{x(\uk)_{\ell,\alpha}}$ using
their definition \eqref{eq_new_overline_a}, we have an expression for the action
of $g_{D_{\pi}}$ on $\HG_{\uk^{\ii}}$ strictly in terms of variables
$x(\uk)_{\ell,\alpha}$.

The bimodule maps $f_D \otimes_{\HG_{\uk}} g_{D_{\pi}}$ are linearly independent
operators since the bimodule maps $f_D$ and $g_{D_{\pi}}$ are separately
independent and act on algebraically independent generators of $\HG_{\uk^{\ii}}$:
\[
\begin{pspicture}[.5](6,1.5)
  \rput(.8,0){\Elinedot{\;\;i_1}{\;\alpha_1}}
  \rput(2,0){\Elinedot{\;\;i_2}{\;\alpha_2}}
  \rput(3,.8){$\cdots$}
  \rput(4,0){\Elinedot{\;\;i_m}{\;\alpha_m}}
  \rput(5.5,1.3){$\lambda$}
\end{pspicture}
\quad \longmapsto\quad \sum_{\beta_1,\dots,\beta_m, \gamma_1, \dots,\gamma_k}
\begin{pspicture}[.5](8,1.5)
  \rput(.8,0){\Elinedot{\;\;j_1}{\;\beta_1}}
  \rput(2,0){\Elinedot{\;\;j_2}{\;\beta_2}}
  \rput(3,.8){$\cdots$}
  \rput(4,0){\Elinedot{\;\;j_m}{\;\beta_m}}
  \rput(5.5,.8){\chern{\ell_1,\gamma_1}}
  \rput(8.5,.8){\chern{\ell_k,\gamma_k}}
  \rput(7,1.5){$\lambda$}\rput(7,.8){$\cdots$}
\end{pspicture}
\]
where the sum over the $\beta_{a}$'s is determined by the action of $f_D$ on
$\xi_1^{\alpha_1} \otimes \cdots \otimes \xi_{m}^{\alpha_m}$, and the sum over
the $x(\uk)_{\ell_b,\gamma_b}$, represented by labelled boxes on the far right,
is determined from \eqref{eq_one_otimes_gDpi} by expanding the
$\overline{x(\uk)}_{j,\alpha}$.   In particular, $f_D$ fixes variables
$x(\uk)_{\ell,\alpha}$ represented by labelled boxes on the far right and $g_D$
acts only on such boxes.
\end{proof}

\bigskip

This concludes the proof of Theorem~\ref{thm-nondegenerate} stated in the
introduction, so that
\begin{equation}
  \gamma \maps \UA(\mf{sl}_n) \longrightarrow K_0(\UcatD(\mf{sl}_n))
\end{equation}
is an isomorphism.   Proposition~\ref{prop_tilde_lifts} allows us to view
2-functors $\tilde{\psi}$, $\tilde{\omega}$, $\tilde{\sigma}$, and $\tilde{\tau}$
as categorifications of symmetries $\psi$, $\omega$, $\sigma$, and $\tau$ of
$\UA(\mf{sl}_n)$.  The graded homs between 1-morphisms in $\UcatD(\mf{sl}_n)$
categorify the semilinear form $\sla,\sra_{\pi}$ given by
\eqref{eq_semilinear_form_pi}.

\vspace{0.1in}

\paragraph{Simple nondegeneracy observations.}

We can assume that $\Bbbk$ is only a commutative ring, work over this ring from
the start, and say that the calculus is nondegenerate over $\Bbbk$ if $\Ucatq(
\mc{E}_{\ii}{\mathbf 1}_{\lambda}, \mc{E}_{\jj} {\mathbf 1}_{\lambda})$ is a free
$\Bbbk$-module with a basis $B_{\ii,\jj,\lambda}$ for all $\jj,\ii$ and
$\lambda$. We don't know any examples of a root datum and ring $\Bbbk$ when the
calculus is degenerate. Over a field the nondegeneracy of the calculus depends
only on the root datum and the characteristic of $\Bbbk$. If the calculus is
nondegenerate over $\Q$, it is nondegenerate over $\Z$ and over any commutative
ring $\Bbbk$.

\vspace{0.1in}

\paragraph{Multigrading.}

The multigrading introduced at the end of \cite{KL2} on rings $R(\nu)$ extends to
a multigrading on each hom space $\Ucatq(\cal{E}_{\ii}\onel,
\cal{E}_{\jj}\onel)$.  The Karoubian envelope of the corresponding multigraded
2-category should categorify the multi-parameter deformation of $\UA$.

\addcontentsline{toc}{section}{References}


\bibliographystyle{hplain}

%

\vspace{0.1in}

\noindent M.K.: { \sl \small Department of Mathematics, Columbia University, New
York, NY 10027} \newline \noindent {\tt \small email: khovanov@math.columbia.edu}

\vspace{0.1in}

\noindent A.L:  { \sl \small Department of Mathematics, Columbia University, New
York, NY 10027} \newline \noindent
  {\tt \small email: lauda@math.columbia.edu}

%
\end{document}